\documentclass{amsart}
\usepackage[foot]{amsaddr}
\usepackage{float, graphicx}
\usepackage[]{epsfig}
\usepackage{amsmath, amsthm, amssymb,}
\usepackage{epsfig}
\usepackage{verbatim}
\usepackage{multicol}
\usepackage{url}
\usepackage{latexsym}
\usepackage{mathrsfs}
\usepackage[colorlinks, bookmarks=true]{hyperref}
\usepackage{graphicx}
\usepackage{enumerate}
\usepackage[normalem]{ulem}
\usepackage{bm}
\usepackage{stmaryrd}
\usepackage{enumitem}
\usepackage{mathtools}

\usepackage{color}
\usepackage{xcolor}
\usepackage{geometry}


\numberwithin{equation}{section}

\newcommand\cA{{\mathcal A}}
\newcommand\cB{{\mathcal B}}
\newcommand{\cC}{{\mathcal C}}
\newcommand{\cD}{{\mathcal D}}
\newcommand{\cE}{{\mathcal E}}

\newcommand{\cG}{{\mathcal G}}

\newcommand{\cL}{{\mathcal L}}

\newcommand\cW{{\mathcal W}}

\newcommand{ \sfx }{{\mathsf x}}

\newcommand{ \sft }{{\mathsf t}}

\newcommand{ \sfh}{{\mathsf h}}

\newcommand{\sfD}{{\mathsf D}}

\newcommand{\sfH}{{\mathsf H}}

\newcommand{\fa}{{\mathfrak a}}
\newcommand{\fm}{{\mathfrak m}}
\newcommand{\fA}{{\mathfrak A}}
\newcommand{\fU}{{\mathfrak U}}
\newcommand{\fb}{{\mathfrak b}}
\newcommand{\fc}{{\mathfrak c}}
\newcommand{\fC}{{\mathfrak C}}
\newcommand{\fL}{{\mathfrak L}}
\newcommand{\fF}{{\mathfrak F}}
\newcommand{\fd}{{\mathfrak d}}

\newcommand{\fr}{{\mathfrak r}}

\newcommand{\ft}{{\mathfrak t}}

\newcommand{\fs}{{\mathfrak s}}

\newcommand{\fM}{{\mathfrak M}}

\newcommand{\fD}{{\mathfrak D}}
\newcommand{\fP}{{\mathfrak P}}
\newcommand{\fR}{{\mathfrak R}}

\newcommand{\bms}{\bm{s}}
\newcommand{\bme}{{\bm{e}}}

\newcommand{\bmv}{{\bm{v}}}

\newcommand{\bmx}{{\bm{x}}}

\newcommand{\rd}{{\rm d}}

\newcommand{\ri}{\mathrm{i}}

\newcommand{\bC}{{\mathbb C}}
\newcommand{\bD}{{\mathbb D}}

\newcommand{\bE}{\mathbb{E}}
\newcommand{\bH}{\mathbb{H}}

\newcommand{\bP}{\mathbb{P}}
\newcommand{\bQ}{\mathbb{Q}}
\newcommand{\bR}{{\mathbb R}}

\newcommand{\bZ}{\mathbb{Z}}

\newcommand{\al}{\alpha}

\newcommand{\la}{\lambda}

\DeclareMathOperator{\supp}{supp}
\DeclareMathOperator{\dist}{dist}

\DeclareMathOperator{\OO}{\mathcal{O}}
\DeclareMathOperator{\oo}{o}
\DeclareMathOperator{\argmax}{argmax}

\DeclareMathOperator{\Res}{Res}

\renewcommand{\Re}{\mathop{\mathrm{Re}}}
\renewcommand{\Im}{\mathop{\mathrm{Im}}}

\newcommand{\deq}{\mathrel{\mathop:}=} 
 
\renewcommand{\leq}{\leqslant}
\renewcommand{\geq}{\geqslant}

\newcommand{\ceil}[1]  {\lceil  {#1} \rceil}

\newcommand{\del}{\partial}

\newcommand{\qq}[1]{[\![{#1}]\!]}

\newcommand{\beq}{\begin{equation}}
\newcommand{\eeq}{\end{equation}}

\newcommand{\Adm}{\mathrm{Adm}}
\usepackage{tikz}
\usetikzlibrary{arrows}
\usepackage{cleveref}

\DeclareMathOperator{\Real}{Re}		
\DeclareMathOperator{\Imaginary}{Im}

\DeclareMathOperator{\north}{no}
\DeclareMathOperator{\so}{so}
\DeclareMathOperator{\ea}{ea}
\DeclareMathOperator{\we}{we}

\newtheorem{thm}{Theorem}[section]
\newtheorem{prop}[thm]{Proposition}
\newtheorem{lem}[thm]{Lemma}

\theoremstyle{remark}
\newtheorem{rem}[thm]{Remark}
\theoremstyle{definition}
\newtheorem{definition}[thm]{Definition}
\newtheorem{assumption}[thm]{Assumption}

\title{Edge Statistics for Lozenge Tilings of Polygons, I:\\
 Concentration of Height Function on Strip Domains}

    \author{Jiaoyang Huang}
    \address{University of Philadelphia, PA}
    \email{huangjy@wharton.upenn.edu}
    
\begin{document}

\begin{abstract}
		
		In this paper we study uniformly random lozenge tilings of strip domains. Under the assumption that the limiting arctic boundary has at most one cusp, we prove a nearly optimal concentration estimate for the tiling height functions and arctic boundaries on such domains: with overwhelming probability the tiling height function is within $n^\delta$ of its limit shape, and the tiling arctic boundary is within $n^{1/3+\delta}$ to its limit shape, for arbitrarily small $\delta>0$. This concentration result will be used in \cite{AH2} to prove that the edge statistics of  simply-connected polygonal domains, subject to a technical assumption on their limit shape, converge to the Airy line ensemble.
		
	\end{abstract}

\maketitle
{
  \hypersetup{linkcolor=black}
  \tableofcontents
}

	\section{Introduction}

	\label{Introduction}

	A central feature of random lozenge tilings is that they exhibit boundary-induced phase transitions. Depending on the shape of the domain, they can admit \emph{frozen} regions, where the associated height function is flat almost deterministically, and \emph{liquid} regions, where the height function appears more rough and random. We refer to the papers \cite{TSP} and \cite{ECM} for some of the earlier analyses of this phenomenon in lozenge tilings of hexagonal domains, and to the book \cite{RT} for a comprehensive review. 
	
	The interface or edge,  is the boundary between frozen and liquid regions, and in the limit it converges to a non-random curve, the \emph{arctic curve}.  A thorough study of arctic curves on arbitrary polygons was pursued in \cite{LSCE, DMCS}, where it was shown that their limiting trajectories are algebraic curves. Before the limit, it was predicted that the fluctuations of tiling boundary curves are of order $n^{1/3}$ and $n^{2/3}$ in the directions transverse and parallel to their limiting trajectories, respectively. And after taking appropriate scaling limit, they converge to the \emph{Airy$_2$ process}, a universal scaling limit introduced in \cite{SIDP}. We refer to \cite{EFLS} for a detailed survey. 
	
	Following the initial works \cite{SFRM,NPRTRM,ACP}, this prediction has been established for various tiling models on specific domains. For example, we refer to \cite{CFPALG,RSPPP,SFFC} for certain $q$-weighted random plane partitions; \cite{DP} for lozenge tilings of hexagons; and \cite{ARS,UEFDIPS} for lozenge tilings of trapezoids (hexagons with cuts along a single side). These results are all based on exact and analyzable formulas, specific to the domain of study, for the correlation kernel for which the tiling forms a determinantal point process. Although for lozenge tilings of arbitrary polygonal domains such explicit formulas are not known, it is believed the edge universality results, namely the convergence of the tiling boundary curves to the Airy$_2$ process, still hold; see Conjecture 18.7 of the book \cite{RT}.

For the above edge universality results to hold, it is necessary that the tiling arctic boundary concentrates around its limit shape, with fluctuations bounded by $n^{1/3+\delta}$. One challenge to prove the  edge universality results is the lack of such concentration estimates for general domains. In this paper we study uniformly random lozenge tilings of strip domains. This is equivalent to a family of $m$ non-intersecting random Bernoulli walks conditioned to start and end at specified locations.  Under the assumption that the limiting arctic boundary has at most one cusp (see the left and middle of \Cref{walkscusp} for examples) and some other technical assumptions, we prove a concentration estimate for the tiling height functions and arctic boundaries on such domains. An informal formulation of this result is provided as follows; we refer to Theorem \ref{estimategamma} below for a more precise statement.
\begin{thm}\label{t:infoconcentration}
For uniformly random lozenge tilings of strip domains (equivalently, non-intersecting random Bernoulli bridges), with limiting arctic boundary containing at most one cusp, it holds that  for any $\delta>0$
the tiling height function is within $n^\delta$ of its limit shape, and the tiling arctic boundary is within $n^{1/3+\delta}$ of its limit shape with overwhelming probability.
\end{thm}

The above concentration result is optimal up to the $n^\delta$ factor, and it will be used in Part II of this series \cite{AH2} to prove the edge universality for random lozenge tilings: the edge statistics of simply-connected polygonal domains, subject to a technical assumption on their limit shape (that we believe to hold generically), converge to the Airy line ensemble at any point (that is not a cusp or tangency location) around the arctic boundary. Recently, universality results for lozenge tilings at other points inside the domain  have been established. For example, it was shown that local statistics in the interior of the liquid region converge to the unique translation-invariant, ergodic Gibbs measure of the appropriate slope for hexagons \cite{DP,gorin2008nonintersecting,borodin2010q}, domains covered by finitely many trapezoids \cite{ARS,gorin2017bulk}, bounded perturbations of these \cite{LLTS}, and finally for general domains \cite{ULS}. It was also recently shown in \cite{ERT} that, at tangency locations between the arctic boundary and sides of general domains, the limiting statistics converge to the corners process of the Gaussian Unitary Ensemble. 
	
	The concentration results imply that particles of non-intersecting random Bernoulli walks conditioned to start and end at specified locations strongly concentrate around their classical locations. Such concentration phenomenon is ubiquitous in random matrix theory, and is known as \emph{eigenvalue rigidity}. In the context of Wigner matrices, it was first proven in \cite{erdHos2012rigidity}, and later for more general classes of random matrices \cite{erdHos2013local,knowles2013isotropic,alex2014isotropic,he2018isotropic,bao2017convergence, ajanki2019stability, bao2020spectral}.

	To prove Theorem \ref{t:infoconcentration}, for $m$ non-intersecting random Bernoulli walks conditioned to start and end at specified locations (equivalently, random lozenge tilings of a strip), we follow \cite{HFRTNRW} and approximate the random bridge model by a family of non-intersecting Bernoulli random walks with a space and time dependent drift, and derive a discrete stochastic differential equation for the empirical particle density. A dynamical version of loop equations has been developed in \cite{HFRTNRW}, to control the martingale term and errors in the stochastic differential equation on a \emph{macroscopic} scale. And a macroscopic central limit theorem has been proven in \cite{HFRTNRW} for random lozenge tilings of polygonal domains with exactly one horizontal upper boundary edge.

	In the past, loop equations have served as important tools to study fluctuations for interacting particle systems. They were introduced to the mathematical community in \cite{MR1487983} by Johansson  to derive macroscopic central limit theorems for continuous $\beta$-ensembles; see also \cite{borot-guionnet2, MR3010191,KrSh}. For discrete $\beta$-ensembles, \cite{ADE} proved macroscopic central limit theorems through making using of a family of discrete loop equations originating from \cite{Nekrasov, Nek_PS,Nek_Pes}, which could be analyzed similarly to the continuous ones. Loop equations have also played important roles in understanding rigidity phenomena and local statistics for $\beta$-ensembles \cite{MR3253704,MR3192527,MR3987722,MR3351052,bourgade2021optimal}.

To prove the concentration results on optimal scale, Theorem \ref{t:infoconcentration}, we need to analyze the discrete stochastic differential equation for the empirical particle density of non-intersecting Bernoulli random walks on any \emph{mesoscopic} scale. We analyze the dynamical loop equations by a multi-scale analysis, which gives nearly optimal estimates of the martingale term and errors in the discrete stochastic differential equation on any mesoscopic scale. With this as an input, we analyze the discrete stochastic differential equation by studying its behavior along the characteristics of the limiting complex Burgers equation, an idea which has previously been used in \cite{MR4009708, adhikari2020dyson,bourgade2018extreme,huang2020edge} to study Dyson Brownian motion. However, the fact that non-intersecting Bernoulli random walks behave differently close to cusps and tangency locations poses a new challenge. We need to construct the spectral domains for the characteristics and estimate the error terms of the stochastic differential equation  adapted to these singularities.   Through this framework, our results also give optimal estimates of extreme Bernoulli walks close to these singular locations for Bernoulli random walks with at most one cusp.

	\begin{figure}
		
		\begin{center}		
			
			\begin{tikzpicture}[
				>=stealth,
				auto,
				style={
					scale = .425
				}
				]
				
				\draw[black, very thick] (-4, 0) -- (5.5, 0); 
				\draw[black, very thick] (-4, 3) -- (5.5, 3);
				
				\draw[black, thick] (-3, 0) arc (180:120:3.464);
				\draw[black, thick] (3.928, 0) arc (0:60:3.464);
				
				\draw[black, thick] (-1, 0) arc (-60:0:1.732);
				\draw[black, thick] (.75, 0) arc (240:180:1.732);
				
				\draw[black, very thick] (8, 0) -- (17.5, 0); 
				\draw[black, very thick] (8, 3) -- (17.5, 3);
				
				\draw[black, thick] (9, 3) arc (180:240:3.464);
				\draw[black, thick] (15.928, 3) arc (0:-60:3.464);
				
				\draw[black, thick] (11, 3) arc (60:0:1.732);
				\draw[black, thick] (12.75, 3) arc (120:180:1.732);
				
				\draw[black, very thick] (20, 0) -- (29.5, 0); 
				\draw[black, very thick] (20, 3) -- (29.5, 3);
				
				\draw[black, thick] (21, 0) arc (180:120:3.464);
				\draw[black, thick] (27.928, 0) arc (0:60:3.464);
				
				\draw[black, thick] (24.5, 0) arc (-60:0:1.15);
				\draw[black, thick] (25.65, 0) arc (240:180:1.15);
				
				\draw[black, thick] (23.35, 3) arc (60:0:1.15);
				\draw[black, thick] (24.5, 3) arc (120:180:1.15);

			\end{tikzpicture}
			
		\end{center}
		
		\caption{\label{walkscusp} Shown to the left and middle are arctic boundaries exhibiting a single cusp. Shown to the right is an arctic boundary exhibiting two cusps that point in opposite directions.}
		
	\end{figure}
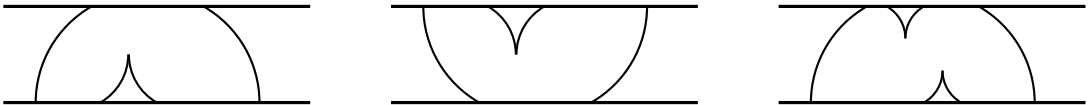 
		
	The remainder of this paper is organized as follows. In \Cref{Walks} we define the model, the associated variational problem, and state our main results. In Section \ref{s:CBE0} we collect some properties of complex slopes associated with tilings of double sided trapezoid domains with general boundary height function. 
	In Section \ref{s:nbb} we introduce a weighted version of non-intersecting Bernoulli bridges, and use it to study tilings on the double-sided trapezoid domains. We prove our main result Theorem \ref{estimategamma}, using concentration estimates  for this weighted non-intersecting Bernoulli bridge model. The concentration estimates  for weighted non-intersecting Bernoulli bridge model is proven in Section \ref{s:pcW} assuming \Cref{p:improve} below.
	Finally, \Cref{p:improve} is proven in Sections \ref{s:loopeq} and \ref{s:improveEE} by analyzing the dynamical loop equations.
	In Appendix \ref{s:thin}, we check that thin slices of polygonal domains provide examples of strips with limiting continuum height function satisfies our assumptions. In Appendices \ref{s:initialEs}, \ref{s:proof} and \ref{QEquation}, we give the proofs of some propositions from  Section \ref{s:CBE0}. In Appendix \ref{s:ansatz}, we sketch the derivation of transition probability of the weighted non-intersecting Bernoulli bridge model we used in Section \ref{s:nbb}.

\subsection*{Notations}	Throughout, we let $\overline{\mathbb{C}} = \mathbb{C} \cup \{ \infty \}$, $\mathbb{H}^+ = \big\{ z \in \mathbb{C} : \Imaginary z > 0 \big\}$, and $\mathbb{H}^- = \{ z \in \mathbb{C} : \Imaginary z < 0 \}$ denote the  compactified complex plane, upper complex plane, and lower complex plane, respectively. For any real numbers $a, b \in \mathbb{R}$, we set $a\vee b=\max\{a,b\}$ and $a\wedge b=\min\{a,b\}$. For any integer $n \geq 1$, we also denote $\bZ_n=\{\cdots, -2/n, -1/n, 0,1/n, 2/n,\cdots\}$ to be the set of integers, rescaled by $n$. We denote the set $\qq{1, r}=\{1,2,\cdots,r\}$. We write $X = \OO(Y )$ or $X\lesssim Y$  if  $|X| \leq \fC Y$;  $X\gtrsim Y$ if $X \geq  Y/\fC$; and $X\asymp Y$ if  $ Y/\fC \leq |X| \leq  \fC Y$ for some universal constant $\fC>0$.  We write $X = \oo(Y )$ or $X \ll Y$ if the ratio $|X|/Y$ tends to $0$, and $X = \oo(Y )$ or $X \gg Y$ if $Y/X$ tends to $0$. For two probability measures $\mu$ and $\nu$, we let $d(\mu,\nu)$ be any metric compatible with the weak topology.
We further denote by $|u - v|$ the Euclidean distance between any elements $u, v \in \mathbb{R}^2$. For any subset $\mathfrak{R} \subseteq \mathbb{R}^2$, we let $\partial \mathfrak{R}$ denote its boundary, $\overline{\mathfrak{R}}$ denote its closure.  For any real number $c \in \mathbb{R}$, we also define the rescaled set $c \mathfrak{R} = \{ cr : r \in \mathfrak{R} \}$, and for any $u \in \mathbb{R}^2$, we define the shifted set $\mathfrak{R} + u = \big\{ r + u : r \in \mathfrak{R} \big\}$. We say an event holds with overwhelming probability, if it holds with probability $1-n^{-D}$ for any constant $D > 1$ (assuming $n$ is sufficiently large).

	\subsection*{Acknowledgements}

 The work of Jiaoyang Huang was partially supported by the Simons Foundation as a Junior Fellow at the Simons Society of Fellows and NSF grant DMS-2054835. Jiaoyang Huang would like to thank Amol Aggarwal for fruitful discussions and careful reading of the manuscript, which has significantly improved the presentation. Jiaoyang Huang also wants to thank Vadim Gorin and Erik Duse for enlightening discussions, and the anonymous referees for helpful comments and suggestions.

	\section{Main Results} 
	
	\label{Walks}

	\subsection{Tilings and Height Functions}
	
	\label{FunctionWalks}
	
	We denote by $\mathbb{T}$ the \emph{triangular lattice}, namely, the graph whose vertex set is $\mathbb{Z}^2$ and whose edge set consists of edges connecting $(x, t), (x', t') \in \mathbb{Z}^2$ if $(x' - x, t' - t) \in \{ (1, 0), (0, 1), (1, 1)\}$. The axes of $\mathbb{T}$ are the lines $\{ x = 0 \}$, $\{ t = 0 \}$, and $\{ x  = t \}$, and the faces of $\mathbb{T}$ are triangles with vertices of the form $\big\{ (x, t), (x + 1, t), (x + 1, t + 1) \big\}$ or $\big\{ (x, t), (x, t + 1), (x + 1, t + 1) \big\}$. A \emph{domain} $\mathsf{R} \subseteq \mathbb{T}$ is a simply-connected induced subgraph of $\mathbb{T}$. With a slight abuse of notation, we also denote by $\mathsf{R} \subseteq \mathbb{R}^2$ the union of triangular faces with vertices in $\mathsf R$. When viewing $\mathsf{R}$ as a vertex set, the \emph{boundary} $\partial \mathsf{R} \subseteq \mathsf{R}$ is the set of vertices $\mathsf{v} \in \mathsf{R}$ adjacent to a vertex in $\mathbb{T} \setminus \mathsf{R}$; when viewing $\mathsf{R}$ as a union of triangular faces, $\partial \mathsf{R}$ is the union of its boundary edges. 

	A \emph{dimer covering} of a domain $\mathsf{R} \subseteq \mathbb{T}$ is defined to be a perfect matching on the dual graph of $\mathsf{R}$. A pair of adjacent triangular faces in any such matching forms a parallelogram, which we will also refer to as a \emph{lozenge} or \emph{tile}. Lozenges can be oriented in one of three ways; see the right side of \Cref{tilinghexagon} for all three orientations. We refer to the topmost lozenge there (that is, one with vertices of the form $\big\{ (x, t), (x, t + 1), (x + 1, t + 2), (x + 1, t + 1) \big\}$) as a \emph{type $1$} lozenge. Similarly, we refer to the middle (with vertices of the form $\big\{ (x, t), (x + 1, t), (x + 2, t + 1), (x + 1, t + 1) \big\}$) and bottom (vertices of the form $\big\{ (x, t), (x, t + 1), (x + 1, t + 1), (x + 1, t) \big\}$) ones there as \emph{type $2$} and \emph{type $3$} lozenges, respectively. A dimer covering of $\mathsf{R}$ can equivalently be interpreted as a tiling of $\mathsf{R}$ by lozenges of types $1$, $2$, and $3$. Therefore, we will also refer to a dimer covering of $\mathsf{R}$ as a \emph{(lozenge) tiling}. We call $\mathsf{R}$ \emph{tileable} if it admits a tiling.
	
	Associated with any tiling of $\mathsf{R}$ is a \emph{height function} $\mathsf{H}: \mathsf{R} \rightarrow \mathbb{Z}$, namely, a function on the vertices of $\mathsf{R}$ that satisfies 
	\begin{align*}
		\mathsf{H} (\mathsf{v}) - \mathsf{H} (\mathsf{u}) \in \{ 0, 1 \}, \quad \text{whenever $\mathsf{u} = (x, t)$ and $\mathsf{v} \in \big\{ (x + 1, t), (x, t - 1), (x + 1, t + 1) \big\}$},
	\end{align*}
	
	\noindent for some $(x, t) \in \mathbb{Z}^2$.

	We refer to the restriction $\mathsf{h} = \mathsf{H}|_{\partial \mathsf{R}}$ as a \emph{boundary height function}. For any boundary height function $\mathsf{h} : \partial \mathsf{R} \rightarrow \mathbb{Z}$, let $\mathscr{G} (\mathsf{h})$ denote the set of all height functions $\mathsf{h}: \mathsf{R} \rightarrow \mathbb{Z}$ with $\mathsf{H}|_{\partial \mathsf{R}} = \mathsf{h}$. In what follows, any height function on a domain $\mathsf{R} \subseteq \mathbb{T}$ will always implicitly be extended by linearity to the faces of $\mathsf{R}$, so that it may be viewed as a piecewise linear function on $\mathsf{R} \subseteq \mathbb{R}^2$.
	
	\begin{figure}
		
		\begin{center}		
			
			\begin{tikzpicture}[
				>=stealth,
				auto,
				style={
					scale = .5
				}
				]		
				
				\draw[-, black] (15, 9) node[above, scale = .7]{$\mathsf{H}$}-- (15, 7) node[right, scale = .7]{$\mathsf{H}$} -- (13, 5) node[below, scale = .7]{$\mathsf{H}$} -- (13, 7) node[left, scale = .7]{$\mathsf{H}$} -- (15, 9); 
				\draw[-, dashed, black] (13, 7) -- (15, 7);
				
				\draw[-, black] (14, 3) node[above, scale = .7]{$\mathsf{H}$} -- (16, 3) node[right, scale = .7]{$\mathsf{H} + 1$} -- (14, 1) node[below, scale = .7]{$\mathsf{H} + 1$} -- (12, 1) node[left, scale = .7]{$\mathsf{H}$}-- (14, 3);
				\draw[-, dashed, black] (14, 1) -- (14, 3);
				
				\draw[-, black] (13, -3) node[below, scale = .7]{$\mathsf{H}$} -- (13, -1) node[above, scale = .7]{$\mathsf{H}$} -- (15, -1) node[above, scale = .7]{$\mathsf{H} + 1$} -- (15, -3) node[below, scale = .7]{$\mathsf{H} + 1$} -- (13, -3);
				\draw[-, dashed, black] (13, -3) -- (15, -1);
				
				\draw[-, black] (-9, 1.5) node[below, scale = .7]{$\mathsf{H}$} -- (-7, 1.5) node[below, scale = .7]{$\mathsf{H} + 1$} -- (-5, 3.5) node[right, scale = .7]{$\mathsf{H} + 1$} -- (-5, 5.5) node[above, scale = .7]{$\mathsf{H} + 1$} -- (-7, 5.5) node[above, scale = .7]{$\mathsf{H}$} -- (-9, 3.5) node[left, scale = .7]{$\mathsf{H}$} -- (-9, 1.5);
				\draw[-, black] (-9, 3.5) -- (-7, 3.5)node[right, scale = .7]{$\mathsf{H}+1$} -- (-7, 1.5);
				\draw[-, black] (-7, 3.5) -- (-5, 5.5);
				
				\draw[-, black] (0, 0) -- (5, 0);
				\draw[-, black] (0, 0) -- (0, 4);
				\draw[-, black] (5, 0) -- (8, 3);
				\draw[-, black] (0, 4) -- (3, 7);
				\draw[-, black] (8, 3) -- (8, 7); 
				\draw[-, black] (8, 7) -- (3, 7);	
				
				\draw[-, black] (1, 0) -- (1, 3) -- (0, 3) -- (1, 4) -- (1, 5) -- (2, 5) -- (4, 7);
				\draw[-, black] (0, 2) -- (2, 2) -- (2, 0);
				\draw[-, black] (0, 1) -- (3, 1) -- (3, 0) -- (4, 1) -- (6, 1);
				\draw[-, black] (2, 6) -- (5, 6) -- (6, 7) -- (6, 6) -- (8, 6);
				\draw[-, black] (4, 6) -- (5, 7);
				\draw[-, black] (1, 4) -- (2, 4) -- (2, 5); 
				\draw[-, black] (2, 4) -- (3, 5) -- (3, 6);
				\draw[-, black] (1, 3) -- (2, 4);
				\draw[-, black] (2, 2) -- (3, 3) -- (3, 4) -- (2, 4) -- (2, 3) -- (1, 2);
				\draw[-, black] (2, 3) -- (3, 3) -- (3, 2) -- (4, 3) -- (4, 6) -- (5, 6) -- (5, 3) -- (4, 3);
				\draw[-, black] (3, 5) -- (5, 5) -- (6, 6);
				\draw[-, black] (7, 7) -- (7, 4) -- (8, 4) -- (7, 3) -- (7, 2) -- (6, 2) -- (4, 0);
				\draw[-, black] (8, 5) -- (7, 5) -- (6, 4) -- (6, 5) -- (7, 6);
				\draw[-, black] (3, 3) -- (4, 4) -- (6, 4);
				\draw[-, black] (3, 4) -- (4, 5);
				\draw[-, black] (6, 5) -- (5, 5);
				\draw[-, black] (7, 4) -- (6, 3) -- (7, 3);
				\draw[-, black] (2, 1) -- (3, 2) -- (5, 2) -- (6, 3);
				\draw[-, black] (3, 1) -- (5, 3) -- (6, 3) -- (6, 4); 
				\draw[-, black] (4, 1) -- (4, 2); 
				\draw[-, black] (5, 1) -- (5, 2); 
				\draw[-, black] (6, 2) -- (6, 3); 
				
				\draw[] (0, 0) circle [radius = 0] node[below, scale = .5]{$0$};
				\draw[] (1, 0) circle [radius = 0] node[below, scale = .5]{$1$};
				\draw[] (2, 0) circle [radius = 0] node[below, scale = .5]{$2$};
				\draw[] (3, 0) circle [radius = 0] node[below, scale = .5]{$3$};
				\draw[] (4, 0) circle [radius = 0] node[below, scale = .5]{$4$};
				\draw[] (5, 0) circle [radius = 0] node[below, scale = .5]{$5$};
				
				\draw[] (0, 1) circle [radius = 0] node[left, scale = .5]{$0$};
				\draw[] (1, 1) circle [radius = 0] node[above = 3, left = 0, scale = .5]{$1$};
				\draw[] (2, 1) circle [radius = 0] node[above = 3, left = 0, scale = .5]{$2$};
				\draw[] (3, 1) circle [radius = 0] node[left = 1, above = 0, scale = .5]{$3$};
				\draw[] (4, 1) circle [radius = 0] node[right = 1, below = 0, scale = .5]{$3$};
				\draw[] (5, 1) circle [radius = 0] node[right = 1, below = 0, scale = .5]{$4$};
				\draw[] (6, 1) circle [radius = 0] node[below, scale = .5]{$5$};
				
				\draw[] (0, 2) circle [radius = 0] node[left, scale = .5]{$0$};
				\draw[] (1, 2) circle [radius = 0] node[above = 3, left = 0, scale = .5]{$1$};
				\draw[] (2, 2) circle [radius = 0] node[left = 1, above = 0, scale = .5]{$2$};
				\draw[] (3, 2) circle [radius = 0] node[left = 3, above = 0, scale = .5]{$2$};
				\draw[] (4, 2) circle [radius = 0] node[below = 3, right = 0, scale = .5]{$3$};
				\draw[] (5, 2) circle [radius = 0] node[below = 1, right = 0, scale = .5]{$4$};
				\draw[] (6, 2) circle [radius = 0] node[right = 1, below = 0, scale = .5]{$4$};
				\draw[] (7, 2) circle [radius = 0] node[below, scale = .5]{$5$};
				
				\draw[] (0, 3) circle [radius = 0] node[left, scale = .5]{$0$};
				\draw[] (1, 3) circle [radius = 0] node[left = 1, above = 0, scale = .5]{$1$};
				\draw[] (2, 3) circle [radius = 0] node[above = 1, left = 0, scale = .5]{$1$};
				\draw[] (3, 3) circle [radius = 0] node[left = 3, above = 0, scale = .5]{$2$};
				\draw[] (4, 3) circle [radius = 0] node[right = 3, above = 0, scale = .5]{$2$};
				\draw[] (5, 3) circle [radius = 0] node[right = 3, above = 0, scale = .5]{$3$};
				\draw[] (6, 3) circle [radius = 0] node[below = 3, right = 0, scale = .5]{$4$};
				\draw[] (7, 3) circle [radius = 0] node[below = 1, right = 0, scale = .5]{$5$};
				\draw[] (8, 3) circle [radius = 0] node[right, scale = .5]{$5$};
				
				\draw[] (0, 4) circle [radius = 0] node[left = 1, scale = .5]{$0$};
				\draw[] (1, 4) circle [radius = 0] node[above = 1, left = 0, scale = .5]{$0$};
				\draw[] (2, 4) circle [radius = 0] node[left = 3, above = 0, scale = .5]{$1$};
				\draw[] (3, 4) circle [radius = 0] node[above, scale = .5]{$2$};
				\draw[] (4, 4) circle [radius = 0] node[above = 1, left = 0, scale = .5]{$2$};
				\draw[] (5, 4) circle [radius = 0] node[left = 3, above = 0, scale = .5]{$3$};
				\draw[] (6, 4) circle [radius = 0] node[left = 3, above = 0, scale = .5]{$4$};
				\draw[] (7, 4) circle [radius = 0] node[right = 3, above = 0, scale = .5]{$4$};
				\draw[] (8, 4) circle [radius = 0] node[right, scale = .5]{$5$};
				
				\draw[] (1, 5) circle [radius = 0] node[above, scale = .5]{$0$};
				\draw[] (2, 5) circle [radius = 0] node[left = 1, above = 0, scale = .5]{$1$};
				\draw[] (3, 5) circle [radius = 0] node[above = 2, left = 0, scale = .5]{$1$};
				\draw[] (4, 5) circle [radius = 0] node[left = 3, above = 0, scale = .5]{$2$};
				\draw[] (5, 5) circle [radius = 0] node[left = 3, above = 0, scale = .5]{$3$};
				\draw[] (6, 5) circle [radius = 0] node[left = 1, above = 0, scale = .5]{$4$};
				\draw[] (7, 5) circle [radius = 0] node[right = 3, above = 0, scale = .5]{$4$};
				\draw[] (8, 5) circle [radius = 0] node[right, scale = .5]{$5$};			
				
				\draw[] (2, 6) circle [radius = 0] node[above, scale = .5]{$0$};
				\draw[] (3, 6) circle [radius = 0] node[above, scale = .5]{$1$};
				\draw[] (4, 6) circle [radius = 0] node[left = 2, above = 0, scale = .5]{$2$};
				\draw[] (5, 6) circle [radius = 0] node[left = 2, above = 0, scale = .5]{$3$};
				\draw[] (6, 6) circle [radius = 0] node[right = 3, above = 0, scale = .5]{$3$};
				\draw[] (7, 6) circle [radius = 0] node[right = 3, above = 0, scale = .5]{$4$};
				\draw[] (8, 6) circle [radius = 0] node[right, scale = .5]{$5$};
				
				\draw[] (3, 7) circle [radius = 0] node[above, scale = .5]{$0$};
				\draw[] (4, 7) circle [radius = 0] node[above, scale = .5]{$1$};
				\draw[] (5, 7) circle [radius = 0] node[above, scale = .5]{$2$};
				\draw[] (6, 7) circle [radius = 0] node[above, scale = .5]{$3$};
				\draw[] (7, 7) circle [radius = 0] node[above, scale = .5]{$4$};
				\draw[] (8, 7) circle [radius = 0] node[above, scale = .5]{$5$};
				
			\end{tikzpicture}
			
		\end{center}
		
		\caption{\label{tilinghexagon} Depicted to the right are the three types of lozenges. Depicted in the middle is a lozenge tiling of a hexagon. One may view this tiling as a packing of boxes (of the type depicted on the left) into a large corner, which gives rise to a height function (shown in the middle). }
		
	\end{figure}
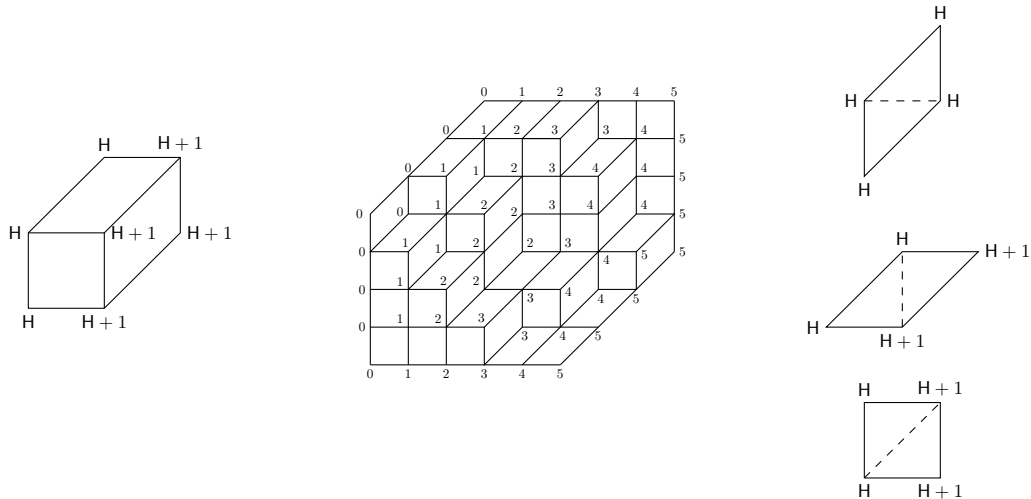

	For a fixed vertex $\mathsf{v} \in \mathsf{R}$ and integer $h_0 \in \mathbb{Z}$, one can associate with any tiling of $\mathsf{R}$ a height function $\mathsf{H}: \mathsf{R} \rightarrow \mathbb{Z}$ as follows. First, set $\mathsf{H} (\mathsf{v}) = h_0$, and then define $\mathsf{H}$ at the remaining vertices of $\mathsf{R}$ in such a way that the height functions along the four vertices of any lozenge in the tiling are of the form depicted on the right side of \Cref{tilinghexagon}. In particular, we require that $\mathsf{H} (x + 1, t) = \mathsf{H} (x, t)$ if and only if $(x, t)$ and $(x + 1, t)$ are vertices of the same type $1$ lozenge, and that $\mathsf{H} (x, t) - \mathsf{H} (x, t + 1) = 1$ if and only if $(x, t)$ and $(x, t + 1)$ are vertices of the same type $2$ lozenge. Since $\mathsf{R}$ is simply-connected, a height function on $\mathsf{R}$ is uniquely determined by these conditions (and the value of $\mathsf{H}(\mathsf{v}) = h_0$).
	
	We refer to the right side of \Cref{tilinghexagon} for an example; as depicted there, we can also view a lozenge tiling of $\mathsf{R}$ as a packing of $\mathsf{R}$ by boxes of the type shown on the left side of \Cref{tilinghexagon}. In this case, the value $\mathsf{H} (\mathsf{u})$ of the height function associated with this tiling at some vertex $\mathsf{u} \in \mathsf{R}$ denotes the height of the stack of boxes at $\mathsf{u}$. Observe in particular that, if there exists a tiling $\mathscr{M}$ of $\mathsf{R}$ associated with some height function $\mathsf{H}$, then the boundary height function $\mathsf{h} = \mathsf{H} |_{\partial \mathsf{R}}$ is independent of $\mathscr{M}$ and is uniquely determined by $\mathsf{R}$ (except for a global shift, which was above fixed by the value of $\mathsf{H}(\mathsf{v}) = h_0$).

	\subsection{Limit Shapes and Arctic Boundaries}

	\label{HeightLimit} 
	
	To analyze the limits of height functions of random tilings, it will be useful to introduce continuum analogs of the notions considered in \Cref{FunctionWalks}. So, set 
	\begin{flalign} \label{t}
		\mathcal{T} = \big\{ (s, t) \in (0,1) \times (-1,0): s+t>0 \big\} \subset \mathbb{R}^2, 
	\end{flalign} 

	\noindent and its closure $\overline{\mathcal{T}} = \big\{ (s, t) \in [0,1] \times [-1,0]: s+t\geq 0 \big\}$. We interpret $\overline{\mathcal{T}}$ as the set of possible gradients, also called \emph{slopes}, for a continuum height function; $\mathcal{T}$ is then the set of ``non-frozen'' or ``liquid'' slopes, whose associated tilings contain tiles of all types. For any simply-connected subset $\mathfrak{R} \subset \mathbb{R}^2$, we say that a function $H : \mathfrak{R} \rightarrow \mathbb{R}$ is \emph{admissible} if $H$ is $1$-Lipschitz and $\nabla H(u) \in \overline{\mathcal{T}}$ for almost all $u \in \mathfrak{R}$. We further say a function $h: \partial \mathfrak{R} \rightarrow \mathbb{R}$ \emph{admits an admissible extension to $\mathfrak{R}$} if $\Adm (\mathfrak{R}; h)$, the set of admissible functions $H: \mathfrak{R} \rightarrow \mathbb{R}$ with $H |_{\partial \mathfrak{R}} = h$, is not empty.
	
	We say that a sequence of domains $\mathsf{R}_1, \mathsf{R}_2, \ldots \subset \mathbb{T}$ \emph{converges} to a simply-connected subset $\mathfrak{R} \subset \mathbb{R}^2$ if $n^{-1} \mathsf{R}_n \subseteq \mathfrak{R}$ for each $n \geq 1$ and  $\lim_{n \rightarrow \infty} \dist (n^{-1} \del\mathsf{R}_n, \del\mathfrak{R}) = 0$. We further say that a sequence $\mathsf{h}_1, \mathsf{h}_2, \ldots $ of boundary height functions on $\mathsf{R}_1, \mathsf{R}_2, \ldots $, respectively, \emph{converges} to a boundary height function $h : \partial \mathfrak{R} \rightarrow \mathbb{R}$ if $\lim_{n \rightarrow \infty} n^{-1} \mathsf{h}_n (n v_n) = h (v)$ if $v_n$ is any point in $n^{-1} \del\mathsf{R}_n$ with limit $v\in \mathfrak \del R$.  
	
	To state results on the limiting height function of random tilings, for any $x \in \mathbb{R}_{\geq 0}$ and $(s, t) \in \overline{\mathcal{T}}$ we denote the \emph{Lobachevsky function} $L: \mathbb{R}_{\geq 0} \rightarrow \mathbb{R}$ and the \emph{surface tension} $\sigma : \overline{\mathcal{T}} \rightarrow \mathbb{R}$ by 
	\begin{flalign}
		\label{sigmal} 
		L(x) = - \displaystyle\int_0^x \log |2 \sin z| \mathrm{d} z; \qquad \sigma (s, t) = \displaystyle\frac{1}{\pi} \Big( L(\pi (1-s)) + L (- \pi t) + L \big( \pi ( s + t) \big) \Big).
	\end{flalign}
	
	\noindent For any $H \in \Adm (\mathfrak{R})$, we further denote the \emph{entropy functional}
	\begin{flalign}
		\label{efunctionh} 
		\mathcal{E} (H) = \displaystyle\int_{\mathfrak{R}} \sigma \big( \nabla H (z) \big) \mathrm{d}z.
	\end{flalign}
	
	The following variational principle of \cite{VPT}  states that the height function associated with a uniformly random tiling of a sequence of domains corresponding to $\mathfrak{R}$ converges to the maximizer of $\mathcal{E}$ with high probability.

	\begin{lem}[{\cite[Theorem 1.1]{VPT}}]
		
		\label{hzh} 
		
		Let $\mathsf{R}_1, \mathsf{R}_2, \ldots \subset \mathbb{T}$ denote a sequence of tileable domains, with associated boundary height functions $\mathsf{h}_1, \mathsf{h}_2, \ldots $, respectively. Assume that they converge to a simply-connected subset $\mathfrak{R} \rightarrow \mathbb{R}^2$ with piecewise smooth boundary, and a boundary height function $h : \partial \mathfrak{R} \rightarrow \mathbb{R}$, respectively. Denoting the height function associated with a uniformly random tiling of $\mathsf{R}_n$ with boundary height function $\mathsf h_n$ by $\mathsf{H}_n$, we have
		\begin{flalign*}
			\displaystyle\lim_{n \rightarrow \infty} \mathbb{P} \bigg( \displaystyle\max_{\mathsf{v} \in \mathsf{R}_n} \big| n^{-1} \mathsf{H}_n (\mathsf{v}) - H^* (n^{-1} \mathsf{v}) \big| > \varepsilon \bigg) = 0,
		\end{flalign*}  
	
		\noindent where $H^*$ is the unique maximzer of $\mathcal{E}$ on $\mathfrak{R}$ with boundary data $h$,
		\begin{flalign}
			\label{hmaximum}
			H^* = \displaystyle\argmax_{H \in \Adm (\mathfrak{R}; h)} \mathcal{E} (H).
		\end{flalign}
	\end{lem} 

	\noindent The fact that there is a unique maximizer described as in \eqref{hmaximum} follows from Proposition 4.5 of \cite{MCFARS}. 
	
	\subsection{Complex Burgers Equation}

	\begin{figure}
	\begin{center}
	 \includegraphics[scale=0.3,trim={0cm 5cm 0 7cm},clip]{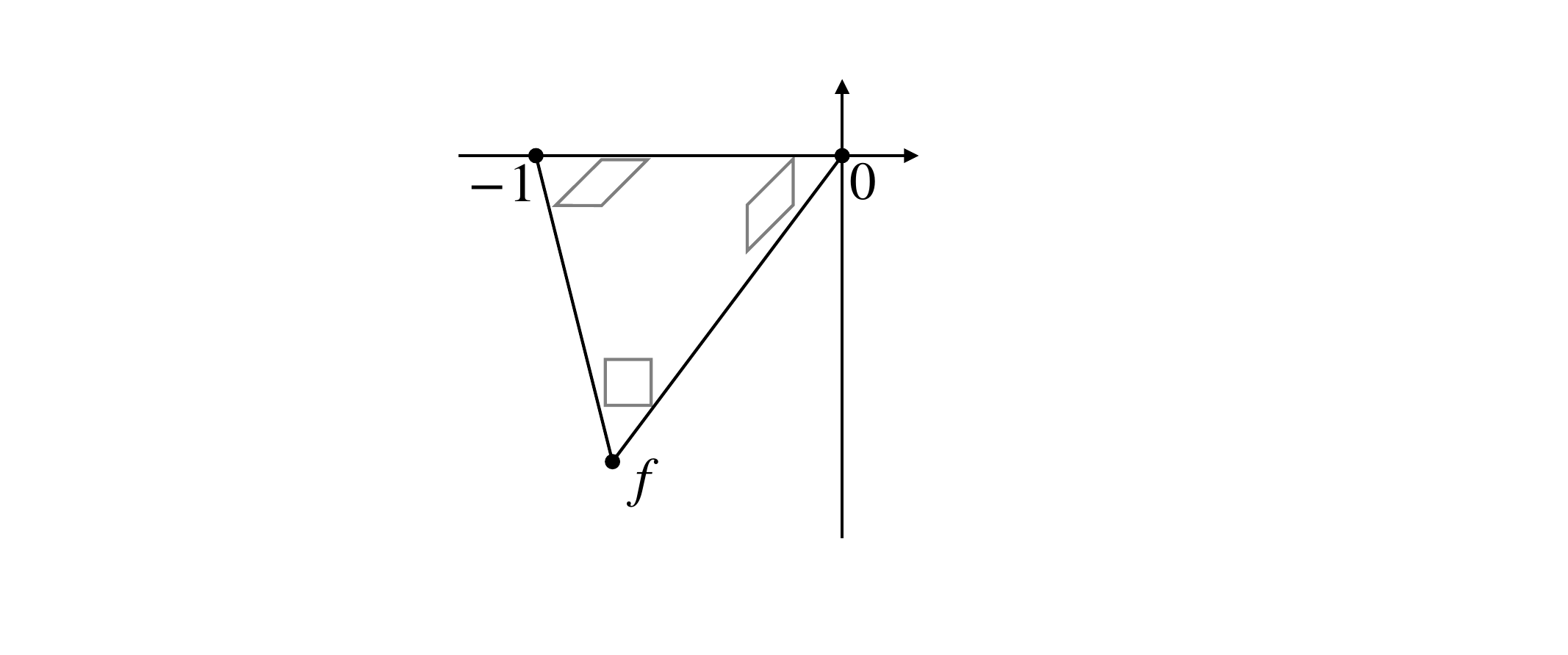}
	 \caption{Shown above the complex slope $f = f (u)$.}
	 \label{slope1}
	 \end{center}
	 \end{figure}	
	Under a suitable change of coordinates, this maximizer $H^*$ in \eqref{hmaximum} solves a complex variant of the Burgers equation, which makes it amenable to further analysis. To explain this point further, we require some additional notation. For any simply-connected open subset $\mathfrak{R} \subset \mathbb{R}^2$ with Lipschitz boundary, and boundary height function $h: \partial \mathfrak{R} \rightarrow \mathbb{R}^2$ admitting an admissible extension to $\mathfrak{R}$, define the \emph{liquid region} $\mathfrak{L} = \mathfrak{L} (\mathfrak{R}; h) \subset \mathfrak{R}$ 
	\begin{align}\begin{split}
		\label{al} 
		&\mathfrak{L} = \big\{ u=(x,t) \in \mathfrak{R}: \big( \partial_x H^* (u), \partial_t H^* (u) \big) \in \mathcal{T} \big\},
\end{split}	\end{align}
and the \emph{arctic boundary} $\mathfrak{A} = \mathfrak{A} (\mathfrak{R}; h) \subset \overline{\mathfrak{R}}$ as the set of points $ u=(x,t) \in \del \fL$, such that for any sequence of points $u_n \in \fL$ converging to $u$,
\begin{align}\begin{split}
		\label{e:arctic} 
		   \big( \partial_x H^* (u_n), \partial_t H^* (u_n) \big) \rightarrow \del\mathcal{T}, 
\end{split}	\end{align}
	where $H^*$ is as in \eqref{hmaximum}. The complement of the liquid region $\fR\setminus \fL$ is called the \emph{frozen region}. 
Then we define the \emph{complex slope} $f^*: \mathfrak{L} \rightarrow \mathbb{H}^-$ by, for any $u \in \mathfrak{L}$, setting $f^*(u) \in \mathbb{H}^-$ to be the unique complex number satisfying 
	\begin{flalign}
		\label{fh}
		\arg^* f^*(u) = - \pi \partial_x H^* (u); \qquad \arg^* \big( f^*(u) + 1 \big) = \pi \partial_t H^* (u),
	\end{flalign}

	\noindent where for any $z \in \overline{\mathbb{H}^-}\setminus \{0\}$ we have set $\arg^* z = \theta \in [-\pi, 0]$ to be the unique number in $[-\pi, 0]$ satisfying $e^{-\mathrm{i} \theta} z \in \mathbb{R}_{> 0}$; see \Cref{slope1} for a depiction, where there we interpret $1 - \partial_x H^* (u)$ and $-\partial_t H^* (u)$ as the approximate proportions of tiles of types $1$ and $2$ around $nu \in \mathsf{R}_n$, respectively (which follows from the definition of the height function from \Cref{FunctionWalks}).

		 The following result from \cite{LSCE} indicates that the complex slope $f^*$ satisfies the complex Burgers equation on the liquid region.
	  \begin{prop}[{\cite[Theorem 1]{LSCE}}]
	 	
	 	\label{fequation}
	 	
	 	For any $(x, t) \in \mathfrak{L}$, let $f^*_t (x) = f^* (x, t)$ we have
	 	\begin{flalign}
	 		\label{ftx}
	 		\partial_t f^*_t (x) + \partial_x f^*_t (x) \displaystyle\frac{f^*_t (x)}{f^*_t (x) + 1} = 0.
	 	\end{flalign} 
	 	 	
	 	\end{prop}

The following result from \cite{DMCS} gives regularity of the arctic curve.
	 \begin{thm}[{\cite[Theorem 1.2, Theorem 1.5, Theorem 1.6]{DMCS}}]
	 \label{t:localr}
	 	The arctic curve $\fA$ is piecewise analytic. Take any connected component $\Gamma$ of $\fA$. There are  at most finitely many singularities $\{\xi_j\}_{j=1}^m\subset \fA$ on $\Gamma$, and they are all
either first order (inward) cusps or tacnodes. The complex slope $f^*$ extends continuously to $\Gamma$. For any $(x, t) \in \Gamma$, $f_t^*(x)\in \bR\cup \{\infty\}$ and the slope of the arctic boundary at $(x,t)$ is given by
\begin{align}\label{e:slope}
\frac{f_t^*(x)+1}{f_t^*(x)}
\end{align} 	 	
    \end{thm} 
For a nonsingular point $(x, t) \in \mathfrak{A}$, we call it a \emph{tangency location} of $\mathfrak{A}$, if the tangent line to $\mathfrak{A}$ has slope in $\{ 0, 1, \infty \}$.

\subsection{Lozenge Tiling of Strip Domains}
%
%
%

	In this section, we give precise definitions of strip domains (also called ``double-sided trapezoid domains'').
	These domains are different from the ones considered in earlier works, such as \cite{ARS,AURTPF,UEFDIPS}, since they will accommodate non-frozen boundary conditions along both their north and south edges (instead of only their south ones).
	
	We fix a real number $\mathfrak{t}>0$, and linear functions $\mathfrak{a}, \mathfrak{b} : [0, \mathfrak{t}]\mapsto \bR$ with $\mathfrak{a}' (s), \mathfrak{b}' (s) \in \{ 0, 1 \}$, such that $\mathfrak{a} (s) \leq \mathfrak{b} (s)$ for each $s \in [0, \mathfrak{t}]$. Define the double-sided trapezoid domain
	\begin{flalign}
		\label{d}
		\mathfrak{D} = \mathfrak{D} (\mathfrak{a}, \mathfrak{b}; \mathfrak{t}) = \big\{ (x, t) \in \mathbb{R} \times [0,\ft] : \mathfrak{a} (t) \leq x \leq \mathfrak{b} (t) \big\},
	\end{flalign}

	\noindent and denote its four boundaries by 
	\begin{flalign}
		\label{dboundary}
		\begin{aligned} 
		&\partial_{\so} (\mathfrak{D}) = \big\{ (x,t) \in {\mathfrak{D}}: t = 0 \big\}; \qquad \quad
		\partial_{\north} (\mathfrak{D}) = \big\{ (x, t) \in {\mathfrak{D}}: t = \mathfrak{t} \big\}; \\ 
		& \partial_{\we} (\mathfrak{D}) = \big\{ (x, t) \in {\mathfrak{D}}: x = \mathfrak{a} (t) \big\}; \qquad 
		\partial_{\ea} (\mathfrak{D}) = \big\{ (x, t) \in {\mathfrak{D}}: x = \mathfrak{b} (t) \big\}.
		\end{aligned} 
	\end{flalign}

	\noindent We refer to \Cref{ddomain} for a depiction.

	\begin{figure}
		
		\begin{center}		
			
			\begin{tikzpicture}[
				>=stealth,
				auto,
				style={
					scale = .52
				}
				]
				
				\draw[black, very thick] (0, 0) node[left, scale = .7]{$y =0$} -- (4, 0) -- (7, 3) -- (0, 3) node[left, scale = .7]{$y = \mathfrak{t}$} -- (0, 0);
				\draw[black, very thick] (6.5, 0) -- (10.5, 0) -- (13.5, 3) -- (9.5, 3) -- (6.5, 0);
				\draw[black, very thick] (15, 0) -- (19, 0) -- (19, 3) -- (15, 3) -- (15, 0);
				\draw[black, very thick] (20.5, 0) -- (25.5, 0) -- (25.5, 3) -- (23.5, 3) -- (20.5, 0);
				
				\draw[] (3.5, 3) node[above, scale = .7]{$\partial_{\north} (\mathfrak{D})$};
				\draw[] (5.5, 1.5) node[below = 2, right, scale = .7]{$\partial_{\ea} (\mathfrak{D})$};
				\draw[] (0, 1.5) node[left, scale = .7]{$\partial_{\we} (\mathfrak{D})$};
				\draw[] (2, 0) node[below, scale = .7]{$\partial_{\so} (\mathfrak{D})$};	
				
			\end{tikzpicture}
			
		\end{center}
		
		\caption{\label{ddomain} Shown above are the four possibilities for $\mathfrak{D}$.}
		
	\end{figure}
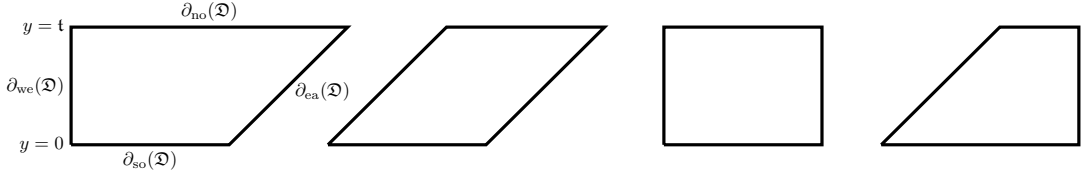 

	Let $h: \partial \mathfrak{D} \rightarrow \mathbb{R}$ denote a function admitting an admissible extension to $\mathfrak{D}$. We fix a real number $\fm>0$, and assume throughout that 
	\begin{align}\label{e:defhh}
	h\equiv 0 \text{  on } \partial_{\we} (\mathfrak{D}),\quad h\equiv \fm \text{  on } \partial_{\ea} (\mathfrak{D}).
	\end{align}
 Let $H^* \in \Adm (\mathfrak{D}; h)$ denote the maximizer of $\mathcal{E}$ as in \eqref{hmaximum}, and let the liquid region $\mathfrak{L} = \mathfrak{L} (\mathfrak{D}; h)$ and arctic boundary $\mathfrak{A} = \mathfrak{A} (\mathfrak{D}; h)$ be as below \eqref{al}. Recall that a point on $\mathfrak{A}$ is a \emph{tangency location} if the tangent line to $\mathfrak{A}$ through it has slope in $\{ 0, 1, \infty \}$. 
	
	We may then define the complex slope $f^*: \mathfrak{L} \rightarrow \mathbb{H}^-$ as in \eqref{fh}, which upon denoting $f^*_t (x) = f^* (x, t)$ satisfies the complex Burgers equation \eqref{ftx}. Further let $\mathfrak{L}_{\north} = \mathfrak{L}_{\north} (\mathfrak{D}; h)$ denote the interior of $\overline{\mathfrak{L}} \cap \partial_{\north} (\mathfrak{D})$, and let $\mathfrak{L}_{\so} = \mathfrak{L}_{\so} (\mathfrak{D}; h)$ denote the interior of $\overline{\mathfrak{L}} \cap \partial_{\so} (\mathfrak{D})$; these are the extensions of the liquid region to the north and south boundaries of $\mathfrak{L}$. For all $t \in (0, \mathfrak{t})$, we define slices of the liquid region (along the horizontal line $y = t$) by
	\begin{flalign}\label{e:defItIt}
		I_t^* = \big\{ x : (x, t) \in \overline{\mathfrak{L}} \big\}; \qquad I^*_{0} = \overline{\mathfrak{L}}_{\so}; \qquad I^*_{\mathfrak{t}} = \overline{\mathfrak{L}}_{\north}.
	\end{flalign}

	We need to formulate certain conditions on the limit shape $H^*$. We will study the following two cases ( We refer to \Cref{dldomain} for a depiction). 
	In the first case, we call $\partial_{\north} (\mathfrak{D})$ \emph{packed} (with respect to $h$) if $\partial_x h (v) = 1$ for each $v \in \partial_{\north} (\mathfrak{D})$. This case has been extensively studied in \cite{ARS,AURTPF}, due to its exact solvability. In this case, the arctic curve is tangent to $\del_{\north}(\mathfrak{D})$, and $\overline{\mathfrak{L}} \cap \partial_{\north} (\mathfrak{D})$ constitutes a single point, and so $\mathfrak{L}_{\north}$ is empty. We refer to \Cref{dldomain} for a depiction. 
	
	In the second case, $\mathfrak{L}_{\north} = I^*_{\mathfrak{t}}$ is a single interval. Take $\ft'>\ft$. We denote the extended double-sided trapezoid domain $\fD'$ on the time strip $[0,\ft']$:
\begin{align*}
\fD'=\Big\{ (x,t): t\in [0,\ft'], x\in [ \mathfrak{a}(t), \mathfrak{b}(t)] \Big\}, \quad \fa'(t)\in\{0,1\},\quad \fb'(t)\in \{0,1\}.
\end{align*}
We say that $H^*$ can be \emph{extended to time $\mathfrak{t}'$} if there exists an extended height function $\tilde{H}^*: \mathfrak{D}' \rightarrow \mathbb{R}$, which agrees with $H^*$ on $\mathfrak{D}$ and it is  the minimizer of the variational principle
\begin{align}\label{e:varW}
\tilde H^*=\arg\sup_{ H\in \Adm({\fD'};h)}\int_ 0^ {\ft'}\int_\bR\sigma(\nabla  H)\rd x \rd t,
\end{align}
where $\tilde h=\tilde H^*$ on $\del\fD'$ with $\tilde h\equiv 0$ on $\del_{\rm we}(\mathfrak D')$,  and $\tilde h\equiv \fm$ on $\del_{\rm ea}(\mathfrak D')$. 
The corresponding liquid region $\tilde \cL$ is  simply-connected, open subset $\tilde{\mathfrak{L}} \subset \mathbb{R}^2$ containing $\mathfrak{L}$, such that the set $\big\{ x: (x, \mathfrak{t}') \in \tilde{\mathfrak{L}} \big\}$ is non-empty and connected. The complex slope $\tilde{f}^*: \tilde{\mathfrak{L}} \rightarrow \mathbb{H}^-$ extends $f^*_t (x)$, and satisfies the complex Burgers equation \eqref{ftx}. We denote the corresponding arctic boundary by $\tilde \fA$  as in \eqref{e:arctic}. For all $t \in [0, \mathfrak{t'}]$, we define slices of the extended liquid region (along the horizontal line at time $t$) still as $I_t^*$ as in \eqref{e:defItIt}.


		\begin{assumption}
		\label{xhh} 
		
		Assume the following constraints hold.
		
		\begin{enumerate} 
			\item The boundary height function $h$ is constant along both $\partial_{\ea} (\mathfrak{D})$ and $\partial_{\we} (\mathfrak{D})$ as in \eqref{e:defhh}.
			\item Either $\partial_{\north} (\mathfrak{D})$ is packed with respect to $h$ or $\mathfrak{L}_{\north} = I^*_{\mathfrak{t}}$ is a single interval and there exists $\mathfrak{t}' > \mathfrak{t}$ such that $H^*$ admits an extension to time $\mathfrak{t}'$.
			\item $\mathfrak{L}_{\north} = I^*_{\mathfrak{t}}$ is a single interval and there exists $\mathfrak{t}' > \mathfrak{t}$ such that $H^*$ admits an extension to time $\mathfrak{t}'$, we further assume the following holds:
			\begin{enumerate}
			\item The arctic curve $\tilde \fA$ contains at most one cusp location. Either $
I^*_t= \big[ E_1(t), E_2(t) \big],
$
for $0\leq t\leq \ft'$, $\tilde \fA=\{(E_1(t), t), (E_2(t),t): 0\leq t\leq \ft'\}$ and there is no cusp. Or there is one cusp point at $(E_c, t_c)$, then for $t_c\leq t\leq \ft'$, the slice $I^*_t$ is a single interval
$
I^*_t= \big[ E_1(t), E_2(t) \big],
$
 for $0\leq t\leq t_c$, the slice $I^*_t$ consists of two intervals
$
I^*_t= \big[ E_1(t), E'_1(t) \big] \cup \big[ E'_2(t), E_2(t) \big]$, and $\tilde \fA=\{(E_1(t), t), (E_2(t),t): 0\leq t\leq \ft'\}\cup \{(E'_1(t), t), (E'_2(t),t): 0\leq t\leq t_c\}$ (See Figure \ref{f:It}). We assume that the cusp location is not a tangency location.

\item Any tangency location along $\tilde \fA$ is of the form $\min I_t^*$ or $\max I_t^*$, for some $t \in (0, \ft)$. At most one tangency location is of the form $\min I_t^*$, and at most one is of the form $\max I_t^*$. Moreover, these tangent locations are contained in either $\del_{\we}(\fD')$ or $\del_{\ea}(\fD')$.

\item On the frozen region $\fD'\setminus \tilde \fL$, $\nabla \tilde H^*(x,t)$ is piecewise constant, taking values in $\{(0,0), (1,0), (1,-1)\}$. 

\item The extended complex slope $\tilde f^*$ extends to the closure of $\tilde{\fL}$.	
Fix any $(x_0,t_0)$ in the closure of $\tilde \fL$. There exists a neighborhood $\fU\subset \bR^2$ of $(x_0,t_0)$, and a real analytic function $Q_0$ \footnote{Here, a function $g$ defined on a domain of $\bC$ is called real analytic if it is analytic and $\overline{g(z)}=g(\overline z)$.}  in one variable, such that for any $(x,t)$ in the intersection of $\fU$ and the closure of $\tilde\fL$,	 	\begin{flalign}
	 		\label{q0f} 
	 		Q_0 \big( \tilde f^*_t (x) \big) = x \big( \tilde f^*_t (x) + 1 \big) - t \tilde f^*_t (x).
	 	\end{flalign}
Moreover, $(x,t)\in \tilde \fA$ if and only if $w=\tilde f_t^*(x)$ is a double root of \eqref{q0f}: $Q_0(w)=x(w+1)-tw$.

\item For any $0\leq t\leq \ft$ the map
			\begin{align}\label{e:injecthi}
			(x,r)\in \tilde \fL\cap \left(\bR\times [t, \ft']\right) \mapsto x+(t-r)\frac{\tilde f^*_r(x)}{\tilde f^*_r(x)+1}\in \bH^+
			\end{align}
is a bijection to its image.

			\end{enumerate}
	\end{enumerate}			
	\end{assumption}

	Let us briefly comment on these constraints. The first guarantees that the associated tiling is one of a double-sided trapezoid, as depicted in \Cref{ddomain}. The second states the limit shape is one of the two cases in \Cref{dldomain}. Item (a) in the third statement guarantees that the arctic boundary for the tiling has only one cusp, and the cusp location is not a tangency location. Item (b) in the third statement implies that there are at most two tangency locations along the arctic boundary (and are along the leftmost and rightmost components of the arctic curve, if they exist); Item (c), (d) and (e) in the third statement are technical. Item (c) holds if there exists a height function $H^\circ$ on $\fD'$ extending the boundary data $\tilde h$, and all the slopes $\nabla H^\circ(x,t)\in \mathcal T$; see \cite[Proposition 7.10]{RT}. In particular Item (c) and (d) hold for any polygonal domain and slices of (an explicit perturbation of) one given by a polygonal domain. 
For \eqref{e:injecthi}, it is crucial that the limit shape does not have cusps pointing at opposite directions. In our case, the limit shape has at most one cusp, it could in principle be weakened; we impose it since doing so will substantially simplify our proofs later.

	  For our main application in Part II of this series \cite{AH2}, we will take the limit shape from the tiling of a thin slice of (an explicit perturbation of) one given by a polygonal domain $\fP$, and Assumption \ref{xhh} holds. We postpone a detailed discussion to Appendix \ref{s:thin}.

	\subsection{Main Results}

		Now let $n \geq 1$ be an integer; denote $\mathsf{t} = \mathfrak{t} n$, and $m=\fm n$ (recall from \eqref{e:defhh}). Suppose $m\in \bZ_>0$, $\mathsf{D} = n \mathfrak{D} \subset \mathbb{T}$, so that $ \mathsf{t} \in \mathbb{Z}_{> 0}$. Let $\mathsf{h}: \partial \mathsf{D} \rightarrow \mathbb{Z}$ denote a boundary height function.  We further fix a real number $\delta > 0$ and a (large) positive integer $n$.
			We define the augmented variant of $\mathfrak{L}$ by
	\begin{flalign*}
		\mathfrak{L}_+ (\mathfrak{D}) =\mathfrak{L}_+^{\delta} (\mathfrak{D}) = \left( \mathfrak{L} \cup \bigcup_{u \in \mathfrak{A}} \big\{ v \in \mathbb{R}^2 : |v - u| \leq n^{\delta-2/3}\big\}\right)\cap \mathfrak{D}
	\end{flalign*}
		 Below, we view the quantities $0 < \mathfrak{t} < \mathfrak{t}'$, functions $\mathfrak{a}, \mathfrak{b}$, and domain $\mathfrak{D}$ as independent of $n$. The next assumption indicates how the tiling boundary data $\mathsf{h}$ approximates $h$ (recall from \eqref{e:defhh}) along $\partial \mathfrak{D}$.

	\begin{assumption} 
		\label{xhh2} 
		
		Adopt \Cref{xhh}, and assume the following on how $\mathsf{h}$ converges to $h$.
		
		\begin{enumerate}
			\item For each $v \in \partial \mathfrak{D}$, we have $\big| \mathsf{h} (nv) - n h(v) \big| < n^{\delta / 50}$.
			\item For each $v \in \partial_{\ea} (\mathfrak{D}) \cup \partial_{\we} (\mathfrak{D})$, and each $v \in \partial_{\north} \mathfrak{D}\cup \partial_{\so} \mathfrak{D}$ such that $v \notin \mathfrak{L}_+^{\delta / 50} (\mathfrak{D})$, we have $\mathsf{h} (nv) = n h(v)$. 
			\item If $\partial_{\north} (\mathfrak{D})$ is packed, then for each $v \in \partial_{\north} (\mathfrak{D})$ we have $\mathsf{h} (nv) = n h(v)$.
		\end{enumerate}
	
	\end{assumption}

	The first assumption states that $\mathsf{h}$ approximates its limit shape; the second and third state that it coincides with its limit shape in the frozen region. Recalling that $\mathscr{G} (\mathsf{h})$ denotes the set of height functions on $\mathsf{D}$ with boundary data $\mathsf{h}$, we can now state our main result on  the  concentration estimate for a uniformly element of $\mathscr{G} (\mathsf{h})$.

	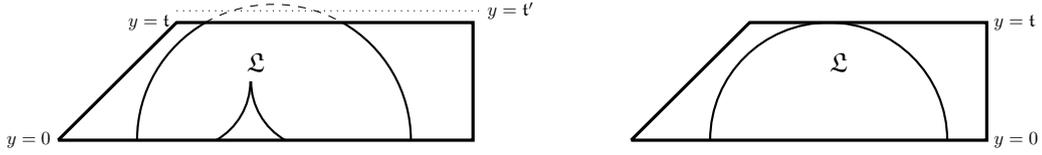
\begin{figure}
		
		\begin{center}		
			
			\begin{tikzpicture}[
				>=stealth,
				auto,
				style={
					scale = .52
				}
				]
				
				\draw[black, very thick] (-5, 0) node[left, scale = .7]{$y =0$}-- (5.5, 0) -- (5.5, 3) -- (-2, 3) node[left, scale = .7]{$y = \ft$} -- (-5, 0);
				\draw[black, very thick] (9.5, 0) -- (18.5, 0) node[right, scale = .7]{$y = 0$} -- (18.5, 3) node[right, scale = .7]{$y = \ft$} -- (12.5, 3) -- (9.5, 0);
				\draw[dotted] (-2, 3.3) -- (5.7, 3.3) node[right, scale = .7]{$y = \mathfrak{t}'$};
				
				\draw[black, thick] (-3, 0) arc (180:120:3.464);
				\draw[black, thick] (3.928, 0) arc (0:60:3.464);
				\draw[black, thick] (11.5, 0) arc (180:0:3);
				
				\draw[black, thick] (-1, 0) arc (-60:0:1.732);
				\draw[black, thick] (.75, 0) arc (240:180:1.732);
				
				\draw[black, dashed] (-1.278, 3) arc (120:60:3.464);
				
				\draw[] (0, 2) circle [radius = 0] node[]{$\mathfrak{L}$};
				\draw[] (14.75, 2) circle [radius = 0] node[]{$\mathfrak{L}$};
				

			\end{tikzpicture}
			
		\end{center}
		
		\caption{\label{dldomain} Shown to the left is an example of limit shape admitting an extension to time $\mathfrak{t}' > \ft$; shown to the right is a liquid region that is packed with respect to $h$.}
		
	\end{figure} 
	
	\begin{thm}
		
		\label{estimategamma}
		
		Adopt \Cref{xhh} and \Cref{xhh2}. Let $\mathsf{H} : \mathsf{D} \rightarrow \mathbb{Z}$ denote a uniformly random element of $\mathscr{G} (\mathsf{h})$. Then, the following two statements hold with overwhelming probability, i.e., $1-n^{-D}$ for any constant $D > 1$ (assuming $n$ is sufficiently large):
		
		\begin{enumerate} 
			\item We have $\big| \mathsf{H} (nu) - n H^* (u) \big| < n^{\delta}$, for any $u \in \mathfrak{D}$.
			\item For any $u \in \mathfrak{D} \setminus \mathfrak{L}_+^{\delta} (\mathfrak{D})$, we have $\mathsf{H} (nu) =n H^* (u)$.
		\end{enumerate}  
	
	\end{thm}

In the following, we give an overview for the proof of \Cref{estimategamma}. The complex slopes which characterize the limit shapes of lozenge tilings, and their analytic extensions play fundamental role in our proof. Section \ref{s:CBE0} is technical, where we collect properties of complex slopes and their analytic extensions. 

	The proof of \Cref{estimategamma} is given in Section \ref{s:nbb}. Lozenge tilings of strip domains are equivalent to non-intersecting random Bernoulli walks conditioned to start and end at specified locations. We follow \cite{HFRTNRW} and approximate the random bridge model by a family of non-intersecting Bernoulli random walks with a space and time dependent drift in Sections \ref{s:NBwe} and \ref{s:NBBdefine}. The drifts of these  random walks are explicitly given by analytic extensions of complex slopes. In particular, we need to understand complex slopes with general boundary data; their properties are collected in Section \ref{s:CBE}.  A key intermediate result is the optimal rigidity \Cref{p:rigidityBB} for these non-intersecting Bernoulli random walks with drifts. Using it as input, together with a monotonicity coupling of lozenge tilings to non-intersecting Bernoulli random walks, we conclude the  optimal rigidity for lozenge tilings of strip domains in Section \ref{s:bridgeop}. 
	
	The proof of \Cref{p:rigidityBB} for the optimal rigidity of non-intersecting Bernoulli random walks is given in Sections \ref{s:pcW}, with intermediate results proven in Sections \ref{s:loopeq} and \ref{s:improveEE}. 
	We analyze the discrete stochastic differential equation for the empirical particle density of non-intersecting Bernoulli random walks. The results from Sections \ref{s:loopeq} and \ref{s:improveEE} are collected in Proposition \ref{p:improve}, which gives nearly optimal controls of the martingale term and errors in the discrete stochastic differential equation on any mesoscopic scale. With them as an input, we analyze the discrete stochastic differential equation by studying its behavior along the characteristics of the limiting complex Burgers equation. This leads to the optimal rigidity \Cref{p:rigidityBB}.

	These intermediate results proven in Sections \ref{s:loopeq} and \ref{s:improveEE} are for integrable transition probabilities, which satisfy the \emph{dynamical loop equations}. In Sections \ref{ProofEquation1} and \ref{EquationsEstimate}, we introduce a general family of transition probabilities, which includes those of non-intersection Bernoulli random walks as a special case and enjoy the dynamical loop equations. In Sections \ref{s:LRE} and \ref{ProofEquation}, utilizing the loop equations and a multi-scale analysis, we obtain estimates of these transition probabilities, which is optimal in the liquid region.  Through bootstrapping we obtain optimal estimates in the frozen region as well in Section \ref{s:improveEE}.

\section{Complex Burgers Equation}
\label{s:CBE0}

We recall the double-sided trapezoid domain $\fD$ on the  strip $[0,\mathfrak{t}]$ from \eqref{d}. Under \Cref{xhh}, either the north boundary $\del_{\rm no}(\fD)$ of $\fD$ is packed, or $\mathfrak{L}_{\north} = I^*_{\mathfrak{t}}$ is a single interval and $H^*$ admits an extension to time $\ft'$ with $\ft'>\ft$. The first of these cases corresponds to lozenge tilings of trapezoid domains with tightly packed north boundary, whose fluctuations have been studied extensively by this point \cite{ARS,AURTPF}, due to its exact solvability. We will therefore focus on the second case, and remark the minor necessary changes for the first case; see Remark \ref{r:trapzoid1} below. 
In this section, we collect some properties of complex slopes associated with tilings of double sided trapezoid domains with general boundary height function. 

In the rest of this paper, we assume that $H^*$ admits an extension to time $\mathfrak{t}'$. By slight abusing of notations, we denote the extended height function, its liquid region, arctic boundary and complex slope as $H^*, \fL(\fD'), \fA(\fD')$ and $f^*$. Thanks to Item (a) and (c) in the third statement of \Cref{xhh}, 
\begin{align}\label{e:fLA}
\fL(\fD')=\{(x,t)\in \fD': 0<\del_x H^*(x,t)<1\},\quad
\fA(\fD')=\overline{\{(x,t)\in \del \fL(\fD'): 0<t<\ft'\}}. 
\end{align}
We recall the relation between the slope of the arctic boundary and the complex slope from \eqref{e:slope}. Thanks to Item (b) in the third statement of \Cref{xhh}, there is no horizontal tangency location (slope equals $0$) along the arctic boundary $\fA(\fD')$. Thus, the slope of the tangent line at every point along this arctic boundary is bounded away from $0$. Thus it follows that $f_t^*(x)$ is bounded away from $-1$, and
\begin{align}\label{e:ft+1b}
\big| f_t^*(x)+1 \big|\gtrsim 1,\quad \left|\frac{f_t^*(x)}{f_t^*(x)+1}\right|\lesssim 1.
\end{align}
which will be used repeatedly in this section.

For any $0\leq t\leq \ft$, we denote the part of the trapezoid domain restricted to the time strip $[t,\ft']$ as
\begin{align}\label{e:defDs}
\fD^t=\fD'\cap \big( \bR\times [t,\ft'] \big) =\Big\{ (x,r): r\in [t,\ft'], x\in [ \mathfrak{a}(r), \mathfrak{b}(r)] \Big\}.
\end{align}
There is a map from the liquid region $\fL(\fD')\cap \fD^t$ to the pair $(f_t, z)$,
\begin{align}\label{e:xrmap}
(x,r)\in \fL(\fD')\cap \fD^t\mapsto (f_t,z)=\left(f_r^*( x),  x + (t-r)\frac{   f^*_r( x)}{ f^*_r( x)+1}\right) \in \mathbb{H}^-\times \bH^+.
\end{align}
By \Cref{t:localr} and \Cref{xhh}, the above map extends continuously to the boundary of the liquid region.  
We remark that in \eqref{e:xrmap}, we have 
\begin{flalign*} 
	\Im[f^*_r(x)]< 0; \qquad  \Im \bigg[ \displaystyle\frac{f^*_r( x )}{( f^*_r( x)+1)} \bigg]= \displaystyle\frac{\Imaginary \big[ f^*_r( x ) \big]}{| f^*_r( x)+1|^2}< 0.
\end{flalign*} 
By Assumption \ref{e:injecthi}, if we further project the map \eqref{e:xrmap} to the second coordinate, 
\begin{align}\label{e:proj0}
(x,r)\in \fL(\fD')\cap \fD^t\mapsto z= x + (t-r)\frac{   f^*_r( x)}{ f^*_r( x)+1} \in \bH^+,
\end{align}
we get an injection. Thus we can view $f_t$ on the righthand side of \eqref{e:xrmap} as a function of $z$, i.e. $f_t=f_t(z)$. By taking $r=t$ in \eqref{e:xrmap}, we get $f_t(x)=f_t^*(x)$ for $(x,t)\in \overline{\mathfrak L(\mathfrak D')}$.  $f_t(z)$ can be viewed as an extension of $f_t^*(x)$ to the complex plane. Moreover, The map \eqref{e:xrmap} and its complex conjugate $(\overline f, \overline z) \in \mathbb{H}^+\times \bH^-$
glue along the arctic curve to a map from the double of
the liquid region $\fL(\fD') \cap \fD^t\mapsto (f_t,z)$ to a domain $\mathscr U_t\subseteq \bC$.

The next proposition collects some properties of $f_t(z)$.
\begin{prop}
Under \Cref{xhh},  the function $f_t(z)$ as defined in \eqref{e:xrmap} is a holomorphic function of $z$ and satisfies the complex Burgers equation
\begin{align}\label{e:Burgeq}
 \del_{t}  f_t( z)+\del_{z}  f_t( z ) \frac{f_t( z )}{ f_t( z)+1}=0.
\end{align}
\end{prop}
\begin{proof}
Denote $z=z(x,r,t)=x+(t-r) f_r^*(x)/(f_r^*(x)+1)$. First we show that $\del_x z= 0$ implies that $\del_r z=0$ and $(x,r)$ is a critical point of the map $(x,r)\mapsto z$. The complex Burgers equation \eqref{ftx} gives
\begin{align}\begin{split}\label{e:dzhi}
&\phantom{{}={}}\del_r z=\del_r\left(x+(t-r) \frac{f_r^*(x)}{f_r^*(x)+1}\right)
=- \frac{f_r^*(x)}{f_r^*(x)+1}+(t-r)\frac{\del_r f_r^*(x)}{(f_r^*(x)+1)^2}\\
&=- \frac{f_r^*(x)}{f_r^*(x)+1}-(t-r) \frac{f_r^*(x)}{f_r^*(x)+1}\frac{\del_x f_r^*(x)}{(f_r^*(x)+1)^2}
=- \frac{f_r^*(x)}{f_r^*(x)+1}\del_x\left(x+(t-r) \frac{f_r^*(x)}{f_r^*(x)+1}\right)=0.
\end{split}\end{align}
Since $f_r^*(x)$ is analytic in $x,r$, so is $z(x,r,t)$. The collection of critical points are discrete. Next we can view $f_r^*(x)=f_t(z(x,r,t))$ as a function of $x,r,t$.  By taking derivatives on both sides we get
\begin{align}\begin{split}\label{e:dzft}
\del_x f_r^*(x)=\del_z f_t \del_x z+\del_{\overline z} f_t \del_x \overline z,\quad\del_r f_r^*(x)=\del_z f_t \del_r z+\del_{\overline z} f_t \del_r \overline z,\quad \del_t f_t+\del_z f_t \del_t z+\del_{\overline z} f_t \del_t \overline z=0.
\end{split}\end{align}
The expression \eqref{e:dzhi} implies that $\del_r z=-(\del_x z) f_r^*(x)/(f_r^*(x)+1)$. Using this as input, we can solve for $\del_{\overline z}f_t$ using the first  two relations in \eqref{e:dzft}, 
\begin{align*}
0=\del_{\overline z} f_t \del_x \overline z \Im\left[\frac{f^*_r(x)}{f_r^*(x)+1}\right].
\end{align*}
When $\del_x \overline z\neq 0$ ($(x,r)$ is not a critical point), and inside liquid region $\Im[f_r^*(x)]<0$, we conclude that $\del_{\overline z}f_t=0$, and $f_t(z)$ is  holomorphic at $z$. The set of critical points is discrete, and from our construction $f_t(z)$ is continuous as a function $z$. By Riemann's theorem on removable singularities, we conclude that $f_t(z)$ is a holomorphic function of $z$.  Using the last relation in \eqref{e:dzft}, we get the complex Burgers equation \eqref{e:Burgeq}.
\end{proof}

 The complex Burgers equation \eqref{e:Burgeq} now can be solved using the characteristic flow, 
\begin{align}\label{e:ccff}
\del_t f^*_t(z_t(u))=0,\qquad \text{where $z_t(u)$ satisfies} \qquad \del_t z_t(u)= \displaystyle\frac{f^*_t(z_t(u))}{f^*_t(z_t(u))+1}, \quad z_0(u)=u,
\end{align}

\noindent In particular, $z_t$ is linear in $t$, given explicitly by
\begin{flalign}
	\label{zlineart}
	z_t (u) = u + t \displaystyle\frac{f^*_0 (u)}{f^*_0 (u) + 1}.
\end{flalign}

\noindent If the context is clear, we will simply write $z_t(u)$ as $z_t$, omitting its dependence on the initial value $u$.


From the third statement in \Cref{xhh},  the arctic curve $\fA(\fD')$ has at most one cusp point. In the rest of this section, we study the most general case provided as follows.

\begin{assumption}\label{a:xhh}
Under \Cref{xhh} and that $H^*$ admits an extension to time $\ft'$ with $\ft'>\ft$,  we further restrict ourselves to the case: 
 The arctic curve $\fA(\fD')$ has a cusp point at $(E_c, t_c)$, and it has two tangent locations 
 at $(E_1(t_1),t_1)$ and $(E_2(t_2),t_2)$ with slopes $1$ and $\infty$ respectively 
 \end{assumption}

\begin{rem}\label{r:replace}
Under \Cref{xhh} and \Cref{a:xhh}, these two tangent locations $(E_1(t_1),t_1)$ and $(E_2(t_2),t_2)$ are on the west $\del_{\rm we}\fD'$ and east $\del_{\rm ea}\fD'$ boundary of $\fD'$ respectively. It follows that $\fa'(t)=1$ and $\fb'(t)=0$ (otherwise the west and east boundary will not be outside the liquid region). We conclude that the arctic curve $\fA(\fD')$ is tangent to the west $\del_{\rm we}\fD'$ and east $\del_{\rm ea}\fD'$ boundary of $\fD'$
at $(E_1(t_1),t_1)$ and $(E_2(t_2),t_2)$ respectively. 

Next we show that on each connected component of frozen region, $\nabla H^*$ is constant and takes values in $\{(0,0), (1,0), (1,-1)\}$. We will only prove this for the region $\{(x,t)\in \fD': 0\leq t\leq t_1, \fa(t)\leq x\leq E_1(t)\}$. By the third statement in \Cref{xhh}, the complex slope $f_t^*(x)$ extends continuously to the arctic curve $\{(E_1(t), t): 0\leq t\leq \ft'\}$. Thanks to the relation \eqref{e:slope} between $f_t^*(x)$ and the slope, we have $f_t^*(E_1(t))>0$ for $0\leq t< t_1$.  Then by the defining relation \eqref{fh} of the complex slope, we have $\nabla H^*(E_1(t), t)=(0,0)$ for  $0\leq t< t_1$. It follows that $H^*$ is constant on the arctic curve $\{(E_1(t), t): 0\leq t< t_1\}$. Moreover on the west boundary of $\fD'$, $H^*(\fa(t), t)\equiv 0$. As a consequence, $H^*\equiv 0$ on the boundary of the region $\{(x,t)\in \fD': 0\leq t\leq t_1, \fa(t)\leq x\leq E_1(t)\}$. We conclude that $H^*\equiv 0$ on $\{(x,t)\in \fD': 0\leq t\leq t_1, \fa(t)\leq x\leq E_1(t)\}$ and $\nabla H^*=(0,0)$. 
\end{rem}

We recall the time slices of the liquid region from \eqref{e:defItIt} and \eqref{e:fLA}
\begin{align}\label{e:defI*_t}
I^*_t\deq \big\{x: (x,t)\in \overline{\fL(\fD') }\big\}=\overline{\{\fa(t)<x<\fb(t): 0<\del_x H(x,t)<1\}},\quad 0\leq t\leq \ft'.
\end{align}
The complement $[\fa(t),\fb(t)]\setminus I^*_t$ consists of several intervals. On each interval we either have $\del_xH_t^*(x,t)\equiv 0$ 
or $\del_xH_t^*(x,t)\equiv 1$.
In the former case, we call the interval a \emph{void region}, and in the latter case we call it a \emph{saturated region}.

\begin{figure}
		
		\begin{center}		
			
			\begin{tikzpicture}[
				>=stealth,
				auto,
				style={
					scale = .52
				}
				]
				
				\draw[black, very thick] (-4.9, 0) node[left, scale = .7]{$y = 0$}-- (3.45, 0)--(3.45,3)--(-1.9,3)node[left, scale = .7]{$y = \ft'$}-- (-4.934877870428601, 0);
				
				\draw[black, very thick] (8.6, 0) node[left, scale = .7]{$y = 0$}-- (18, 0);
				\draw[black, very thick] (8.6, 3) node[left, scale = .7]{$y = \ft'$}-- (18, 3);
				
				\draw[black] (10.4, 2)-- (17.2, 2);
				\draw[] (13.6, 2) node[above, scale = .7]{$I_t^+$};
				
				\draw[black,fill=black] (-2.65,2.25) circle (.1);
				\draw[](-2.7,2.25) node[left, scale=.7]{$(E_1(t_1),t_1)$};
				
				\draw[black,fill=black] (3.45,1.5) circle (.1);
				\draw[](3.45,1.5) node[right, scale=.7]{$(E_2(t_2),t_2)$};
				
				\draw[black,fill=black] (-0.1,1.5) circle (.1);
				\draw[](-0.1,1.5) node[above, scale=.7]{$(E_c,t_c)$};
				
				\draw[](-3.6,0.4) node[right, scale=.7]{$E_1(t)$};
				\draw[](-2,0.4) node[right, scale=.7]{$E_1'(t)$};
				\draw[](0.4,0.4) node[right, scale=.7]{$E_2'(t)$};
				\draw[](1.8,0.4) node[right, scale=.7]{$E_2(t)$};

				\draw[black, thick] (-3.5, 0) arc (180:110:3.2);

				\draw[black, thick] (3, 0) arc (-30:30:3);

				\draw[black, thick] (-1, 0) arc (-60:0:1.732);
				\draw[black, thick] (.75, 0) arc (240:180:1.732);
				
				\draw[black, thick] (10, 0) arc (180:110:3.2);
				\draw[black, dashed] (9.7, 0) arc (180:110:3.2);
				
				\draw[black, thick] (16.5, 0) arc (-30:30:3);
				\draw[black, dashed] (16.8, 0) arc (-30:30:3);

				\draw[black, thick] (12.5, 0) arc (-60:0:1.732);
				\draw[black, thick] (14.25, 0) arc (240:180:1.732);
				\draw[black, dashed] (12.8, 0) arc (-60:-33:1.732);
				\draw[black, dashed] (13.95, 0) arc (240:212:1.732);	
		
				\draw[] (1.5, 1.2)  node[scale=0.7]{$\mathfrak{L}(\fD')$};
				\draw[] (15, 1.2)  node[scale=0.7]{$\mathfrak{L}(\fD')$};
				

			\end{tikzpicture}
			
		\end{center}
		
		\caption{\label{f:It} Shown to the left is the liquid region $\fL(\fD')$, the arctic curve $\fA(\fD')$, the tangency locations $(E_1(t_1), t_1), (E_2(t_2), t_2)$, and the cusp $(E_c,t_c)$. Shown to the right is the enlarged liquid region, with its time slices $I_t^+$ as in \eqref{e:defI_t+1}.}
		
	\end{figure}
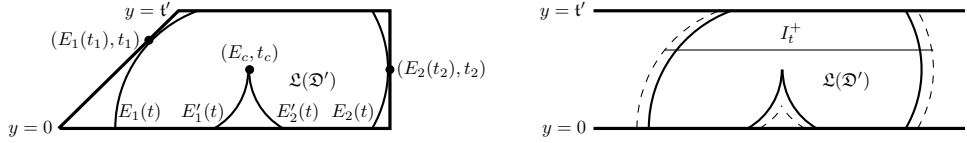

\subsection{Complex Burgers Equation With General Boundary}\label{s:CBE}

We recall $\fD^t$ from \eqref{e:defDs}. We call $\beta: \{x: (x,t)\in \del_{\rm so}\fD^t\}\mapsto [0,\fm]$ a \emph{boundary profile} on the south  boundary $\del_{\rm so}\fD^t$ of $\fD^t$, if it satisfies $\beta(\fa(t))=0$, $\beta(\fb(t))=\fm$ (recall from \eqref{e:defhh}), and it is non-decreasing and is Lipschitz with Lipschitz constant $1$.

In this section, we consider the following variational problem 
on $\fD^t$. Define $\Adm({\fD^t};\beta) = \Adm (\mathfrak{D}^t; \beta, h)$ to be the set of functions $H \in \Adm (\mathfrak{D}^t)$ such that the restriction of $H$ to the south boundary $\del_{\rm so}\fD^t$ is given by $H(x,t)=\beta(x)$, and the restriction of $H$ on the remaining boundary of $\fD^t$ (i.e., $\del\fD^t\setminus \del_{\rm so}\fD^t$) is given by $h$ as defined below \eqref{e:varW}. Then, set
\begin{align}\label{e:varWt}
W_ t ( \beta )=\sup_{ H\in \Adm({\fD^t};\beta)}\int_ t^ {\ft'}\int_\bR\sigma(\nabla  H)\rd x \rd r .
\end{align}
We denote the maximizer of \eqref{e:varWt} by $ H_r(x;\beta,t)$ and the corresponding complex slope by $ f_r( x ;\beta, t)$, defined by
\begin{align}\label{e:ft2}
(\del_{x}  H_r(x;\beta,t), \del_{r}  H_t(x;\beta,t))=\frac{1}{\pi}\left(-\arg^*( f_r( x ;\beta, t)),\arg^*( f_r( x ;\beta, t)+1)\right),
\end{align}
then $ f_r( x ;\beta, t)$ satisfies the complex Burgers equation
\begin{align}\label{e:burgereq22}
 \frac{\del_{r}  f_r( x ;\beta, t)}{ f_r( x ;\beta, t)}+\frac{\del_{x}  f_r( x ;\beta, t)}{ f_r( x ;\beta, t)+1}=0,
\end{align}
for $(x,r)$ in the liquid region
\begin{align}\label{e:defliquid}
\fL({\fD^t};\beta)\deq \big\{ (x,r)\in \cD^t:  \big( \partial_x H^* (x;\beta,t), \partial_r H_r^* (x;\beta,t) \big) \in \mathcal{T},  t<r<\ft' \big\}.
\end{align}

 
Similarly to \eqref{e:xrmap}, we consider the following map from the liquid region $\fL(\fD^t;\beta)$ to $\mathbb{H}^-\times \bH^+$,
\begin{align}\label{e:idliq}
(x,r)\in \fL(\fD^t;\beta)\mapsto (f,z)=\left(f_r( x ;\beta, t),  x - (r-t)\frac{   f_r( x ;\beta, t)}{ f_r( x ;\beta, t)+1}\right) \in \mathbb{H}^-\times \bH^+,
\end{align}
We further project the map \eqref{e:idliq}  to the $z$ coordinate:
\begin{align}\label{e:proj}
(x,r)\in \fL(\fD^t;\beta) \mapsto z=x -(r-t)\frac{   f_r( x ;\beta, t)}{ f_r( x ;\beta, t)+1} \in \bH^+.
\end{align} 

Thanks to Item (c) in the third statement in \Cref{xhh},  for $\beta(x)=H_t^*(x)$,  the  boundary of the liquid region $\fL(\fD^t;\beta)$ consists of three parts: the north boundary, the south boundary, and the arctic curve boundary
\begin{align}\label{e:boundaryxtc}
\big\{(x,\ft')\in  \overline{\fL({\fD^t};\beta)} \big\}\cup \big\{ (x,t)\in \overline{\fL({\fD^t};\beta)} \big\}\cup \overline{\big\{ (x,r)\in \del \fL({\fD^t};\beta): r\in(t,\ft') \big\}}.
\end{align}
We will prove in \Cref{p:fdecompose} below that, for $\beta(x)$ sufficiently close to the height profile $H_t^*(x)$,  
the decomposition \eqref{e:boundaryxtc} is also true. In particular the arctic boundary $\fA(\fD^t;\beta)= \overline{\big\{ (x,r)\in \del \fL({\fD^t};\beta): r\in(t,\ft') \big\}}$. Moreover, $f_r(x;\beta,t)$ extends continuously to the boundary of $\fL(\fD^t;\beta)$ and  the projection map \eqref{e:proj} is a bijection. See Figure \ref{f:riemannS} for an illustration. We establish Proposition \ref{p:fdecompose} for $\beta(x)=H_t^*(x)$ in Appendix \ref{s:initialEs} and the general case in Appendix \ref{s:proof} below.

\begin{prop}\label{p:fdecompose}
Fix any $0\leq t\leq \ft$, and a boundary profile $\beta$ on the south boundary $\del_{\rm so}\fD^t$ of $\fD^t$. Let $m(z;\beta,t)$ be the Stieltjes transform of the boundary profile $\beta$,
\begin{align}\label{e:mta}
m(z;\beta,t)=\int_{\fa(t)}^{\fb(t)} \frac{\del_{x} \beta(x)}{z- x }\rd  x.
\end{align}
Assume that the boundary profile $\beta$ is sufficiently close to $H^*_t$, that is, $d(\del_x \beta(x), \del_x H_t^*(x))\leq \fc$, for some small constant $\mathfrak{c} > 0$. Then the following statements hold. 

\begin{enumerate} 
	
	\item \label{11} 
	The projection map \eqref{e:proj} extends to the  boundary \eqref{e:boundaryxtc} of the liquid region $\mathfrak L(\fD^t;\beta)$, and it maps
the north boundary to a curve in the upper half-plane, and the remaining boundary is mapped bijectively to an interval in $\bR$. The map \eqref{e:proj} and its complex conjugate together give a bijection from two copies of the liquid region $\fL(\fD^t, \beta)$, glued along the arctic curve, with a domain $\mathscr U_t^{\beta}\subseteq \bC$. 

\item  \label{22} We can interpret $f$ in \eqref{e:idliq} as a function of $z\in \mathscr U_t^{\beta}$, $f=f(z;\beta,t)$. For $x\in \mathscr U_t^{\beta}\cap \bR$ with $x<\fa(t)$, $f(x;\beta,t)\in (-\infty, -1)$ and for $x\in \mathscr U_t^{\beta}\cap \bR$ with $x>\fb(t)$, we have $f(x;\beta,t)\in (-1,0)$. The following quantity is uniformly bounded away from $0$,
		\begin{align}\label{e:imratio0}
			\Bigg| \displaystyle\frac{1}{\Imaginary z} \cdot  \Imaginary \bigg[ \displaystyle\frac{f(z;\beta,t)}{f(z;\beta,t)+1} \bigg] \Bigg| \gtrsim 1,\quad \text{for $z\in \mathscr U_t^{\beta}$}.
		\end{align}

\item \label{33} 
On $\mathscr U_t^{\beta}$, 
 $f(z;\beta,t)$ has a decomposition 
\begin{align}\label{e:gszmut}
 f(z;\beta, t)=e^{m(z;\beta, t)}g (z;\beta, t)=e^{\tilde m(z;\beta, t)}\tilde g (z;\beta, t),
\end{align}
where 
\begin{align}\begin{split}\label{e:deftgt}
\tilde m(z;\beta,t) & \deq m(z;\beta,t)-\log \big(z-\fa(t) \big)+\log \big(\fb(t)-z \big),\\
\log  \tilde g(z;\beta,t) & \deq\log g(z;\beta,t)+\log \big(z-\fa(t) \big)-\log \big(\fb(t)-z \big).
\end{split}\end{align}
The function $\tilde m(z;\beta,t)$ is the Stieltjes transform of the measure $\del_{x}  \beta+\bm1_{(-\infty,\fa(t)]}+\bm1_{[\fb(t),\infty)}$. The function $\tilde g (z;\beta, t)$ is real analytic on $\tilde {\mathscr U}_t^{\beta}:= \mathscr U_t^{\beta}\cup \{x: (x,t)\in \overline{\fL(\fD^t;\beta)}\}$, i.e.,
$\tilde g(\bar z;\beta, t)=\overline{\tilde g(z;\beta, t)}$. Moreover, it is positive on $\tilde {\mathscr U}_t^{\beta}\cap \bR$, and it does not have zeros in $\tilde {\mathscr U}_t^{\beta}$.
The function $g(z;\beta, t)$ is meromorphic on $\tilde {\mathscr U}_t^{\beta}$. It is positive on $(\fa(t), \fb(t))$, has a simple pole at $\fa(t)$ and a simple zero at $\fb(t)$, with no other poles or zeros.

\item \label{44} For any $z\in \tilde{\mathscr U}_t^{\beta}$ we have the following perturbation formula
\begin{align}\label{e:lngsbound}
\big| \del_z^k\log g(z;\beta,t)-\del_z^k\log g(z;H_t^*,t) \big|\lesssim \oint_{\omega}  \big| m(w;\beta,t)-m(w;H^*_t,t) \big|  |\rd w|,\quad 0\leq k\leq 2,
\end{align}
where $\omega$ is a contour enclosing $[\fa(t),\fb(t)]$, and the implicit constant is uniform in $\beta$ satisfying $d(\del_x \beta, \del_x H_t^*(x)) < \mathfrak{c}$.

\end{enumerate}

\end{prop}

We recall the continuum limit height function $H_t^*(x)$ associated with the trapezoid domain $\fD'$ from \eqref{e:varW}. 
 If we take  $\beta(x)=H_t^*(x)$  in \eqref{e:varWt}, the solution of the variational principle \eqref{e:varWt} coincides with that of \eqref{e:varW}, and we define
\begin{align}\label{e:limits}
 \rho_t^*(x):=\del_x H_t^*(x),\quad f(z;H_t^*,t)=f_t^*(z).
\end{align}
We recall the liquid region 
$\fL(\fD')=\{(x,t)\in \fD': t\in[0,\ft'],0<\del_x H_t^*(x)<1\}$. Via the map \eqref{e:proj},  we identify the gluing of two copies of $\fL(\fD')\cap \fD^t$ along the arctic curve with a domain $\mathscr U_t:=\mathscr U_t^{H_t^*}$. The complex slope $f_t^*(x)$ extends to $f_t^*(z)$ on $\mathscr U_t$. 
By \Cref{p:fdecompose}, and we can decompose $f_t^*(z)$ as
\begin{align}\label{e:decompft*}
\log f_t^*(z)=m_t^*(z)+\log g_t^*(z),\quad g_t^*(z)=g(z;H_t^*,t), \quad m_t^*(z)=\int \frac{\rho_t^*(x)}{z-x}\rd x.
\end{align}
With this notation, the complex Burgers equation from \eqref{e:Burgeq} can be written as
\begin{align}\label{e:newBB}
\del_t m^*_t(z)+\del_t \log g^*_t(z)+\del_z \log (f_t^*(z)+1)=0.
\end{align}
Similarly to \eqref{e:deftgt}, we can also write \eqref{e:decompft*} as
\begin{align}\label{e:decompft*2}
\log f_t^*(z)=\tilde m_t^*(z)+\log\tilde g_t^*(z),
\end{align}
where $\tilde m_t^*=m_t^*-\log(z-\fa(t))+\log(\fb(t)-z)$, and $\log  \tilde g^*_t=\log  g^*_t+\log \big( z-\fa(t) \big)-\log \big( \fb(t)-z \big)$ is analytic in $\tilde {\mathscr U}_t:=\tilde {\mathscr U}_t^{H_t^*}$ and does not have zeros. Thus, $|\tilde g_t^*(z)|$ is uniformly bounded away from $0$ and $\infty$ on $\tilde {\mathscr U}_t$, namely, $|\tilde g_t^*(z)|\asymp 1$.

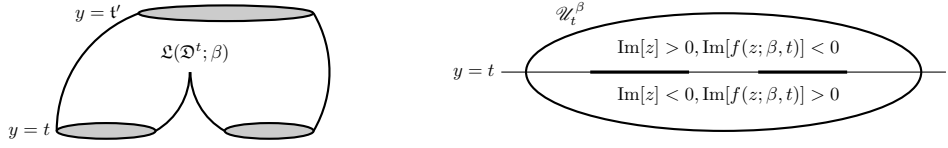
\begin{figure}
		
		\begin{center}		
			
			\begin{tikzpicture}[
				>=stealth,
				auto,
				style={
					scale = .52
				}
				]
				
				\draw [black,thick, fill=gray!40!white](0.78,	3) ellipse (2.21 and 0.2) ;
				\draw [] (-1.7, 3) node[left, scale=.7]{$y=\ft'$};


				\draw [black,thick, fill=gray!40!white] (-2.25,	0) ellipse (1.25 and 0.2);
				\draw [black,thick, fill=gray!40!white] (1.875,	0) ellipse (1.125 and 0.2);

				\draw[black, thick] (-3.5, 0) arc (180:110:3.2);
				\draw [] (-3.5, 0) node[left, scale=.7]{$y=t$};
				
				\draw[black, thick] (3, 0) arc (-30:30:3);

				\draw[black, thick] (-1, 0) arc (-60:0:1.732);
				\draw[black, thick] (.75, 0) arc (240:180:1.732);
				\draw[] (0, 2) node[scale=0.7]{$\fL(\fD^t;\beta)$};
				
				\draw[black] (7.75, 1.5) node[left, scale = .7]{$y = t$}-- (19, 1.5);
				\draw[black, very thick] (10, 1.5) -- (12.5, 1.5);
				\draw[black, very thick] (14.25, 1.5) -- (16.5, 1.5);
				\draw[black, thick] (13.375,1.5) ellipse (5 and 1.5);

				\draw[] (9.5, 3)  node[ scale=0.7]{$\mathscr U_t^{\beta}$};
				\draw[] (13.5, 2.5)  node[below, scale=0.7]{$\Im[z]>0, \Im[f(z;\beta,t)]<0$};
				\draw[] (13.5, 1.3)  node[below, scale=0.7]{$\Im[z]<0, \Im[f(z;\beta,t)]>0$};

			\end{tikzpicture}
			
		\end{center}
		
		\caption{\label{f:riemannS} Shown to the left is the gluing of two copies of the liquid regions $\fL(\fD^t;\beta)$;
Shown to the right is an illustration of the map \eqref{e:proj} and its complex conjugate. It maps the gluing of two copies of the liquid region $\fL(\fD^t;\beta)$ bijectively to a domain $\mathscr U_t^{\beta}$.
		}
		
	\end{figure}

In the following proposition, we collect some quantitative estimates for  the extended complex slope $f_t^*(z)$ around the arctic boundary, under  \Cref{a:xhh}.  We postpone its proof to Appendix \ref{QEquation} below. 

\begin{prop}\label{p:ftzbehave}
Under  \Cref{a:xhh}, we have 
\begin{align}\label{e:ft+1b2}
\big| f_t^*(z)+1 \big|\gtrsim 1, \quad \left|\frac{f^*_t(z)}{f_t^*(z)+1}\right|\lesssim1.
\end{align}

\noindent Moreover, the following holds with a sufficiently small constant $\fc>0$.
\begin{enumerate}
\item \label{1} If $\big( E(t),t \big)\in \fA(\fD')$ is bounded away from cusp points, in a small neighborhood of $E(t)$, then the quantity $f_t^*(z)/(f_t^*(z)+1)$ exhibits the square root behavior
\begin{align}\label{e:squareeq}
\frac{f_t^*(z)}{f_t^*(z)+1}=\frac{f_t^*(E(t))}{f_t^*(E(t))+1}+\sqrt{\fC \big( z-E(t) \big)}+\OO \big( |z-E(t)|^{3/2} \big),\quad \big| z-E(t) \big| \leq \fc,
\end{align}
where $|\fC|\asymp 1$ may depend on $t$, and for $\Imaginary z>0$, the square root has negative imaginary part. If we further assume that $\big( E(t),t \big)\in \fA(\fD')$ is bounded away from tangency locations, then 
\begin{align}\label{e:mtsquare}
-\sin(\Im[m_t^*(z)])
\asymp -\sin(\Im[\log f_t^*(z)]) \asymp
\left\{
\begin{array}{cc}
\frac{\eta}{\sqrt{|\kappa|+\eta}}, & \kappa\geq 0,\\
\sqrt{|\kappa|+\eta}, & \kappa \leq 0,
\end{array}
\right.
\end{align}
where $\eta=\Imaginary z\geq 0$, and $\kappa=\dist(\Re z, I_t^*)$ if $\Re z \not\in I_t^*$; otherwise, $\kappa=-\dist(\Re z, \bR\setminus I_t^*)$.

\item \label{2} Around the tangency location $\big( E_1(t_1),t_1 \big)$, we have $\big| \fa(t)-E_1(t) \big|\asymp |t_1-t|^2$. If $(z,t)=(E_1(t)-\kappa+\ri\eta,t)$ is in a neighborhood of the tangency location $\big( E_1(t_1), t_1 \big)$, i.e., $\big| z-E_1(t) \big|, |t-t_1|\leq \fc$, then for $t\leq t_1$ the interval $[ \fa(t), E(t) ]$ is a void region,
\begin{align}\label{e:mtbehave2copy}
-\arg^* f_t^*(z)\asymp\left\{
\begin{array}{cc}
\frac{\eta}{\sqrt{|\kappa|+\eta}(\sqrt{|\kappa|+\eta}+|t-t_1|)} +\arg(z-\fa(t)), & \kappa\geq 0,\\
\frac{\sqrt{|\kappa|+\eta}}{\sqrt{|\kappa|+\eta}+|t-t_1|}, & \kappa \leq 0,
\end{array}
\right.
\end{align}
and
\begin{align}\label{e:mtbehave1copy}
-\Im[m_t^*(z)]
\asymp
\left\{
\begin{array}{cc}
\frac{\eta}{\sqrt{|\kappa|+\eta}(\sqrt{|\kappa|+\eta}+|t-t_1|)}, & \kappa\geq 0,\\
\frac{\sqrt{|\kappa|+\eta}}{\sqrt{|\kappa|+\eta}+|t-t_1|}, & \kappa \leq 0.
\end{array}
\right.
\end{align}
For $t\geq t_1$, the interval $[\fa(t), E(t)]$ is a saturated region, and 
\begin{align}\label{e:mtbehave3copy}
\pi+\Im[\tilde m_t^*(z)]
\asymp \pi+\arg^*f_t^*(z)\asymp
\left\{
\begin{array}{cc}
\frac{\eta}{\sqrt{|\kappa|+\eta}(\sqrt{|\kappa|+\eta}+|t-t_1|)}, & \kappa\geq 0,\\
\frac{\sqrt{|\kappa|+\eta}}{\eta/\sqrt{|\kappa|+\eta}+|t-t_1|}, & \kappa \leq 0.
\end{array}
\right.
\end{align}
The analogous statement holds for $(z,t)= \big( E_2(t)+\kappa+\ri\eta,t \big)$  in a neighborhood of the tangency location $(E_2(t_2), t_2)$.
\item  \label{3}
Around the cusp singularity $(E_c,t_c)$, let $c(t)=E_c+(t-t_c) f_{t_c}^*(E_c)/ \big( f_{t_c}^*(E_c)+1 \big)$. If  $0\leq t_c-t\leq \fc$, then we have $E_2'(t)-E_1'(t)\asymp (t_c-t)^{3/2}$.  For $\big| z-E_1'(t) \big| \leq \fc(t_c-t)^{3/2}$, we have 
\begin{align}\label{e:cubiceq}
\frac{f_t^*(z)}{f_t^*(z)+1}=\frac{f_t^*(E_1'(t))}{f_t^*(E_1'(t))+1}+\frac{\sqrt{\fC(z-E'_1(t))}}{(t_c-t)^{1/4}} +\OO\left((t_c-t)^{1/4}|z-E_1'(t)|^{1/2}+\frac{|z-E_1'(t)|}{t_c-t}\right),
\end{align}
where $|\fC|\asymp 1$, and for $\Imaginary z>0$, the square root has negative imaginary part. It follows that
\begin{align}\label{e:mtcubic}
-\sin(\Im[m_t^*(z)])
\asymp -\sin(\Im[\log f_t^*(z)]) \asymp
\left\{
\begin{array}{cc}
\frac{\eta}{\sqrt{|\kappa|+\eta}(t_c-t)^{1/4}}, & \kappa\geq 0,\\
\frac{\sqrt{|\kappa|+\eta}}{(t_c-t)^{1/4}}, & \kappa \leq 0. 
\end{array}
\right.
\end{align}

\noindent The analogous statement holds for $|z-E_2'(t)|\leq \fc(t_c-t)^{3/2}$. 

If $z$ is farther from the interval $[E_1'(t), E_2'(t)]$, namely, $\fc(t_c-t)^{3/2}\leq \dist \big(z, [E_1'(t), E_2'(t)] \big) \leq \fc$, then $f_t^*(z)/ \big( f_t^*(z)+1 \big)$ exhibits the cube root behavior
\begin{align}\label{e:localft3}
-\Im \bigg[ \displaystyle\frac{f^*_t(z)}{f^*_t(z)+1} \bigg] \asymp \big| z-c(t) \big|^{1/3}.
\end{align}

\end{enumerate}
\end{prop}

	\begin{rem}
		
		\label{derivativeh} 
		
		Suppose $u = (x, t) \in \mathfrak{L}(\fD')$ is bounded away from a cusp or tangency location of $\mathfrak{A}(\fD')$; let $d_u = \inf \big\{ |x - x_0| : (x_0, t) \in \mathfrak{A}(\fD') \big\}$. Then, the fact that $f^*_t (x_0) \in \mathbb{R}$ if $(x_0, t) \in \mathfrak{A}(\fD')$ from \Cref{t:localr}, the first statement of \Cref{p:ftzbehave}, and \eqref{fh} together imply that there exists a small constant $\fc > 0$ such that $\fc d_u^{1/2} < \big| \partial_x H^* (x, t) - \partial_x H^* (x_0, t) \big| < \fc^{-1} d_u^{1/2}$.
		
	\end{rem}

\section{Weighted Non-intersecting Bernoulli Bridges}\label{s:nbb}
In this section we first explain the correspondence between lozenge tilings of double-sided trapezoid domains and non-intersecting Bernoulli walk ensembles in Section \ref{s:NBwe}. In that way, we can interpret random lozenge tilings of double-sided trapezoid domains as uniformly random non-intersecting Bernoulli walks (or bridges) with specified starting and ending data.
%
%
Then we introduce a weighted version of non-intersecting Bernoulli bridges, and use it to study tilings (equivalently, height functions) on the double-sided trapezoid domain $\fD$ on the  strip $[0, \ft]$, under Assumptions \ref{a:xhh} and \ref{xhh2}, in Section \ref{s:NBBdefine}. Then \Cref{estimategamma} will follow from concentration estimates for this weighted non-intersecting Bernoulli bridge model.

\subsection{Non-intersecting Bernoulli walk ensembles}\label{s:NBwe}

	A  \emph{Bernoulli walk} is a walk $\mathsf{q} = \big( \mathsf{q} (0), \mathsf{q} (1), \ldots , \mathsf{q} (\sft) \big) \in \mathbb{Z}^{\sft +1}$ such that $\mathsf{q} (r + 1) - \mathsf{q} (r) \in \{ 0, 1 \}$ for each $r \in \qq{0, \sft - 1}$; viewing $r$ as a time index, $\big( \mathsf{q} (r) \big)$ denotes the space-time trajectory for a Bernoulli walk, which may either not move or jump to the right at each step. For this reason, the interval $\qq{0, \sft}$ is called the \emph{time span} of the Bernoulli walk $\mathsf{q}$, and a step $(r,r + 1)$ of this Bernoulli walk may be interpreted as an ``non-jump'' or a ``right-jump'' if $\mathsf{q} (r + 1) = \mathsf{q} (r)$  or $\mathsf{q} (r + 1) = \mathsf{q} (r)+1$, respectively. A family of Bernoulli walks $\mathsf{Q} = \big( \mathsf{q}_1, \mathsf{q}_{2}, \ldots , \mathsf{q}_m \big)$ is called \emph{non-intersecting} if $\mathsf{q}_i (r) < \mathsf{q}_j (r)$ whenever $1 \leq i < j \leq m$ and $0\leq r\leq \sft$. 
	
	We recall the double-sided trapezoid domain $\fD$ from \eqref{d}, and the boundary height function $h$ satisfying \eqref{e:defhh}. Now let $n \geq 1$ be an integer; denote $\mathsf{t} = \mathfrak{t} n$, and $m=\fm n$ (recall from \eqref{e:defhh}). Suppose $m\in \bZ_>0$, $\mathsf{D} = n \mathfrak{D} \subset \mathbb{T}$, so that $ \mathsf{t} \in \mathbb{Z}_{> 0}$. Let $\mathsf{h}: \partial \mathsf{D} \rightarrow \mathbb{Z}$ denote a boundary height function satisfying Assumption \ref{xhh2}. It is necessary that $\sfh=0$ on $\del_{\rm we}\sfD$ and $\sfh=m$ on $\del_{\rm es}\sfD$. 
	
	Now take any height function $\mathsf{H}: \mathsf{D} \rightarrow \mathbb{T}$ corresponding to a tiling $\mathscr{M}$ of $\mathsf{D}$ with boundary height function $\sfh$. We may interpret $\mathscr{M}$ as a family of $m$ particle non-intersecting Bernoulli walks by first omitting all type $1$ lozenges from $\mathscr{M}$, and then viewing any type $2$ or type $3$ tile as a right-jump or non-jump of a Bernoulli walk, respectively; see \Cref{walksfigure} for a depiction. 
	
	It will be useful to set more precise notation on this correspondence. Recalling that $\mathsf{H}$ is extended by linearity to the faces of $\mathsf{D}$, $\partial_x \mathsf{H} (x, t) \in \{ 0, 1 \}$ for all $(x, t)$. In particular, there exist integers $\mathsf{x}_{1} (t) < \mathsf{x}_{2} (t) < \cdots < \mathsf{x}_{m} (t)$ such that 
	\begin{flalign}\label{e:gobetween}
		\partial_x \mathsf{H} (x, t) = \displaystyle\sum_{i = 1}^{m} \bm1 \Big( x \in \big[ \mathsf{x}_i (t), \mathsf{x}_i (t) + 1 \big] \Big),
	\end{flalign}

	\noindent which are those such that $\mathsf{H} \big( \sfx_i (t) + 1, t \big) = \mathsf{H} \big( \mathsf{x}_i (t), t \big) + 1$.  This defines a Bernoulli walk $\mathsf{x}_i = \big( \mathsf{x}_i (t) \big)$, and a non-intersecting ensemble of Bernoulli walks $\mathsf{X} = \big( \mathsf{x}_1, \mathsf{x}_{2}, \ldots , \mathsf{x}_m \big)$.

	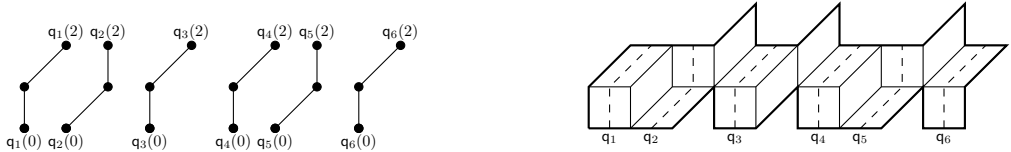
\begin{figure}	
		\begin{center}		
			\begin{tikzpicture}[
				>=stealth,
				auto,
				style={
					scale = .55
				}
				]
				
				\draw[-, black]  (1.5, 2) -- (1.5, 3) -- (2.5, 4) ; 
				\draw[-, black] (2.5, 2) -- (3.5, 3) -- (3.5, 4); 
				\draw[-, black] (4.5, 2) -- (4.5, 3) -- (5.5, 4);
				\draw[-, black] (6.5, 2) -- (6.5, 3) -- (7.5, 4) ; 
				\draw[-, black]  (7.5, 2) -- (8.5, 3) -- (8.5, 4);
				\draw[-, black]  (9.5, 2) -- (9.5, 3) -- (10.5, 4) ;
				
				\filldraw[fill=black] (1.5, 2) circle [radius = .1] node[below , scale = .6]{$\mathsf{q}_{1} (0)$};	
				\filldraw[fill=black] (2.5, 2) circle [radius = .1] node[ below, scale = .6]{$\mathsf{q}_{2} (0)$};	
				\filldraw[fill=black] (4.5, 2) circle [radius = .1] node[below, scale = .6]{$\mathsf{q}_3 (0)$};	
				\filldraw[fill=black] (6.5, 2) circle [radius = .1] node[below, scale = .6]{$\mathsf{q}_4 (0)$};	
				\filldraw[fill=black] (7.5, 2) circle [radius = .1] node[below, scale = .6]{$\mathsf{q}_5 (0)$};	
				\filldraw[fill=black] (9.5, 2) circle [radius = .1] node[below, scale = .6]{$\mathsf{q}_6 (0)$};

				\filldraw[fill=black] (1.5, 2) circle [radius = .1];
				\filldraw[fill=black] (2.5, 2) circle [radius = .1];
				\filldraw[fill=black] (4.5, 2) circle [radius = .1];
				\filldraw[fill=black] (6.5, 2) circle [radius = .1];
				\filldraw[fill=black] (7.5, 2) circle [radius = .1];
				\filldraw[fill=black] (9.5, 2) circle [radius = .1];
				
				\filldraw[fill=black] (1.5, 3) circle [radius = .1];
				\filldraw[fill=black] (3.5, 3) circle [radius = .1];
				\filldraw[fill=black] (4.5, 3) circle [radius = .1];
				\filldraw[fill=black] (6.5, 3) circle [radius = .1];
				\filldraw[fill=black] (8.5, 3) circle [radius = .1];
				\filldraw[fill=black] (9.5, 3) circle [radius = .1];
				
				\filldraw[fill=black] (2.5, 4) circle [radius = .1]node[above , scale = .6]{$\mathsf{q}_{1} (2)$};
				\filldraw[fill=black] (3.5, 4) circle [radius = .1]node[above , scale = .6]{$\mathsf{q}_{2} (2)$};
				\filldraw[fill=black] (5.5, 4) circle [radius = .1]node[above , scale = .6]{$\mathsf{q}_{3} (2)$};
				\filldraw[fill=black] (7.5, 4) circle [radius = .1]node[above , scale = .6]{$\mathsf{q}_{4} (2)$};
				\filldraw[fill=black] (8.5, 4) circle [radius = .1]node[above , scale = .6]{$\mathsf{q}_{5} (2)$};
				\filldraw[fill=black] (10.5, 4) circle [radius = .1]node[above , scale = .6]{$\mathsf{q}_{6} (2)$};

				\draw[-, black, thick] (15, 2) -- (15, 3) -- (16, 4) -- (17, 4) -- (18, 4) -- (19, 5) -- (19, 4) -- (20, 4)  -- (21, 5) -- (21, 4) -- (23, 4)  -- (24, 5)--(24,4) -- (25, 4) -- (24, 3) -- (24, 2)  -- (23, 2) --(23,3)--(22,2)-- (20, 2) --(20,3)-- (19, 2) -- (18, 2) -- (18, 3) -- (17, 2) -- (15, 2);
				
				\draw[-, black, dashed] (15.5, 2)node[below, scale = .6]{$\mathsf{q}_{1}$}  -- (15.5, 3) -- (16.5, 4); 
				\draw[-, black, dashed] (16.5, 2) node[below , scale = .6]{$\mathsf{q}_{2}$} -- (17.5, 3) -- (17.5, 4); 
				\draw[-, black, dashed] (18.5, 2) node[ below , scale = .6]{$\mathsf{q}_3$} -- (18.5, 3) -- (19.5, 4);
				\draw[-, black, dashed] (20.5, 2) node[below, scale = .6]{$\mathsf{q}_4$}-- (20.5, 3) -- (21.5, 4) ; 
				\draw[-, black, dashed]  (21.5, 2)node[below, scale = .6]{$\mathsf{q}_5$} -- (22.5, 3) -- (22.5, 4);
				\draw[-, black, dashed]  (23.5, 2)node[below, scale = .6]{$\mathsf{q}_6$} -- (23.5, 3) -- (24.5, 4);
				\draw[-, black] (18, 3) -- (18, 4);
				\draw[-, black] (16, 2) -- (17, 3) -- (17, 4) ; 
				\draw[-, black] (18, 2) -- (18, 3) -- (19, 4);
				\draw[-, black]  (20, 2) -- (20, 3) -- (21, 4) ; 
				\draw[-, black] (21, 2) -- (22, 3) -- (22, 4) ;
				\draw[-, black ] (23, 2) -- (23, 3) -- (24, 4);

				\draw[-, black] (15, 3) -- (16, 3); 
				\draw[-, black] (16, 4) -- (17, 4); 
				\draw[-, black] (16, 3) -- (17, 4);
				
				\draw[-, black] (17, 2) -- (18, 3) -- (18, 2);
				
				\draw[-, black] (20, 4) -- (19, 4);
				\draw[-, black] (20, 4) -- (19, 3) -- (19, 2);
				\draw[-, black] (18, 2) -- (19, 2);
				\draw[-, black] (18, 3) -- (19, 3);
				\draw[-, black] (19, 2) -- (20, 3) -- (20, 4);
				\draw[-, black] (21, 2) -- (21, 3) -- (22, 4) ;
				\draw[-, black]  (22, 2) -- (23, 3) -- (23, 4) ;

				\draw[-, black] (20, 2) -- (22, 2);
				\draw[-, black] (20, 3) -- (21, 3);
				\draw[-, black] (21, 4) -- (23, 4);
				\draw[-, black] (22, 3) -- (23, 3);

				\draw[-, black] (17, 3) -- (18, 3);
				\draw[-, black] (24, 3) -- (25, 4);
				\draw[-, black] (23, 2) -- (24, 2); 
				\draw[-, black] (23, 3) -- (24, 3); 
				\draw[-, black] (24, 4) -- (25, 4);
				\draw[-, black] (16, 2) -- (16, 3);		
			\end{tikzpicture}	
		\end{center}
		\caption{\label{walksfigure} Depicted to the left is an ensemble $\mathsf{Q} = \big( \mathsf{q}_{1}, \mathsf{q}_{2}, \mathsf{q}_3, \mathsf{q}_4, \mathsf{q}_5, \mathsf{q}_6 \big)$ consisting of six non-intersecting Bernoulli walks. Depicted to the right is an associated lozenge tiling. }		
	\end{figure}

	To proceed, we require some additional notation on non-intersecting Bernoulli walk ensembles. Let $\mathsf{X} = (\mathsf{x}_1, \mathsf{x}_{2}, \ldots , \mathsf{x}_m)$ denote a family of non-intersecting Bernoulli walks, each with time span $\qq{0,\sft}$, so that $\mathsf{x}_j = \big( \mathsf{x}_j (0), \mathsf{x}_j (1), \ldots , \mathsf{x}_j (\sft) \big)$ for each $j \in \qq{1, m}$. Given functions $\mathsf{f}, \mathsf{g}: \qq{0, \sft} \rightarrow \mathbb{R}$, we say that $\mathsf{X}$ has $(\mathsf{f}; \mathsf{g})$ as a \emph{boundary condition} if $\mathsf{f} (r) \leq \mathsf{x}_j (r) \leq \mathsf{g} (r)$ for each $r \in \qq{0, \sft}$. We refer to $\mathsf{f}$ and $\mathsf{g}$ as a \emph{left boundary} and \emph{right boundary} for $\mathsf{X}$, respectively, and allow $\mathsf{f}$ and $\mathsf{g}$ to be $-\infty$ or $\infty$. We further say that $\mathsf{X}$ has \emph{entrance data} $\mathsf{d} = (\mathsf{d}_1, \mathsf{d}_{2}, \ldots , \mathsf{d}_m)$ and \emph{exit data} $\mathsf{e} = (\mathsf{e}_1, \mathsf{e}_{2}, \ldots , \mathsf{e}_m)$ if $\mathsf{x}_j (0) = \mathsf{d}_j$ and $\mathsf{x}_j (\sft) = \mathsf{e}_j$, for each $j \in \qq{1, m}$. Then, there is a finite number of non-intersecting Bernoulli walk ensembles with any given entrance and exit data $(\mathsf{d}; \mathsf{e})$ and (possibly infinite) boundary conditions $(\mathsf{f}; \mathsf{g})$.

	The following lemma from \cite{LSRT} provides a monotonicity property for non-intersecting Bernoulli walk ensembles randomly sampled under the uniform measure on the set of such families with prescribed entrance, exit, and boundary conditions. In what follows, for any functions $\mathsf{f}, \mathsf{f}': \qq{0,\sft} \rightarrow \mathbb{R}$ we write $\mathsf{f} \leq \mathsf{f}'$ if $\mathsf{f} (r) \leq \mathsf{f}' (r)$ for each $r \in \qq{0,\sft}$. Similarly, for any sequences $\mathsf{d} = (\mathsf{d}_1, \mathsf{d}_{2}, \ldots , \mathsf{d}_m) \subset \mathbb{R}$ and $\mathsf{d}' = (\mathsf{d}_1', \mathsf{d}_{2}', \ldots , \mathsf{d}_m') \subset \mathbb{R}$, we write $\mathsf{d} \leq \mathsf{d}'$ if $\mathsf{d}_j \leq \mathsf{d}_j'$ for each $j \in \qq{1,m}$.

	\begin{lem}[{\cite[Lemma 18]{LSRT}}]
		
		\label{comparewalks}
		
		Fix integers $\sft,m\geq 1$; functions $\mathsf{f}, \mathsf{f}, \mathsf{g}, \mathsf{g}' : \qq{0,\sft} \rightarrow \mathbb{R}$; and $m$-tuples $\mathsf{d}, \mathsf{d}', \mathsf{e}, \mathsf{e}'$ with coordinates indexed by $\qq{1,m}$. Let $\mathsf{X} = (\mathsf{x}_1, \mathsf{x}_{2}, \ldots , \mathsf{x}_m)$ denote a uniformly random non-intersecting Bernoulli walk ensemble with boundary data $(\mathsf{f}; \mathsf{g})$; entrance data $\mathsf{d}$; and exit data $\mathsf{e}$. Define $\mathsf{X}' = (\mathsf{x}_1', \mathsf{x}_{2}', \ldots , \mathsf{x}_m')$ similarly, but with respect to $(\mathsf{f}'; \mathsf{g}')$ and $ (\mathsf{d}'; \mathsf{e}')$. If $\mathsf{f} \leq \mathsf{f}'$, $\mathsf{g} \leq \mathsf{g}'$, $\mathsf{d} \leq \mathsf{d}'$, and $\mathsf{e} \leq \mathsf{e}'$, then there exists coupling between $\mathsf{X}$ and $\mathsf{X}'$ such that $\mathsf{x}_j \leq \mathsf{x}_j'$ almost surely, for each $j \in \qq{1,m}$.
		
	\end{lem}

	\begin{rem}
		
		\label{heightcompare}
		
		An equivalent way of stating \Cref{comparewalks} (as was done in \cite{LSRT}) is through the height functions associated with the Bernoulli walk ensembles $\mathsf{X}$ and $\mathsf{X}'$. Specifically, let $\mathsf{D} \subset \mathbb{T}$ be a finite domain, and let $\mathsf{h}, \mathsf{h}' : \partial \mathsf{H} \rightarrow \mathbb{Z}$ denote two boundary height functions such that $\mathsf{h} (v) \geq \mathsf{h}' (v)$, for each $v \in \partial \mathsf{D}$. Let $\mathsf{H}, \mathsf{H}' : \mathsf{D} \rightarrow \mathbb{Z}$ denote two uniformly random height functions on $\mathsf{D}$ with boundary data $\mathsf{H} |_{\partial \mathsf{D}} = \mathsf{h}$ and $\mathsf{H}' |_{\partial \mathsf{D}} = \mathsf{h}'$. Then, \Cref{comparewalks} implies (and is equivalent to) the existence of a coupling between $\mathsf{H}$ and $\mathsf{H}'$ such that $\mathsf{H} (\mathsf{u}) \geq \mathsf{H}' (\mathsf{u})$ almost surely, for each $\mathsf{u} \in \mathsf{D}$. 
		
	\end{rem}

\subsection{Weighted Non-intersecting Bernoulli Bridges}\label{s:NBBdefine}
In this section, we introduce a weighted measure on non-intersecting Bernoulli bridges, which corresponds to non-intersecting Bernoulli bridges with random boundary data at the final time $\ft$.

To implement this, in the rest of this paper, we denote $\bZ_n=\{\cdots, -2/n, -1/n, 0,1/n, 2/n,\cdots\}$ to be the set of integers rescaled by $n$, and rescale the non-intersecting Bernoulli walk ensemble \eqref{e:gobetween} (both time and space) by a factor $n$. Then time and particle locations are all indexed by $\bZ_n$.

A \emph{particle configuration} is an increasing sequence $\bmx=(x_1<x_2< \cdots<x_m) \in \bZ_n^m $. We encode any such particle configuration by the empirical measure 
\begin{align}\label{e:emprho}
\rho(x; \bmx)=\sum_{i=1}^m \bm1(x\in [x_i, x_i+1/n]).
\end{align}
For any $t\in [0, \ft)\cap \bZ_n$ and particle configuration $\bmx\subseteq [ \fa(t), \fb(t)-1/n]$, we construct a boundary profile $\beta(x)=\beta(x;\bmx)$ corresponding to $\bmx$ with the empirical measure \eqref{e:emprho}:
\begin{align}\label{e:defbx}
\beta(x;\bmx)= \int_{\fa(t)}^ x  \rho( x;\bmx )\rd  x,\quad \text{meaning} \quad \del_x\beta(x;\bmx)=\rho(x;\bmx).
\end{align}
By construction it satisfies $\beta(x_i+1/n;\bmx)=i/n$ for $1\leq i\leq m$.

 We will choose the weights in a way that the height functions of weighted non-intersecting Bernoulli bridges at time $\ft$ concentrate around $H_{\ft}^*(x)$. Then \Cref{estimategamma} will follow from concentration estimates for this weighted non-intersecting Bernoulli bridge model.
We define
\begin{align}\begin{split}\label{e:defYt+}
&\phantom{{}={}}Y_{t}(\bmx)
\deq W_t(\beta(x;\bmx))-\frac{1}{2}\int \log |x-y|\rho(x;\bmx)\rho(y;\bmx)\rd x\rd y\\
&\qquad \qquad \qquad +\int \int_{\fa(t)}^x \log(x-y)\rho(x;\bmx)\rd y\rd x 
+\int\int_x^{\fb(t)}\log (y-x)\rho(x;\bmx)\rd y\rd x,
\end{split}\end{align}

\noindent For any $\bmx=(x_1,x_2,\cdots, x_m)\subseteq \big[ \fa(\ft), \fb(\ft) \big]$, we then define the partition function
\begin{align}\label{e:deftN}
N_{\ft}(\bmx)=\frac{V(\bmx)e^{n^2Y_{\ft}(\bmx)}}{\prod_{j=1}^m \Gamma_n \big(x_j-\fa(\ft) \big) \Gamma_n \big( \fb(\ft)-x_j-1/n \big)},
\end{align}
where $V(\bmx)=\prod_{i<j} (x_j-x_i)$ is the Vandermonde determinant, $\Gamma_n(k/n)=(k/n)((k-1)/n)\cdots (1/n)=k!/n^k$, and for any $\bmx\not \subseteq \big[ \fa(\ft), \fb(\ft) \big]$, we simply set $N_{\ft}(\bmx)=0$.
Then for any $t\in[0,\ft)\cap \bZ_n$, and $\bmx=(x_1,x_2,\cdots, x_m)\subseteq [ \fa(t), \fb(t)]$, $N_t(\bmx)$ is defined recursively as
\begin{align}\label{e:recurss}
N_{t}(\bmx)=\sum_{\bme\in\{0,1\}^m} N_{t+1/n}(\bmx+\bme/n).
\end{align}

\noindent We then define a family of weighted non-intersecting Bernoulli bridge $\big\{\bmx_t=(x_1(t), x_2(t), \cdots, x_m(t)) \big\}_{t\in [0,\ft]\cap \bZ_n}$ whose transition probability is given by 
\begin{align}\label{e:wNBB}
\bP(\bmx_{t+1/n}=\bmx+\bme/n|\bmx_t=\bmx)=\frac{N_{t+1/n}(\bmx+\bme/n)}{N_t(\bmx)},\quad \bme\in \{0,1\}^m,
\end{align}
for any $t\in [0, \ft)\cap \bZ_n$ and $\bmx\subseteq [\fa(t), \fb(t)]$. This weighted model is one for non-intersecting Bernoulli bridges with random boundary data at time $\ft$. Indeed, given any initial particle configuration $\bmx_0\subseteq [ \fa(0), \fb(0)]$, it can be sampled in two steps. First we sample the boundary configuration $\bmx_\ft$ using \eqref{e:deftN}:
\begin{align}\label{e:sampleNt}
\bP(\bmx_\ft)=\frac{\#\{\text{non-intersecting Bernoulli bridges from $\bmx_0$ to $\bmx_\ft$}\} \cdot N_\ft(\bmx_\ft)}{\sum_{\bmx}\#\{\text{non-intersecting Bernoulli bridges from $\bmx_0$ to $\bmx$} \} \cdot N_\ft(\bmx)}.
\end{align}
 Then we sample a non-intersecting Bernoulli bridge with boundary data $\bmx_0, \bmx_\ft$ uniformly from all possible non-intersecting Bernoulli bridges from $\bmx_0$ to $\bmx_\ft$. 


The recursion \eqref{e:recurss} can be solved using the Feynman--Kac formula. The transition probability \eqref{e:wNBB} is approximated through the following proposition. Its proof is essentially the same as in \cite[Section 9]{NRWLT}; we outline it in Appendix \ref{s:ansatz} below.

\begin{prop}\label{p:formula}
Adopt  \Cref{a:xhh}.
Given any particle configuration $\bmx \subseteq [ \fa(t), \fb(t)]$ with density $\rho (x; \boldsymbol{x})$, define the boundary profile $\beta$ by $\del_x \beta=\rho(x;\bmx)$. Assume that $\beta$ is close to the continuum limiting profile at time $t$, i.e., $d \big( \rho(x;\bmx), \del_x H_t^*(x) \big) \leq \fc$ for some small constant $\fc$. Then the transition probability \eqref{e:wNBB} is given by
\begin{align}\label{e:defLtnew}
 \frac{1}{Z_t(\bmx)}\frac{V(\bmx+\bme/n)}{V(\bmx)}\prod_{i=1}^m\phi^+(x_i;\beta,t)^{e_i}\phi^-(x_i;\beta,t)^{1-e_i} \exp \Bigg( \displaystyle\sum_{1 \leq i,j \leq m} \frac{e_ie_j}{n^2}\kappa(x_i,x_j;\beta,t)+\OO(1/n) \Bigg),
\end{align}
where 
\begin{align}\label{e:defvphit}
\phi^+(z;\beta,t)= \frac{\fb(t)-1/n-z}{\fb(t)-z}\varphi^+(z;\beta,t)e^{\frac{1}{n}\psi_t^+(z;\beta,t)},\quad 
\phi^-(z;\beta,t)=\varphi^-(z;\beta,t).
\end{align}
and
\begin{align}\label{e:defvarphi}
\varphi^-(z;\beta,t)=z-\fa(t), \quad \varphi^+(z;\beta,t)= \tilde g(z;\beta,t) \big( \fb(t)-z \big),
\end{align}
where $\tilde g(z;\beta,t)$ is from \eqref{e:deftgt}; $\psi^+(z;\beta,t)$ and $\kappa(z,z';\beta,t)$ are analytic in $z,z'$ and uniformly bounded in a neighborhood of $[ \fa(t), \fb(t)]$; and the error term of size $\OO(1/n)$ uniform in $\beta$. 
\end{prop}

\subsection{Optimal Rigidity for Weighted Non-intersecting Bernoulli Bridges}\label{s:bridgeop}
We denote the height function for the weighted non-intersecting Bernoulli bridge model as $H(x,t)=H_t(x)$, so that
\begin{align}\label{e:defHt}
H(x,t)=H_t(x)=\int_{\fa(t)}^x \rho(y;\bmx_t)\rd y,\quad \del_x H(x,t)=\rho(x;\bmx_t)=\sum_{i=1}^m \bm1([x_i(t),x_i(t)+1/n]),
\end{align}
where $\bmx_t= \big( x_1(t), x_2(t), \cdots, x_m(t) \big)$.

To understand the height fluctuations around the boundary of the liquid region, we need an enlarged version of 
$I^*_t$ as defined in \eqref{e:defI*_t}. We recall that under \Cref{a:xhh}, for $t_c\leq t\leq \ft'$, the slice $I^*_t$ is a single interval
$
I^*_t= \big[ E_1(t), E_2(t) \big],
$
and for $0\leq t\leq t_c$, the slice $I^*_t$ consists of two intervals
$
I^*_t= \big[ E_1(t), E'_1(t) \big] \cup \big[ E'_2(t), E_2(t) \big].
$ 
In the rest of this section, we fix an arbitrarily small constant $\fd>0$. For any $\big( E(t),t \big)= \big( E_i(t),t \big)$ on the arctic curve $\fA(\fD')$ with $i \in \{ 1,2 \}$, we define the distance function 
\begin{align}\label{e:disf}
\tau \big( E(t),t \big)=
n^{-2/3+6\fd}|t-t_i|^{2/3}\vee n^{ -1+10\fd}.
\end{align}

\noindent Moreover, for any $\big( E(t),t \big)= \big( E'_i(t), t \big)$ on the arctic curve $\fA(\fD')$ with $i \in \{ 1,2 \}$ and $t\leq t_c$,  we define the distance function
\begin{align*}
\tau \big( E(t),t \big)=
n^{-2/3+6\fd}(t_c-t)^{1/6}. 
\end{align*}
We then define the enlarged intervals  
\begin{align}\label{e:defI_t+1}
	\begin{aligned}
I_t^+ = \big[ E_1  (t) - \tau (E_1 (t), t), E_2 (t) + \tau (E_2 (t), t) \big], & \qquad \text{for $t_c \leq t \leq \mathfrak{t}'$}; \\
I_t^+= \big[ E_{1}(t)-\tau(E_{1}(t), t), E'_{1}(t)+\tau(E'_{1}(t), t) \big] \cup \big[ E'_{2}(t)-\tau(E'_{2}(t), t), E_{2}(t)+\tau(E_{2}(t), t) \big], & \qquad \text{for $0 \leq t < t_c$}.
\end{aligned} 
\end{align}
We remark that from the third statement in \Cref{p:ftzbehave}, $E_2'(t)-E_1'(t)\asymp (t_c-t)^{3/2}$. Thus for $t\geq t_c-n^{-1/2+6\fd}$, \eqref{e:defI_t+1} reduces to a single interval, 
\begin{align}\label{e:defI_t+2}
&I_t^+= \Big[E_{1}(t)-\tau \big( E_{1}(t), t \big), E_{2}(t)+\tau \big(E_{2}(t), t \big) \Big].
\end{align}
See Figure \ref{f:It} for an illustration.

In Sections \ref{s:pcW}, we will prove the following concentration result on the height function for the weighted non-intersecting Bernoulli bridge model \eqref{e:wNBB}. We remark that our proof applies to any Markov process with transition probability in the form \eqref{e:defLtnew}.
\begin{prop}\label{p:rigidityBB}
Fix an arbitrarily small constant $\fd>0$, and adopt  \Cref{a:xhh} and \Cref{xhh2}.
For any initial data $\bmx_0\subseteq \big[ \fa(0),\fb(0) \big]$ with the corresponding height function $H_0(x)$ satisfying 
\begin{align}\label{e:boundaryc0}
	\begin{aligned}
\big|H_0(x)- H_0^*(x) \big| \leq n^{\mathfrak{d}/3 - 1},\qquad & \text{if $\dist(x, I_0^*)\leq n^{\fd/3-2/3}$}; \\
H_0(x)= H_0^*(x),\qquad & \text{if $\dist(x, I_0^*)\geq n^{\fd/3-2/3}$},
\end{aligned}
\end{align}

\noindent with overwhelming probability the height function $H(x,t)$ for the weighted non-intersecting Bernoulli bridge \eqref{e:wNBB} satisfies, for $0\leq t\leq \ft$, 
\begin{align}\label{e:Optrigidity}
	\begin{aligned}
\big| H(x,t)- H^*(x , t) \big| \leq n^{3\fd - 1},\qquad \text{if $x\in I_t^+$}, \\
H(x,t)= H^*(x ,  t ),\qquad \text{if $x\not\in I_t^+$}.
\end{aligned} 
\end{align}
\end{prop}

Given this, we can quickly prove \Cref{estimategamma}. 

 \begin{proof}[Proof of \Cref{estimategamma}]
We recall that the weighted non-intersecting Bernoulli random walks can be sampled in two steps. First we sample the boundary configuration $\bmx_\ft$ using \eqref{e:sampleNt}. Then we sample a non-intersecting Bernoulli bridge with boundary data given by $\bmx_0$ and $\bmx_\ft$. 
 
We fix the initial data $\bmx_0$ corresponding to boundary height function $\sfh$ as in \Cref{xhh2} restricted to ${\del_{\so}\fD}$. And we take $\mathfrak{d} =\delta/15$ in \Cref{p:rigidityBB}, then \Cref{xhh2} implies \eqref{e:boundaryc0}. By the $t = \mathfrak{t}$ case of \eqref{e:Optrigidity}, a boundary configuration $\bmx_\ft$ sampled from \eqref{e:sampleNt} satisfies with overwhelming probability the bound 
\begin{align}\label{e:boundaryc2}
\big| \beta(x;\bmx_\ft)- H^*(x , \ft ) \big| \leq n^{3\fd- 1},\quad \text{if $x\in I_\ft^+$;} \qquad \qquad \beta(x;\bmx_\ft)= H^*(x ,  \ft ),\quad \text{if $x\not\in I_\ft^+$}.
\end{align}
Moreover, for the (fixed) boundary data $\bmx_0, \bmx_\ft$, the height functions $H(x,t)$ for non-intersecting Bernoulli bridges between $\bmx_0$ and $\bmx_\ft$ satisfy the optimal rigidity estimates \eqref{e:Optrigidity} for $0 \leq t \leq \mathfrak{t}$. In the following, we fix a sample of such boundary configuration $\bmx_\ft$. 

The boundary height function estimates \eqref{e:boundaryc0} and \eqref{e:boundaryc2} imply that 
\begin{align}\label{e:betaxt}
\beta(x;\bmx_t)-(n^{3\fd - 1} +n^{\delta/2 - 1}) \leq n^{-1} \sfh(nx,nt) \leq \beta(x;\bmx_t)+ (n^{3\fd - 1}+n^{\delta/2 - 1}),\quad t\in\{0,\ft\}. 
\end{align}
 By \Cref{comparewalks} (and \Cref{heightcompare}), we may couple $H$ and $\mathsf{H}$ so that with overwhelming probability we have
\begin{align}
  H(x,t)- (n^{3 \mathfrak{d} - 1} + n^{\delta / 2 - 1}) \leq  n^{-1} \sfH(nx,nt) \leq  H(x,t) + (n^{3 \mathfrak{d} - 1}+n^{\delta/2 - 1}),
\end{align}
for each $(x,t) \in \fD$. Thus, we conclude from this coupling and \eqref{e:Optrigidity} that with overwhelming probability
 \begin{align*}
 \big| n^{-1} \sfH(nx,nt)- H^*(x ,  t ) \big|
 \leq \big| n^{-1} \sfH(nx,nt)-  H(x ,  t ) \big|+ \big| H(x,t)- H^*(x ,  t ) \big|\leq n^{\delta/2 - 1}+2n^{3\fd - 1} \leq n^{\delta - 1}, 
 \end{align*}
for any $ (x,t)\in \fD$, provided that $\delta > 4 \mathfrak{d}$.
This gives the first statement in \Cref{estimategamma}.

For the second statement in \Cref{estimategamma}, we need to show that $n^{-1} \sfH(nx,nt)=H^*(x,t)$ for any $(x,t)\in \fD\setminus \fL_+^\delta(\fD)$. For any time $0\leq t\leq \ft$, the time slice $\{x: (x,t)\in \fD\setminus \fL_+^\delta(\fD)\}$ consists of several intervals. On each interval either $\del_x H^*(x,t)\equiv 0$, or $\del_x H^*(x,t)\equiv 1$, see \Cref{r:replace}. We denote by $J_t$ one such interval, with $\del_x H^*(x,t)\equiv 0$ corresponding to a void region, and prove  $n^{-1} \sfH(nx,nt)= H^*(x ,  t )$ for $x\in J_t$. The proof in the other case when $\del_x H^*(x,t)\equiv 1$ on $J_t$ is very similar and thus omitted.

In the following we first show that 
\begin{align}\begin{split}\label{e:heightshift}
&\beta\left(x-\frac{n^{7\fd}+n^{\delta/2}}{n^{2/3}};\bmx_t\right)\leq n^{-1} \sfh(nx,nt) \leq \beta\left(x+\frac{n^{7\fd}+n^{\delta/2}}{n^{2/3}};\bmx_t\right),\quad t\in \{0,\ft\}.
\end{split}\end{align}
We prove \eqref{e:heightshift} for $t=\ft$; the case $t=0$ is essentially the same.
The measure $\del_x H^*(x,\ft)=\rho^*_\ft(x)$ is supported on $I_\ft^*=[E_1(\ft),E_2(\ft)]$. Thanks to \Cref{derivativeh}, it has square root behavior close to the edges $E_1(\ft)$ and $E_2(\ft)$. In particular, for $x\leq E_2(\ft)-\fC(n^{2\fd}+n^{\delta/3})/n^{2/3}$, \eqref{e:betaxt} and this square root behavior together imply
\begin{align}\label{e:betalow1}
\beta(x;\bmx_\ft)\leq H^*(x,\ft) + n^{3\fd - 1}+n^{\delta/2 - 1} \leq H^*\left(x+\frac{\fC(n^{2\fd}+n^{\delta/3})}{n^{2/3}},\ft\right),
\end{align}
provided the constant $\fC$ is large enough. For $x\geq E_2(\ft)-\fC(n^{2\fd}+n^{\delta/3})/n^{2/3}$, the second statement of \eqref{e:boundaryc2} implies that
\begin{align}\label{e:betalow2}
\beta(x;\bmx_\ft)
\leq \beta\left(x+\frac{\fC(n^{2\fd}+n^{\delta/3}+n^{6\fd})+n^{\delta/2}}{n^{2/3}};\bmx_\ft\right)=H^*\left(x+\frac{\fC(n^{2\fd}+n^{\delta/3}+n^{6\fd})+n^{\delta/2}}{n^{2/3}} ,  \ft\right ).
\end{align}
The lower bound in \eqref{e:heightshift} follows from combining \eqref{e:betalow1} and \eqref{e:betalow2}, and recalling that $n^{-1} \sfh(nx,n\ft)  = H^*(x,\ft)$. The upper bound in \eqref{e:heightshift} can be proven in the same way. With \eqref{e:heightshift} as input, the coupling  \Cref{heightcompare} gives
\begin{align}\begin{split}\label{e:couplingfrozen}
& H\left(x-\frac{n^{7\fd}+n^{\delta/2}}{n^{2/3}},t\right)\leq n^{-1} \sfH(nx,nt) \leq H\left(x+\frac{n^{7\fd}+n^{\delta/2}}{n^{2/3}},t\right),
\end{split}\end{align}
for any $(x,t)\in \fD$. For $x\in J_t$, $H^*(x,t)$ is a constant. This, with the second statement of \eqref{e:Optrigidity}, together imply that with overwhelming probability the left and right sides of \eqref{e:couplingfrozen} coincide, and both equal $H^*(x,t)$, provided we take $\delta>14\fd$. Thus, the second statement in \Cref{estimategamma} follows.
\end{proof}

\section{Optimal Concentration for Weighted Non-intersecting Bernoulli Bridges}\label{s:pcW}
In this section we prove the concentration estimate for the weighted non-intersecting Bernoulli bridge model, \Cref{p:rigidityBB}. 
We recall the weighted Bernoulli bridges from Section \ref{s:NBBdefine}.  The weighted Bernoulli bridges with initial data given by the particle configuration $\bmx_0=\big\{ x_1(0), x_2(0),\cdots, x_{m}(0) \big\} \subseteq \big[ \fa(0),\fb(0) \big]$ is a Markov process $\big\{ \bmx_t=(x_1(t), x_2(t), \cdots, x_m(t)) \big\}_{t\in[0, \ft]\cap \bZ_n}$ with transition probability given by \eqref{e:defLtnew} 
\begin{align}\begin{split}\label{e:transitDE}
&\phantom{{}={}}\bP(\bmx_{t+1/n}=\bmx+\bme/n|\bmx_t=\bmx) =:a_t(\bme;\bmx)\\
&= \frac{1}{Z_t(\bmx)}\frac{V(\bmx+\bme/n)}{V(\bmx)}\prod_{i=1}^m\phi^+(x_i;\beta,t)^{e_i}\phi^-(x_i;\beta,t)^{1-e_i}  \exp \Bigg(\sum_{1 \leq i,j \leq m} \frac{e_ie_j}{n^2}\kappa(x_i,x_j;\beta,t)+\OO(1/n) \Bigg),
\end{split}\end{align}
where the height profile $\beta$ is determined by $\boldsymbol{x}_t = \boldsymbol{x}$ through $\del_x \beta=\rho(x;\bmx)$.
We will study the dynamics of the  Stieltjes transform of its empirical particle density
\begin{align}\begin{split}\label{e:defrhot}
&\rho(x;\bmx_t)=\del_x H_t(x)=\sum_{i=1}^{m}\bm1 \Big( x\in \big[x_i(t), x_i(t)+1/n\big] \Big),\\
& m_t(z)=\int \frac{\rho(x;\bmx_t)}{z-x}\rd x=\int^{z}_{z-1/n}\sum_{i=1}^{m}\frac{1}{u-x_i(t)} \rd u.
\end{split}\end{align}
We recall the continuum limit height function $H^*_t(x)$, the limiting particle density $\rho_t^*(x)$ and its Stieltjes transform from \eqref{e:limits} and \eqref{e:decompft*}
\begin{align}
\rho_t^*(x)=\del_x H_t^*(x),\quad m_t^*(z)=\int\frac{\rho_t^*(x)}{z-x}\rd x.
\end{align}

Concentration for the height function will eventually follow from  showing that the difference of $m_t(z)$ and $m_t^*(z)$
\begin{align}\label{e:defDt}
\Delta_t(z)=m_t(z)-m_t^*(z),
\end{align}
is small on a certain spectral domain. In the Section \ref{s:De} we construct this spectral domain. We collect some preliminary estimates in Section \ref{s:PreR}. In Section \ref{s:DEq} we derive dynamical equations for $\Delta_t(z)$. Finally in Section \ref{s:OptimalRigidity} we analyze the dynamical equations for $\Delta_t(z)$ and prove \Cref{p:rigidityBB}.

\subsection{Spectral Domain}\label{s:De}

We recall the discussion after \eqref{e:limits}. Two copies of $\fL(\fD')\cap \fD^t$ glued along the arctic curve is mapped bijectively to the domain $\mathscr U_t$ via \eqref{e:proj}. For any $0\leq t\leq \ft$ and $0\leq r\leq \ft'-t$, we define the domain ${\mathscr D}_t(r)$ as
\begin{align}\begin{split}\label{e:defDtr}
 \Bigg\{ x+ (t - s) \displaystyle\frac{f^*_s(x)}{f^*_s(x)+1}\in \mathscr U_t:  (x,s)\in \fL(\fD'), t+r\leq s\leq \ft' \Bigg\},
\end{split}\end{align}
together with its complex conjugate. Under the identification of the liquid region $\fL(\fD')\cap \fD^t$ with $\mathscr U_t$ via \eqref{e:proj}, $\mathscr D_t(r)$ corresponds to the region $\{(x,s)\in \fL(\fD'): t+r\leq s\leq \ft'\}$.  See Figure \ref{f:Dtr} for an illustration.

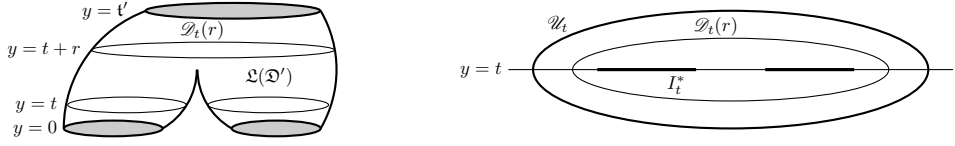
\begin{figure}
		
		\begin{center}		
			
			\begin{tikzpicture}[
				>=stealth,
				auto,
				style={
					scale = .52
				}
				]
				
				\draw [black,thick, fill=gray!40!white](0.78,	3) ellipse (2.21 and 0.2) ;
				\draw [] (-1.7, 3) node[left, scale=.7]{$y=\ft'$};
				\draw [black](0.275,	2) ellipse (3.075 and 0.2);
				\draw[] (-2.9, 2) node[left, scale = .7]{$y=t+r$};
				\draw[] (0, 2.1) node[above, scale = .7]{$\mathscr D_t(r)$};

				\draw [black,thick, fill=gray!40!white] (-2.25,	0) ellipse (1.25 and 0.2);
				\draw [black,thick, fill=gray!40!white] (1.875,	0) ellipse (1.125 and 0.2);
				
				\draw [black] (-1.9,	0.6) ellipse (1.5 and 0.2);
				\draw [black] (1.675,	0.6) ellipse (1.525 and 0.2);
				\draw[] (-3.5, 0.6) node[left, scale = .7]{$y=t$};

				\draw[black, thick] (-3.5, 0) arc (180:110:3.2);
				\draw [] (-3.5, 0) node[left, scale=.7]{$y=0$};
				
				\draw[black, thick] (3, 0) arc (-30:30:3);

				\draw[black, thick] (-1, 0) arc (-60:0:1.732);
				\draw[black, thick] (.75, 0) arc (240:180:1.732);
				\draw[] (1.7, 1.3) node[scale=0.7]{$\mathfrak{L}(\fD')$};
				
				\draw[black] (7.75, 1.5) node[left, scale = .7]{$y = t$}-- (19, 1.5);
				\draw[black, very thick] (10, 1.5) -- (12.5, 1.5);
				\draw[black, very thick] (14.25, 1.5) -- (16.5, 1.5);
				\draw[black, thick] (13.375,1.5) ellipse (5 and 1.5);
				
				\draw[black] (13.375,1.5) ellipse (4 and 0.8);

				\draw[] (13, 2.2)  node[above, scale=0.7]{$\mathscr D_t(r)$};
				\draw[] (12, 1.5)  node[below, scale=0.7]{$I^*_t$};
				\draw[] (9, 2.3)  node[above, scale=0.7]{$\mathscr U_t$};

			\end{tikzpicture}
			
		\end{center}
		
		\caption{\label{f:Dtr} 
		Shown to the left is the gluing of two copies of the liquid region $\fL(\fD')$ (with $s=0$ and $\beta=H_0^*(x)$), which can be identified with $\mathscr U_t$ via the map \eqref{e:proj}. Under this identification, $\mathscr D_t(r)$ as defined in \eqref{e:defDtr} corresponds to the region $\{(x,s)\in \fL(\fD'): t+r\leq s\leq \ft'\}$; shown to the right is an illustration of $\mathscr D_t(r)$ inside $\mathscr U_t$.
		}
		
	\end{figure}

In the rest of this section, we fix a small constant $0<\fr<\ft'-\ft$.
Since the spectral domain $\mathscr D_t(\fr)$ is bounded away from $I_t^*$, we will see that the behavior of $\Delta_t(z)$ for $z\in\mathscr D_t(\fr)$ is relatively easy to understand.
To understand the Stieltjes transform close to $I_t^*$, we define the \emph{spectral domains} (observe that we exclude $\mathscr D_t(\fr)$ from them):
\begin{align}\begin{split}\label{e:bulkdomain}
{\mathscr D}_t^{L}=
\Big\{ z\in \mathscr U_t: n |  \Imaginary [z]\Im [f_t^*]/f_t^*|\geq n^{ 2\fd}/2 \Big\}\setminus {\mathscr D}_t(\fr).
\end{split}\end{align}
We recall that under \Cref{a:xhh}, for $t_c\leq t\leq \ft'$, the slice $I^*_t$ is a single interval
$
I^*_t= \big[ E_1(t), E_2(t) \big],
$
and for $0\leq t\leq t_c$, the slice $I^*_t$ consists of two intervals
$
I^*_t= \big[ E_1(t), E'_1(t) \big] \cup \big[ E'_2(t), E_2(t) \big].
$ 
We define ${\mathscr D}_t^{\rm F}$ as the union of  
\begin{align}\begin{split}\label{e:edgedomain1}
&\phantom{\bigcup{}} \Big\{z\in \mathscr U_t: \min\big\{(n\Im z)^2, n\sqrt{\dist(z, I_t^*) |\Im z|} \big\} \cdot |\Im [f_t^*]/f_t^*|\geq n^{ 2\fd},\Re z \leq E_1(t)-\tau \big( E_1(t),t \big) \Big\}\\
&\cup \Big\{z\in \mathscr U_t: \min \big\{(n\Im z)^2, n\sqrt{\dist(z, I_t^*) |\Im z|} \big\} \cdot |\Im [f_t^*]/f_t^*|\geq n^{ 2\fd}, \Re z \geq E_{2}(t)+\tau \big(E_2(t),t \big) \Big\}
\end{split}\\
\begin{split}
& \cup\bm1(t\leq t_c-n^{-1/2+\fd})  \bigg\{z\in \mathscr U_t:  \min \big\{(n\Im z)^2, n\sqrt{\dist(z, I_t^*) |\Im z|} \big\} \cdot |\Im f_t^*/f_t^*|\geq n^{2 \fd}, \\
&\qquad \qquad \qquad \qquad \qquad \qquad \qquad \quad \phantom{\bigcup{}}\Re z\in \Big[E'_{1}(t)+\tau \big( E'_1(t),t \big), E_{2}'(t)+\tau \big(E'_2(t), t \big) \Big] \bigg\}\setminus {\mathscr D}_t(\fr).\label{e:edgedomain2}
\end{split}\end{align}

\noindent Estimates of Stieltjes transform on ${\mathscr D}_t^{\rm L}$ will give the information of particles inside the liquid region, and the estimates on ${\mathscr D}_t^{\rm F}$ will be used to control the location of particles in the frozen region. We refer to Figure \ref{f:D_t} for a depiction.

\begin{rem}
In context of Wigner matrices \cite{DRT}, the spectral domains are given in the bulk by $\big\{ z\in \bH^+: \Im[m_{sc}(z)]\Im[z]\gg 1/n \big\}$ and at the edge by $\big\{ z\in \bH^+: \Im[m_{sc}(z)]\sqrt{\Im[z]\dist(z, [-2,2])}\gg 1/n \big\}$, where $m_{sc}$ denotes the Stieltjes transform of the semi-circle distribution. In our construction, $\big| \Im [f_t^*] / f_t^* \big| =-\sin \big( \arg^*( f_t^*) \big)$ plays the same role as the Stieltjes transform of the semi-circle distribution; this quantity $| \Imaginary f_t^*/f_t^*|$ will also take into account the singular behaviors of $f_t^*$ close to the tangency locations.
\end{rem}

Thanks to \Cref{p:ftzbehave}, we have explicit estimates for $\big| \Im[f_t^*(z)]/f_t^*(z) \big| =-\sin \big( \arg^*( f_t^*(z)) \big)$ close to the arctic curve. We can rewrite the condition \eqref{e:bulkdomain} as $\big| \sin(\arg^*( f_t^*(z))) \big| \geq n^{2\fd-1}/|\Im z|$. For $z$ in the domain \eqref{e:edgedomain1}, say $z=E_1(t)-\kappa+\ri\eta \in \mathbb{H}^+$, \eqref{e:mtbehave2copy} gives
\begin{align}\label{e:f/fedge}
\left|\frac{\Im f_t^*(z)}{f_t^*(z)}\right|\asymp \frac{\eta}{\sqrt{\kappa+\eta} \big( \sqrt{\kappa+\eta}+|t_1-t| \big)}+\sin\arg \big( z-\fa(t) \big).
\end{align}
For $\kappa+\eta\leq \fc(t_1-t)^2$, we can rewrite the condition in \eqref{e:edgedomain1} as
\begin{align*}
\min\bigg\{(n\eta)^2, n\sqrt{(\kappa+\eta)\eta}\bigg\}\frac{\eta}{\sqrt{\kappa+\eta}|t_1-t|}\gtrsim n^{2\fd},
\end{align*}
and it simplifies to
$\eta\gtrsim \big( n^{2\fd - 1}|t_1-t| \big)^{2/3}$. For $z$ in the domain \eqref{e:edgedomain2}, say $z=E'_1(t)+\kappa+\ri\eta \in \mathbb{H}^+$, with $\kappa+\eta\leq \fc(t_c-t)^{3/2}$, \eqref{e:mtcubic} gives
\begin{align}\label{e:f/fcusp}
\left|\frac{\Im f_t^*(z)}{f_t^*(z)}\right|\asymp \frac{\eta}{(t_c-t)^{1/4}\sqrt{\kappa+\eta}},
\end{align}
and we can rewrite the condition in \eqref{e:edgedomain2} as
$\eta\gtrsim \big( n^{2\fd - 1}|t_c-t|^{1/4})^{2/3}$.

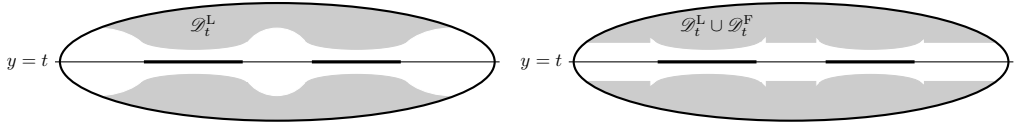
\begin{figure}
		
		\begin{center}		
			
			\begin{tikzpicture}[
				>=stealth,
				auto,
				style={
					scale = .52
				}
				]

				\draw[white, fill=gray!40!white] (0.375,1.5) ellipse (5.5 and 1.5);
				\draw[white, fill=white] (0.3625,1.5) ellipse (0.9 and 0.86 );
				\draw[white, fill=white] (4.8,1.5) ellipse (1.4 and 0.8625);
				\draw[white, fill=white] (-4.2,1.5) ellipse (1.4 and 0.8625);
				
				\draw[white, fill=white] (-3.2,1.0) rectangle ++(2.9,1);
				\draw[white, fill=white] (1.025,1.0) rectangle ++(2.65,1);
				\draw[gray!40!white, fill=gray!40!white] (2.38,2.07) ellipse (1.29 and 0.25);
				\draw[gray!40!white, fill=gray!40!white] (2.38,0.93) ellipse (1.29 and 0.25);
				\draw[gray!40!white, fill=gray!40!white] (-1.73,2.07) ellipse (1.34 and 0.25);
				\draw[gray!40!white, fill=gray!40!white] (-1.73,0.93) ellipse (1.34 and 0.25);
			
				\draw[black, thick] (0.375,1.5) ellipse (5.5 and 1.5);
				\draw[] (-1.5, 2)  node[above, scale=0.7]{$\mathscr D^{\rm L}_t$};
				
				\draw[black] (-5.25, 1.5) node[left, scale = .7]{$y = t$}-- (6, 1.5);
				\draw[black, very thick] (-3, 1.5) -- (-0.5, 1.5);
				\draw[black, very thick] (1.25, 1.5) -- (3.5, 1.5);

				\draw[white, fill=gray!40!white] (13.375,1.5) ellipse (5.5 and 1.5);
				\draw[white, fill=white] (13.3625,1.5) ellipse (0.9 and 0.86 );
				\draw[white, fill=white] (17.8,1.5) ellipse (1.4 and 0.8625);
				\draw[white, fill=white] (8.8,1.5) ellipse (1.4 and 0.8625);
				
				\draw[white, fill=white] (9.8,1.0) rectangle ++(2.9,1);
				\draw[white, fill=white] (14.025,1.0) rectangle ++(2.65,1);
				\draw[gray!40!white, fill=gray!40!white] (15.38,2.07) ellipse (1.29 and 0.25);
				\draw[gray!40!white, fill=gray!40!white] (15.38,0.93) ellipse (1.29 and 0.25);
				\draw[gray!40!white, fill=gray!40!white] (11.27,2.07) ellipse (1.34 and 0.25);
				\draw[gray!40!white, fill=gray!40!white] (11.27,0.93) ellipse (1.34 and 0.25);
				
				\draw[gray!40!white, fill=gray!40!white] (8,2) rectangle ++(1.8,0.5);
				\draw[gray!40!white, fill=gray!40!white] (8,0.5) rectangle ++(1.8,0.5);
				
				\draw[gray!40!white, fill=gray!40!white] (12.75,2) rectangle ++(1.25,0.5);
				\draw[gray!40!white, fill=gray!40!white] (12.75,0.5) rectangle ++(1.25,0.5);
				
				\draw[gray!40!white, fill=gray!40!white] (16.75,2) rectangle ++(1.8,0.5);
				\draw[gray!40!white, fill=gray!40!white] (16.75,0.5) rectangle ++(1.8,0.5);

				\draw[] (11.5, 2)  node[above, scale=0.7]{$\mathscr D^{\rm L}_t\cup \mathscr D^{\rm F}_t$};

				 \fill[white,even odd rule] (13.375,1.5) ellipse (5.5 and 1.5) (13.375,1.5) ellipse (6 and 2.5);
				 \draw[black, thick] (13.375,1.5) ellipse (5.5 and 1.5);

				\draw[black] (7.75, 1.5) node[left, scale = .7]{$y = t$}-- (19, 1.5);
				\draw[black, very thick] (10, 1.5) -- (12.5, 1.5);
				\draw[black, very thick] (14.25, 1.5) -- (16.5, 1.5);

			\end{tikzpicture}
			
		\end{center}
		
		\caption{\label{f:D_t} Shown to the left is an illustration of the spectral domain $\mathscr D_t^{\rm L}$; shown to the right is an illustration of the spectral domain $\mathscr D_t^{\rm L}\cup \mathscr D_t^{\rm F}$.}
		
	\end{figure}

Next, we define the control parameter
\begin{align}\label{e:deftz2}
\fM(t):=M e^{M t}n^{ {\fd/3}},
\end{align}
and the stopping time
\begin{align}\begin{split}\label{stoptime2}
\sigma&\deq \inf_{0\leq t\in \bZ_n}\left\{\exists z\in {\mathscr D}_t(\fr): \left|\Delta_t(z)\right|\geq \frac{\fM(t)}{n}\right\} \wedge\inf_{0\leq t\in \bZ_n}\left\{\exists z\in {\mathscr D}_t^{\rm L}: \left|\Delta_t(z)\right|\geq \frac{M n^{ \fd}}{n\Im z}\right\}\\
& \qquad \wedge\inf_{0\leq t\in \bZ_n}\left\{\exists z\in {\mathscr D}_t^{\rm F} : \left|\Delta_t(z)\right|\geq \frac{ M n^{ \fd}}{n\sqrt{\Im[z]\dist(z, I_t^*)}}+\frac{M n^\fd}{(n\Im z)^2}\right\}\wedge \ft,
\end{split}\end{align}
where the large constant $M>0$ will be chosen later, and $a\wedge b=\min\{a,b\}$.

In the rest of this section, we restrict our analysis to complex numbers $z$ in these spectral domains satisfying $z \in \mathbb{H}^+$; the estimates we show on these quantities will follow for $z \in \mathbb{H}^-$ by symmetry, so we will not comment on this case in the below. We further recall the characteristic $z_t (u)$ from \eqref{e:ccff} and \eqref{zlineart}. Then $\Im[f^*_0(u)]=\Im[f^*_t(z_t(u))]\leq 0$, and $\Im[z_t(u)]=\Imaginary u+t \Im[f_0^*(u)]/|f_0^*(u)+1|^2$ is nonincreasing in $t$.

\begin{prop}\label{p:gap}
	
Adopt the assumptions of \Cref{p:rigidityBB}; fix $u \in \mathbb{H}^+$; and abbreviate $z_s=z_s(u)$ for $s \in [0, \ft]$. Then the following statements hold.

\begin{enumerate} 
	
	\item If $z_t$ is in the domain \eqref{e:edgedomain1}, then for any $0\leq s\leq t$, 
\begin{align}\label{e:gap1}
\dist(z_s, I^+_s) \gtrsim \Im z_s+\sqrt{\dist(z_s, I_s^*)}(t-s).
\end{align}

\item If $z_t$ is in the domain \eqref{e:edgedomain2}, then for any $0\leq s\leq t\leq t_c$, 
\begin{align}\label{e:gap2}
\dist(z_s, I^+_s) \gtrsim \Im z_s +\sqrt{\dist(z_s,I_s^*)}\frac{t-s}{(t_c-s)^{1/4}}.
\end{align}

\item If $z_t \in {\mathscr D}_t(\fr)\cup{\mathscr D}_t^{\rm L}\cup{\mathscr D}_t^{\rm F}$,  then for any $0\leq s\leq t$, we have $z_s \in {\mathscr D}_s(\fr)\cup{\mathscr D}_s^{\rm L}\cup{\mathscr D}_s^{\rm F}$.

\item Recall  $I_t^+$ from \eqref{e:defI_t+1}, and the height function $H_t(x)$ corresponding to the particle configuration $\bmx_t$ from \eqref{e:defHt}. If we further assume that $t\in [0,\sigma)\cap \bZ_n$ then we have 
\begin{align}\label{e:Htbound}
	\begin{aligned}
|H_t(x)-H_t^*(x)|\leq n^{3\fd - 1},\qquad & \text{if $x\in I_t^+$}; \\
H_t(x)=H_t^*(x),\qquad & \text{if $x\not\in I_t^+$}.
\end{aligned}
\end{align}

\end{enumerate} 
\end{prop}
\begin{proof}

We begin by establishing the first statement of the proposition. For $z_t$ in the domain \eqref{e:edgedomain1}, we consider the case when $\Re z_t \leq E_{1}(t)$, for the case when $\Real z_t \geq E_2 (t)$ is addressed very similarly. We denote $E(r)\deq E_1(r)$ and $z_r=E_1(r)-\kappa_r+\ri\eta_r$ for $r \in [s, t]$.
 From \eqref{e:squareeq}, $f^*_r(z)/(f^*_r(z)+1)$ has square root behavior in a small neighborhood of $E(r)$,
\begin{align}\label{e:localft}
\displaystyle\frac{f^*_r(z)}{f^*_r(z)+1} = \displaystyle\frac{f^*_r \big( E(r) \big)}{f^*_r \big(E(r) \big)+1}+\fC\sqrt{E(r)-z}+\OO \big( |z-E(r)| \big),\quad |z-E(r)|\leq \fc,
\end{align}
where $\fC\gtrsim 1$ and the square root is the branch with nonpositive imaginary part. 
Moreover, by \eqref{e:slope}, $E(r)$ satisfies the differential equation 
\begin{align}\label{e:edgeEqn}
\del_r E(r)= \displaystyle\frac{f^*_r \big( E(r) \big)}{f^*_r \big( E(r) \big)+1}.
\end{align}
By taking difference between \eqref{e:edgeEqn} and  the characteristic flow $\del_r z_r=f^*_r(z_r)/ \big( f^*_r(z_r)+1 \big)$, and then taking the real part we get
\begin{align}\label{e:diff}
\del_r \kappa_r
=\Re \bigg[ \displaystyle\frac{f^*_r(E(r))}{f^*_r \big(E(r) \big)+1} - \displaystyle\frac{f^*_r(z_r)}{f^*_r(z_r)+1} \bigg].
\end{align}

If there is some time $s\leq r\leq t$ such that $\Im z_r \geq \fc/2$, then $\Im z_s \geq \Im z_r \geq \fc/2$, and so
\begin{align*}
\dist(z_s, I^+_s)
\geq \Im z_s \gtrsim \Im z_s +\sqrt{\dist(z_s,I_s^*)}(t-s),
\end{align*}

\noindent holds trivially. In the following, we assume that for all $s\leq r\leq t$ we have $\Im z_r \leq \fc/2$. 
If $\big| z_r-E(r) \big| \leq \fc$ and $\kappa_r\geq 0$, then \eqref{e:localft} and \eqref{e:diff} together give
\begin{align}\label{e:diff2}
\del_r \kappa_r
=\Re \bigg[ \displaystyle\frac{f^*_r(E(r))}{f^*_r \big( E(r) \big)+1} - \displaystyle\frac{f^*_r(z_r)}{f^*_r(z_r)+1} \bigg] \lesssim -\sqrt{\eta_r+\kappa_r}.
\end{align}
Since $\kappa_t\geq 0$, we conclude that  $\kappa_r$ will remain positive for $s\leq r\leq t$. In the following we further assume that for all $s\leq r\leq t$, we have $0\leq \kappa_r\leq\fc/2$. Otherwise, we can simply repeat the same argument with $t=r$ such that $\kappa_r= \fc/2$ (if $r=s$, then $\kappa_s \geq \fc/2 \gtrsim 1$, and so \eqref{e:gap1} holds again). Under this assumption, $\big| z_r-E(r) \big|\leq \fc$ and \eqref{e:diff2} holds for $s\leq r\leq t$.
Moreover, the characteristic flow gives (using \eqref{e:mtsquare})
\begin{align}\label{e:diff2a}
\del_s \eta_s
=\Im \bigg[ \displaystyle\frac{f^*_s(z_s)}{f^*_s(z_s)+1} \bigg] \lesssim -\frac{\eta_s}{\sqrt{\kappa_s+\eta_s}}.
\end{align}
By taking sum of \eqref{e:diff2} and \eqref{e:diff2a}, we have
\begin{align}\label{e:diff2b}
\del_s (\kappa_s+\eta_s)\lesssim -\sqrt{\eta_s+\kappa_s}.
\end{align}
The differential equation \eqref{e:diff2b} can be solved directly; it gives
\begin{align}\label{e:diff2c}
\sqrt{\kappa_s+\eta_s}-\sqrt{\kappa_t+\eta_t}\gtrsim t-s,\quad (\kappa_s+\eta_s)-(\kappa_t+\eta_t)\gtrsim (t-s)\sqrt{\kappa_s+\eta_s}.
\end{align}
We also notice that $\kappa_t\geq \tau \big( E(t),t \big)$, since $z_t$ is in the domain \eqref{e:edgedomain1}. Recall from \eqref{e:disf} that if $|t_1-r|\geq n^{6\fd-1/2}$ then we have
$\tau \big(E(r),r \big)=(t_1-r)^{2/3}/n^{2/3- 6 \fd}$; otherwise we have $\tau \big( E(t),t \big)=n^{10\fd-1}$. Thus, for $|t_1-t|, |t_1-s|\geq n^{6\fd-1/2}$
\begin{align}\label{e:diff2d}
\tau \big( E(s),s \big)- \tau \big( E(t),t \big)= \frac{(t_1-s)^{2/3}-(t_1-t)^{2/3}}{n^{2/3-6\fd}}\ll (t-s)\sqrt{\kappa_s+\eta_s}.
 \end{align}
Similarly, for the other cases, one can quickly verify that $\big| \tau(E(s),s)- \tau(E(t),t) \big|\ll(t-s)\sqrt{\kappa_s+\eta_s}$. Hence,
\begin{align}
	\label{distancezi} 
	\begin{aligned}
&\dist(z_s, I_s^+)  \gtrsim \eta_s + \max \Big\{ \eta_s + \kappa_s - \tau \big( E(s), s \big), 0 \Big\} \gtrsim \eta_s + \max \Big\{ \eta_t + \kappa_t - \tau \big( E(s), s \big), 0 \Big\} + (t - s) \sqrt{\kappa_s + \eta_s} \\
& \gtrsim \eta_s + \max \Big\{ \eta_t + \kappa_t - \tau \big( E(t), t \big), 0 \Big\} + (t - s) \sqrt{\kappa_s + \eta_s}  \gtrsim \eta_s+(t-s)\sqrt{\kappa_s+\eta_s},
\end{aligned}
\end{align}

\noindent This gives \eqref{e:gap1}.

We next establish the second statement of the proposition.  For $z_t$ in the domain \eqref{e:edgedomain2}, we have $\Re z_t \in \big[ E_{1}'(t), E'_{2}(t) \big]$. We denote $E(r)\deq E_{1}'(r)$, $E'(r)\deq E_{2}'(r)$ and $z_r=E(r)+\kappa_r+\ri\eta_r$ for any $s\leq r\leq t\leq t_c$. 
Take small $\fc>0$ which will be chosen later. First assume that  $r=\inf\{\tau \in [s, t]: \Imaginary z_\tau \geq \mathfrak{c} (t_c - \tau)^{3/2}/2\}$ exists. Next we show that $\Im z_s \gtrsim (t_c-s)^{3/2}$. If $r=s$, there is nothing to prove. Otherwise $r>s$, then we have $\Im z_r = \fc(t_c-r)^{3/2}/2$ and $\del_r \Im z_r\geq  -(3\fc/4)(t_c-r)^{1/2}$.  From \eqref{e:localft3}, $f^*_r(z)/ \big( f^*_r(z)+1 \big)$ has cube root behavior 
\begin{align}\label{e:localft3copy}
-\Imaginary \bigg[ \displaystyle\frac{f^*_r(z)}{f^*_r(z) + 1} \bigg]\geq C\Im[z]^{1/3},
\end{align}
for $\Imaginary z\geq \fc(t_c-r)^{3/2}/2$. Hence, if $\Im z_r \geq \fc(t_c-r)^{3/2}/2$, then \eqref{e:localft3copy} yields
\begin{align*}
\del_r \eta_r
\leq -C(\fc/2)^{1/3}(t_c-r)^{1/2}.
\end{align*}
This contradicts with $\del_r \Im z_r\geq  -(3\fc/4)(t_c-r)^{1/2}$, provided we take $(\fc/2)^{2/3}\leq (2C/3)$. 
It follows that we always have $\Im z_s \gtrsim (t_c-s)^{3/2}$, and so
\begin{align*}
\dist(z_s, I_s^+)\geq \Im z_s \gtrsim \Im z_s +\sqrt{\dist(z_s, I_s^*)}\frac{t-s}{(t_c-s)^{1/4}},
\end{align*}

\noindent where we have used the fact that $\kappa_s \leq \big| E_1' (s) - E_2' (s) \big| \lesssim (t_c - s)^{3/2}$, by the third statement of \Cref{p:ftzbehave}.

Next, we assume that $\Im[z_r]\leq \fc(t_c-r)^{3/2}$ for $r\in [s,t]$. If $\big| z_r - E(r) \big| \leq \mathfrak{c} (t_c - r)^{3/2}$ for some $r \in [s, t]$, then, from \eqref{e:cubiceq}, $f^*_r(z)/ \big( f^*_r(z)+1 \big)$ has square root behavior with a constant depending on $t_c-r$, namely
\begin{align}\label{e:localft2}
	\frac{f^*_r(z)}{f^*_r(z)+1}=\frac{f^*_r \big( E(r) \big)}{f^*_r \big( E(r) \big)+1}+\frac{\fC\sqrt{z-E(r)}}{(t_c-r)^{1/4}}+\OO\left((t_c-r)^{1/4} \big| z-E(r) \big|^{1/2}+\frac{\big| z-E(r) \big|}{t_c-r}\right),
\end{align}
where the square root is the branch with nonpositive imaginary part. Similarly to \eqref{e:diff2}, we have
\begin{align}\label{e:diff3a}
	\del_r \kappa_r
	=\Re\left[\frac{f^*_r \big( E(r) \big)}{f^*_r \big( E(r) \big)+1}-\frac{f^*_r(z_r)}{f^*_r(z_r)+1}\right]\lesssim -\frac{\sqrt{\eta_r+\kappa_r}}{(t_c-r)^{1/4}}.
\end{align}

\noindent Thanks to \eqref{e:diff3a}, $\kappa_r$ will remain positive for $s\leq r\leq t$. In the following we further assume that for all $s\leq r\leq t$ we have $0\leq \kappa_r\leq \fc(t_c-r)^{3/2}/2$. Otherwise, we can simply repeat the same argument as used above with $t=r$ such that $\kappa_r=\fc(t_c-r)^{3/2}/2$ (if $r=s$, then $\kappa_s \geq (t_c-s)^{3/2}$, and so \eqref{e:gap2} holds again). Under this assumption, we have $\big| z_r-E(r) \big|\leq \fc(t_c-r)^{3/2}$ and \eqref{e:diff3a} holds for each $r \in [s, t]$. Moreover, the characteristic flow gives (by \eqref{e:mtcubic})
\begin{align}\label{e:diff3b}
\del_r \eta_r
=\Im\left[\frac{f^*_r(z_r)}{f^*_r(z_r)+1}\right]\lesssim -\frac{\eta_r}{(t_c-r)^{1/4}\sqrt{\kappa_r+\eta_r}}.
\end{align}
By taking sum of \eqref{e:diff3a} and \eqref{e:diff3b}, we have
\begin{align}\label{e:diff3c}
\del_r (\kappa_r+\eta_r)\lesssim  -\frac{\sqrt{\kappa_r+\eta_r}}{(t_c-r)^{1/4}}.
\end{align}
The differential equation \eqref{e:diff3c} can be solved directly by
\begin{align}\label{e:kskt}
\sqrt{\kappa_s+\eta_s}-\sqrt{\kappa_t+\eta_t}\gtrsim (t_c-s)^{3/4}-(t_c-t)^{3/4},\quad (\kappa_s+\eta_s)-(\kappa_t+\eta_t)\gtrsim \frac{(t-s)}{(t_c-s)^{1/4}}\sqrt{\kappa_s+\eta_s}.
\end{align}
We recall from \eqref{e:defI_t+2}, if $t_c - t\leq n^{6\fd-1/2}$,  $I_t^+$ is a single interval and the  domain \eqref{e:edgedomain2} is empty. For $t_c - t\geq n^{6\fd-1/2} \gg n^{9 \mathfrak{d} / 2 - 1/2}$ we have
\begin{align}\label{e:sigkskt}
\tau \big( E(s),s \big)- \tau \big( E(t),t \big)=\frac{(t_c-s)^{1/6}-(t_c-t)^{1/6}}{n^{2/3-6\fd}}\ll \frac{(t-s)}{(t_c-s)^{1/4}}\sqrt{\kappa_s+\eta_s}.
 \end{align}
We also notice that $\kappa_t\geq \tau \big( E(t),t \big)$.

\noindent Then, similarly to as in \eqref{distancezi}, it follows from combining \eqref{e:kskt} and \eqref{e:sigkskt} that
\begin{align*}
\dist(z_s, I_s^+)\geq \eta_s+ \max \Big\{ \kappa_s-\tau \big(E(s),s) \big), 0 \Big\} \gtrsim \eta_s+ \frac{(t-s)}{(t_c-s)^{1/4}}\sqrt{\kappa_s+\eta_s},
\end{align*}

\noindent which gives \eqref{e:gap2}.

We now address the third statement of the proposition. If $z_t\in \mathscr D_t(\fr)$, from our definition \eqref{e:defDtr}, we have $z_s\in \mathscr D_s(\fr)$. If $z_t\in \mathscr D_t^{\rm L}$, from the definition of the characteristic flow \eqref{e:ccff}, we have $\Im z_s \geq \Im z_t $ and $f_s^*(z_s)=f_t^*(z_t)$. Therefore, 
\begin{flalign*} 
	n\Im[z_s] \bigg| \displaystyle\frac{\Im[f_s^*(z_s)]}{f_s^* (z_s)} \bigg|\geq n\Im[z_t] \bigg| \displaystyle\frac{\Im[f_t^*(z_t)]}{f_t^* (z_t)} \bigg| \geq \displaystyle\frac{n^{2\fd}}{2}, 
\end{flalign*} 

\noindent and thus $z_s\in \mathscr D_s(\fr)\cup \mathscr D_s^{\rm L}$. Finally, if $z_t\in \mathscr D_t^{\rm F}$, we show that $z_s\in \mathscr D_s(\fr)\cup  \mathscr D_s^{\rm L} \cup \mathscr D_s^{\rm F}$. We discuss the case that $z_t$ is in the domain \eqref{e:edgedomain1}, namely, $z_t=z_t(u)=E_1(t)-\kappa_t+\ri \eta_t$ with $\kappa_t > 0$; the other cases follow from the same argument. If $\Im z_s \geq \mathfrak{c}$ for some constant $\mathfrak{c}$, then we have $z_s\in \mathscr D_s(\fr)$; so, let us suppose that $\Im z_s\leq \mathfrak{c}$; then set $z_s=E_1(s)-\kappa_s+\ri \eta_s$. 
Thanks to \eqref{e:diff2}, we have that $\kappa_s\geq \kappa_t$, and $\dist(z_s, I_s^*)\geq \dist(z_t, I_t^*)$.
If $\eta_s\geq \kappa_s$, then 
\begin{flalign*} 
	n\Im[z_s] \bigg| \displaystyle\frac{\Im[f_s^*(z_s)]}{f_s^* (z_s)} \bigg |\geq \displaystyle\frac{1}{2} n\sqrt{\dist(z, I_s^*)\Im[z_s]} \cdot \bigg| \displaystyle\frac{\Im[f_s^*(z_s)]}{f_s^* (z_s)} \bigg| \geq \displaystyle\frac{1}{2} n \sqrt{ \dist (z_t, I_t^*) \Imaginary [z_t]} \cdot \bigg| \displaystyle\frac{\Imaginary [f_t^* (z_t)]}{f_t^* (z_t)} \bigg| \geq \displaystyle\frac{n^{2 \mathfrak{d}}}{2},
\end{flalign*} 

\noindent which gives $z_s\in \mathscr D_s(\fr)\cup\mathscr D_{s}^{\rm L}$. If instead $\eta_s\leq \kappa_s$, then \eqref{e:diff2} gives that $\del_r \kappa_r\lesssim -\sqrt{\eta_r+\kappa_r}\leq -\sqrt{\kappa_r}$ for $s\leq r\leq t$. As in \eqref{e:diff2c}, we can solve this differential equation to obtain the bound 
\begin{align*}
\kappa_s-\kappa_t\gtrsim (t-s)\sqrt{\kappa_s}\gtrsim (t-s)\sqrt{\kappa_s+\eta_s}\gg \tau \big( E(s),s \big)-\tau \big( E(t),t \big),
\end{align*}
where we used \eqref{e:diff2d} for the last inequality. Since $\kappa_t\geq \tau(E(t),t)$, it follows by rearranging that
$\kappa_s\geq \tau \big( E(s),s \big)$. We conclude that  $z_s\in \mathscr D_s(\fr)\cup\mathscr D_s^{\rm F}$.

It remains to establish the fourth statement of the proposition. From our construction of the stopping time \eqref{stoptime2} for $t\in[0,\sigma)\cap \bZ_n$, uniformly for any $z\in{\mathscr D}_t^{\rm L}\cup \mathscr D_t(\fr)$ as defined in \eqref{e:bulkdomain}, we have 
\begin{align}\label{e:mtbb}
|\Delta_t(z)|\leq \frac{(\log n) n^\fd}{n\Imaginary z}. 
\end{align}
Thanks to \Cref{p:ftzbehave}, we have always $|\Im[f^*_t(z)]/f^*_t(z)|=|\sin(\Im[\log f_t^*(z)])|\gtrsim |\sin(\Im[m_t^*(z)])|$. Therefore, ${\mathscr D}_t^{\rm L}\cup \mathscr D_t(\fr)$ contains $z$ satisfying $\Imaginary z \gtrsim n^{2\fd-1}/|\sin(\Im[m_t^*(z)])|$; thus, \eqref{e:mtbb} holds whenever $\Imaginary z \gtrsim n^{2\fd-1}/|\sin(\Im[m_t^*(z)])|$. By a standard argument \cite[Corollary 3.2]{MR4009708}, this implies the first half of \eqref{e:Htbound}: $|H_t(x)-H_t^*(x)|\leq n^{3 \mathfrak{d}}/n$.

For the second half of \eqref{e:Htbound}, we need to show that the intervals $\big\{ x : \fa(t)\leq x\leq E_1(t)-\tau(E_1(t),t) \big\}$, $\big\{ x :  E_{2}(t)-\tau(E_{2}(t),t)\leq x\leq \fb(t) \big\}$ and $\big\{ x : E'_{1}(t)+\tau(E'_{1}(t),t)\leq x\leq E'_{2}(t) + \tau(E'_{2}(t),t) \}$ are either void or saturated. 

We first prove it for $\fa(t)\leq x\leq E_1(t)-\tau \big(E_1(t),t \big)$; the case $ E_{2}(t) + \tau \big(E_{2}(t),t \big)\leq x\leq \fb(t)$ follows from the same argument and will therefore be omitted. Let $E(t)=E_1(t)$. We recall $\tau \big( E(t), t \big)=n^{-2/3+6\fd}|t-t_1|^{2/3}\vee n^{ -1+10\fd}$ from \eqref{e:disf} and $ \big|E(t)-\fa(t) \big|\asymp (t-t_1)^2$ from the second statement of \Cref{p:ftzbehave}. Thus, if $|t-t_1| \le \mathfrak{c} n^{5\fd-1/2}$ for sufficiently small $\mathfrak{c} > 0$, then $\big\{ x : \fa(t)\leq x\leq E_1(t)-\tau(E_1(t),t) \big\}$ is empty, so there is nothing to prove. Otherwise, for $|t-t_1| \geq \mathfrak{c} n^{5\fd-1/2}$, we fix $\mathfrak{b} = 6 \mathfrak{d}$, and we take $\kappa=n^{\fb - 2/3} |t-t_1|^{2/3}\geq \tau \big( E(t), t \big)$ and $\eta=n^{3\fb/13 - 2/3}  |t-t_1|^{2/3}$. Using \eqref{e:f/fedge}, we can check $E(t)-\kappa+\ri\eta\in \mathscr D_t^{\rm F}\cup \mathscr D_t(\fr)$. Thus, for $t\leq \sigma$, we have
\begin{align}\label{e:mtbound}
    \Big| m_t \big(E(t)-\kappa+\ri \eta \big)- m^*_t \big( E(t)-\kappa+\ri \eta \big) \Big| \lesssim\frac{n^{ 2\fd}}{n\sqrt{\eta(\kappa+\eta)}}.
\end{align}
Thanks to the square root behavior of $m^*_t(z)$ from \eqref{e:mtbehave1copy}, if $t\leq t_1$, then $\big[ \fa(t), E(t) \big]$ is a void region, and
\begin{align}\label{e:mtbb2}
-\Im m_t^*\big( E(t)-\kappa+\ri \eta \big)
\asymp \frac{\eta}{\sqrt{\kappa+\eta} \big(\sqrt{\kappa+\eta}+(t_1-t) \big)}.
\end{align}
By our choice $\kappa=n^{\fb - 2/3} |t-t_1|^{2/3}\geq \tau \big( E(t), t \big)$ with $\fb\geq 6\fd$, and $\eta=n^{3\fb/13 - 2/3} |t-t_1|^{2/3}$, \eqref{e:mtbound} and \eqref{e:mtbb2} imply that  $-\Im[ m_t(E_1 (t)-\kappa+\ri \eta)]\ll 1/n\eta$. 
However, if there existed a particle $x_i(t)$ such that $\big| x_i(t)-E(t)+\kappa \big| \leq \eta$, we would have
 that 
\begin{align}\label{e:sumerror}
   - \Im \Big[m_t \big( E(t)-\kappa+\ri \eta \big) \Big]=\frac{1}{n}\sum_{i=1}^n \frac{\eta}{ \big( x_i(t)-\kappa-E(t) \big)^2+\eta^2}\geq \frac{1}{2n\eta}.
\end{align}
This leads to a contradiction, and we conclude that there is no particle on $\big[ \fa(t), E(t)-\tau(E(t), t) \big]$.
If $t\geq t_1$, then $\big[ \fa(t), E(t) \big]$ is a saturated region, and \eqref{e:mtbehave3copy} gives
\begin{align}\label{e:pimtbb}
 \pi+\Im \tilde m_t^* \big( E(t)-\kappa+\ri \eta \big)
\asymp \frac{\eta}{\sqrt{\kappa+\eta} \big(\sqrt{\kappa+\eta}+(t_1-t) \big)}.
\end{align}
With our choice $\kappa=n^{\fb - 2/3} |t-t_1|^{2/3}$ with $\fb\geq 6\fd$, and $\eta=n^{3\fb/13 - 2/3} |t-t_1|^{2/3}$, \eqref{e:mtbound} and \eqref{e:pimtbb} implies that  $\pi+\Im[ \tilde m_t(E(t)-\kappa+\ri \eta)]\ll 1/n\eta$. 
Similarly to \eqref{e:sumerror}, we conclude that  $\big[ \fa(t), E(t)-\tau(E(t), t) \big]$ is fully packed.

For the interval $E'_{1}(t)+\tau(E'_{1}(t),t)\leq x\leq E'_{2}(t)-\tau(E'_{2}(t),t)$, we take $E(t)=E'_{1}(t)$,  $E'(t)=E'_{2}(t)$,  $\kappa=n^{\fb - 2/3} (t_c-t)^{1/6} \geq \tau \big( E(t), t \big)$ with $\fb\geq 6\fd$, and $\eta=n^{3\fb/13 - 2/3} (t_c-t)^{1/6}$. Then using \eqref{e:f/fcusp}, we can check that $E(t)+\kappa+\ri\eta\in \mathscr D_t^{\rm F}\cup \mathscr D_t(\fr)$. Thus, for $t\leq \sigma$, we have
\begin{align}\label{e:mtbound2}
    \Big| m_t \big( E(t)+\kappa+\ri \eta \big) - m_t^* \big(E(t)+\kappa+\ri \eta \big) \Big| \lesssim\frac{n^{2 \fd}}{n\sqrt{\eta(\kappa+\eta)}}.
\end{align}
If $\big[ E(t), E'(t) \big]$ is a void region, then \eqref{e:mtcubic} implies that
\begin{align}\label{e:cuspmtbb}
-\Im \Big[ m^*_t \big( E(t)+\kappa+\ri \eta \big) \Big]\lesssim \frac{\eta}{(t_c-t)^{1/4}\sqrt{\kappa+\eta}}.
\end{align}
By our choice that $\kappa=n^{\fb - 2/3} (t_c-t)^{1/6}$ with $\fb\geq 6\fd$, and $\eta=n^{3\fb/13 - 2/3}\fb(t_c-t)^{1/6}$,  \eqref{e:mtbound2} and \eqref{e:cuspmtbb} together imply that  $-\Im[ m_t(E(t)+\kappa+\ri \eta)]\ll 1/n\eta$. Through the same reasoning as above, we conclude that there is no particle on $\big[ E(t)+\tau(E(t),t), E'(t)-\tau(E'(t), t) \big]$. Similarly, if $\big[ E(t), E'(t) \big]$ is a saturated region, by considering $\pi+\Im[m_t(E(t)-\kappa+\ri \eta)]$, we will have that $[E(t)+\tau(E(t),t), E'(t)-\tau(E'(t), t)]$ is fully packed. This finishes the proof of \eqref{e:Htbound}.
\end{proof}

\subsection{Preliminary Results}\label{s:PreR}
In this section, we collect some preliminary estimates, which will be used in Sections \ref{s:DEq} and \ref{s:OptimalRigidity}. We can reorganize the complex Burgers equation \eqref{e:newBB} as
\begin{align}\label{e:newBB2}
\del_t m^*_t(z)+\del_t \log \tilde g^*_t(z)+\del_z \log \big(\cB^*_t(z) \big)=0,\quad \cB_t^*(z)=\big( f_t^*(z)+1 \big) \big( z-\fa(t) \big).
\end{align}
Since $\del_t m^*_t(z)= \OO(1/z^2)$ as $z\rightarrow \infty$, we can do a contour integral on both sides of \eqref{e:newBB2} to get rid of $\del_t m^*_t(z)$, and express $\del_t \log \tilde g^*_t(z)$ as 
\begin{align}\label{e:dmtta}
\del_t \log \tilde g^*_t(z)&=-\frac{1}{2\pi \ri}\oint_{\omega+}\frac{\del_w \log \cB_t^*(w)\rd w}{w-z}
=-\frac{1}{2\pi \ri}\oint_{\omega+}\frac{ \log \cB_t^*(w)\rd w}{(w-z)^2}.
\end{align}
where the contour $\omega+\subseteq \mathscr D_t(\fr)$, which encloses $[ \mathfrak{a}(t), \mathfrak{b}(t)]$ and $z$.  The second statement of \Cref{p:fdecompose} implies that for $x>\fb(t)$ or  $x< \fa(t)$,  $\cB_t^*(x) = \big( f_t^*(x)+1 \big) \big(x-\fa(t) \big)>0$. Moreover, we also have $\Im \big[ f_t^*(z)+1 \big]<0$ for $\Imaginary z>0$, and $\Im \big[f_t^*(z)+1 \big]>0$ for $\Imaginary z<0$. Therefore, $\cB_t^*(z)\in \bC\setminus (-\infty,0]$  on $\mathscr U_t\setminus[\fa,\fb]$, and $\log \cB(z)$ is well defined on $\mathscr U_t\setminus[\fa(t),\fb(t)]$. In what follows, we take the branch of the logarithm to be so that $\Imaginary \log u \in [-\pi, \pi)$.

For $m_t(z)$, we recall the empirical density $\rho(x;\bmx_t)$, and corresponding height function $H_t(x)$ from \eqref{e:defHt}. By taking $\beta=H_t$,  the variational problem \eqref{e:varWt} can be encoded by the complex slope $f(z;H_t,t)$. Thanks to \Cref{p:fdecompose}, we have the following decomposition 
\begin{align}\label{e:ftdda}
f(z;H_t,t)=e^{m(z;H_t,t)}g(z;H_t,t)=e^{\tilde m_t(z)}\tilde g(z;H_t,t),
\end{align}
where by our construction, $m(z;H_t,t)=m_t(z)$; $ \log \tilde g(z;H_t,t)=\log g(z;H_t,t)+\log(z-\fa(t))-\log(\fb(t)-z)$ is analytic in a neighborhood of $[\fa(t),\fb(t)]$, and $\tilde m_t(z)=m_t(z)-\log \big(z-\fa(t) \big)+\log \big( \fb(t)-z \big)$.
Similarly to \eqref{e:newBB2}, we also define
\begin{align}\label{e:defBtt}
\cB_t(z)=(f(z;H_t,t)+1)(z-\fa(t))=e^{m_t(z)}\varphi^+(z;H_t,t)+\varphi^-(z;H_t,t),
\end{align}
where $\varphi^+(z;H_t,t), \varphi^-(z;H_t,t)$ are defined in \eqref{e:defvarphi}. 
For simplicity of notations, if the context is clear, we will simply write  
\begin{align}\label{e:}
f_t(z)=f(z;H_t,t),\quad  g_t(z)= g(z;H_t,t),\quad \tilde g_t(z)= \tilde g(z;H_t,t).
\end{align}

\noindent Observe that $f_t(z)$ and $g_t(z)$ defined above are random depending on the particle configuration $\bmx_t$ through $H_t$, and that $f_t(z)$ does not necessarily satisfy the complex Burgers equation (unlike $f_t^* (z)$).

In the following proposition, we collect some estimates of $f_t(z)$ and $g_t(z)$, which will be used later to derive a dynamical equation for $m_t(z)$.
\begin{prop}\label{p:gfbound}
Adopt the assumptions of \Cref{p:rigidityBB}, and recall the control parameter $\fM(t)$ from \eqref{e:deftz2}. For  $0\leq t\leq \sigma$ as defined in \eqref{stoptime2}, we have the following. 
\begin{enumerate}
\item 
For $z$ in the neighborhood $\tilde{\mathscr U}_t$ of $[ \mathfrak{a}(t), \mathfrak{b}(t)]$, we have
\begin{align}\begin{split}\label{e:gfdidd}
&\big| \del^k_z\log g_t(z)-\del^k_z\log g^*_t(z) \big|
\lesssim \frac{\fM(t)}{n}, \quad 0\leq k\leq 2,\\
&  \log f_t(z)-\log  f^*_t(z)
=\Delta_t(z)+\OO\left(\frac{\fM(t)}{n}\right),
\end{split}\end{align}
it follows that $\big| \tilde g_t(z) \big|, \big| \del_z \tilde g_t(z) \big|, \big| \del^2_z \tilde g_t(z) \big| \asymp 1$ in the  neighborhood $\tilde{\mathscr U}_t$ of $[ \mathfrak{a}(t), \mathfrak{b}(t)]$.  

\item For $z \in {\mathscr D}_t(\fr)\cup{\mathscr D}_t^{\rm L}\cup{\mathscr D}_t^{\rm F}$, the difference of the Stieltjes transform satisfies
\beq\label{e:replacebound}
\Delta_t(z)\lesssim n^{ -\fd}\left|\frac{\Im f^*_{t}(z)}{f^*_t(z)}\right|\leq n^{-\fd};
\eeq
 the difference of the complex slopes satisfies
\beq\label{e:replacep2}
\left|\frac{f_t(z)}{f_t(z)+1}-\frac{f^*_t(z)}{f^*_t(z)+1}\right|\lesssim -n^{-\fd}\Im\left[\frac{f_t^*(z)}{f_t^*(z)+1}\right] \lesssim n^{-\fd},
\eeq
and it follows that
\beq\label{e:replacep3}
-\Im\left[\frac{f_t(z)}{f_t(z)+1}\right]\lesssim-\Im\left[\frac{f_t^*(z)}{f_t^*(z)+1}\right]\lesssim 1.
\eeq

\item The function $f_t$ defines a negative measure 
 $\Im \big[ f_t(x+0\ri)/(f_t(x+0\ri)+1) \big]$ on $[\fa(t),\fb(t)]$, whose support
\begin{align}\label{e:defItft}
I_t=\supp \bigg( \displaystyle\frac{\Im \big[ f_t(x+0\ri) \big]}{f_t(x+0\ri)+1} \bigg),
\end{align}
is contained in $I_t^+$,
\begin{align}\label{e:extremepp2}
 I_t\subseteq I^+_t.
\end{align}
\item For $k \geq 1$, the $k$-th derivative of $f_t(z)/ \big( f_t(z)+1 \big)$ satisfies
\begin{align}
\label{e:derf/f+1}
\left|\del^k_z \frac{f_t(z)}{f_t(z)+1}\right|
\lesssim \frac{1}{\Im[z]\dist(z,I_t)^{k-1}} \cdot \bigg| \Imaginary \displaystyle\frac{f_t (z)}{f_t (z) + 1} \bigg|.
\end{align} 
If we further assume that $z\in {\mathscr D}_t(\fr)\cup{\mathscr D}_t^{\rm L}\cup{\mathscr D}_t^{\rm F}$, then 
\begin{align}\label{e:imbb}
\frac{1}{\Im z} \cdot \bigg| \Imaginary \displaystyle\frac{f_t (z)}{f_t (z) + 1}\bigg| \lesssim \frac{1}{\dist(z,I_t^+)}.
\end{align}
\end{enumerate}

\end{prop}
\begin{proof}
For $0\leq t\leq \sigma$, \eqref{e:Htbound} implies that the two profiles $H_t(x)$ and $H_t^*(x)$ are close. In particular, $d(\del_x H_t(x), \del_x H_t^*(x))\leq \fc$, and the assumptions in  \Cref{p:fdecompose} holds.
For the difference $\log g_t-\log g^*_t$, \eqref{e:lngsbound} gives for $z\in \tilde{\mathscr U}_t$
\begin{align*}
\big|\del_z^k\log g_t(z)-\del_z^k\log g^*_t(z) \big| \lesssim
\displaystyle\oint_{\omega} \big| m_t(w)-m^*_t(w) \big|  |\rd w|\lesssim \frac{\fM(t)}{n},
\end{align*}
where $\omega\subseteq \mathscr D_t(\fr)$ is a contour enclosing $\big[ \fa(s),\fb(s) \big]$, and we used the fact that $|m_t(w)-m^*_t(w)|\leq \fM(t)/n$ on  $\mathscr D_t(\fr) $. Since $|\del_z^k\tilde g_t^*(z)|\asymp 1$, we conclude that $|\del_z^k\tilde g_t(z)|\asymp 1$ in the  neighborhood $\tilde{\mathscr U}_t$ of $[\fa(t),\fb(t)]$ for $0\leq k\leq 2$. 
Thanks to the decomposition \eqref{e:gszmut}, the difference $\log f_t-\log f^*_t$ is given by
\begin{align*}\begin{split}
\log f_t(z)-\log  f^*_t(z)&= \big( m_t(z)- m^*_t(z) \big)+\log g_t(z)-\log g^*_t(z)\\
&= \big( m_t(z)- m^*_t(z) \big)+\OO\left(\frac{\fM(t)}{n}\right)
=\Delta_t(z)+\OO\left(\frac{\fM(t)}{n}\right).
\end{split}\end{align*}
This finishes the proof of \eqref{e:gfdidd}.

The estimate \eqref{e:replacebound} follows from the definition of the spectral domains ${\mathscr D}_t^{\rm L}, {\mathscr D}_t^{\rm F}$ from \eqref{e:bulkdomain},  \eqref{e:edgedomain1} and \eqref{e:edgedomain2}, and of the stopping time $\sigma$ from \eqref{stoptime2}.
The estimate \eqref{e:replacep2} follows from \eqref{e:gfdidd}, \eqref{e:replacebound} and \eqref{e:ft+1b2},
\begin{flalign*}
	\begin{aligned} 
\left|\frac{f_t(z)}{f_t(z)+1}-\frac{f^*_t(z)}{f^*_t(z)+1}\right|&\lesssim \Bigg( \big| \Delta_t (z) \big| + \mathcal{O} \bigg( \displaystyle\frac{\mathfrak{M} (t)}{n} \bigg) \Bigg) \frac{\big| f_t^* (z) \big|}{\big|f_t^*(z)+1 \big|^2} \\
 &\leq \frac{n^{-\fd} | \Im[f_t^*(z)] |}{|f_t^*(z)+1|^2}  = -n^{-\fd}\Im\left[\frac{f_t^*(z)}{f_t^*(z)+1}\right] \lesssim n^{-\fd}.
\end{aligned} 
\end{flalign*}

For $\Imaginary z\geq 0$, $\Im \big[f_t(z)/(f_t(z)+1) \big] =\Im [f_t(z)]/ \big| f_t(z)+1 \big|^2\leq 0$; therefore, $\Im \big[f_t(x+0\ri)/(f_t(x+0\ri)+1) \big]$ defines a negative measure on $[\fa(t),\fb(t)]$. Thanks to \eqref{e:Htbound} and the decomposition $\log f_t(z)=m_t(z)+\log g_t(z)$,  for $x\in[\fa(t),\fb(t)]\setminus I_t^+$ we have $m_t(x)\in \bR$ and thus $f_t(x)\in \bR$. By the third estimate of \eqref{e:gfdidd}, we have $\big| f_t^* (x) \big|^{-1} \big| f_t(x)-f_t^*(x) \big| = \big| m_t(x)-m_t^*(x) \big|+\OO \big( \fM(t)/n \big)$ for $x\in[\fa(t),\fb(t)]\setminus I_t^+$. 
Moreover, through essentially the same argument as for the fourth statement of \Cref{p:gap}, we have $ \big| m_t(x)-m_t^*(x) \big|=\oo(1)$ for $x\in[\fa(t),\fb(t)]\setminus I_t^+$. We conclude that $\big| f_t^* (x) \big|^{-1} \big| f_t(x)-f_t^*(x) \big| =\oo(1)$ for $x\in[\fa(t),\fb(t)]\setminus I_t^+$. By our assumption, $f_t^*(x)$ is bounded away from $-1$; see \eqref{e:ft+1b2}. We conclude that $f_t(x)\neq -1$ for $x\in[\fa(t),\fb(t)]\setminus I_t^+$. Therefore, on $x\in[\fa(t),\fb(t)]\setminus I_t^+$, $f_t(x+0\ri)/ \big( f_t(x+0\ri)+1 \big)\in \bR$ and does not have any poles. Therefore, its imaginary part vanishes outside $I_t^+$, i.e. $\Im \big[ f_t(x+0\ri)/(f_t(x+0\ri)+1) \big]=0$, and the claim \eqref{e:extremepp2} follows.

We can recover the singular part of  $f_t(z)/ \big( f_t(z)+1 \big)$ from the measure $\Im \big [f_t(x+\ri0)/(f_t(x+\ri0)+1) \big]$. In particular, since for  almost every $E \in \mathbb{R}$ we have
\begin{flalign*}
	\Imaginary \bigg[ \displaystyle\frac{f_t (E)}{f_t (E) + 1} \bigg] = \displaystyle\lim_{\eta \rightarrow 0} \Imaginary \bigg[ \displaystyle\frac{1}{\pi} \displaystyle\int_{I_t} \displaystyle\frac{\Im \big[ f_t (x + \ri0) / (f_t (x + \ri0) + 1) \big] \mathrm{d} x}{x - E - \mathrm{i} \eta} \bigg],
\end{flalign*}

\noindent the Schwarz reflection principle implies that
\begin{align*}
\frac{f_t(z)}{f_t(z)+1}-\frac{1}{\pi}\int_{I_t}\frac{\Im \big[ f_t(x+\ri0)/(f_t(x+\ri0)+1) \big] \rd x}{x-z},
\end{align*}
is a real analytic function in $z$ on a neighborhood of $I_t$. Observe that, when $z$ is bounded away from $I_t$, this analytic function is bounded. Thus, by the maximum principle, it is uniformly bounded in a neighborhood of $I_t$, and we get
\begin{align}\label{e:f/f+1}
\frac{f_t(z)}{f_t(z)+1}=\frac{1}{\pi}\int_{I_t}\frac{\Imaginary \big[ f_t(x+\ri0)/(f_t(x+\ri0)+1) \big]\rd x}{x-z}+\OO(1),
\end{align}
where the $\OO(1)$ error and its derivatives are uniformly bounded in a neighborhood of $I_t$. 
 Using \eqref{e:f/f+1}, we have the following bounds on the derivatives of $f_t(z)/(f_t(z)+1)$, for any $k\geq 1$
\begin{align}
	\label{ff1derivative} 
	\begin{split}
 \left|\del^k_z \frac{f_t(z)}{f_t(z)+1}\right|
&=\left| \displaystyle\frac{1}{\pi} \int_{I_t}\frac{k!\Im \big[f_t(x+\ri0)/(f_t(x+\ri0)+1) \big]\rd x}{(x-z)^{k+1}}\right|+\OO(1)\\
&\lesssim \frac{1}{\dist(z,I_t)^{k-1}}\left|\int_{I_t}\frac{\Im \big[f_t(x+\ri0)/(f_t(x+\ri0)+1) \big]\rd x}{|x-z|^2}\right|+\OO(1)\\
& \lesssim \frac{1}{\dist(z,I_t)^{k-1}}\left( \frac{\big| \Im[f_t(z)/(f_t(z)+1)] \big|}{\Im[z]}+\OO(1)\right)+\OO(1)
\lesssim \frac{\big| \Im[f_t(z)/(f_t(z)+1)] \big|}{\Im[z]\dist(z,I_t)^{k-1}},
\end{split}\end{align}
where we used \eqref{e:imratio0} that $| \Im[f_t(z)/(f_t(z)+1)]|/\Im[z]\gtrsim 1$ for the last inequality. 

If $\dist(z, I_t)\leq 2\Im z$, then \eqref{e:imbb} follows from \eqref{e:replacep3}. Otherwise if $\dist(z,I_t)\geq 2\Im z$, we will use \eqref{e:f/f+1}. Since the $\OO(1)$ error is a real analytic function, its imaginary part is bounded by $\OO(\Im z)$. By taking imaginary part, and dividing by $\Im z$, we get
\begin{align}\begin{split}\label{e:Imf/z}
\displaystyle\frac{1}{\Imaginary z} \cdot \bigg| \Imaginary \frac{f_t(z)}{f_t (z) + 1} \bigg|
&\lesssim \int_{I_t} \frac{-\Im \big[f_t(x+\ri0)/(f_t(x+\ri0)+1) \big] \rd x}{|x-z|^2}+\OO(1).
\end{split}\end{align}
If $\dist(z,I_t)\geq 2\Im z$, then for any $x\in I_t$, $|\Re z -x|\gtrsim \dist(z,I_t)$. 
Denoting $z'=\Re z +\ri \dist(z,I_t)$, we have $|x-z|^2=|\Re z-x|^2+\Im[z]^2\gtrsim |\Re z-x|^2+\Im[z']^2=|x-z'|^2$. Moreover, we also have that
\begin{align*}\begin{split}
\frac{1}{\Im z} \cdot \bigg| \Imaginary \displaystyle\frac{f_t (z)}{f_t (z) + 1} \bigg| 
&\lesssim \int_{I_t} \frac{-\Im \big[f_t(x+\ri0)/(f_t(x+\ri0)+1) \big]\rd x}{|x-z'|^2}+\OO(1)\\
&\lesssim \frac{1}{\Imaginary z'} \bigg| \Imaginary \displaystyle\frac{f_t(z')}{f_t(z') + 1} \bigg|+\OO(1)
\lesssim \frac{1}{\Im z'} \lesssim \frac{1}{\dist(z, I_t)}.
\end{split}\end{align*}
This finishes the proof of \eqref{e:imbb}.
\end{proof}

We define  the following lattice on the domain ${\mathscr D}_0(\fr)\cup {\mathscr D}_0^{\rm L}\cup {\mathscr D}_0^{\rm F}$,
\beq\label{def:L}
\mathscr L = \mathscr{L}_t \deq \big\{ z_t^{-1}(E+\ri \eta)\in {\mathscr D}_0(\fr)\cup {\mathscr D}_0^{\rm L}\cup  {\mathscr D}_0^{\rm F}: E\in \bZ/ n^3, \eta\in \bZ/n^3 \big\}.
\eeq
From the characteristic flow \eqref{e:ccff}, as a function of $u\in {\mathscr D}_0(\fr)\cup {\mathscr D}_0^{\rm L} \cup {\mathscr D}_0^{\rm F}$, we have that $z_t(u)$ is Lipschitz with Lipschitz constant at most $\OO(n)$. Indeed, using \eqref{zlineart}, \eqref{e:derf/f+1} and \eqref{e:imbb} we have
\begin{align*}
\big| \del_u z_t(u) \big| \lesssim 1+t\left|\del_u \frac{f_0^*(u)}{f_0^*(u)+1}\right|\lesssim 
1+\frac{t}{\Im u} \cdot \bigg| \Imaginary \displaystyle\frac{f_0^* (u)}{f_0^* (u) + 1} \bigg| \lesssim
\frac{1}{\dist(u, I_t^*)}\lesssim n.
\end{align*}
 Thus the image of the lattice ${\mathscr L}$  gives a $\OO(1/n^2)$ mesh of $z_t \big({\mathscr D}_0(\fr)\cup {\mathscr D}_0^{\rm L}\cup{\mathscr D}^{\rm F}_0 \big)$. In particular we have: for any $0\leq t\leq 1$ and $w\in z_t \big( {\mathscr D}_0(\fr)\cup {\mathscr D}_0^{\rm L}\cup{\mathscr D}^{\rm F}_0 \big)$, there exists some lattice point $u = u(w) \in \mathscr L$ as in \eqref{def:L}, such that
\begin{align}\label{e:approxw}
\big| z_t(u)-w \big| \lesssim 1/n^2.
\end{align}

In the following proposition we collect some estimates on various integrals, which we will use later.

\begin{prop}\label{p:integral}
Adopt the assumptions of \Cref{p:rigidityBB} and fix $0\leq t\leq \sigma$ (recall from \eqref{stoptime2}). If $z_t=z_t(u)\in \mathscr D_t^{\rm L}$, then
\begin{align}
\int_0^{t} \bigg| \Imaginary \displaystyle\frac{f_s^* (z_s)}{f_s^* (z_s) + 1} \bigg| \cdot \frac{\mathrm{d} s}{\Im[z_{s}]^{k+1} }
 \lesssim
 \frac{(\log n)^{\bm1(k=0)}}{\Im[z_t]^k}.\label{e:integral0}
\end{align}
If $z_t=z_t(u)\in \mathscr D_t^{\rm F}$, then
\begin{align}
 &\int_0^{t} \bigg| \Imaginary \displaystyle\frac{f_s^* (z_s)}{f_s^* (z_s) + 1} \bigg| \cdot \displaystyle\frac{\mathrm{d} s}{ \sqrt{\Im[z_{s}] \dist(z_s, I^*_s)}}
 \lesssim
 \frac{\log n\sqrt{\Im[z_{t}]}}{\sqrt{\dist(z_{t}, I^*_{t})}};\label{e:integral1}\\
&\int_0^t \bigg| \Imaginary \displaystyle\frac{f_s^* (z_s)}{f_s^* (z_s) + 1} \bigg| \cdot \displaystyle\frac{\rd s}{\Im[z_s]\dist(z_s,I^+_s)}
\lesssim 
\frac{\log n}{\sqrt{\Im[z_{t}]\dist(z_{t}, I^*_{t})}};\label{e:integral2}\\
&\int_0^{t} \bigg| \Imaginary \displaystyle\frac{f_s^* (z_s)}{f_s^* (z_s) + 1} \bigg| \cdot \displaystyle\frac{\mathrm{d} s}{\Im[z_{s}] \dist(z_s, I^+_s)^2}
 \lesssim
 \frac{1}{\Im[z_{t}]\dist(z_{t}, I^*_{t})}.
\label{e:integral3}
\end{align}
\end{prop}

\begin{proof}
	
We recall from the characteristic flow \eqref{e:ccff}, $\big| \Im[f^*_s(z_s)/(f^*_s(z_s)+1)] \big| =-\del_s \Im z_s$. We begin by establishing \eqref{e:integral0}. For $k=0$, we have
\begin{align*}
\int_0^{t} \bigg| \Imaginary \displaystyle\frac{f_s^* (z_s)}{f_s^* (z_s) + 1} \bigg| \cdot \displaystyle\frac{\mathrm{d} s}{\Im z_{s}}
=\int_0^{t}\frac{-\del_s\Im[z_s]}{\Im[z_{s}] }\rd s
=\log \frac{\Imaginary z_0}{\Im z_t}\lesssim \log n,
\end{align*}

\noindent where in the last bound we used the fact that $n^{-1} \ll \Imaginary z_s \ll n$ for each $s \in [0, t]$. For $k\geq 1$, we have
\begin{align}
	\label{integralkf}
\int_0^{t} \bigg| \Imaginary \displaystyle\frac{f_s^* (z_s)}{f_s^* (z_s) + 1} \bigg| \cdot \displaystyle\frac{\mathrm{d} s}{\Im[z_{s}]^{k+1} }
=\int_0^{t}\frac{-\del_s\Im[z_s]}{\Im[z_{s}]^{k+1} }\rd s
\lesssim  \frac{1}{\Im[z_t]^k}.
\end{align}
This finishes the proof of \eqref{e:integral0}.

The proofs of \eqref{e:integral1}, \eqref{e:integral2} and \eqref{e:integral3} are all very similar, and therefore we only establish the third. We must consider two cases, whether $z_t$ is in the domain \eqref{e:edgedomain1} or in the domain \eqref{e:edgedomain2}. For $z_t$ in the domain \eqref{e:edgedomain1}, we will use the square root behavior \eqref{e:squareeq} to bound $\big| \Im[f^*_s(z_s)/(f^*_s(z_s)+1)] \big|$, and \eqref{e:gap1} to lower bound $\dist(z_s, I_s^+)$. For $z_t$ in the domain \eqref{e:edgedomain2}, we use the square root behavior \eqref{e:cubiceq} to bound $\big| \Im[f^*_s(z_s)/(f^*_s(z_s)+1)] \big|$, and \eqref{e:gap2} to lower bound $\dist(z_s, I_s^+)$.

For $z_t$ in the domain \eqref{e:edgedomain1}, we only prove \eqref{e:integral3} when $z_t$ is closer to the left boundary $E_1 (t)$, since the proof is entirely analogous when it is closer to the right boundary $E_2 (t)$. Thus, we denote $z_t=E(t)-\kappa_t+\ri \eta_t$ where $E(t)=E_1(t)$ and $\kappa_t > 0$. Let $z_s=E(s)-\kappa_s+\eta_s$ for any $0\leq s\leq t$. Fix small $\fc>0$ and let $r=\sup\{0\leq s\leq t: \dist(z_s, E(s))\geq \fc\}$. We decompose the integral \eqref{e:integral3} into two integrals corresponding to $0\leq s\leq r$ and $r\leq s\leq t$. For $0\leq s\leq r$, we have $\dist(z_s, E_s)\gtrsim 1$. By the same argument as for \eqref{e:integral0}, we have 
\begin{align}\label{e:intbb1}
\int_0^{r} \bigg| \Imaginary \displaystyle\frac{f_s^* (z_s)}{f_s^* (z_s) + 1} \bigg| \cdot \displaystyle\frac{\mathrm{d} s}{\Im[z_{s}] \dist(z_s, I^+_s)^2}
 \lesssim 1\lesssim 
 \frac{1}{\Im[z_{t}]\dist(z_{t}, I^*_{t})}.
\end{align}
For $r\leq s\leq t$, we have $\dist(z_s, E_s)\leq \fc$, then the square root behavior \eqref{e:squareeq} gives 
\begin{align}\label{e:sqbb}
\frac{\eta_s}{\sqrt{\kappa_s+\eta_s}}\asymp
\bigg| \Im  \displaystyle\frac{f^*_s(z_s)}{f^*_s(z_s)+1} \bigg| = \bigg| \Im  \displaystyle\frac{f^*_{t}(z_{t})}{f^*_{t}(z_{t})+1} \bigg|\asymp \frac{\eta_{t}}{\sqrt{\kappa_{t}+\eta_{t}}},
\end{align}
 We can bound the integral \eqref{e:integral3} as
\begin{align}\begin{split}\label{e:intbb2}
\int_r^{t} \bigg| \Imaginary \displaystyle\frac{f_s^* (z_s)}{f_s^* (z_s) + 1} \bigg|\cdot \displaystyle\frac{\mathrm{d} s}{\Im[z_{s}] \dist(z_s, I^+_s)^2}
& \lesssim\int_r^{t} \bigg| \Imaginary \displaystyle\frac{f_s^* (z_s)}{f_s^* (z_s) + 1} \bigg| \cdot \frac{\mathrm{d} s}{\eta_s \big( (\kappa_s+\eta_s)^{1/2}(t-s)+ \eta_{s} \big)^2}\\
&\lesssim
\int_{r}^{t} \frac{ \eta_{s}/{\sqrt{\kappa_{s}+\eta_s}}}{\eta_{s}\big( (\kappa_s+\eta_s)^{1/2}(t-s)+ \eta_{s} \big)^2}\rd s \\
& \lesssim \int_{r}^{t} \frac{\rd s}{(\kappa_{s}+\eta_s)^{3/2} \big(t-s+\eta_s/(\kappa_s+\eta_s)^{1/2} \big)^2}\\
&\leq  \int_{r}^{t} \frac{\rd s}{(\kappa_{t}+\eta_{t})^{3/2} \big( t-s+\eta_{t}/(\kappa_{t}+\eta_{t})^{1/2} \big)^2}
\leq   \frac{1}{\eta_{t} (\kappa_{t}+\eta_{t})},
\end{split}
\end{align} 
where in the first line we used \eqref{e:gap1}; in the last two lines we used 
\eqref{e:sqbb}
and $\kappa_s+\eta_s \geq \kappa_{t}+\eta_{t}$ from \eqref{e:diff2c}. The claim \eqref{e:integral3} follows from combining \eqref{e:intbb1} and \eqref{e:intbb2}. 

For the second case when $z_t$ is in the domain \eqref{e:edgedomain2} we again only analyze the integral when $z_t$ is closer to $E'_{1}(t)$ (as the proof is entirely analogous when $z_t$ is closer to $E_2' (t)$). We denote $E(t)=E'_{1}(t)$. Since $E_2'(t)-E_1'(t)\lesssim (t_c-t)^{3/2}$, we have $z_t=E(t)+\kappa_t+\ri \eta_t$ with $0<\kappa_t\lesssim (t_c-t)^{3/2}$. If $\eta_t\gtrsim(t_c-t)^{3/2}\gtrsim \kappa_t$, then \eqref{integralkf} gives \eqref{e:integral3}. Thus, we further restrict ourselves to the case $\eta_t\lesssim(t_c-t)^{3/2}$. Let $z_s=E(s)+\kappa_s+\ri\eta_s$ for any $0\leq s\leq t$. Fix small $\fc>0$ and let $r=\sup\{0\leq s\leq t: \Im[z_s]\geq \fc (t_c-s)^{3/2}\}$.  We decompose the integral \eqref{e:integral3} into two integrals corresponding to $0\leq s\leq r$ and $r\leq s\leq t$. For $0\leq s\leq r$, by the same argument as for \eqref{e:integral0}, we have 
\begin{align}\label{e:intbb3}
\int_0^{r} \bigg| \Imaginary \displaystyle\frac{f_s^* (z_s)}{f_s^* (z_s) + 1} \bigg| \cdot \displaystyle\frac{\mathrm{d} s}{\Im[z_{s}] \dist(z_s, I^+_s)^2}
 \lesssim \frac{1}{\Im[z_r]^2}\lesssim 
 \frac{1}{\Im[z_{t}]\dist(z_{t}, I^*_{t})}.
\end{align}
For $r\leq s\leq t$, the square root behavior \eqref{e:cubiceq} gives 
\begin{align}\begin{split}\label{e:1/4squre}
\phantom{{}={}}\frac{\eta_s}{(t_c-s)^{1/4}\sqrt{\kappa_s+\eta_s}} & \asymp \bigg| \Im \displaystyle\frac{f^*_s(z_s)}{f^*_s(z_s)+1} \bigg| =\bigg| \Imaginary \displaystyle\frac{f^*_{t}(z_{t})}{f^*_{t}(z_{t})+1} \bigg| \asymp \frac{\eta_{t}}{(t_c-t)^{1/4}\sqrt{\kappa_{t}+\eta_{t}}}.
\end{split}\end{align}
 We can bound the integral \eqref{e:integral3} as, 
\begin{align}\begin{split}\label{e:edgevariance2}
		\displaystyle\int_r^{t} & \bigg| \Imaginary \displaystyle\frac{f_s^* (z_s)}{f_s^* (z_s) + 1} \bigg| \cdot \displaystyle\frac{\rd s}{\eta_{s}\dist(z_s, I^+_s)^2} \\
& \lesssim \int_r^{t} \bigg| \Imaginary \displaystyle\frac{f_s^* (z_s)}{f_s^* (z_s) + 1} \bigg| \cdot \displaystyle\frac{\mathrm{d} s}{\eta_{s}\big( (\kappa_s+\eta_s)^{1/2}(t-s)/(t_c-s)^{1/4}+ \eta_{s} \big)^2}\\
&\lesssim
 \int_{r}^{t} \frac{ \eta_{s}/(t_c-s)^{1/4}{\sqrt{\kappa_{s}+\eta_s}}}{\eta_{s} \big( (\kappa_s+\eta_s)^{1/2}(t-s)/(t_c-s)^{1/4}+ \eta_s \big)^2}\rd s \\
&\lesssim  \int_{r}^{t} \frac{\rd s}{(t_c-s)^{3/4}(\kappa_{s}+\eta_s)^{3/2} \big( (t-s)/(t_c-s)^{1/2}+\eta_s/(t_c-s)^{1/4}(\kappa_s+\eta_s)^{1/2} \big)^2}\\
& \lesssim \frac{1}{(\kappa_{t}+\eta_{t})^{3/2}} \int_{r}^{t} \frac{\rd s}{(t_c-s)^{3/4} \big( (t-s)/(t_c-s)^{1/2}+\eta_{t}/(t_c-t)^{1/4}(\kappa_{t}+\eta_{t})^{1/2} \big)^2} \lesssim   \frac{1}{\eta_{t} (\kappa_{t}+\eta_{t})},
\end{split}
\end{align} 
where in the first bound we used  \eqref{e:gap2}, and in the second and fourth bounds we used that \eqref{e:1/4squre} and $\kappa_s+\eta_s \geq \kappa_{t}+\eta_{t}$ from \eqref{e:kskt}. To verify the last inequality of \eqref{e:edgevariance2}, we must show that the integral in the fourth line of \eqref{e:edgevariance2} is  bounded by $\sqrt{\kappa_{t}+\eta_{t}}/\eta_{t}$. There are two cases to consider, namely, if $\eta_{t}/ \sqrt{\kappa_{t}+\eta_{t}}$ is less than $(t_c-t)^{3/4}$ or is greater than $(t_c-t)^{3/4}$. In the first case, let $y=t-s$. We divide the integral into three integrals according to the dominant term in the denominator, so that
\begin{align*}\begin{split}
\int_{0}^{t} & \frac{\rd y}{(y+t_c-t)^{3/4} \big( y/(y+t_c-t)^{1/2}+\eta_{t}/(t_c-t)^{1/4}(\kappa_{t}+\eta_{t})^{1/2} \big)^2} \\
&\leq
\int_0^{(t_c-t)^{1/4}\eta_{t}/\sqrt{\kappa_{t}+\eta_{t}}} \frac{\rd y}{(t_c-t)^{1/4}\eta^2_{t}/(\kappa_{t}+\eta_{t})}
+
\int_{t_c-t}^{t}\frac{\rd y}{y^{7/4}} +\int_{(t_c-t)^{1/4}\eta_{t}/\sqrt{\kappa_{t}+\eta_{t}}}^{t_c-t}\frac{(t_c-t)^{1/4}\rd y}{y^2} \\
& \lesssim \frac{\sqrt{\kappa_{t}+\eta_{t}}}{\eta_{t}}.
\end{split}\end{align*}

\noindent The proof in the second case $\eta_{t}/ \sqrt{\kappa_{t}+\eta_{t}}\geq(t_c-t)^{3/4}$ is very similar and thus omitted.
\end{proof}

\subsection{Dynamical Equation}\label{s:DEq}

In this section we derive difference equations for $\Delta_t(z)=m_t(z)-m_t^*(z)$. We recall $\tilde m_t^*(z)=m_t^*(z)-\log \big(z-\fa(t) \big)+\log \big( \fb(t)-z \big)$. By plugging the characteristic flow \eqref{e:ccff} in \eqref{e:decompft*2}, we get
\begin{align}\label{e:floweq}
	\log f_t^*(z_t)=\tilde m_t^*(z_t)+\log \tilde g_t^*(z_t).
\end{align} 

To get the difference equation of $\tilde m_t^*(z_t)$, we recall the expression of $\del_t\log \tilde g_t^*(z_t)$ from \eqref{e:dmtta}. We can take one more time derivative, 
\begin{align}\begin{split}\label{e:dttgt}
		\del_t^2 \log \tilde g^*_t(z)
		&=-\frac{1}{2\pi \ri}\oint_{\omega+}\frac{\del_t \log \cB_t^*(w)\rd w}{(w-z)^2} \\
		& =-\frac{1}{2\pi \ri}\oint_{\omega+} \left(\frac{\del_t f_t^*(w)}{f_t^*(w)+1}-\frac{1}{w-\fa(t)}\right)\frac{\rd w}{(w-z)^2}\\
		&=\frac{1}{2\pi \ri}\oint_{\omega+} \frac{\del_w f_t^*(w)f_t^*(w)}{(f_t^*(w)+1)^2}\frac{\rd w}{(w-z)^2}
		=\frac{1}{2\pi \ri}\oint_{\omega+} \del_w\left(\frac{ f_t^*(w)}{f_t^*(w)+1}\right) \frac{f_t^*(w)\rd w}{(w-z)^2},
\end{split}\end{align}
where the contour $\omega+\subseteq \mathscr D_t(\fr)$, which encloses $[ \mathfrak{a}(t), \mathfrak{b}(t)]$ and $z$. For $w\in \mathscr D_t(\fr)$, we have $\big| \del_w (f^*_t/(f^*_t+1)) \big| \lesssim 1$ (which holds by \eqref{e:derf/f+1} and \eqref{e:imbb}), and $\big| f_t^*(w) \big|\lesssim 1$. Thus we conclude that $\big| \del_t^2 \log \tilde g^*_t(z) \big| \lesssim 1$. 

We can obtain the following difference equation of $\tilde m_t^*$ by taking the difference of \eqref{e:floweq} at times $t$ and $t+1/n$,
\begin{align}\begin{split}\label{e:meq}
		\phantom{{}={}}\tilde m_{t+1/n}^*(z_{t+1/n})-\tilde m_{t}^*(z_{t})
		& =-\log \tilde g_{t+1/n}^*(z_{t+1/n})+\log \tilde g_t^*(z_t)\\
		&=-\log \tilde g_{t+1/n}^*(z_{t+1/n})+\log \tilde g_t^*(z_{t+1/n})-\log \tilde g_{t}^*(z_{t+1/n})+\log \tilde g_t^*(z_t)\\
		&=-\frac{(\del_t\log \tilde g_t^*)(z_{t+1/n})}{n}-\del_z \tilde g_t^*(z_t)(z_{t+1/n}-z_t)+\OO\left(\frac{1}{n^2}\right),
\end{split}\end{align}
where we used the bound $|\del_t^2\log \tilde g_t^*|\lesssim 1$ from the discussion after \eqref{e:dttgt}, $|\del_z^2\log \tilde g_t^*|\lesssim 1$ from the analyticity of $\tilde g_t^*$, and  \eqref{e:ft+1b2} so $|z_{t+1/n}-z_t|= |f^*(z_t)/n(f_t^*(z_t)+1)|\lesssim 1/n$. By plugging  \eqref{e:dmtta} into \eqref{e:meq}, we get that 
\begin{align}\begin{split}\label{e:Deltamt}
\tilde m_{t+1/n}^*(z_{t+1/n})-\tilde m_{t}^*(z_{t})
=\frac{1}{2n\pi \ri}\oint_{\omega+}\frac{\del_w \log \cB_t^*(w)\rd w}{w-z_{t+1/n}}-\del_z \tilde g_t^*(z_t)(z_{t+1/n}-z_t)+\OO\left(\frac{1}{n^2}\right).
\end{split}\end{align}

For the transition probability \eqref{e:transitDE}, we can decompose the time difference of $m_{t+1/n}(z)-m_t(z)$ as
\begin{align}\begin{split}\label{e:dm}
		m_{t+1/n}(z)- m_t(z)&=\int_{z - 1/n}^z \left(\sum_{i=1}^{m}\frac{1}{u-x_i(t+1/n)}-\sum_{i=1}^{m}\frac{1}{u-x_i(t)}\right)\rd u\\
		&=\Delta M_t(z) + \int_{z-1/n}^z\cL_t \big( m_t(u) \big) \rd u,
\end{split}\end{align}
where $\Delta M_t(z)$ is a martingale difference,
\begin{align*}\begin{split}
		\Delta M_t(z)
		&=\int_{z-1/n}^z \Bigg( \sum_{i=1}^{m}\frac{1}{u-x_i(t+1/n)} - \mathbb{E} \bigg[ \displaystyle\sum_{i=1}^{m}\frac{1}{u-x_i(t+1/n)} \bigg| \bmx_t \bigg] \Bigg) \rd u\\
		&=\int_{z-1/n}^z \Bigg( \sum_{i=1}^{m}\frac{1}{u-x_i(t+1/n)}- \sum_{\bme\in\{0,1\}^{m}} a_t(\bme;\bmx_t)\sum_{i=1}^{m}\frac{1}{u-x_i(t)-e_i/n} \Bigg)\rd u,
\end{split}\end{align*}
and  the drift term is given by
\begin{align}\label{e:keyt}
	\begin{aligned}
		\cL_t \big( m_t (u) \big)
		& = \mathbb{E} \Bigg[ \displaystyle\sum_{i = 1}^m \bigg( \displaystyle\frac{1}{u - x_i (t + 1/n)} - \displaystyle\frac{1}{u - x_i (t)} \bigg)\bigg| \bmx_t  \Bigg] \\
		&=\sum_{\bme\in\{0,1\}^{m}} a_t(\bme;\bmx_t)\left(\sum_{i=1}^{m}\frac{1}{u-x_i(t)-e_i/n}-\frac{1}{u-x_i(t)}\right),
	\end{aligned}
\end{align}
where $a_t(\bme;\bmx)$ is as defined in \eqref{e:transitDE}.

Transition probabilities in the form of \eqref{e:transitDE} can be analyzed using dynamical loop equations. The following proposition gives us the leading order term of the drift term $\cL_t(m_t(z))$, and bounds on the moments of the martingale term $\Delta M_t(z)$. 
We postpone the proof of \Cref{p:improve} to Sections \ref{s:loopeq} and \ref{s:improveEE}  below.
\begin{prop}\label{p:improve}
	Under the assumptions of \Cref{p:rigidityBB},  for $0\leq t\leq \sigma$, uniformly for any $z \in {\mathscr D}_t(\fr)\cup{\mathscr D}_t^{\rm L}\cup{\mathscr D}_t^{\rm F}$, we have 
	\begin{align}\begin{split}\label{e:improve1}
			\bE\left[\sum_{i = 1}^m \frac{1}{z-x_i(t + 1/n)}-\frac{1}{z-x_i(t)}\right]
			&=\frac{1}{2\pi \ri}\oint_{\omega}\frac{\log \cB_t(w)\rd w}{(w-z)^2}+\OO\left(\frac{|\Im[f_t(z)/(f_t(z)+1)]|}{n\Im[z] \dist(z,I_t)}\right),
	\end{split}\end{align}
	where $\cB_t$ is from \eqref{e:defBtt}, the expectation is with respect to the transition probability \eqref{e:transitDE}, and the contour $\omega$ encloses  $[\fa(t),\fb(t)]$ but not $z$. Moreover,
	\begin{align}
	\begin{split}\label{e:improve2}
		&\phantom{{}={}}\bE\left|\sum_{i = 1}^m \frac{1}{z-x_i(t)-e_i/n}-\frac{1}{z-x_i(t)}-\bE\left[\sum_{i = 1}^m \frac{1}{z-x_i(t)-e_i/n}-\frac{1}{z-x_i(t)}\right]\right|^{2p}\\
		& \qquad \qquad = \OO\left(\left(\frac{\Im[f_t(z)/(f_t(z)+1)]}{n\Im[z] \dist(z,I_t)^{2}}\right)^{p}+\frac{|\Im[f_t(z)/(f_t(z)+1)]|}{n^{2p-1}\Im[z] \dist(z,I_t)^{4p-2}}\right).
	\end{split}\end{align}
\end{prop}

\begin{rem}
	We recall the complex Burgers equation from \eqref{e:newBB2}. From the expression \eqref{e:dmtta}, $\del_t \log \tilde g_t^*(z)$ is analytic in a neighborhood of $[\fa(t),\fb(t)]$. We can do a contour integral on both sides of \eqref{e:newBB2} to get rid of $\del_t \log \tilde g_t^*(z)$ on both sides of \eqref{e:newBB2}
	\begin{align}\label{e:newBB2c}
		\del_t m_t^*(z)=-\frac{1}{2\pi\ri}\oint_{\omega}\frac{\del_t m_t^*(w)\rd w}{w-z}
		=\frac{1}{2\pi\ri}\oint_{\omega}\frac{\del_w \cB_t^*(w)\rd w}{w-z}=\frac{1}{2\pi\ri}\oint_{\omega}\frac{ \cB_t^*(w)\rd w}{(w-z)^2},
	\end{align}
	where the contour $\omega$ encloses  $[\fa(t),\fb(t)]$ but not $z$. 
	The equation \eqref{e:improve1} can therefore be viewed as the discrete analog of the complex Burgers equation \eqref{e:newBB2c}.
\end{rem}

We next use \Cref{p:improve} to approximate the difference $\Delta_{t + 1/n} (z_{t + 1/n}) - \Delta_t (z_t)$, through the following proposition.

\begin{prop}\label{p:dynamic}
	Under the assumptions of \Cref{p:rigidityBB}, for $0\leq t\leq \sigma$ (recall from \eqref{stoptime2}), uniformly for any $z_t=z_t(u) \in {\mathscr D}_t(\fr)\cup{\mathscr D}_t^{\rm L}\cup{\mathscr D}_t^{\rm F}$, we have the following equation for the difference of the Stieltjes transforms
	\begin{align}\begin{split}\label{e:mainEq}
			\Delta_{t+1/n}(z_{t+1/n})-\Delta_{t}(z_{t})
			&=\Delta M_t(z_{t+1/n})+\frac{\del_z f_t (z_t)}{n}\frac{e^{\log f^*_t (z_t) -\log  f_t (z_t)}-1}{ \big(f_t (z_t) +1 \big) \big(f^*_t (z_t)+1 \big)}\\
			& \qquad +\OO\left(\frac{\fM(t)}{n^2}+\frac{\Im [f^*_t(z_t)/(f^*_t(z_t)+1)]}{n^2\Im[z_t]\dist(z_t,I^+_t)}\right),
	\end{split}\end{align}
where the control parameter $\fM(t)$ is from \eqref{e:deftz2}. 
\end{prop}

\begin{proof}

	\Cref{p:improve} gives us that
	\begin{align}\begin{split}\label{e:dm0}
			&\phantom{{}={}} \int_{z-1/n}^{z}\cL_t \big( m_t(u) \big) \rd u=\frac{1}{2\pi \ri} \int_{z-1/n}^{z}\oint_{\omega}\frac{\log \cB_t(w)\rd w\rd u}{(w-u)^2} 
			+\OO\left(\frac{|\Im[f_t(z)/(f_t(z)+1)]|}{n^2\Im[z]\dist(z, I_t)}\right),
	\end{split}\end{align}
	where the contour $\omega$ encloses $[\fa(t), \fb(t)]$ but not any number on the segment $[z - 1/n, z]$. For the contour integral on the right side of \eqref{e:dm0}, we can rewrite it as
	\begin{align}\begin{split}\label{e:simp}
			\phantom{{}={}}\frac{1}{2\pi \ri} & \int_{z-1/n}^{z}\oint_{\omega}\frac{\log \cB_t(w)\rd w\rd u}{(w-u)^2} \\
			& =-\int_{z-1/n}^{z}\del_u \log \cB_t(u)\rd u +\frac{1}{2\pi \ri} \int_{z-1/n}^{z}\oint_{\omega+}\frac{\log \cB_t(w)\rd w\rd u}{(w-u)^2}\\
			&=- \big( \log \cB_t(z)-\log \cB_t(z-1/n) \big)+\frac{1}{2n\pi \ri} \oint_{\omega+}\frac{\log \cB_t(w)\rd w}{(w-z)^2}+\OO\left(\frac{1}{n^2}\right)\\
			&=- \Big(\log \big(f_t(z)+1 \big) \big(z-\fa(t) \big)-\log \big( f_t(z-1/n)+1 \big) \big(z-1/n-\fa(t) \big) \Big)\\
			& \qquad +\frac{1}{2n\pi \ri} \oint_{\omega+}\frac{\log \cB_t(w)\rd w}{(w-z)^2}+\OO\left(\frac{1}{n^2}\right),
	\end{split}\end{align}
	where the contour $\omega+\subseteq \mathscr D_t(\fr) $ encloses $[\fa(t), \fb(t)]$ and $[z - 1/n, z]$. Here, we have used the fact that the derivatives of $\log \mathcal{B}_t (w)$ are uniformly bounded for $w \in \omega+$, which holds since $\mathcal{B}_t (w) = \big( f_t (w) + 1 \big) \big( w - \mathfrak{a} (t) \big)$ and $\omega+$ is bounded away from $\big[ \mathfrak{a} (t), \mathfrak{b} (t) \big]$. 
	
	Recall $\tilde m_t(z)=m_t(z)-\log \big(z-\fa(t) \big)+\log \big( \fb(t)-z \big)$. By plugging \eqref{e:dm0} and \eqref{e:simp} into \eqref{e:dm}, and noticing that $\fa(t+1/n)=\fa(t)+1/n$ and $\mathfrak{b} (t + 1/n) = \mathfrak{b} (t)$, we get 
	\begin{align}\begin{split}\label{e:dmsimp}
			\phantom{{}={}}\tilde m_{t+1/n}(z)- \tilde m_t(z)
			& =m_{t+1/n}(z) -\log \big( z-\fa(t+1/n) \big)-  m_t(z)+\log \big( z-\fa(t) \big)\\
			&=\Delta M_t(z) - \Big( \log \big(f_t(z)+1 \big) - \log \big(f_t(z-1/n)+1 \big) \Big)\\
			& \qquad +\frac{1}{2n\pi \ri} \oint_{\omega+}\frac{\log \cB_t(w)\rd w}{(w-z)^2}+\OO\left(\frac{|\Im[f_t(z)/(f_t(z)+1)]|}{n^2\Im[z]\dist(z, I_t)}\right).
	\end{split}\end{align}
	We decompose the difference equation of $\tilde m_t(z_t)$ as
	\begin{align}\begin{split}\label{e:Deltatmt2}
			&\phantom{{}={}}
			\tilde m_{t+1/n}(z_{t+1/n})-\tilde m_{t}(z_t) =
			\big( \tilde m_{t+1/n}(z_{t+1/n})- \tilde m_{t}(z_{t+1/n}) \big)
			+ \big( \tilde m_{t}(z_{t+1/n})- \tilde m_{t}(z_{t}) \big).
	\end{split}\end{align}
	For the second term on the righthand side of \eqref{e:Deltatmt2}, thanks to \eqref{e:ftdda} 
	\begin{align}\label{e:Deltamt3}\begin{split}
			\tilde m_{t}(z_{t+1/n})- \tilde m_{t}(z_{t})&=\log f_t(z_{t+1/n})-\log f_t(z_t)- \big( \log\tilde g_t(z_{t+1/n})-\log\tilde g_t(z_t) \big)\\
			&=\log f_t(z_{t+1/n})-\log f_t(z_t)-\del_z\log\tilde g_t(z_{t})(z_{t+1/n}-z_t)+\OO\left(\frac{1}{n^2}\right),
	\end{split}\end{align}
	where we used $|\del_z^2\log \tilde g_t|\lesssim 1$ from the first statement of \Cref{p:gfbound}, and  as a consequence of \eqref{e:ft+1b2}: $|z_{t+1/n}-z_t|= |f^*(z_t)/n(f_t^*(z_t)+1)|\lesssim 1/n$.
	By plugging \eqref{e:dmsimp} and \eqref{e:Deltamt3} into \eqref{e:Deltatmt2}, we obtain the following difference equation of $\tilde m_t(z_t)$,
	\begin{align}\begin{split}\label{e:dmsimp2}
			&\phantom{{}={}}\tilde m_{t+1/n}(z_{t+1/n})- \tilde m_t(z_t)\\
			& \qquad =\Delta M_t(z_{t+1/n}) +\log f_t(z_{t+1/n})-\log f_t(z_t)-\log \big( f_t(z_{t+1/n})+1 \big) + \log \big( f_t(z_{t+1/n}-1/n) + 1 \big) \\
			& \qquad \quad +\frac{1}{2n\pi \ri} \oint_{\omega+}\frac{\log \cB_t(w)\rd w}{(w-z_{t+1/n})^2}-\del_z\log\tilde g_t(z_{t})(z_{t+1/n}-z_t)+\OO\left(\frac{|\Im[f_t(z_t)/(f_t(z_t)+1)]|}{n^2\Im[z_t]\dist(z_t, I_t)}\right),
	\end{split}\end{align}
	where by \eqref{e:imratio0} the error term is at least of order $1/n^2$. Finally by taking difference of \eqref{e:Deltamt} and \eqref{e:dmsimp2}, we get the difference equation of $\Delta_t(z_t)=m_t(z_t)-m_t^*(z_t)=\tilde m_t(z_t)-\tilde m_t^*(z_t)$,
	\begin{align}\begin{split}\label{e:Deltaeq4}
			&\phantom{{}={}}\Delta_{t+1/n}(z_{t+1/n})-\Delta_t(z_t)=\Delta M_t(z_{t+1/n})+\OO\left(\frac{|\Im[f_t(z_t)/(f_t(z_t)+1)]|}{n^2\Im[z_t]\dist(z_t, I_t)}\right)\\
			& + \log f_t(z_{t+1/n})-\log f_t(z_t)-\log \big(f_t(z_{t+1/n})+1 \big)+\log \big(f_t(z_{t+1/n}-1/n)+1 \big) \\
			&+\frac{1}{2n\pi \ri} \oint_{\omega+}\frac{(\log \cB_t(w)-\log \cB^*_t(w))\rd w}{(w-z_{t+1/n})^2}-(\del_z\log\tilde g_t(z_{t})-\del_z\log\tilde g^*_t(z_{t}))(z_{t+1/n}-z_t).
	\end{split}\end{align}
	In the following we estimate \eqref{e:Deltaeq4} term by term. By Taylor expansion around $z_t$, the first order term for the second line in \eqref{e:Deltaeq4} is given by
	\begin{align}\begin{split}\label{e:t1}
			&\phantom{{}={}}\del_z \log f_t(z_t) (z_{t+1/n}-z_t)-\frac{\del_z \log (f_t(z_t)+1)}{n}\\
			& \qquad =\frac{\del_z f_t(z_t)}{n f_t(z_t)}\frac{f_t^*(z_t)}{f_t^*(z_t)+1}-\frac{1}{n}\frac{\del_z f_t(z_t)}{f_t(z_t)+1}
			=\frac{\del_z f_t(z_t)}{n}\frac{f_t^*(z_t)/f_t(z_t)-1}{( f_t(z_t)+1)(f^*_t(z_t)+1)},
	\end{split}\end{align}
	where we used the equality $z_{t+1/n}-z_t=f_t^*(z_t)/n(f_t^*(z_t)+1)$.
	
	Let us show that the next order term for the second line in \eqref{e:Deltaeq4} is negligible and can be absorbed in the error term. To that end, first observe that $z_t$ is bounded away from either $\mathfrak{b} (t)$ or $\mathfrak{a} (t)$ (since $\mathfrak{b} (t) - \mathfrak{a} (t)$ is of order $1$). We only consider the former case, when $z_t$ is bounded away from $\fb(t)$, as the analysis for the latter is very similar. Then, $f_t(z_t)$ is bounded away from $0$; combining with \eqref{e:ft+1b2}, we deduce $\big| f_t(z_t)/(f_t(z_t)+1) \big|\asymp 1$. To cancel the possible singularity at $\fa(t)$, we reorganize the second line in \eqref{e:Deltaeq4} as
	\begin{align*}
		\log \frac{f_t(z_{t+1/n})}{f_t(z_{t+1/n})+1}-\log \frac{f_t(z_t)}{f_t(z_t)+1}-\log \big( f_t(z_t)+1 \big)+\log \big( f_t(z_{t+1/n}-1/n)+ 1 \big).
	\end{align*}
	The second order term by a Taylor expansion around $z_t$ is
	\begin{align}\begin{split}\label{e:tt2}
			&\phantom{{}={}}\del_z^2 \left.\log\frac{f_t(z)}{f_t(z)+1}\right|_{z=z_t}\frac{(z_{t+1/n}-z_t)^2}{2}+\del^2_z \log \big( f_t(z)+1 \big)|_{z=z_t}\frac{(z_{t+1/n}-z_t-1/n)^2}{2}\\
			& \qquad =\del_z^2 \left.\log\frac{f_t(z)}{f_t(z)+1}\right|_{z=z_t}\frac{f_t^*(z_t)^2}{2n^2 \big( f_t^*(z_t)+1 \big)^2}+\frac{\del^2_z \log (f_t(z)+1)|_{z=z_t}}{2 n^2 \big( f_t^*(z_t) + 1 \big)^2}.
	\end{split}\end{align}
	For the first term on the right side of \eqref{e:tt2}, we have
	\begin{align}\begin{split}\label{e:tt3}
			\left|\del_z^2 \left.\log\frac{f_t(z)}{f_t(z)+1}\right|_{z=z_t}\frac{f_t^*(z_t)^2}{2n^2 \big( f_t^*(z_t)+1 \big)^2}\right| 
			& \lesssim \frac{1}{n^2} \left(\left|\del^2_z\left. \frac{f_t(z)}{f_t(z)+1}\right|_{z=z_t}\right|+\left|\del_z\left. \frac{f_t(z)}{f_t(z)+1}\right|_{z=z_t}\right|^2\right)\\
			&\lesssim \frac{|\Im[f_t(z_t)/(f_t(z_t)+1)]|}{n^2\Im[z_t]\dist(z_t, I_t)}+ \frac{|\Im[f_t(z_t)/(f_t(z_t)+1)]|^2}{n^2\Im[z_t]^2} \\
			& \lesssim \frac{|\Im[f_t(z_t)/(f_t(z_t)+1)]|}{n^2\Im[z_t]\dist(z_t, I_t)},
	\end{split}\end{align}
	where we used that $\big| f^*_t(z_t)/(f^*_t(z_t)+1) \big|, \big| f_t (z_t) / (f_t (z_t) + 1) \big| \asymp 1$ in the first inequality, \eqref{e:derf/f+1} in the second inequality, and \eqref{e:imbb} for the last inequality. For the second term on the righthand side of \eqref{e:tt2}, we have
	\begin{align}\begin{split}\label{e:tt4}
			\phantom{{}={}}\frac{\big| \del^2_z \log (f_t(z)+1)|_{z=z_t} \big|}{2n^2 \big|f_t^*(z_t)+1 \big|^2}
			& \lesssim \frac{\big| \del^2_z \log (f_t(z)+1)|_{z=z_t} \big|}{2n^2 \big|f_t(z_t)+1 \big|^2}\\
			&\lesssim \frac{1}{n^2} \left(\frac{1}{ \big| f_t(z_t)+1 \big|}\left|\del^2_z\left. \frac{f_t(z)}{f_t(z)+1}\right|_{z=z_t}\right|+\left|\del_z\left. \frac{f_t(z)}{f_t(z)+1}\right|_{z=z_t}\right|^2\right) \\
			& \lesssim \frac{|\Im[f_t(z_t)/(f_t(z_t)+1)|]}{n^2\Im[z_t]\dist(z_t, I_t)},
	\end{split}\end{align}
	where we used \eqref{e:gfdidd} in the first inequality, and \eqref{e:derf/f+1} in the second inequality. 
	Both error terms \eqref{e:tt3} and \eqref{e:tt4} are in the same form as the error term in \eqref{e:Deltaeq4}, and can thus be absorbed into it.

	For the first term in the last line in \eqref{e:Deltaeq4}, the contour $\omega+\subseteq \mathscr D_t(\fr) $, and $w\in \omega+$ is bounded away from $[\fa(t), \fb(t)]$.  Thus, since $t \leq \sigma$ we have $|m_t(w)-m_t^*(w)|\leq \fM(t)/n$ from \eqref{stoptime2}. It thus follows from \eqref{e:gfdidd} that $|\log f_t(w)-\log f_t^*(w)|\lesssim \fM(t)/n$, and 
	\begin{align*}
		\big|\log\cB_t(w)-\log \cB_t^*(w) \big|
		&= \left|\log\left(1+(e^{\log f_t(w)-\log f_t^*(w)}-1)\frac{f_t^*(w)}{f_t^*(w)+1}\right)\right|\lesssim \frac{\fM(t)}{n}.
	\end{align*}
	
	\noindent Hence,
	\begin{align}\label{e:t2}
		\left|\frac{1}{2n\pi \ri} \oint_{\omega+}\frac{\big( \log \cB_t(w)-\log \cB^*_t(w) \big)\rd w}{(w-z_{t+1/n})^2}\right|\lesssim  \frac{\fM(t)}{n^2}.
	\end{align}
	For the last term in the last line in \eqref{e:Deltaeq4}, using \eqref{e:gfdidd} and $|z_{t+1/n}-z_t|\lesssim 1/n$, we have
	\begin{align}\label{e:t3}
		\Big| \big(\del_z\log\tilde g_t(z_{t})-\del_z\log\tilde g^*_t(z_{t}) \big)(z_{t+1/n}-z_t) \Big|\lesssim \frac{\fM(t)}{n^2}.
	\end{align}
	For the error term in \eqref{e:Deltaeq4},  thanks to \eqref{e:extremepp2}, for $0\leq t\leq \sigma$, it holds $I_t\subset I_t^+$, and thus $\dist(z_t, I^+_t)\leq \dist(z_t, I_t)$. Using \eqref{e:replacep3}, we can rewrite the error as
	\begin{align}\label{e:t4}
		\frac{|\Im[f_t(z_t)/(f_t(z_t)+1)]|}{n^2\Im[z_t]\dist(z_t, I_t)}=
		\OO\left(\frac{|\Im[f^*_t(z_t)/(f^*_t(z_t)+1)]|}{n^2\Im[z_t]\dist(z_t, I^+_t)}\right).
	\end{align}
	The claim \eqref{e:mainEq} follows from combining \eqref{e:t1}, \eqref{e:t2}, \eqref{e:t3} and \eqref{e:t4}.
\end{proof}

\subsection{Optimal Rigidity Estimates}\label{s:OptimalRigidity}

In this section, we analyze \eqref{e:mainEq} and prove \Cref{p:rigidityBB}. By abbreviating $f_s = f_s (z_s)$ and $f_s^* = f_s^* (z_s)$, and then summing over \eqref{e:mainEq} from time $0$ to $t\wedge \sigma$, we get
\begin{align}\begin{split}\label{e:mainEq2}
		\Delta_{t\wedge\sigma}(z_{t\wedge\sigma})&=\Delta_{0}(z_{0})
		+\sum_{s\in [0,t\wedge \sigma)\cap \bZ_n}\Delta M_s(z_{s+1/n})+\sum_{s\in [0,t\wedge \sigma)\cap \bZ_n}\frac{\del_z f_s}{n}\frac{e^{\log f^*_s-\log  f_s}-1}{( f_s+1)(f^*_s+1)}\\
		& \qquad +\sum_{s\in [0,t\wedge \sigma)\cap \bZ_n}\OO\left(\frac{\fM(s)}{n^2}+\frac{|\Im[f^*_s(z_s)/(f^*_s(z_s)+1)]|}{n^2\Im[z_s]\dist(z_s,I^+_s)}\right).
\end{split}\end{align}

The following proposition bounds the martingale terms in \eqref{e:mainEq2}.\begin{prop}\label{c:error}
	Under the assumptions of \Cref{p:rigidityBB},  there exists an event $\cW$  with overwhelming probability on which the following holds for every  $u \in \mathscr L$.
	If $z_t=z_t(u)\in {\mathscr D}_t(\fr)\cup \mathscr D_t^{\rm L}$, then
	\begin{align}\label{eq:estet1}
		\left|\sum_{s\in [0,t\wedge\sigma)\cap \bZ_n}\Delta M_s(z_{s+1/n})\right|\lesssim \frac{ n^{\fd/4}}{n\Im[z_{t\wedge \sigma}]}.
	\end{align}
	If $z_t\in {\mathscr D}_t^{\rm F}$ then
	\begin{align}\label{eq:estet}
		\left|\sum_{s\in [0,t\wedge\sigma)\cap \bZ_n} \Delta M_s(z_{s+1/n})\right|\lesssim n^{\fd/4}\left(\frac{1}{n\sqrt{\Im[z_{t\wedge \sigma}] \dist(z_{t\wedge \sigma}, I^*_{t\wedge \sigma})}}+\frac{1}{(n\Im[z_{t\wedge \sigma}])^2}\right).
	\end{align}
\end{prop}

\begin{proof}
	In the rest of the proof, we simply write $\Delta M_s=\Delta M_s(z_{s+1/n})$. For any time  $0\leq t\leq \ft$ such that $z_t\in \mathscr D_t(\fr)\cup \mathscr D_t^{\rm L}\cup \mathscr D_t^{\rm F}$, the Burkholder--Rosenthal inequality gives
	\begin{align}\label{e:Ros}
		\bE\left|\sup_{r\leq t}\sum_{s\in [0,r\wedge\sigma)\cap \bZ_n} \Delta M_s\right|^{2p}
		\leq \fC_p \max\left\{ \bE \Bigg[ \bigg(\sum_{s\in [0,t\wedge \sigma)\cap \bZ_n}\bE \big[ | \Delta M_s|^2 \big| \bmx_s \big] \bigg)^p \Bigg],\sum_{s\in [0,t\wedge \sigma)\cap \bZ_n}\bE \big[ |\Delta M_s|^{2p} \big] \right\}
	\end{align}
	To bound the first term in \eqref{e:Ros2}, thanks to \eqref{e:improve2} in \Cref{p:improve} by taking $p=1$,
	\begin{align}\label{e:bfirst}
		\sum_{s\in [0,t\wedge \sigma)\cap \bZ_n}\bE \big[ |\Delta M_s|^2 \big| \bmx_{s} \big]&
		\lesssim\frac{1}{n^3}\sum_{s\in [0,t\wedge \sigma)\cap \bZ_n} \frac{|\Im[f_s(z_s)/(f_s(z_s)+1)]|}{\Im[z_s]\dist(z_s, I_s)^2}.
	\end{align}
	For the second term in \eqref{e:Ros}, thanks to \eqref{e:improve2}
	\begin{align}\label{e:bsecond}
		\bE \big[ |\Delta M_s|^{2p} \big| \bmx_{s} \big] \lesssim 
		\left(\left(\frac{1}{n^3} \frac{|\Im[f_s(z_s)/(f_s(z_s)+1)]|}{\Im[z_s]\dist(z_s, I_s)^2}\right)^p +\frac{|\Im[f_s(z_s)/(f_s(z_s)+1)]|}{n^{4p-1}\Im[z_s] \dist(z_s,I_s)^{4p-2}}\right),
	\end{align}
	
	\noindent Observe that the sum over $s$ of the first term on the right side of \eqref{e:bsecond} can be bounded by the $p$-th moment of \eqref{e:bfirst}. Thus, \eqref{e:Ros} implies
	\begin{align}\begin{split}\label{e:Ros2}
			\eqref{e:Ros}
			&\lesssim \bE\left[\left(\frac{1}{n^3}\sum_{s\in [0,t\wedge \sigma)\cap \bZ_n} \frac{|\Im[f_s(z_s)/(f_s(z_s)+1)]|}{\Im[z_s]\dist(z_s, I_s)^2}\right)^p\right]\\
			& \qquad + \sum_{s\in [0,t\wedge\sigma)\cap \bZ_n}\bE\left[\frac{|\Im[f_s(z_s)/(f_s(z_s)+1)]|}{n^{4p-1}\Im[z_s] \dist(z_s,I_s)^{4p-2}}\right].
	\end{split}\end{align}
	
	For the first term on the righthand side of \eqref{e:Ros2}, 
	since we always have that $|z_s-x|\geq \Im z_s$ for $x\in I_s$,   
	it follows that 
	\begin{align}\begin{split}\label{e:bbterma}
			\frac{1}{n^3}\sum_{s\in [0,t\wedge \sigma)\cap \bZ_n} \frac{|\Im[f_s(z_s)/(f_s(z_s)+1)]|}{\Im[z_s]\dist(z_s, I_s)^2} & \lesssim \frac{1}{n^3}\sum_{s\in [0,t\wedge \sigma)\cap \bZ_n}  \frac{|\Im[ f^*_s(z_s)/( f^*_s(z_s)+1)]|}{\Im[z_s]^3}\\
			&   \lesssim  
			\frac{1}{n^2}\int_0^{t\wedge\sigma}\frac{|\Im[ f^*_s(z_s)/( f^*_s(z_s)+1)]|}{\Im[z_s]^3} \mathrm{d} s 
			\lesssim \frac{1}{n^2\Im[z_{t\wedge \sigma}]^2},
		\end{split}
	\end{align} 
	where we used \eqref{e:replacep3} for the first inequality, and \eqref{e:integral0} for the last inequality. 
	
	For the second term in \eqref{e:Ros2}, we have
	\begin{align}\begin{split}\label{e:netas}
			\sum_{s\in [0,t\wedge \sigma)\cap \bZ_n}\bE\left[\frac{|\Im[f(z_s)/(f(z_s)+1)]|}{n^{4p-1}\Im[z_s] \dist(z_s,I_s)^{4p-2}}\right] & \lesssim  \frac{1}{n^{4p-1}}\sum_{s\in [0,t\wedge \sigma)\cap \bZ_n}  \frac{|\Im[ f^*_s(z_s)/( f^*_s(z_s)+1)]|}{\Im[z_s]^{4p-1}}\\
			&\lesssim \frac{1}{n^{4p-2}}\int_0^{t\wedge\sigma}\frac{|\Im[ f^*_s(z_s)/( f^*_s(z_s)+1)]|}{\Im[z_s]^{4p-1}} \\
			& \leq \frac{1}{n^{4p-2}\Im[z_{t\wedge \sigma}]^{4p-2}},
	\end{split}\end{align}
	where we used \eqref{e:integral0} for the last inequality. By plugging \eqref{e:bbterma} and \eqref{e:netas} into \eqref{e:Ros2}, we get
	\begin{align*}
		\bE\left|\sup_{r\leq t}\sum_{s\in [0,r\wedge\sigma)\cap \bZ_n} \Delta M_s\right|^{2p}
		\leq \fC_p\left( \frac{1}{n^{2p}\Im[z_{t\wedge \sigma}]^{2p}}+\frac{1}{n^{4p-2}\Im[z_{t\wedge \sigma}]^{4p-2}}\right).
	\end{align*}
	Therefore, by taking $p$ large in Rosenthal's inequality \eqref{e:Ros2}, for any $u\in {\mathscr L}$ and $z_t=z_t(u)\in {\mathscr D}_t(\fr)\cup \mathscr D_t^{\rm L}$, the following holds with overwhelming probability, i.e., $1-n^{-D}$ for any constant $D > 1$ (assuming $n$ is sufficiently large),
	\begin{align}\begin{split}\label{e:continuityarg}
			\sup_{r\leq t}\left|\sum_{s\in [0, r\wedge \sigma)\cap \bZ_n} \Delta M_s\right|\lesssim  \frac{n^{\fd/4}}{n\Im[z_{t\wedge\sigma}]}.
		\end{split}
	\end{align}
	
	If $z_t\in \mathscr D_t^{\rm F}$, we have
	\begin{align}\begin{split}\label{e:edgevariance}
			\frac{1}{n^3}  \sum_{s\in [0,t\wedge \sigma)\cap \bZ_n}\frac{| \Im[f_s(z_s)/( f_s(z_s)+1)]|}{\eta_s \dist(z_s,  I_s)^2}
			&  \lesssim  
			\frac{1}{n^3}\sum_{s\in [0,t\wedge \sigma)\cap \bZ_n} \frac{|\Im[ f^*_s(z_s)/( f^*_s(z_s)+1)]|}{\eta_s \dist(z_s,  I^+_s)^2}\\
			&\lesssim \frac{1}{n^2}\int_0^{t\wedge \sigma}\frac{|\Im[ f^*_s(z_s)/( f^*_s(z_s)+1)]|\rd s}{\eta_s \dist(z_s,  I^+_s)^2}
			\lesssim  \frac{1} {n^{2}} \frac{1}{\eta_{t \wedge \sigma} (\kappa_{t \wedge \sigma}+\eta_{t \wedge \sigma})},
	\end{split}\end{align}
	where we have denoted $z_s = \kappa_s + \mathrm{i} \eta_s$, and have used  \eqref{e:replacep3} and \eqref{e:extremepp2} for the first inequality, and \eqref{e:integral3} for the last inequality. Therefore, by plugging \eqref{e:netas} and \eqref{e:edgevariance} into \eqref{e:Ros2}, we get
	\begin{align}\label{e:Ros3}
		\bE\left|\sup_{r\leq t}\sum_{s\in [0,r\wedge\sigma)\cap \bZ_n} \Delta M_s\right|^{2p}
		\leq \fC_p\left( \frac{1}{n^{2p} \big( \Im[z_{t\wedge \sigma}]\dist(z_{t\wedge \sigma}, I_{t\wedge \sigma}^*) \big)^{p}}+\frac{1}{n^{4p-2}\Im[z_{t\wedge \sigma}]^{4p-2}}\right).
	\end{align}
	By taking $p$ large in \eqref{e:Ros3},
	for any $u\in {\mathscr L}$ and $z_t(u)\in \mathscr D_t^{\rm F}$, we deduce that, with overwhelming probability,
	\begin{align}\begin{split}\label{eq:2term}
			\sup_{r\leq t}\left|\sum_{s\in [0,r\wedge\sigma)\cap \bZ_n} \Delta M_s\right|\leq n^{\fd/4}\left(\frac{ 1}{n \sqrt{\Im[z_{t\wedge \sigma}]\dist(z_{t\wedge \sigma}, I_{t\wedge \sigma}^*)}}+\frac{1}{ \big( n\Im[z_{t\wedge \sigma}] \big)^2}\right).
		\end{split}
	\end{align}
	We define $\cW$ to be the set  on which \eqref{e:continuityarg} and \eqref{eq:2term} hold for  any $u\in \mathscr L$ and any time $0\leq t\leq \ft$ such that $z_t=z_t(u)\in {\mathscr D}_t(\fr)\cup \mathscr D_t^{\rm L}\cup \mathscr D_t^{\rm F}$.
	The above discussions  imply that  $\cW$ holds with overwhelming probability, thereby establishing the proposition.
\end{proof}

Now we can establish \Cref{p:rigidityBB}. 

\begin{proof}[Proof of \Cref{p:rigidityBB}]
	We can now start analyzing the difference equation \eqref{e:mainEq2} for $\Delta_t(z_t)$. By the third statement of \Cref{p:gap}, for any lattice point $u\in {\mathscr L} = \mathscr{L}_t$ with $z_t(u)\in \mathscr D_t(\fr)$, we have that for any $0\leq s\leq t$  that $z_s=z_s(z)\in \mathscr D_s(\fr)$. As a consequence,  \eqref{e:gfdidd} and our construction of the stopping time \eqref{stoptime2} imply that $\big| \log f_s (z_s) -\log  f^*_s (z_s) \big|\lesssim \fM(s)/ n$ for $s\leq \sigma$, and so
	\begin{align}\label{e:sectt}
		\left|\frac{\del_z f_s}{n}\frac{e^{\log f^*_s-\log  f_s}-1}{( f_s+1)(f^*_s+1)}\right|
		=\Bigg|\frac{\del_z f_s (e^{\log f^*_s (z_s)-\log  f_s (z_s)}-1)}{n (f_s +1)^2} \bigg(1+\frac{f_s^* (z_s) (e^{\log f_s (z_s) -\log  f^*_s (z_s)}-1)}{f_s^* (z_s)+1} \bigg)  \Bigg|
		\lesssim \frac{\fM(s)}{n^2},
	\end{align}
	where we used \eqref{e:ft+1b2} and the bound $\big| \del_z (f_s/(f_s+1)) \big|\lesssim 1/\dist(z, I_s^+)\lesssim 1$ (which holds by \eqref{e:derf/f+1} and \eqref{e:imbb}). By \Cref{c:error}, on the event $\cW$, we also have 
	\begin{align}\label{e:Deltafar}
		\left|\sum_{s\in [0,t\wedge \sigma)\cap \bZ_n}\Delta M_s(z_{s+1/n})\right|\lesssim n^{\fd/4 - 1}.
	\end{align}
	Under our assumption \eqref{e:boundaryc0} for $z_0=u \in {\mathscr D}_0(\fr)\cup{\mathscr D}_0^{\rm L}\cup{\mathscr D}_0^{\rm F}$, we have
	\begin{align}\label{e:boundarycm2}
		\big| m_0^*(u)-m_0(u) \big|\leq \int_{-\infty}^{\infty} \frac{\big| H^*_0(x)-H_0(x) \big|\rd x}{|u-x|^2}\lesssim \frac{n^{\fd/3 - 1}}{\dist(u, I_0^*)}.
	\end{align}
	In particular, for $z_0\in {\mathscr D}_0(\fr)$,
	\eqref{e:boundarycm2} gives $\big| \Delta_0(z_0) \big|\lesssim n^{ \fd/3 - 1}$.
	By plugging \eqref{e:sectt} and \eqref{e:Deltafar} in \eqref{e:mainEq2}, we obtain
	\begin{align}\label{e:Deltafar2}
		\big|\Delta_{t\wedge\sigma}(z_{t\wedge\sigma}) \big|\lesssim n^{ \fd/3 - 1}+\sum_{s\in [0,t\wedge \sigma)\cap \bZ_n}\frac{\fM(s)}{n^2}\lesssim n^{ \fd/3 - 1}+\int_0^{t\wedge \sigma}\frac{\fM(s)\rd s}{n}.
	\end{align}
	We recall from \eqref{e:deftz2}, that $\fM(t)=M e^{M t}n^{ \fd/3}$, so we can further simplify the right side of \eqref{e:Deltafar2} and conclude that on the event $\cW$,
	\begin{align*}
		\big| \Delta_{t\wedge\sigma}(z_{t\wedge\sigma})\big|\lesssim 
		(1+e^{M (t\wedge \sigma)})n^{ \fd/3 - 1}, \qquad \text{so that} \qquad \big| \Delta_{t \wedge \sigma} (z_{t \wedge \sigma} ) \big| \leq \frac{M}{2} e^{M (t\wedge \sigma)}n^{ \fd/3- 1},
	\end{align*}
	provided $M$ is large enough. By our construction of the lattice set $\mathscr L$, \eqref{e:approxw} implies that for any $w\in {\mathscr D}_{t\wedge\sigma}(\fr)$, there exists some $u\in {\mathscr L}$ such that $z_{t\wedge\sigma}(u)\in {\mathscr D}_{t\wedge\sigma}(\fr)$, and $\big| z_{t\wedge\sigma}(u)-w \big|\lesssim 1/n^2$. Moreover, on the domain $\mathscr D_{t\wedge \sigma}(\fr)$, both $ m_{t\wedge \sigma}$ and $m^*_{t\wedge\sigma}$ are Lipschitz with Lipschitz constant $\OO(1)$. Therefore 
	\begin{align}\begin{split}\label{e:useLip}
			\big| m_{t\wedge\sigma}(w)-m^*_{t\wedge\sigma}(w) \big| & \leq 
			\Big|m_{t\wedge\sigma} \big(z_{t\wedge\sigma}(u) \big)-m^*_{t\wedge\sigma} \big(z_{t\wedge\sigma}(u) \big) \Big|\\
			& \qquad +
			\Big| m_{t\wedge\sigma}(w)- m_{t\wedge\sigma} \big(z_{t\wedge\sigma}(u) \big) \Big|+
			\Big| m^*_{t\wedge\sigma}(w)-m^*_{t\wedge\sigma} \big(z_{t\wedge\sigma}(u) \big) \Big|\\
			&\leq \frac{M}{2} e^{M (t\wedge \sigma)} n^{ \fd/3 - 1}+\OO\left(\frac{1}{n^2}\right) < M e^{M (t\wedge \sigma)}n^{\fd/3 - 1}=\frac{\fM(t\wedge \sigma)}{n}.
	\end{split}\end{align}

	In the following we analyze \eqref{e:mainEq2} for $z_t$ close to the real interval $\big[ \mathfrak{a} (t), \mathfrak{b} (t) \big]$. For the last error term in \eqref{e:mainEq2}, 
	\begin{align}\label{e:errorMt}
		\sum_{s\in [0,t\wedge \sigma)\cap \bZ_n}\frac{\fM(s)}{n^2}\lesssim \int_0^{t\wedge \sigma}\frac{\fM(s)\rd s}{n}
		=e^{M (t\wedge \sigma)}n^{ \fd/3 - 1}\ll n^{\fd/2-1}.
	\end{align}
	Since we always have that $\dist(z_s, I_s^+)\geq \Im z_s$, we can bound the error term, as in \eqref{e:bbterma} 
	\begin{align}\begin{split}\label{e:errort1}
			\phantom{{}={}}\sum_{s\in [0,t\wedge \sigma)\cap \bZ_n}\frac{|\Im[f^*_s(z_s)/(f^*_s(z_s)+1)]|}{n^2\Im[z_s]\dist(z_s,I^+_s)}
			& \lesssim 
			\frac{1}{n}\int_0^{t\wedge\sigma}\frac{|\Im[ f^*_s(z_s)/( f^*_s(z_s)+1)]|\rd s}{\Im[z_s]^2}
			\lesssim
			\frac{1}{n\Im[z_{t\wedge \sigma}]},
		\end{split}
	\end{align} 
	
	\noindent where we have used \eqref{e:integral0} for the last inequality. For $z_t\in \mathscr D_t^{\rm F}$, 
	we can bound the error term using \eqref{e:integral2}
	\begin{align}\begin{split}\label{e:errort2}
			\sum_{s\in [0,t\wedge \sigma)\cap \bZ_n}\frac{|\Im[f^*_s(z_s)/(f^*_s(z_s)+1)]|}{n^2\Im[z_s]\dist(z_s,I^+_s)}
			&\lesssim 
			\frac{1}{n}\int_0^{t\wedge\sigma}\frac{|\Im[ f^*_s(z_s)/( f^*_s(z_s)+1)]|\rd s}{\Im[z_s]\dist(z_s, I_s^+)}\\
			&\lesssim\frac{\log n}{n\sqrt{\Im[z_{t\wedge \sigma}]\dist(z_{t\wedge \sigma}, I^*_{t\wedge \sigma})}}.
		\end{split}
	\end{align}

	Therefore, for $u\in {\mathscr L}$ with $z_t(u)\in {\mathscr D}_t^{\rm L}$, by \eqref{eq:estet1}, \eqref{e:errorMt} and \eqref{e:errort1} we can write \eqref{e:mainEq2} as 
	\begin{align}\label{e:deltaee1}
		\big| \Delta_{t\wedge\sigma}(z_{t\wedge\sigma}) \big|
		\leq \big|\Delta_{0}(z_{0}) \big|+\frac{1}{n}\sum_{s\in [0,t\wedge \sigma)\cap \bZ_n}\left|\del_z  f_s\frac{e^{\log f^*_s-\log  f_s}-1}{( f_s+1)(f^*_s+1)}\right|+\OO\left(\frac{n^{ \fd/2 - 1}}{\Im[z_{t\wedge\sigma}] }\right).
	\end{align}
	Similarly, for $u\in {\mathscr L}$ with $z_t(u)\in {\mathscr D}_t^{\rm F}$, using \eqref{eq:estet}, \eqref{e:errorMt} and \eqref{e:errort2} we can write \eqref{e:mainEq2} as 
	\begin{align}\begin{split}\label{e:deltaee2}
			\big|\Delta_{t\wedge\sigma}(z_{t\wedge\sigma}) \big|
			\leq \big| \Delta_{0}(z_{0}) \big|&+\frac{1}{n}\sum_{s\in [0,t\wedge \sigma)\cap \bZ_n}\left|\del_z  f_s\frac{e^{\log f^*_s-\log  f_s}-1}{( f_s+1)(f^*_s+1)}\right|\\
			&+\OO\left(\frac{n^{ \fd/2}}{n\sqrt{\Im[z_{t\wedge\sigma}] \dist(z_{t\wedge\sigma}, I^*_{t\wedge\sigma})}}+\frac{n^{\fd/2}}{(n\Im[z_{t\wedge\sigma}])^2}\right).
	\end{split}\end{align}
	For the summation in \eqref{e:deltaee1} and \eqref{e:deltaee2}, we slightly rewrite the summand as,
	\begin{align*}\begin{split}
			\left|\del_z  f_s (z_s) \frac{e^{\log f^*_s (z_s) -\log  f_s (z_s)}-1}{ \big( f_s (z_s) +1 \big) \big(f^*_s (z_s) +1 \big)}\right| & = \Bigg|\frac{\del_z f_s (z_s)}{\big( f_s (z_s)+1 \big)^2} \bigg(1+\frac{f_s^* (z_s) (e^{\log f_s (z_s) -\log  f^*_s (z_s)}-1)}{f_s^* (z_s)+1} \bigg) \\
			& \qquad \times \frac{e^{\log f^*_s (z_s)-\log  f_s (z_s)}-1}{\log f_s (z_s) -\log  f^*_s (z_s)} \Bigg| \big| \log f_s (z_s) -\log  f^*_s (z_s) \big|.
	\end{split}\end{align*}
	We denote the first factor on the right side by $\alpha_s$,
	\begin{align}\label{e:defbetas}
		\alpha_s:=\Bigg|\frac{\del_z f_s}{(f_s+1)^2} \bigg(1+\frac{f_s^*(e^{\log f_s-\log  f^*_s}-1)}{f_s^*+1} \bigg)\frac{e^{\log f^*_s-\log  f_s}-1}{\log f_s-\log  f^*_s} \Bigg|,
	\end{align}
	where we omitted the dependence on $z_s$.
	For $0\leq s< t\wedge \sigma$, by \eqref{e:gfdidd} and \eqref{e:replacebound} we have $|\log f^*_s-\log  f_s|\lesssim n^{-\fd}$, and by \eqref{e:ft+1b2} we have $|f_s^*/(f_s^*+1)|\lesssim 1$. So,
	\begin{align}\begin{split}\label{e:aterm}
			\alpha_s = \left(1+\frac{\OO(1)}{n^{{\fd}}}\right)\left|\frac{\del_z  f_s(z_s)}{( f_s(z_s)+1)^2}\right|
			& = \left(1+\frac{\OO(1)}{n^{{\fd}}}\right)\left|\left.\del_z  \left(\frac{f_s}{f_s+1}\right)\right|_{z=z_s}\right|\\
			& \lesssim  \left(1+\frac{\OO(1)}{n^{{\fd}}}\right)\frac{|\Im[  f_s(z_s)/( f_s(z_s)+1)]|}{\Im z_s}\\
			& \lesssim \left(1+\frac{\OO(1)}{n^{{\fd}}}\right)\frac{|\Im[ f^*_s(z_s)/(1+f^*_s(z_s))]|}{\Im z_s}
			= \left(1+\frac{\OO(1)}{n^{{\fd}}}\right)\frac{-\del_s \Im z_s}{\Im z_s},
	\end{split}\end{align}
	where in the second line we used \eqref{e:derf/f+1}, in the third line we used \eqref{e:replacep2}, and in the last equation we used the equality $\partial_s z_s = f_s^* (z_s) / \big( f_s^* (z_s) + 1 \big)$. With $\alpha_s$, we can by \eqref{e:gfdidd} rewrite \eqref{e:deltaee1} as
	\begin{align*}\begin{split}
			\big| \Delta_{t\wedge \sigma}(z_{t\wedge \sigma}) \big|
			\leq \frac{1}{n}\sum_{s\in [0,t\wedge \sigma)\cap \bZ_n}\alpha_s(z_s) \big|\Delta_s(z_s) \big|
			+\OO\left( \big| \Delta_0(z_0) \big|+\frac{n^{\fd/2}}{n\Im[z_{t\wedge\sigma}]}\right).
		\end{split}
	\end{align*}
	
	\noindent By discrete Gr{\" o}nwall's inequality \cite{holte2009discrete}, this implies the estimate
	\begin{align}\begin{split}\label{e:midgronwall0}
			\left|\Delta_{t\wedge\sigma}(z_{t\wedge\sigma})\right|
			&\lesssim \big| \Delta_0(z_0) \big|+\frac{n^{\fd/2}}{n\Im[z_{t\wedge\sigma}]} +
			\frac{1}{n}\sum_{s\in [0,t\wedge \sigma)\cap \bZ_n}\alpha_s\prod_{\tau\in (s, t\wedge \sigma)\cap \bZ_n}\left(1+\frac{\alpha_\tau}{n}\right) \cdot \left(|\Delta_0(z_0)|+\frac{n^{\fd/2}}{n\Im[z_{s}]} \right)\\
			&\lesssim \big| \Delta_0(z_0) \big|+\frac{n^{\fd/2}}{n\Im[z_{t\wedge\sigma}]} +\int_0^{t\wedge\sigma}\alpha_se^{\int_s^{t\wedge\sigma} \alpha_\tau\rd \tau}\left(|\Delta_0(z_0)|+\frac{n^{\fd/2}}{n\Im[z_{s}]} \right)\rd s.
		\end{split}
	\end{align}
	For the integral of $\alpha_\tau$, we have
	\begin{align}\label{e:betaint}
		e^{\int_s^{t\wedge\sigma} \alpha_\tau\rd \tau}
		\leq e^{\left(1+\frac{\OO(1)}{n^{\fd}}\right)\log \left(\frac{\Im[z_{s}]}{\Im[z_{t\wedge\sigma}]}\right)}
		= \left(\frac{\Im[z_{s}]}{\Im[z_{t\wedge\sigma}]}\right)^{1+\frac{\OO(1)}{n^{\fd}}}
		\leq \OO(1)  \frac{\Im[z_s]}{\Im[z_{t\wedge\sigma}]}.
	\end{align}
	
	\noindent In the last inequality, we used that by our construction of ${\mathscr D}_t^{\rm L}$,  we have $\Im[z_{s}]/\Im[z_{t\wedge\sigma}]=\OO(n)$.
	Combining the above inequality \eqref{e:betaint} with \eqref{e:aterm} we can bound the last term in \eqref{e:midgronwall0} by
	\begin{align}\begin{split}\label{e:term20}
			\phantom{{}={}}\int_0^{t\wedge\sigma}\alpha_s e^{\int_s^{t\wedge\sigma} \alpha_\tau\rd \tau}\left(|\Delta_0(z_0)|+\frac{n^{\fd/2}}{n\Im[z_s]} \right)\rd s
			& \lesssim
			\int_0^{t\wedge\sigma}\frac{-\del_s\Im[ z_s]}{\Im[z_{t\wedge\sigma}]}\left(|\Delta_0(z_0)|+\frac{n^{\fd/2}}{n \Im[z_{s}] }\right)\rd s\\
			&\lesssim 
			\frac{n^{ \fd/2}}{n\Im[z_{t\wedge \sigma}]}\int_0^{t\wedge\sigma}\frac{-\del_s\Im[ z_s]}{\Im[z_{s}] }\rd s
			+\frac{\big| \Delta_0(z_0) \big|\Im[z_0]}{\Im[z_{t\wedge \sigma}]}\\
			&\lesssim 
			\frac{n^{ \fd/2}}{n\Im[z_{t\wedge \sigma}]}
			\log\left(\frac{\Im[z_0]}{\Im[z_{t\wedge \sigma}]}\right)+\frac{\big| \Delta_0(z_0) \big|\Im[z_0]}{\Im[z_{t\wedge \sigma}]}\\
			&\lesssim\frac{(\log n)n^{\fd/2}+n^{ \fd/3}}{n\Im[z_{t\wedge \sigma}]}
			\lesssim \frac{n^{ \fd - 1}}{\Im[z_{t\wedge \sigma}]},
	\end{split}\end{align} 
	where in the last line we used \eqref{e:boundarycm2}, which implies $|\Delta_0(z_0)|\lesssim n^{\mathfrak{d}/3 - 1} / \Im[z_0]$. It follows by combining  \eqref{e:midgronwall0} and \eqref{e:term20} that
	\begin{align}\label{e:Deltabb}
		\left|\Delta_{t\wedge\sigma}(z_{t\wedge\sigma})\right|
		\leq  \frac{Mn^{\fd - 1}}{2\Im[z_{t\wedge \sigma}]},
	\end{align}
	provided we take $M$ large enough. 
	
	For $u\in {\mathscr L}$ with $z_t=z_t(u)\in {\mathscr D}_t^{\rm F}$, similarly to \eqref{e:midgronwall0}, we deduce from \eqref{e:deltaee2} that
	\begin{align}\begin{split}\label{e:midgronwall}
			\big|\Delta_{t\wedge\sigma}(z_{t\wedge\sigma}) \big|
			&\lesssim \big| \Delta_0 (z_0) \big|+\frac{n^{ \fd/2}}{n\sqrt{\Im[z_{t\wedge\sigma}] \dist(z_{t\wedge\sigma}, I^*_{t\wedge\sigma})}}+\frac{n^{\fd/2}}{\big( n\Im[z_{t\wedge\sigma}] \big)^2}\\
			& \qquad +\int_0^{t\wedge\sigma}\alpha_se^{\int_s^{t\wedge\sigma} \alpha_\tau\rd \tau}\left( \big| \Delta_0 (z_0) \big|+\frac{n^{ {\fd/2}}}{n\sqrt{\Im[z_{s}] \dist(z_s, I^*_s)}}+\frac{n^{\fd/2}}{\big( n\Im[z_s] \big)^2}\right)\rd s.
		\end{split}
	\end{align}
	Thanks to \eqref{e:aterm} and \eqref{e:betaint}, we can bound the last term in \eqref{e:midgronwall} by
	\begin{align}\begin{split}\label{e:term2}
			\phantom{{}={}}\int_0^{t\wedge\sigma}\alpha_s & e^{\int_s^{t\wedge\sigma} \alpha_\tau\rd \tau}\left( \big| \Delta_0 (z_0) \big|+\frac{n^{ {\fd/2}}}{n\sqrt{\Im[z_{s}] \dist(z_s, I^*_s)}}+\frac{n^{\fd/2}}{ \big( n\Im[z_{s}] \big)^2}\right)\rd s\\
			&\lesssim \int_0^{t\wedge\sigma}\frac{-\del_s\Im[ z_s]}{\Im[z_{t\wedge\sigma}]}\left( \big| \Delta_0 (z_0) \big|+\frac{n^{ {\fd/2}}}{n\sqrt{\Im[z_{s}] \dist(z_s, I^*_s)}}+\frac{n^{\fd/2}}{\big( n\Im[z_{s}] \big)^2}\right)\rd s\\
			&\lesssim 
			\frac{\big| \Delta_0(z_0) \big|\Im[z_0]}{\Im[z_{t\wedge \sigma}]}+\frac{n^{ {\fd/2}}}{n\Im[z_{t\wedge \sigma}]}\int_0^{t\wedge\sigma}\frac{-\del_s\Im[ z_s]}{\sqrt{\Im[z_{s}] \dist(z_s, I^*_s)}}\rd s +\frac{n^{\fd/2}}{\big( n\Im[z_{t\wedge \sigma}] \big)^2},
	\end{split}\end{align}
	
	\noindent where in the last bound we again used the fact that $\Imaginary z_s$ is nonincreasing in $s$. To bound the first term on the right side of \eqref{e:term2} first suppose that $z_t$ is in the domain \eqref{e:edgedomain1}, say $z_t=E_1(t)-\kappa_t-\ri\eta_t$ (or $z_t = E_2 (t) + \kappa_t + \mathrm{i} \eta_t$) with $\kappa_t>0$. Then set $z_s=E_1(s)-\kappa_s-\ri\eta_s$ (or $z_s = E_2 (s) + \kappa_s + \mathrm{i} \eta_s$, respectively), for any $0\leq s\leq t$. Using the square root behavior \eqref{e:squareeq}, we find 
	\begin{align*}
		\eta_0=\eta_{t\wedge \sigma}-(t\wedge\sigma)\Im\left[\frac{f^*(z_{t\wedge\sigma})}{f^*(z_{t\wedge\sigma})+1}\right]\lesssim \eta_{t\wedge\sigma}+(t\wedge\sigma) \frac{\eta_{t\wedge\sigma}}{\sqrt{\kappa_{t\wedge\sigma}+\eta_{t\wedge\sigma}}}.
	\end{align*}
	
	\noindent Thus, by \eqref{e:boundarycm2} and \eqref{e:diff2c}, the first term can be bounded by
	\begin{align}\begin{split}\label{e:term22c}
			\frac{\big| \Delta_0(z_0) \big| \Im[z_0]}{\Im[z_{t\wedge \sigma}]}
			\lesssim \frac{\eta_0 n^{\fd/3 - 1}}{\eta_{t\wedge \sigma}(\kappa_0+\eta_0)}
			& \lesssim \frac{n^{\fd/3 - 1} \big(\eta_{t\wedge\sigma}+(t\wedge\sigma) \eta_{t\wedge\sigma}/\sqrt{\kappa_{t\wedge\sigma}+\eta_{t\wedge\sigma}} \big)}{\eta_{t\wedge \sigma} \big(\kappa_{t\wedge\sigma}+\eta_{t\wedge\sigma}+(t\wedge \sigma)\sqrt{\kappa_{t\wedge\sigma}+\eta_{t\wedge\sigma}} \big)}\\
			&= \frac{n^{\fd/3 - 1}}{\kappa_{t\wedge\sigma}+\eta_{t\wedge\sigma}}
			\lesssim \frac{n^{\fd/3 - 1}}{\sqrt{\Im[z_{t\wedge\sigma}]\dist(z_{t\wedge\sigma}, I^*_{t\wedge\sigma})}},
	\end{split}\end{align}
	where we remark that for the last term in the first line, there is an exact cancelation between the numerator and the denominator.
	We have the same estimate if $z_t$ is in the domain \eqref{e:edgedomain2}, by using \eqref{e:kskt} and the square root behavior \eqref{e:cubiceq}:
	\begin{align}\begin{split}\label{e:term22}
			\frac{\big| \Delta_0(z_0) \big| \Im[z_0]}{\Im[z_{t\wedge \sigma}]}
			\lesssim \frac{\eta_0 n^{\fd/3 - 1}}{\eta_{t\wedge \sigma}(\kappa_0+\eta_0)}
			& \lesssim \frac{n^{\fd/3 - 1} \big( \eta_{t\wedge\sigma}+(t\wedge\sigma) \eta_{t\wedge\sigma}/(t_c-t\wedge \sigma)^{1/4}\sqrt{\kappa_{t\wedge\sigma}+\eta_{t\wedge\sigma}} \big)}{\eta_{t\wedge \sigma} \big(\kappa_{t\wedge\sigma}+\eta_{t\wedge\sigma}+(t\wedge \sigma)\sqrt{\kappa_{t\wedge\sigma}+\eta_{t\wedge\sigma}}/(t_c-t\wedge \sigma)^{1/4} \big)}\\
			&= \frac{n^{\fd/3 - 1}}{\kappa_{t\wedge\sigma}+\eta_{t\wedge\sigma}}
			\lesssim \frac{n^{\fd/3 - 1}}{\sqrt{\Im[z_{t\wedge\sigma}]\dist(z_{t\wedge\sigma}, I^*_{t\wedge\sigma})}}. 
	\end{split}\end{align}
	
	We can bound the integral in \eqref{e:term2} using \eqref{e:integral1},
	\begin{align}\begin{split}\label{e:term3}
			&\int_0^{t\wedge\sigma}\frac{-\del_s\Im[ z_s]}{\sqrt{\Im[z_{s}] \dist(z_s, I^*_s)}}\rd s
			\lesssim
			\frac{\log n\sqrt{\Im[z_{t\wedge\sigma}]}}{\sqrt{\dist(z_{t\wedge\sigma}, I^*_{t\wedge\sigma})}}.
	\end{split}\end{align} 
	
	It follows by combining \eqref{e:midgronwall}, \eqref{e:term2}, \eqref{e:term22}, \eqref{e:term22c} and \eqref{e:term3}, and using the bound \eqref{e:boundarycm2} of $\Delta_0(z_0)$ that
	\begin{align}\label{e:Deltabb2}
		\big| \Delta_{t\wedge\sigma}(z_{t\wedge\sigma}) \big|
		\lesssim   \frac{n^{\fd}}{n\sqrt{\Im[z_{t\wedge \sigma}]\dist(z_{t\wedge\sigma}, I^*_{t\wedge\sigma})}}+\frac{n^{\fd/2}}{\big( n\Im[z_{t\wedge \sigma}] \big)^2}.
	\end{align}

	On the domain ${\mathscr D}_{t\wedge \sigma}^{\rm L}\cup {\mathscr D}_{t\wedge \sigma}^{\rm F}$, we again have that both $ m_{t\wedge \sigma}$ and $m^*_{t\wedge\sigma}$ are Lipschitz with Lipschitz constant at most $\OO(n)$. Similarly to \eqref{e:useLip}, we can approximate $w\in {\mathscr D}_{t\wedge \sigma}^{\rm L}\cup {\mathscr D}_{t\wedge \sigma}^{\rm F}$ by the image $z_{t\wedge \sigma}(u)$ of some lattice point $u\in {\mathscr L}$, and deduce from \eqref{e:Deltabb} and \eqref{e:Deltabb2} that, on the event $\cW$ as defined in \Cref{c:error}, we have
	\begin{align}\begin{split}\label{e:finals}
			& \big|m_{t\wedge\sigma}(w)-m^*_{t\wedge\sigma}(w) \big| <  \frac{M n^{ \fd }}{n\Im[z_{t\wedge \sigma}]},\\
			& \big|m_{t\wedge\sigma}(w)-m^*_{t\wedge\sigma}(w) \big| < \frac{Mn^{ \fd}}{n\sqrt{\Im[z_{t\wedge \sigma}]\dist(z_{t\wedge \sigma}, I_{t\wedge \sigma}) }}+\frac{M n^\fd}{\big( n\Im[z_{t\wedge \sigma}] \big)^2},
	\end{split}\end{align}
	uniformly for $w\in {\mathscr D}_{t\wedge \sigma}^{\rm L}$ and $w\in {\mathscr D}_{t\wedge \sigma}^{\rm F}$. Comparing  \eqref{e:useLip}, \eqref{e:finals} with the definition of $\sigma$ in \eqref{stoptime2}, we conclude that $\sigma=\ft$ on the event $\cW$. Since $\cW$ holds with overwhelming probability, this together with the fourth statement in \Cref{p:gap} finish the proof of 
	\Cref{p:rigidityBB}.
\end{proof}

\begin{rem}\label{r:trapzoid1}
	
	Here, we briefly comment on the modifications needed to address the first case in the second statement of \Cref{xhh}, namely, when the north boundary $\del_{\rm no}(\fD)$ of $\fD$ is packed; this corresponds to random lozenge tilings of (single-sided) trapezoid domains, or equivalently families of non-intersecting Bernoulli bridges with tightly packed ending data, i.e., $\bmx_\ft=(\fa(\ft), \fa(\ft)+1/n, \fa(\ft)+2/n,\cdots, \fb(\ft)-1/n)$. Then, the arctic curve will be tangent to the north boundary of $\fD$ at $(E,\ft)$, and $f^*_\ft(E)=-1$, see Figure \ref{f:trapezoid}. 
	For such non-intersecting Bernoulli bridges, the transition probability is explicitly (we no longer need to consider a weighted version that approximates it) given by
	\begin{align}\label{e:tdynamic}
		\bP(\bmx_{t+1/n}=\bmx+\bme/n|\bmx_t=\bmx)=\frac{1}{Z_t(\bmx)}\frac{V(\bmx+\bme/n)}{V(\bmx)}\prod_{i=1}^m \big( \fb(t)-1/n-x_i \big)^{e_i} \big(x_i-\fa(t) \big)^{1-e_i}.
	\end{align}
	Moreover, in this setting, the solution to the variational principle is also explicit. For any height function $\beta(x;\bmx)$ corresponding to the particle configuration $\bmx=(x_1, x_2,\cdots, x_m)\subseteq [\fa(t),\fb(t)]$, the complex slope satisfies the algebraic equation:
	\begin{align*}
		f_r(x;\beta,t)=-\frac{\big( x-r-\fb(t) \big)f_r(x;\beta,t)+ x-\fb(t)}{\big( x-r-\fa(t) \big)f_r(x;\beta,t)+x-\fa(t)}
		\prod_{i=1}^m \frac{(x-t-x_i)f_r(x;\beta,t)+x-x_i}{(x-r-x_i-1/n)f_r(x;\beta,t)+x-x_i-1/n};
	\end{align*}
	see \cite[Proposition 2.5]{ARS}. The complex slope $f_t(x;\beta,t)$ can be extended to $f(z;\beta,t)$ with $z\in \bC$,
	\begin{align*}
		f(z;\beta,t)=\frac{\fb(t)-z}{z-\fa(t)}\prod_{i=1}^m \frac{z-x_i}{z-x_i-1/n}=\frac{\fb(t)-z}{z-\fa(t)}e^{m(z;\beta,t)}.
	\end{align*}
	This is the case when $g(z;\beta,t)=\big(\fb(t)-z \big) / \big( z-\fa(t) \big)$ in \eqref{e:gszmut}.
	
	For $\ft-t\gtrsim 1$, $f^*_t(x)$ bounded away from $-1$, the dynamics \eqref{e:tdynamic} can be analyzed in the same way as the weighted Bernoulli bridge \eqref{e:defLtnew}. For $\ft-t\ll1$, the behavior of the limiting complex slope $f_t^*(z)$ can be analyzed as in \Cref{p:ftzbehave}, and we have $E_2(t)-E, E-E_1(t)\asymp \sqrt{\ft-t}$. To analyze the dynamics \eqref{e:tdynamic}, we can restrict the spectral domains $\mathscr D_t^{\rm L}$ and $\mathscr D_t^{\rm F}$ from \eqref{e:bulkdomain} and \eqref{e:edgedomain1} to a radius $\fC\sqrt{\ft-t}$ neighborhood of $E$, for some large constant $\fC>0$. See Figure \ref{f:trapezoid}.
	Inside this spectral domain, $\big| f_t^*(z) + 1 \big|$ is nonzero but its size will depend on $t$: $|f^*_t(z)+1|\asymp \sqrt{\ft-t}$. Moreover, when $z$ is close to the spectral edge, say $z=E_1(t)-\kappa+\ri \eta$, $f_t^*(z), m_t^*(z)$ have square root behavior with constants depending on $t$,
	\begin{align*}
		-\sin(\Im[m_t^*(z)])
		\asymp -\sin(\Im[\log f_t^*(z)])=-\Im[f_t^*(z)] \asymp
		\left\{
		\begin{array}{cc}
			(\ft-t)^{1/4}\frac{\eta}{\sqrt{|\kappa|+\eta}},  & \kappa\geq 0,\\
			(\ft-t)^{1/4}\sqrt{|\kappa|+\eta}, & \kappa \leq 0.
		\end{array}
		\right.
	\end{align*}
	Similarly to \eqref{e:defI_t+1}, we define the distance function $\tau(E(t),t)=
	n^{-2/3+6\fd}(\ft-t)^{-1/6}$, and the enlarged intervals  
	\begin{align}\label{e:enla}
		I_t^+ = \Big[ E_1  (t) - \tau \big( E_1 (t), t \big), E_2 (t) + \tau \big(E_2 (t), t \big) \Big].
	\end{align}
	With the notations defined above, \Cref{p:gap} holds with \eqref{e:gap1} replaced by $\dist(z_s, I_s^+)\gtrsim \Im z_s+\sqrt{\dist(z_s, I_s^*)}(t-s)(\ft-s)^{1/4}$; And Propositions \ref{p:gfbound}, \ref{p:integral}, \ref{p:improve}, \ref{p:dynamic} and \ref{c:error} also hold, except now we have $\big| \Im[f_t^*(z)/(f_t^*(z)+1)] \big|\lesssim 1/\sqrt{\ft-t}$, and $g(z;\beta,t)= \big( \fb(t)-z \big) / \big( z-\fa(t) \big)$ (independent of $\beta$).  Using them as input, we get \Cref{p:rigidityBB} with $I_t^+$ given by the enlarged intervals \eqref{e:enla}.

\end{rem}

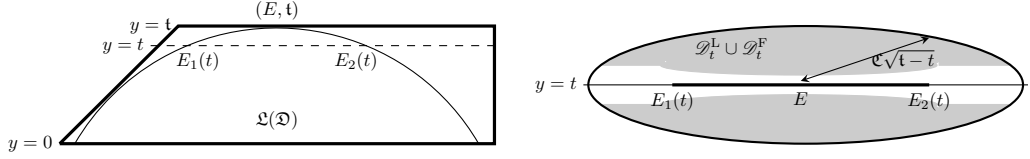
\begin{figure}
	
	\begin{center}		
		
		\begin{tikzpicture}[
			>=stealth,
			auto,
			style={
				scale = .52
			}
			]
			
			\draw[black, very thick] (-5.5, 0) node[left, scale = .7]{$y = 0$}-- (5.5, 0) -- (5.5, 3) -- (-2.5, 3) node[left, scale = .7]{$y = \ft$} -- (-5.5, 0);
			\draw[black, dashed](-3.2, 2.5) node[left,scale=0.7]{$y=t$}--(5.5, 2.5);
			\draw[black] (5.1,0) arc (30:150:5.9);
			
			\draw[](0, 3) node[above, scale=0.7]{$(E,\ft)$};
			\draw[](-2, 2.5) node[below, scale=0.7]{$E_1(t)$};
			\draw[](2, 2.5) node[below, scale=0.7]{$E_2(t)$};
			\draw[](0, 1) node[below, scale=0.7]{$\fL(\fD)$};

			\draw[white, fill=gray!40!white] (13.375,1.5) ellipse (5.5 and 1.5);
			
			\draw[white, fill=white] (17.8,1.5) ellipse (1.4 and 0.8625);
			\draw[white, fill=white] (8.8,1.5) ellipse (1.4 and 0.8625);
			
			\draw[white, fill=white] (9.8,1.0) rectangle ++(6.9,1);

			\draw[gray!40!white, fill=gray!40!white] (8,2) rectangle ++(1.8,0.5);
			\draw[gray!40!white, fill=gray!40!white] (8,0.5) rectangle ++(1.8,0.5);
			\draw[gray!40!white, fill=gray!40!white] (13.2,2) ellipse (3.5 and 0.25);
			\draw[gray!40!white, fill=gray!40!white] (13.2,1) ellipse (3.5 and 0.25);

			\draw[gray!40!white, fill=gray!40!white] (16.75,2) rectangle ++(1.8,0.5);
			\draw[gray!40!white, fill=gray!40!white] (16.75,0.5) rectangle ++(1.8,0.5);

			\draw[] (11.5, 2)  node[above, scale=0.7]{$\mathscr D^{\rm L}_t\cup \mathscr D^{\rm F}_t$};

			\fill[white,even odd rule] (13.375,1.5) ellipse (5.5 and 1.5) (13.375,1.5) ellipse (6 and 2.5);
			\draw[black, thick] (13.375,1.5) ellipse (5.5 and 1.5);
			
			\draw  (13.25, 1.6) edge[<->] node[right,scale=0.7] {$\fC\sqrt{\ft-t}$} (16.5, 2.7);

			\draw[black] (7.75, 1.5) node[left, scale = .7]{$y = t$}-- (19, 1.5);
			\draw[black, very thick] (10, 1.5)node[below, scale=0.7]{$E_1(t)$} -- (16.5, 1.5)node[below, scale=0.7]{$E_2(t)$};
			\draw[black, very thick] (13.25, 1.5)node[below, scale=0.7]{$E$};

		\end{tikzpicture}
		
	\end{center}
	
	\caption{\label{f:trapezoid} Shown to the left  is a liquid region corresponds to lozenge tiling of a (single-sided) trapezoid domain. The arctic curve is tangent to the north boundary at $(E,\ft)$. Shown to the right is the spectral domain restricted to a radius $\fC\sqrt{\ft-t}$ neighborhood of $E$.}
	
\end{figure}

\section{Dynamical Loop Equation}
\label{s:loopeq}

In this section we introduce a general family of transition probabilities $\bP$, which includes those of weighted non-intersecting Bernoulli bridge as a special case.  In Section \ref{ProofEquation1},
We first construct another transition probability $\bQ$, which consists of the leading order terms in $\bP$, and state \Cref{p:improve2}, which is an analogous version of \Cref{p:improve} for $\bP$. In Section \ref{EquationsEstimate} we introduce the dynamical loop equation, which will be used to study the transition probability $\bP$. In Section \ref{s:LRE} we state a weaker version of \Cref{p:improve2}, which is optimal in the liquid region but suboptimal in the frozen region. And its proof is given in Section \ref{ProofEquation}.

\subsection{General Transition Probability}

\label{ProofEquation1}

The transition probability of weighted non-intersecting Bernoulli bridge as in \eqref{e:defLtnew} is in the following form.
Fix a particle configuration $\bmx=(x_1, x_2, \cdots, x_m)\in \bZ_n^m$ with $\fa\leq x_1<x_m\leq \fb-1/n$. For any $\bme=\{0,1\}^m$, the transition probability  is given by
\begin{align}\label{e:defL}
	\bP(\bme)=\frac{1}{Z(\bmx)}\frac{V(\bmx+\bme/n)}{V(\bmx)}\prod_{i=1}^m\phi^+(x_i)^{e_i}\phi^-(x_i)^{1-e_i} \cdot \exp \Bigg( \sum_{1\leq i,j\leq m}\frac{e_ie_j}{n^2}\kappa(x_i,x_j)+\OO(1/n) \Bigg).
\end{align}
We denote the empirical measure of $\bmx$ and its Stieltjes transform by
\begin{align*}
	\rho(x;\bmx)=\sum_{i=1}^m \bm1 \big( x\in [x_i, x_i+1/n] \big),\quad m(z)=\int_{-\infty}^{\infty} \frac{\rho(x;\bmx)\rd x}{z-x}.
\end{align*}

We will study more general transition probabilities of the form \eqref{e:defL} satisfying the following assumption.
\begin{assumption}
	
	\label{a:mz}
	
	Suppose there exists a neighborhood $\Lambda$ of $[\fa,\fb]$ such that the transition probability \eqref{e:defL} satisfies the following three assumptions.
	\begin{enumerate}
		\item The functions $\phi^+ (z)$ and $\phi^- (z)$ have the decompositions
		\begin{align}\label{e:defphi+}
			\phi^+(z)=(\fb-1/n-z)\tilde g(z)e^{\frac{1}{n}\psi(z)},\quad
			\phi^-(z)=z-\fa,
		\end{align}
		where $\tilde g(z)$ and $\psi(z)$ are analytic and uniformly bounded on $\Lambda$. Moreover, $\tilde g(z)$ satisfies  $\tilde g(\bar z)=\overline{\tilde g(z)}$; it does not have zeros or poles in $\Lambda$; and $\tilde g(x)>0$ for $x\in \Lambda\cap \bR$.
		Both $\psi(z)$ and $\kappa(z,w)$ are analytic for $z,w\in \Lambda$, and they map reals to reals.
		
		\item  The complex slope defined by
		\begin{align}\label{e:deffznew}
			f(z)=e^{m(z)}\tilde g(z)\frac{\fb-z}{z-\fa}, 
		\end{align}
		satisfies $\Im f(z)<0$ for $\Imaginary z> 0$. We further assume that $f(x)+1\in (0,1)$ for  $x>\fb$, and $f(x)+1\in (-\infty,0)$ for  $x< \fa$, and that the following quantity is uniformly bounded away from $0$,
		\begin{align}\label{e:imratio}
			\displaystyle\frac{1}{|\Imaginary z|} \cdot \bigg| \Imaginary \frac{f(z)}{f(z)+1} \bigg| \gtrsim 1,\quad \text{for $z\in \Lambda$}.
		\end{align}
		
		\item There exist small constants $\fc,\fr>0$,  and a domain $\mathscr D\subset \big\{z\in \Lambda: \dist(z,[\fa', \fb'])\geq n^{\fc-1} \big\}$, where  $[\fa',\fb']:=[\fa-\fr,\fb+\fr]$, containing the annulus $\big\{ z\in\bC: \fr\leq \dist(z,[\fa',\fb'])\leq 2\fr \big\}$ (see Figure \ref{f:domainD}) such that the following are uniformly bounded
		\begin{align}\label{e:fbound}
			\big| f(z)+1 \big|\gtrsim 1, \quad \left|\frac{f(z)}{f(z)+1}\right|\lesssim 1,
			\qquad \text{for $z\in \mathscr{D}$}.
		\end{align}
		Moreover, $\Im[f(x+0\ri)/(f(x+0\ri)+1)]$ defines a negative measure on $\bR$. We denote its support by $I$, which satisfies
		\begin{align}\label{e:defI}
			I=\supp \Bigg(\Im \bigg[ \displaystyle\frac{f(x+0\ri)}{(f(x+\ri)+1)} \bigg] \Bigg)\subseteq [\fa',\fb'].
		\end{align}
		
	\end{enumerate}
	
\end{assumption}

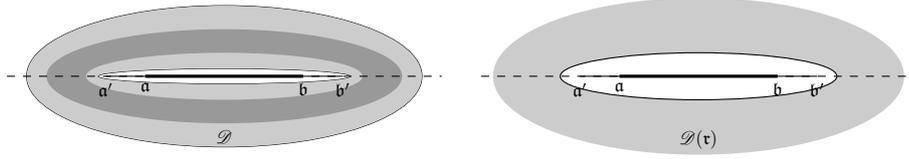
\begin{figure}
	
	\begin{center}		
		
		\begin{tikzpicture}[
			>=stealth,
			auto,
			style={
				scale = .52
			}
			]

			\draw[black] (0,0) ellipse (5 and 1.8 );	
			\draw[black] (0,0) ellipse (3.2 and 0.2 );		
			\fill[gray!40!white,even odd rule] (0,0) ellipse (5 and 1.8) (0,0) ellipse (3.2 and 0.2);

			\fill[gray!80!white,even odd rule] (0,0) ellipse (3.5 and 0.6) (0,0) ellipse (4.5 and 1.2);
			
			\draw[black,dashed] (-5.5,0) -- (5.5,0);
			\draw[black] (-3,0) -- (3,0);
			\draw[] (-3,0) node[below,scale=.7]{$\fa'$};
			\draw[] (3,0) node[below,scale=.7]{$\fb'$};
			\draw[black, very  thick] (-2,0) -- (2,0);
			\draw[] (-2,0) node[below,scale=.7]{$\fa$};
			\draw[] (2,0) node[below,scale=.7]{$\fb$};
			\draw[] (0,-1.2) node[below,scale=.7]{$\mathscr D$};
			
			\fill[gray!40!white,even odd rule] (12,0) ellipse (3.5 and 0.6) (12,0) ellipse (5.2 and 2);
			\draw[black] (12,0) ellipse (3.5 and 0.6 );
			
			\draw[black,dashed] (6.5,0) -- (17.5,0);
			\draw[black] (9,0) -- (15,0);
			\draw[] (9,0) node[below,scale=.7]{$\fa'$};
			\draw[] (15,0) node[below,scale=.7]{$\fb'$};
			\draw[black, very  thick] (10,0) -- (14,0);
			\draw[] (10,0) node[below,scale=.7]{$\fa$};
			\draw[] (14,0) node[below,scale=.7]{$\fb$};
			\draw[] (12,-1.2) node[below,scale=.7]{$\mathscr D(\fr)$};
			
		\end{tikzpicture}
		
	\end{center}
	
	\caption{\label{f:domainD} Shown to the left the interval $[\fa', \fb']=[\fa-\fr, \fb+\fr]$ and the domain $\mathscr D$, which contains the annulus $\{z\in\bC: \fr\leq \dist(z,[\fa',\fb'])\leq 2\fr\}$. Shown to the right is the domain $\mathscr D(\fr)=\{z\in\bC:  \dist(z,[\fa',\fb'])\geq \fr\}$.	}
\end{figure} 

\noindent Throughout the remainder of this paper, we adopt \Cref{a:mz}. Let us define the domains for any $r>0$,
\begin{align}\label{e:defDr}
	\mathscr D(r):= \Big\{z\in\bC:\dist \big(z,[\fa',\fb'] \big)\geq r \Big\}.
\end{align} 
We will mainly use $\mathscr D(\fr)$ and $\mathscr D(2\fr)$. Thanks to the third statement of \Cref{a:mz}, we have $\mathscr D\cap \mathscr D(2\fr)\neq\emptyset$. By \Cref{p:formula} and equation \eqref{e:defvarphi}, the transition probability of weighted non-intersecting Bernoulli bridge as in \eqref{e:defLtnew} satisfies \Cref{a:mz} with $\beta=\beta(x;\bmx)$
\begin{align*}
	\phi^\pm(z)=\phi^\pm(z;\beta,t), \quad \kappa(z,w)=\kappa(z,w;\beta,t),\quad \tilde g(z)=\tilde g(z;\beta,t),\quad f(z)=f(z;\beta,t),
\end{align*}
for $t\leq \sigma$ (recall from \eqref{stoptime2}).

We can take the domains
\begin{align*}
	\Lambda=\mathscr U_t, \quad \mathscr D={\mathscr D}_t(\fr)\cup{\mathscr D}_t^{\rm L}\cup{\mathscr D}_t^{\rm F},
\end{align*}
from Section \ref{s:De}.
The signs of $f(x) + 1$ for $x < \mathfrak{a}$ and $x > \mathfrak{b}$, as well as the lower bound \eqref{e:imratio}, follow from the second statement of  \Cref{p:fdecompose}. 
Also, thanks to the third statement of \Cref{p:fdecompose}, we have the decomposition \eqref{e:defphi+}, and that $\tilde g(z;\beta,t)$ is real analytic and positive on $\Lambda \cap \bR$.
The estimate \eqref{e:fbound} follows from \eqref{e:ft+1b2} and \eqref{e:replacep2}. Thanks to the third statement of \Cref{p:gfbound}, $\Im[f(x+0\ri)/(f(x+0\ri)+1)]$ defines a negative measure on $\bR$, and \eqref{e:defI} holds. 

We need more notation; as in \eqref{e:defvarphi} and \eqref{e:defBtt}, we let
\begin{align*}
	\varphi^+(z)=(\fb-z)\tilde g(z),\quad
	\varphi^-(z)=z-\fa,
\end{align*}
and
\begin{align}\label{e:B}
	\cB(z)&=\cG(z)\varphi^+(z)+\varphi^-(z)= \big( f(z)+1 \big) (z-\fa), \quad \quad \cG(z)=\prod_{i = 1}^m \left(1+\frac{1/n}{z-x_i-1/n}\right)=e^{m(z)}.
\end{align}
The same discussion as after \eqref{e:dmtta} gives that $\log \cB(z)$ (where the branch cut for the logarithm is $\mathbb{R}_{< 0}$) is well defined on $\Lambda\setminus[\fa,\fb]$, and so
\begin{align}\label{e:aB0}
	\frac{1}{2\pi \ri}\oint_{\omega}\del_z \log \cB(z)\rd z=0,
\end{align}
where $\omega\subseteq \Lambda$ is any contour enclosing $[\fa, \fb]$. 
In what follows, we take the branch of the logarithm to be so that $\Imaginary \log u \in [-\pi, \pi)$.  We have the following decomposition of $\log \cB(z)$ as
\begin{align}\label{e:defbz}
	\log \cB(z)=b^+(z)-b(z),\quad b(z)=\frac{1}{2\pi \ri}\oint_{\omega} \frac{\log \cB(w)\rd w}{w-z},\quad b^+(z)=\frac{1}{2\pi \ri}\oint_{\omega+} \frac{\log \cB(w)\rd w}{w-z},
\end{align}
where $\omega\subseteq \Lambda$ is any contour enclosing $[\fa, \fb]$ but not $z$, and $\omega+\subseteq \Lambda$ is any contour enclosing $[\fa, \fb]$ and $z$. 

From the construction, $b(z)$ is analytic in a neighborhood of $\infty$, and we have that $\lim_{z\rightarrow \infty}b(z)=0$; we will abbreviate this by simply writing $b(\infty)=0$. For $b^+(z)$ it is analytic and uniformly bounded in a neighborhood of $[\fa,\fb]$.

\begin{rem}
	The functions $f(z)$, $\tilde g(z)$, and $\cB(z), b^+(z)$ are only defined on the domain $\Lambda$; we will typically use them for $z\in \mathscr D$. The functions $m(z)$, $\cG(z)$, and $b(z)$ are defined for any $z\in \mathbb{C} \setminus [\mathfrak{a}, \mathfrak{b}]$; we will typically use them for $z \in \mathscr{D} \cup \mathscr{D} (\mathfrak{r})$. 
\end{rem}
In the following lemma, we collect some estimates  which will be used repeatedly throughout the remainder of this section.
\begin{lem}\label{l:fbbound}
	Adopting \Cref{a:mz}, for any $z\in \mathscr D$ we have
	\begin{align*}
		\big| \del_z \log f(z) \big|, \left|\del_z \frac{f(z)}{f(z)+1}\right|,  \big| \del_z \log \cB(z) \big|, \big| \del_z \log \cG(z) \big|, \big| \del_z b(z) \big|\lesssim  \frac{1}{\dist(z,[\fa',\fb'])}.
	\end{align*}
	The last two estimates also hold for $z\in \mathscr D(\fr)$.
\end{lem}
\begin{proof}
	Since 
	\begin{flalign*}
		\big| \partial_z m (z) \big| \lesssim \displaystyle\frac{1}{\dist \big( z, [\fa, \fb] \big)}; \qquad \big| \partial_z \log \tilde{g} (z) \big| \lesssim 1; \qquad \bigg| \displaystyle\frac{1}{z - \mathfrak{b}} \bigg| + \bigg| \displaystyle\frac{1}{z - \mathfrak{a}} \bigg| \lesssim \displaystyle\frac{1}{\dist \big( z, [\fa, \fb] \big)},
	\end{flalign*}
	
	\noindent our assumption \eqref{e:deffznew} gives
	\begin{align*}
		\big| \del_{z} \log f(z) \big|= \left|\del_z  m(z)+\del_z\log \tilde g(z)+\frac{1}{z-\fb}-\frac{1}{z-\fa}\right|\lesssim \frac{1}{\dist(z,[\fa,\fb])}\leq \frac{1}{\dist(z,[\fa',\fb'])}.
	\end{align*}
	Since $\cB(z)= \big( f(z)+1 \big)(z-\fa)$ and $\cG(z)=e^{m(z)}$, the same argument (together with \eqref{e:fbound}) gives the bound for $|\del_z \log \cB(z)|$ and $|\del_z \log \cG(z)|$. The bound for $\del_z \big( f(z)/(f(z)+1) \big)$ follows along the same reasoning as the ones used to deduce \eqref{e:derf/f+1} and \eqref{e:imbb}, and is thus omitted. By taking the contour $\omega+$ in \eqref{e:defbz} bounded away from $[\fa',\fb']$, we have $\big| \del_z b(z) \big|= \big| \del_z \log \cB(z) \big|+\OO(1)\lesssim 1/\dist \big( z,[\fa',\fb'] \big)$.
\end{proof}

To study the transition probability \eqref{e:defL}, we introduce a new transition probability
\begin{align}\label{e:defL2}
	\bQ(\bme)=a_\bme=\frac{1}{\tilde Z(\bmx)}\frac{V(\bmx+\bme/n)}{V(\bmx)}\prod_{i=1}^m\phi^+(x_i)^{e_i}\phi^-(x_i)^{1-e_i}.
\end{align}
Then we can write $\bP$ as a change of measure from $\bQ$, namely,
\begin{align}\label{e:changem}
	\bE_\bP[\cdot]=\frac{\bE_\bQ \Bigg[\cdot \exp \Bigg(\displaystyle\sum_{1\leq i,j\leq m} \displaystyle\frac{e_ie_j}{n^2} \kappa(x_i,x_j)+\OO(1/n) \bigg) \Bigg]}{\bE_\bQ \Bigg[ \exp \bigg(\displaystyle\sum_{1\leq i,j \le m\leq m} \displaystyle\frac{e_ie_j}{n^2} \kappa(x_i,x_j)+\OO(1/n) \bigg) \Bigg]}.
\end{align}
We denote the empirical measure of $\bmx+\bme$ by
\begin{align}\label{e:defrho}
	\rho(x;\bmx+\bme/n)=\sum_{i=1}^m \bm1(x\in [x_i+e_i/n, x_i+e_i/n+1/n]).
\end{align}
The main goal of this and next sections is to understand the difference of the empirical measures $\rho(x;\bmx+\bme/n)$ and $\rho(\bmx)$  under the transition probability \eqref{e:defL}, and prove \Cref{p:improve}. 
We will first prove the following proposition for the transition probability $\bQ$ as defined in \eqref{e:defL2}. Then \Cref{p:improve} will be a consequence of it.
\begin{prop}\label{p:improve2}
	Adopting \Cref{a:mz}, for any $z\in \mathscr D$, we have
	\begin{align}\begin{split}\label{e:improve1c}
			\bE_\bQ\left[\sum_{i = 1}^m \frac{1}{z-x_i-e_i/n}-\frac{1}{z-x_i}\right]
			&=\frac{1}{2\pi \ri}\oint_{\omega}\frac{\log \cB(w)\rd w}{(w-z)^2}+\OO\left(\frac{|\Im[f(z)/(f(z)+1)]|}{n\Im[z] \dist(z,I)}\right),
	\end{split}\end{align}
	where the contour $\omega\subseteq \Lambda$ encloses  $[\fa,\fb]$, but not $z$. Moreover, for any integer $p\geq 1$,
	\begin{align}\begin{split}\label{e:improve2c}
			\phantom{{}={}}\bE_\bQ & \left|\sum_{i = 1}^m \frac{1}{z-x_i-e_i/n}-\frac{1}{z-x_i}-\bE_\bQ\left[\sum_{i = 1}^m \frac{1}{z-x_i-e_i/n}-\frac{1}{z-x_i}\right]\right|^{2p}\\
			&= \OO\left(\left(\frac{|\Im[f(z)/(f(z)+1)]|}{n\Im[z] \dist(z,I)^{2}}\right)^{p}+\frac{|\Im[f(z)/(f(z)+1)]|}{n^{2p-1}\Im[z] \dist(z,I)^{4p-2}}\right).
	\end{split}\end{align}
\end{prop}

\subsection{Dynamical Loop Equation}
\label{EquationsEstimate}

Our analysis of the transition probability \eqref{e:defL2} is based on the following dynamical loop equation from \cite{NRWLT}. 
\begin{lem}[{\cite[Lemma 7.2]{NRWLT}}]\label{l:loopeq}
	
	Assume that the functions $\phi^\pm$ from the transition probability $\bQ$ are analytic in a neighborhood of $[\fa,\fb]$. Further assume that the particle configuration $\bmx=(x_1, x_2, \cdots, x_m)\in \bZ_n^m$ satisfies $\fa\leq x_1<x_m\leq \fb-1/n$. Then the following function is analytic in a neighborhood of $[\fa,\fb]$,
	\begin{align}\label{e:sum1}
		\bE_\bQ\left[ \prod_{i=1}^m\frac{z-x_i+(1-e_i)/n}{z-x_i} \phi^+(z)+\prod_{i=1}^m\frac{z-x_i-e_i/n}{z-x_i}\phi^-(z)\right].
	\end{align}
\end{lem}

\begin{rem}\label{r:polesloc}
	If $\phi^\pm$ is meromorphic in a neighborhood $\Lambda$ of $[\fa,\fb]$ and does not have poles on the interval $[\fa,\fb]$, then \Cref{l:loopeq} implies that all possible poles of \eqref{e:sum1} are given by those of $\phi^\pm$.
\end{rem}

Fix any large positive integer $R\geq 1$. We compute observables of the form
\begin{align}\label{e:term0}
	\sum_\bme a_\bme   \prod_{j=1}^r \left(\prod_{i = 1}^m \frac{v_j-x_i-e_i/n}{v_j-x_i}\right)^{s_j} ,\quad \sum_\bme a_\bme   \prod_{j=1}^r \left(\prod_{i = 1}^m \frac{v_j-x_i-e_i/n}{v_j-x_i}\right)^{s_j}\prod_{i=1}^m \frac{z-x_i-e_i/n}{z-x_i},
\end{align}
for either
\begin{align}\label{e:term1}
	0\leq r\leq 2R-1, \quad \bmv=(v_1,v_2,\cdots,v_r), \quad \bms=(s_1,s_2,\cdots, s_r),\quad v_j\in \mathscr D\cup \mathscr D(\fr),\quad s_j=\pm 1,
\end{align}
or 
\begin{align}\label{e:term2a}
	r=2R, \quad \bmv=(v_1,v_2,\cdots,v_r), \quad \bms=(s_1,s_2,\cdots, s_r),\quad v_{2j-1}=\overline v_{2j}\in \mathscr D\cup \mathscr D(\fr),\quad s_j=-1.
\end{align}
For any $r\leq 2R$, to compute \eqref{e:sum1} with vectors $\bmv, \bms$ as described in \eqref{e:term1} or \eqref{e:term2a}, we define the deformed weights 
\begin{align}\label{e:defpsi}
	a_\bme^{(\bmv,\bms)}=a_\bme \prod_{i=1}^m\psi^{(\bmv,\bms)}(x_i)^{1-e_i},\quad
	\psi^{(\bmv,\bms)}(z)= \displaystyle\prod_{j = 1}^r \bigg( \displaystyle\frac{v_j - z - 1/n}{v_j - z} \bigg)^{-s_j}.
\end{align}

\noindent To understand the transition probability \eqref{e:defL2} on mesoscopic scale, we need to take $v_j$ close to $[\fa',\fb']$, where the function $\psi^{(\bmv,\bms)}(z)$ is singular. In fact, $\psi^{(\bmv,\bms)}(z)$ is meromorphic with possible poles at $v_j$ if $s_j=-1$, or $v_j-1/n$ if $s_j=1$ for $1\leq j\leq r$. 
For these deformed weights as in \eqref{e:defpsi}, we can absorb $\psi^{(\bmv,\bms)}(z)$ into $\phi^-(z)$, and then \Cref{l:loopeq} and Remark \ref{r:polesloc} imply that the following quantity is analytic in $\Lambda \setminus \{v_j-(s_j+1)/2n\}_{1\leq j\leq r}$,
\begin{align}\begin{split}\label{e:ABC2}
		\cC^{(\bmv,\bms)}(z)&=\sum_{\bme}a^{(\bmv,\bms)}_\bme \left(\prod_{i=1}^m\frac{z-x_i+(1-e_i)/n}{z-x_i} \phi^+(z)+\prod_{i=1}^m\frac{z-x_i-e_i/n}{z-x_i}\phi^-(z)\psi^{(\bmv,\bms)}(z) \right).
\end{split}\end{align}

We will use \eqref{e:ABC2} to analyze the following quantity, which encodes the information of $\rho(x;\bmx+\bme/n)-\rho(x;\bmx)$,
\begin{align}\label{e:defA}
	\cA^{(\bmv,\bms)}(z)&=\sum_{\bme}a^{(\bmv,\bms)}_\bme
	\prod_{i = 1}^m \frac{z-x_i-e_i/n}{z-x_i}=\bE_\bQ\left[\prod_{i=1}^m\psi^{(\bmv,\bms)}(x_i)^{1-e_i}\prod_{i = 1}^m \frac{z-x_i-e_i/n}{z-x_i}\right].
\end{align}

\noindent We may alternatively write $\cA^{(\bmv,\bms)}(z)$ as 
\begin{align}
	\label{ga}
	\cA^{(\bmv,\bms)}(z)=\prod_{j=1}^r\cG(v_j)^{s_j}\bE_\bQ\left[\prod_{j=1}^r\left(\prod_{i=1}^m\frac{v_j-x_i-e_i/n}{v_j-x_i}\right)^{s_j}\prod_{i = 1}^m \frac{z-x_i-e_i/n}{z-x_i}\right],
\end{align}
which are quantities in \eqref{e:term0} up to some factors of $\cG$. By sending $z\rightarrow \infty$, we have
\begin{align*}
	\cA^{(\bmv,\bms)}(\infty)=\prod_{j=1}^r\cG(v_j)^{s_j}\bE_\bQ\left[\prod_{j=1}^r\left(\prod_{i=1}^m\frac{v_j-x_i-e_i/n}{v_j-x_i}\right)^{s_j}\right].
\end{align*}

By taking $r=1$, $v_1=v$, and $s_1=-1$, $\cA^{(\bmv,\bms)}(z)$ becomes
\begin{align*}
	\cA^{(\bmv,\bms)}(z)=\cG(v)^{-1}\bE_\bQ\left[\prod_{i=1}^m\frac{v-x_i}{v-x_i-e_i/n}\prod_{i=1}^m\frac{z-x_i-e_i/n}{z-x_i}\right],
\end{align*}
and
the derivative $\del_z \log\cA^{(\bmv,\bms)}$ after specializing at $z=v$ gives the Stieltjes transform of $\rho(x;\bmx+\bme/n)-\rho(x;\bmx)$,
\begin{align}\label{e:dzlogA}
	\del_z \log\cA^{(\bmv,\bms)}|_{z=v}=\bE_\bQ\left[\sum_{i = 1}^m \frac{1}{v-x_i-e_i/n}-\sum_{i = 1}^m \frac{1}{v-x_i}\right].
\end{align}
More generally, taking $r=2r'+1$, $s_1=s_2=\cdots=s_{r'}=1$ and $s_{r'+1}=s_{r'+2}=\cdots s_{r}=-1$ the derivative 
\begin{align*}
	\del_z \del_{v_1}\del_{v_2}\cdots \del_{v_{r'}} \log\cA^{(\bmv,\bms)}
	= \del_z \del_{v_1}\del_{v_2}\cdots \del_{v_{r'}} \log\bE_\bQ\left[\prod_{j=1}^r\left(\prod_{i=1}^m\frac{v_j-x_i-e_i/n}{v_j-x_i}\right)^{s_j}\prod_{i = 1}^m \frac{z-x_i-e_i/n}{z-x_i}\right]
\end{align*} 
after specializing at $z=v_r=v, v_j=v_{j+r'}\in \{v,\overline v\}$ for $1\leq j\leq r'$, gives the $(r'+1)$-th joint cumulants of 
\begin{align*}
	\sum_{i = 1}^m \frac{1}{v-x_i-e_i/n}-\sum_{i = 1}^m \frac{1}{v-x_i}, \quad \sum_{i = 1}^m \frac{1}{\overline v-x_i-e_i/n}-\sum_{i = 1}^m \frac{1}{\overline v-x_i}.
\end{align*}

Using the definition \eqref{e:defphi+}, we rewrite the first part of $\cC^{(\bmv,\bms)}(z)$ in \eqref{e:ABC2} as
\begin{align*}\begin{split}
		\phantom{{}={}}\sum_{\bme} & a^{(\bmv,\bms)}_\bme \prod_{i=1}^m\frac{z-x_i+(1-e_i)/n}{z-x_i}\phi^+(z)\\
		& =\cG(z+1/n)\phi^+(z)
		\sum_{\bme}a^{(\bmv,\bms)}_\bme \prod_{i=1}^m\frac{z+1/n-x_i-e_i/n}{z+1/n-x_i}\\
		&=\cG(z+1/n)\phi^+(z)
		\cA^{(\bmv,\bms)}(z+1/n)=\cG(z+1/n)\varphi^+(z+1/n)
		\cA^{(\bmv,\bms)}(z+1/n) \frac{\tilde g(z)}{\tilde g(z+1/n)} e^{\frac{1}{n}\psi(z)}.
\end{split}\end{align*}
The second part of $\cC^{(\bmv,\bms)}(z)$ is given by
\begin{align*}
	\sum_{\bme}a^{(\bmv,\bms)}_\bme \prod_{i=1}^m\frac{z-x_i-e_i/n}{z-x_i}\phi^-(z)\psi^{(\bmv,\bms)}(z) 
	=\cA^{(\bmv,\bms)} (z) \varphi^-(z)\psi^{(\bmv,\bms)}(z). 
\end{align*}
Then in this way we rewrite $\cC^{(\bmv,\bms)}$ as 
\begin{align}\label{e:ABCprecise}
	\cC^{(\bmv,\bms)}(z)=\cA^{(\bmv,\bms)}(z)\cB(z)\left(1+\cE^{(\bmv,\bms)}(z)\right)
\end{align}
where the complex slope $f(z)$ is as in \eqref{e:deffznew}, and 
\begin{align}\begin{split}\label{e:defcE}
		\phantom{{}={}}\cE^{(\bmv,\bms)}(z)& =
		\cE^{(\bmv,\bms)}_0(z)
		+\frac{\psi^{(\bmv,\bms)}(z)-1}{f(z)+1},\\
		\phantom{{}={}}\cE^{(\bmv,\bms)}_0(z)&=\left(\frac{f(z+1/n)}{f(z+1/n)+1} \frac{\cA^{(\bmv,\bms)}(z+1/n)\cB(z+1/n)}{\cA^{(\bmv,\bms)}(z)\cB(z)}\frac{\tilde g(z)}{\tilde g(z+1/n)}e^{\frac{1}{n}\psi(z)}-\frac{f(z)}{f(z)+1}\right).
\end{split}\end{align}

\subsection{Liquid Region Estimates}\label{s:LRE}

The following proposition gives a weak bound on $\big| \del_z(\log \cA^{(\bmv,\bms)}(z)-b(z)) \big|$. It will eventually coincide with the required bound \eqref{e:improve1c} in the liquid region. However, the estimate is suboptimal if $\dist(z,I)\geq \Im z$, and is therefore not sufficient in the frozen region.  In Section \ref{s:improveEE}, we will bootstrap it to obtain optimal estimates in the frozen region.

\begin{prop}\label{p:firstod}
	
	Adopt \Cref{a:mz}, and fix an integer $R>1$. Then for $z\in \mathscr D\cup \mathscr D(\fr)$, and $(\bmv, \bms)=\big( (v_1,v_2,\cdots, v_r), (s_1,s_2,\cdots, s_r) \big)$ as in \eqref{e:term1} or \eqref{e:term2a}, we have
	\begin{align}\begin{split}\label{e:firstod}
			\del_z \big( \log \cA^{(\bmv,\bms)}(z)-b(z) \big)=\frac{1}{2\pi \ri}\oint_{\omega}\frac{\log \big( 1+\cE^{(\bmv,\bms)}(w) \big) \rd w}{(w-z)^2},\quad
			b(z)=\frac{1}{2\pi \ri}\oint_{\omega} \frac{\log \cB(w)\rd w}{w-z}
	\end{split}\end{align}
	where the contour $\omega\subseteq\mathscr D$ encloses  $[\fa',\fb']$ but not $\{z\}\cup \big\{v_j-(s_j+1)/2n \big\}_{1\leq j\leq r}$. Let $1/n\ll \eta \leq \fr$. Then uniformly for $w\in \mathscr D$ satisfying $\dist \big(w, [\fa',\fb']\cup\{v_1,v_2,\cdots, v_r\} \big)\geq \eta $, the error term is bounded by 
	\begin{align}\label{e:errorbb}
		\big| \cE^{(\bmv,\bms)}(w) \big| \lesssim \frac{1}{n\eta}.
	\end{align}
\end{prop}

We will deduce \Cref{p:firstod} as a consequence of the following two propositions. \Cref{c:boost0} gives estimates of $\cA^{(\bmv,\bms)}(z)$ when $z, v_1,v_2,\cdots, v_r$ are bounded away from $[\fa',\fb']$.  \Cref{c:boost} states that if we have some (weak) a priori estimate on $\cA^{(\bmv, \bms)}(z)$ with $\dist \big( \{z,v_1,v_2,\cdots, v_r\}, [\fa',\fb'] \big)\geq \eta$, then we can obtain an improved estimate on $\cA^{(\bmv, \bms)}(z)$ in the larger domain $\dist \big(\{z,v_1,v_2,\cdots, v_r\}, [\fa',\fb'] \big)\geq (1-1/2K)\eta$, for some large $K\geq 0$ (which will be chosen later). Its proof, based on the discrete loop equations, will appear in \Cref{ProofEquation} below.

\begin{prop}\label{c:boost0}
	Adopt \Cref{a:mz}, and fix an integer $R> 1$.  Let $(\bmv,\bms)$ be as in \eqref{e:term1} or \eqref{e:term2a}, with  $\dist \big(\{z,v_1,v_2,\cdots, v_r\}, [\fa',\fb'] \big) \geq 2\fr$. The following estimates hold
		\begin{flalign}\label{e:cAbb0}
			\cA^{(\bmv, \bms)}(z)e^{-b(z)-\sum_{j=1}^r s_j (b(v_j)+\log \cG(v_j))} & =1+\OO\left(\frac{1}{n}\right),
		\end{flalign} 
		
		\noindent and 
		\begin{flalign}
			\label{e:cAbb1}
			\cA^{(\bmv, \bms)}(\infty )e^{-\sum_{j=1}^r s_j (b(v_j)+\log\cG(v_j))} & =1+\OO\left(\frac{1}{n}\right).
		\end{flalign}
	
\end{prop}

\begin{prop}\label{c:boost}
	
	Adopt \Cref{a:mz}, and fix an integer $R> 1$. Further assume that there exists a parameter $n^{\fc-1}\leq \eta\leq 2\fr$, and a constant $\mathfrak{C} > 1$ (independent of $\eta$) such that, for any   $z, v_1,v_2,\cdots, v_r\in \mathscr D\cup \mathscr D(\fr)$ satisfying $\dist \big( \{z,v_1,v_2,\cdots, v_r\}, [\fa',\fb'] \big) \geq \eta$, the following holds. If $(\bmv, \bms)$ is as in \eqref{e:term1} then
	\begin{align}\label{e:ap1}
		\fC^{-1} \leq \Big| \cA^{(\bmv, \bms)}(\infty)e^{-\sum_{j=1}^r s_j (b(v_j)+\log \cG(v_j))} \Big| \leq \mathfrak{C}; \qquad  \mathfrak{C}^{-1} \leq \Big|\cA^{(\bmv, \bms)}(z)e^{-b(z)-\sum_{j=1}^r s_j (b(v_j)+\log \cG(v_j))} \Big|\leq \fC,
	\end{align}
	and if $(\bmv, \bms)$ is as in \eqref{e:term2a} then
	\begin{align}\label{e:ap2}
		\fC^{-1} \leq \Big|\cA^{(\bmv, \bms)}(\infty)e^{-\sum_{j=1}^r s_j (b(v_j)+\log \cG(v_j))}\Big| \leq \fC.
	\end{align}
	Then there is a constant $K = K (R, \mathfrak{C}) > 1$ such that the following holds.
	\begin{enumerate}
		\item
		Let $(\bmv,\bms)$ be as in \eqref{e:term1} or \eqref{e:term2a}, and let $w\in \mathscr D$ satisfy $\dist \big(\{w,v_1,v_2,\cdots, v_r\}, [\fa',\fb'] \big) \geq (1-1/K)\eta$. If $\dist \big( w, \{v_1,v_2,\cdots, v_r\} \big) \geq \eta/2K$, then the error term $\cE^{(\bmv,\bms)}(w)$ from \eqref{e:defcE} satisfies
		\begin{align}\label{e:cebb}
			\big| \cE^{(\bmv,\bms)}(w) \big|\lesssim \frac{1}{n \eta}.
		\end{align}
		\item Let $(\bmv,\bms)$ be as in \eqref{e:term1} or \eqref{e:term2a}, and let $z\in \mathscr D\cup \mathscr D(\fr)$ satisfy $\dist \big( \{z,v_1,v_2,\cdots, v_r\}, [\fa',\fb'] \big)\geq (1-1/2K)\eta$. Then,
		\begin{align}\begin{split}\label{e:firstoda}
				\del_z \log \cA^{(\bmv,\bms)}(z)
				&=\frac{1}{2\pi \ri}\oint_{\omega} \log \cB(w)\frac{\rd w}{(w-z)^2}+\frac{1}{2\pi \ri}\oint_{\omega}\log \big(1+\cE^{(\bmv,\bms)}(w) \big)\frac{\rd w}{(w-z)^2}\\
				&=\frac{1}{2\pi \ri}\oint_{\omega} \log \cB(w)\frac{\rd w}{(w-z)^2}+\OO\left(\frac{1}{n\eta^2}\right).
		\end{split}\end{align}
		Here, the contour $\omega$ encloses $[\fa',\fb']$ but not $\{z,v_1,v_2,\cdots, v_r\}$; is contained inside the region $\big\{ w\in \mathscr D: \dist(w, [\fa',\fb'])\geq (1-1/K)\eta \big\}$; and satisfies $\dist \big(\omega, \{z,v_1, v_2,\cdots, v_r\} \big) \geq \eta / 2K$. We additionally have the improved estimates  
		\begin{flalign}\label{e:toap1}
			\cA^{(\bmv, \bms)}(z)e^{-b(z)-\sum_{j=1}^r s_j (b(v_j)+\log \cG(v_j))} & =1+\OO\left(\frac{1}{n\eta}\right),
		\end{flalign} 
		
		\noindent and 
		\begin{flalign}
			\label{e:toap2}
			\cA^{(\bmv, \bms)}(\infty )e^{-\sum_{j=1}^r s_j (b(v_j)+\log\cG(v_j))} & =1+\OO\left(\frac{1}{n\eta}\right).
		\end{flalign}
	\end{enumerate}
	
\end{prop}

Given above propositions, we can quickly establish \Cref{p:firstod}.

\begin{proof}[Proof of \Cref{p:firstod}]
	
	It follows from \Cref{c:boost0} and the definitions \eqref{e:B} and \eqref{e:defbz} that 
	\begin{flalign*}
		\big| \mathcal{A}^{(\boldsymbol{v}, \boldsymbol{s})} (z) \big|, \big| \mathcal{A}^{(\boldsymbol{v}, \boldsymbol{s})} (\infty) \big| \asymp 1; \qquad \displaystyle\max_{1\leq j\leq r} \Big\{ \big| b(v_j) \big|, \big| \log \mathcal{G} (v_j) \big| \Big\} = \mathcal{O} \left(1\right),
	\end{flalign*}
	
	\noindent whenever $\dist \big( \{ z, v_1, v_2, \ldots , v_r \}, [\fa', \fb'] \big) \geq 2\fr$.
	In particular, if we take $\eta_0=2\fr$, then there exists a constant $\mathfrak{C} > 2$ such that \eqref{e:ap1} holds if $\dist \big( \{ z, v_1, v_2, \ldots , v_r \}, [\fa',\fb'] \big) \geq \eta_0$. Letting $K = K (R, \mathfrak{C}) > 1$ denote the constant guaranteed by \Cref{c:boost}, define the sequence $\eta_0, \eta_1, \ldots $ by 
	\begin{align}
		\eta_{i+1}= \bigg( 1- \displaystyle\frac{1}{2K} \bigg) \eta_i,\quad i\geq 0.
	\end{align}
	
	We show by induction on $i$ that the proposition holds if $\eta = \eta_i$ whenever $\eta_i \geq n^{\fc-1}$ (meaning $i = \mathcal{O} (\log n)$). To that end, recall from the above that \Cref{c:boost} applies for $\eta = \eta_0$. Its second statement yields \eqref{e:firstod}, and its first statement yields \eqref{e:errorbb}, where the $\eta$ there is equal to $(1 - 1/2K) \eta_0 = \eta_1$ here. This verifies the proposition if $\eta = \eta_1$.
	
	Moreover, the second statement of \Cref{c:boost} implies, for $(\bmv, \bms)$ as in \eqref{e:term1} or \eqref{e:term2a}, that we have the improved bounds
	\begin{align}\begin{split}\label{e:toap2b}
			&\cA^{(\bmv, \bms)}(\infty )e^{-\sum_{j=1}^r s_j (b(v_j)+\cG(v_j))}=1+\OO\left(\frac{1}{n\eta}\right),\\
			&\cA^{(\bmv, \bms)}(z)e^{-b(z)-\sum_{j=1}^r s_j (b(v_j)+\log \cG(v_j)))}=1+\OO\left(\frac{1}{n\eta}\right),
	\end{split}\end{align}
	
	\noindent whenever $\dist \big( \{z,v_1,v_2,\cdots, v_r\}, [\fa',\fb'] \big)\geq (1-1/2K)\eta=\eta_1$. In particular, this gives \eqref{e:ap1} and \eqref{e:ap2} if $\eta=\eta_1$ (with the $\mathfrak{C}$ there equal to the $\mathfrak{C} > 2$ here). Through the above argument, we see that if \eqref{e:ap1} and \eqref{e:ap2} hold for $\eta=\eta_i$, then \Cref{c:boost} implies that \eqref{e:firstod}, \eqref{e:errorbb}, \eqref{e:ap1}, and \eqref{e:ap2} all hold for $\eta=\eta_{i+1}$. After repeating this $\OO(\log n)$ times, we conclude that \eqref{e:firstod} and \eqref{e:errorbb} hold for any $z, v_1, v_2,\cdots, v_r \in \mathscr{D} \cup \mathscr{D} (\mathfrak{r})$. This finishes the proof of \Cref{p:firstod}.
\end{proof}

\begin{proof}[Proof of \Cref{c:boost0}]
For $w\in\mathscr D$, we can rewrite the decomposition \eqref{e:ABCprecise} as
\begin{align}\label{e:Cvs}
\cC^{(\bmv,\bms)}(w)=\cA^{(\bmv,\bms)}(w)\cB(w)+\cA^{(\bmv,\bms)}(w)\cB(w)\cE^{(\bmv,\bms)}(w),
\end{align}
where 
\begin{align}\begin{split}\label{e:error}
&\cA^{(\bmv,\bms)}(w)\cB(w)\cE^{(\bmv,\bms)}(w)
=\cA^{(\bmv,\bms)} (w) \varphi^-(w)(\psi^{(\bmv,\bms)}(w)-1)\\
&+\cG(w+1/n)\varphi^+(w+1/n)
		\cA^{(\bmv,\bms)}(w+1/n) \frac{\tilde g(w)}{\tilde g(w+1/n)} e^{\frac{1}{n}\psi(w)}
		-\cG(w)\varphi^+(w)
		\cA^{(\bmv,\bms)}(w).
\end{split}\end{align}
For $w\in\mathscr D$ bounded away from $[\fa',\fb']$ and $\{v_1,v_2,\cdots, v_r\}$, namely $\dist(w, [\fa',\fb']\cup \{v_1,v_2,\cdots, v_r\})\geq \fr/2$, we have $|\cA^{(\bmv,\bms)}(w)|, |\cB(w)|\lesssim 1$, $|\psi^{(\bmv,\bms)}(w)-1|\lesssim 1/n$ (recall from \eqref{e:defpsi}), and  $|\del_w \log \cG(w)|\lesssim 1$ (from \Cref{l:fbbound}). Moreover,  from the definition \eqref{e:defA}, the difference $|\cA^{(\bmv,\bms)}(w+1/n) -\cA^{(\bmv,\bms)}(w)|$ is bounded as
\begin{align}
 \sum_{\bme}|a_\bme^{(\bmv,\bms)}|\left|\prod_{i = 1}^m \frac{w-x_i-e_i/n}{w-x_i}\right|\left|\prod_{i = 1}^m \frac{w+1/n-x_i-e_i/n}{w+1/n-x_i} \frac{w-x_i}{w-x_i-e_i/n}-1\right|
 \lesssim 1/n.
\end{align}
where the implicit constant depends on $\fr$.
The discussions above imply that the error term \eqref{e:error} is bounded by $\OO(1/n)$.

We recall the decomposition of $\log \cB(w)$ from \eqref{e:defbz}, and  multiply $\exp(-b^+(w))$ on both sides of \eqref{e:Cvs}, 
\begin{align}\label{e:ABCorg}
\cC^{(\bmv,\bms)}(w)e^{-b^+(w)}=\cA^{(\bmv,\bms)}(w)e^{-b(w)}+\OO(1/n),
\end{align}
where we used $|b^+(w)|=\OO(1)$ for $\dist(w, [\fa',\fb'])\leq 2\fr$. Since the lefthand side of \eqref{e:ABCorg} is analytic for $\dist(w, [\fa',\fb'])\leq 2\fr$, we can use a contour integral to get rid of it. For any $z$ with $\dist(z, [\fa',\fb'])\geq 2\fr$, we have
\begin{align}\begin{split}\label{e:getridC}
0&=\frac{1}{2\pi\ri}\oint_{\omega}\frac{\cC^{(\bmv,\bms)}(w)e^{-b^+(w)}\rd w}{w-z}=\frac{1}{2\pi\ri}\oint_{\omega}\frac{(\cA^{(\bmv,\bms)}(w)e^{-b(w)}+\OO(1/n))\rd w}{w-z}\\
&=\frac{1}{2\pi\ri}\oint_{\omega}\frac{\cA^{(\bmv,\bms)}(w)e^{-b(w)}\rd w}{w-z}+\OO(1/n)
=\frac{1}{2\pi\ri}\oint_{\omega_+}\frac{\cA^{(\bmv,\bms)}(w)e^{-b(w)}\rd w}{w-z}-\cA^{(\bmv,\bms)}(z)e^{-b(z)}+\OO(1/n)
\end{split}\end{align}
where the contour $\omega$ encloses $[\fa',\fb']$ but not $\{z,v_1,\cdots, v_r\}$, and the contour $\omega_+$ encloses $[\fa',\fb']$ and $z$. For the last equality, we used that $\cA^{(\bmv,\bms)}(w)$ and $b(w)$ are analytic outside the contour $\omega$. 

As $w$ approaches $\infty$, we have $ \cA^{(\bmv,\bms)}(w)= \cA^{(\bmv,\bms)}(\infty)+\OO(1/w)$ and $b(w)=\OO(1/w)$. We can further deform the contour $\omega+$ to infinity, and the first term on the righthand side of \eqref{e:getridC} can be computed
\begin{align}\label{e:getridC2}
\frac{1}{2\pi\ri}\oint_{\omega_+}\frac{\cA^{(\bmv,\bms)}(w)e^{-b(w)}\rd w}{w-z}=\frac{1}{2\pi\ri}\oint_{\omega_+}\frac{(\cA^{(\bmv,\bms)}(\infty)+\OO(1/w)) \rd w}{w-z}=\cA^{(\bmv,\bms)}(\infty).
\end{align}
It follows by plugging \eqref{e:getridC2} into \eqref{e:getridC}, we get for any  $\dist (z, [\fa',\fb'] ) \geq 2\fr$,
\begin{align}\label{e:yao00}
 \cA^{(\bmv,\bms)}(z)e^{-b(z)}= \cA^{(\bmv,\bms)}(\infty)+\OO(1/n).
 \end{align}

Thanks to the relation \eqref{e:yao00}, \eqref{e:cAbb1} implies \eqref{e:cAbb0}. Thus we only need to prove \eqref{e:cAbb1}. We prove \eqref{e:cAbb1} by induction on $r$. If $r=0$, then \eqref{e:cAbb1} holds trivially, since $\cA(\infty)=1$. We assume \eqref{e:cAbb1} holds for $r$ and prove it for $r+1$.
	If  $s_j=-1$ for all $1\leq j\leq r+1$, let $\hat\bmv=(v_1, v_2,\cdots, v_r)$ and $\hat\bms=(s_1, s_2,\cdots, s_r)$, then \eqref{ga} and \eqref{e:yao00} together give
	\begin{align*}\begin{split}
			\log \cA^{(\hat\bmv, \hat\bms)} (\infty ) -\log \cG(v_{r+1})- \log \cA^{(\bmv, \bms)}(\infty)  = \log \cA^{(\bmv, \bms)} (v_{r+1} )- \log \cA^{(\bmv, \bms)}(\infty) =b(v_{r+1})+\OO(1/n).
	\end{split}\end{align*}
	Thus it follows by rearranging that we have
	\begin{align*}\begin{split}
			\log \cA^{(\bmv, \bms)}(\infty) &=\log \cA^{(\hat\bmv, \hat\bms)} (\infty ) -\log \cG(v_{r+1})-b(v_{r+1})+\OO(1/n)\\
			&=\log \cA^{(\hat\bmv, \hat\bms)} (\infty ) +s_{r+1} \big( \log \cG(v_{r+1})+b(v_{r+1}) \big)+\OO(1/n)\\
			&=\sum_{j=1}^{r+1} s_j \big( \log \cG(v_{j})+b(v_{j}) \big)+\OO(1/n),
	\end{split}\end{align*}
	where we used the induction hypothesis in the last step. 
	
	Next suppose some $s_j=1$; we may assume (after permuting the $s_j$ if necessary) that $s_{r+1}=1$. Letting $\hat\bmv=(v_1, v_2,\cdots, v_r)$ and $\hat\bms=(s_1, s_2,\cdots, s_r)$, we have by similar reasoning that 
	\begin{align*}\begin{split}
			\log \cA^{(\hat \bmv, \hat \bms)} (v_{r+1} )- \log \cA^{(\hat \bmv, \hat\bms)}(\infty)
			&= \log \cA^{( \bmv,  \bms)} (\infty )-\log \cG(v_{r+1})- \log \cA^{(\hat \bmv, \hat\bms)} (\infty) =b(v_{r+1})+\OO(1/n).
	\end{split}\end{align*}
	
	\noindent By rearranging and the induction hypothesis, it again follows that  
	\begin{align*}
		\log \cA^{(\bmv, \bms)}(\infty)=\sum_{j=1}^{r+1} s_{j} \big( \log \cG(v_{j})+b(v_{j}) \big)+\OO(1/n).
	\end{align*}
	
	\noindent Hence, \eqref{e:cAbb1} holds for any $(\bmv, \bms)$ in \eqref{e:term1} or \eqref{e:term2a}. This finishes the proof of \Cref{c:boost0}.

\end{proof}

\subsection{Proof of \Cref{c:boost}}

\label{ProofEquation}
Before proving \Cref{c:boost0} and \Cref{c:boost}, we first show the following proposition, which gives estimates on the derivatives of $\log \cA^{(\bmv,\bms)}(z)$.
\begin{prop}\label{p:continuity}
	Adopt the notation and assumptions of \Cref{c:boost}. There is a universal constant $K = K (R, \mathfrak{C}) > 1$ (independent of $\eta$) such that the following holds. Let $z,v_1,v_2,\cdots, v_r\in \mathscr D\cup \mathscr D(\fr)$ be complex numbers such that  $\dist \big( \{z,v_1,v_2,\cdots, v_r\}, [\fa',\fb'] \big) \geq (1-1/K)\eta$. If $(\bmv, \bms)$ is as in \eqref{e:term1} then
	\begin{align}\label{e:derbb1}
		\big| \del_z \log\cA^{(\bmv, \bms)}(z) \big|, \big|\del_{v_j} \log\cA^{(\bmv, \bms)}(z) \big|,
		\big| \del_{v_j} \log\cA^{(\bmv, \bms)}(\infty) \big| \lesssim 1/\eta,\quad 1\leq j\leq r,
	\end{align}
	and if $(\bmv, \bms)$ is as in \eqref{e:term2a} then
	\begin{align}\label{e:derbb2}
		\big| \del_z \log\cA^{(\bmv, \bms)}(z) \big|, \big| \del_{v_j} \log\cA^{(\bmv, \bms)}(\infty)\big| \lesssim 1/\eta,\quad 1\leq j\leq r,
	\end{align}
	
	\noindent where the implicit constants depend on $R$ and $\mathfrak{C}$ (but not $K$). In particular, for $\dist \big( \{z,v_1,v_2,\cdots, v_r\}, [\fa',\fb'] \big)\geq (1-1/K)\eta$, we have the following. If $(\bmv, \bms)$ is as in \eqref{e:term1}, then
	\begin{align}\begin{split}
			\label{a1} 
			& (2 \fC)^{-1} \leq \Big| \cA^{(\bmv, \bms)}(\infty)e^{-\sum_{j=1}^r s_j (b(v_j)+\log \cG(v_j))} \Big| \leq 2 \mathfrak{C}; \\
			&  (2\mathfrak{C})^{-1} \leq \Big|\cA^{(\bmv, \bms)}(z)e^{-b(z)-\sum_{j=1}^r s_j (b(v_j)+\log \cG(v_j))} \Big|\leq 2 \fC,
	\end{split}\end{align}
	and if $(\bmv, \bms)$ is as in \eqref{e:term2a} then
	\begin{align}
		\label{a2}
		(2\fC)^{-1} \leq \Big|\cA^{(\bmv, \bms)}(\infty)e^{-\sum_{j=1}^r s_j (b(v_j)+\log \cG(v_j))}\Big| \leq 2 \fC.
	\end{align}
	
\end{prop}

\begin{proof}

	For $(\bmv, \bms)$ as in  \eqref{e:term2a}, the weights $a_{\bme}^{(\bmv,\bms)}$ in \eqref{e:defpsi} are real and positive. We can rewrite $\cA^{(\bmv,\bms)}(z)$ in terms of its residue decomposition
	\begin{align}\label{e:Aexp}
		\prod_{i = 1}^m \frac{z-x_i-e_i/n}{z-x_i}=1+\sum_{i = 1}^m \frac{\Res_{z=x_i}}{z-x_i},
	\end{align}
	
	\noindent where its residue at $x_i$ is given by
	\begin{align*}
		\Res_{z=x_i}
		:=-\frac{e_i}{n}\prod_{j: j\neq i}\frac{x_i-x_j-e_j/n}{x_i-x_j}
		\leq 0,
	\end{align*}
	is nonpositive. 
	We assume $\Im z \geq 0$, as the case $\Im z \leq 0$ is entirely analogous.
	Our assumption that $\dist \big( z,[\fa',\fb'] \big)\geq (1-1/K)\eta$ implies that  either $|\Imaginary z| \geq\min\{(\fr n/m), 1\}\eta/10$ or $|\Imaginary z| \leq\min\{(\fr n/m), 1\}\eta/10$, $\Re z \not\in [\fa',\fb']$ and $\dist \big( \Re z, [\fa',\fb'] \big) \geq\eta / 2$. We begin by considering the first case. Using \eqref{e:Aexp}, we have the  estimate
	\begin{align}\label{e:derzb}
		\left|\del_z\prod_{i = 1}^m \frac{z-x_i-e_i/n}{z-x_i}\right|\lesssim 
		-\sum_{i=1}^m \frac{\Res_{z=x_i}}{|z-x_i|^2}
		=\frac{1}{\Imaginary z}\Im\left[\prod_{i = 1}^m \frac{z-x_i-e_i/n}{z-x_i}\right],
	\end{align}
	
	\noindent and it follows that
	\begin{align}\label{e:dervlnA}
		\left|\del_z \log \cA^{(\bmv,\bms)}(z)\right|=\left|\frac{\del_z \cA^{(\bmv,\bms)}(z) }{\cA^{(\bmv,\bms)}(z) }\right|
		\lesssim \frac{\Im \big[\cA^{(\bmv,\bms)}(z) \big]}{\big| \cA^{(\bmv,\bms)}(z) \big| \Imaginary z}\lesssim \frac{1}{\eta}.
	\end{align}

	For the second case we recall that $[\fa',\fb']=[\fa-\fr, \fb+\fr]$, and $x_i\in [\fa,\fb]$. Thus we have $|z-x_i|\geq |\Re[z]-x_i|\geq \fr+\eta/2\gtrsim 1+\eta$, and
	\begin{align}\label{e:derzb2}
		\left|\del_z\prod_{i = 1}^m\frac{z-x_i-e_i/n}{z-x_i}\right|\lesssim-\sum_{i=1}^m \frac{\Res_{z=x_i}}{|z-x_i|^2}  \lesssim -\frac{1}{(1+\eta)^2}\sum_{i=1}^m \Res_{z=x_i}			 \lesssim \frac{1}{(1+\eta)^2}.
	\end{align}
	where in the last inequality, we used that from \eqref{e:Aexp}, by comparing the coefficient of $1/z$, the total residual is given by $0\leq \sum_{i=1}^m e_i/n\leq m/n\lesssim 1$.
	In the following we show that 
	\begin{align}\label{e:lowRe}
		\Re\left[\prod_{i = 1}^m \frac{z-x_i-e_i/n}{z-x_i}\right]
		\gtrsim 1.
	\end{align}
	Recalling that $a_\bme^{(\bmv,\bms)}$ are real and nonnegative, then \eqref{e:derzb2} and \eqref{e:lowRe} together gives
	\begin{align*}
		\left|\del_z \log \cA^{(\bmv,\bms)}(z)\right|=\left|\frac{\del_z \cA^{(\bmv,\bms)}(z) }{\cA^{(\bmv,\bms)}(z) }\right|
		\lesssim \frac{\sum_{\bme}a_\bme^{(\bmv,\bms)}\left|\del_z\prod_{i = 1}^m\frac{z-x_i-e_i/n}{z-x_i}\right|}{\sum_{\bme}a_\bme^{(\bmv,\bms)} \Re\left[\prod_{i = 1}^m \frac{z-x_i-e_i/n}{z-x_i}\right]}
		\lesssim \frac{1}{(1+\eta)^2}\lesssim \frac{1}{\eta}. 
	\end{align*}
	To prove \eqref{e:lowRe}, we can rewrite  it as,
	\begin{align}\label{e:norm}
		\left|\prod_{i = 1}^m \frac{z-x_i-e_i/n}{z-x_i}\right|
		\geq \prod_{i = 1}^m\left(1-\frac{1}{n|z-x_i|}\right)\geq \prod_{i = 1}^m\left(1-\frac{1}{n\fr}\right)\gtrsim 1.
	\end{align}
	and its argument is bounded by 
	\begin{align}\begin{split}\label{e:arg}
			0&\leq \sum_{i=1}^m \arg  \frac{z-x_i-e_i/n}{z-x_i}
			=\sum_{i=1}^m \Im\left[\log \frac{z-x_i-e_i/n}{z-x_i}\right]\\
			&=\sum_{i=1}^m \Im\left[-\frac{e_i}{n(z-x_i)}+\OO\left(\sum_{i=1}^n \frac{1}{n^2|z-x_i|^2}\right)\right]\leq \sum_{i=1}^m \frac{\Im[z]}{n|z-x_i|^2}+\OO\left(\frac{1}{n}\right)\\
			&\leq \frac{m}{n}\frac{\Im[z]}{\fr \dist(\Re[z], [\fa', \fb'])}+\OO\left(\frac{1}{n}\right)
			\leq 
			\frac{\min\{(\fr n/m), 1\}\eta/10}{(\fr n/m)(\eta/2)}+\OO\left(\frac{1}{n}\right)\leq 1/2,
	\end{split}\end{align}
	where we used $|z-x_i|\geq \fr$ and $|z-x_i|\geq \dist(\Re[z], [\fa',\fb'])$.
	The claim \eqref{e:lowRe} follows from combining \eqref{e:norm} and \eqref{e:arg}. Thus in both cases we have $\left|\del_z \log \cA^{(\bmv,\bms)}(z)\right|\lesssim 1/\eta$.
		
	For $\del_{v_j}\log \cA^{(\bmv,\bms)}(\infty)$ and \eqref{e:derbb1}, we recall the expression of $\cA^{(\bmv,\bms)}(z)$ from \eqref{ga}. Then $\del_{v_j}\log \cA^{(\bmv,\bms)}(\infty)$ in \eqref{e:derbb2}, or $\del_z \log \cA^{(\bmv,\bms)}(z),\del_{v_j} \log \cA^{(\bmv,\bms)}(z), \del_{v_j}\log \cA^{(\bmv,\bms)}(\infty)$  as in \eqref{e:derbb1} are all of the following form
	\begin{align}\label{e:dervv}
		\del_{v_k}\log \bE_\bQ\left[\prod_{j=1}^r\left(\prod_{i=1}^m\frac{v_j-x_i-e_i/n}{ v_j-x_i}\right)^{s_j}
		\right],\quad 0\leq r\leq 2R, \quad  s_j=\pm 1,\quad  v_j\in\mathscr D\cup \mathscr D(\fr),
	\end{align}
	and $\dist \big(  v_j, [\fa',\fb'] \big) \geq (1-1/K)\eta$,  with possibly an extra $\del_{v_k}\log \cG( v_k)$ term. Thanks to \Cref{l:fbbound}, the derivative $|\del_{v_k}\log \cG( v_k)|\lesssim 1/\eta$. In the following we show that \eqref{e:dervv} is also bounded by $1/\eta$.

	For complex numbers $\dist \big( \{v_1,v_2,\cdots, v_r\}, [\fa',\fb'] \big)\geq (1-1/K)\eta$, we take $ v_j' \in \mathscr{D} \cup \mathscr{D} (\mathfrak{r})$ such that  $\dist \big( v'_j, [\fa',\fb'] \big)\geq \eta$, with $ |v_j'-v_j|\leq \eta/K$ for $1\leq j\leq r$. Then,  we have for any $1\leq j\leq r$,
	\begin{flalign}\label{e:changez}
		\frac{\prod_{i = 1}^m \frac{v_j-x_i-e_i/n}{v_j-x_i}}{ \prod_{i = 1}^m \frac{v_j'-x_i-e_i/n}{v_j'-x_i}}
		=\prod_{i = 1}^m e^{\OO\left(\frac{|v_j'-v_j|}{n|v_j'-x_i|^2}\right)} = \exp \Bigg( \mathcal{O} \bigg( \displaystyle\frac{\eta}{nK} \bigg) \displaystyle\sum_{i = 1}^m \displaystyle\frac{1}{|v_j' - x_i|^2} \Bigg) = e^{\mathcal{O}(m / nK)} = e^{\OO(1/K)}.
	\end{flalign}
	Moreover, our assumptions \eqref{e:ap1} and \eqref{e:ap2} gives that for any $0\leq \hat r\leq 2R$ the following holds
	\begin{align}\label{e:v'v}
		\fC^{-1}\leq \left|\bE_\bQ\left[\prod_{j=1}^{\hat r}\left(\prod_{i=1}^m\frac{\hat v_j-x_i-e_i/n}{ \hat v_j-x_i}\right)^{\hat s_j}
		\right] e^{-\sum_{j=1}^{\hat r} \hat s_j b(\hat v_j)}\right|\leq \fC,\quad \hat s_j=\pm 1,\quad  \hat v_j\in\mathscr D\cup \mathscr D(\fr),
	\end{align}
	where $\dist(\hat v_j, [\fa',\fb'])\geq \eta$.
	
	We may rewrite  \eqref{e:dervv} as
	\begin{align}\label{e:dAv0}
		\eqref{e:dervv}=\frac{\del_{v_k}\bE_\bQ\left[\prod_{j=1}^r\left(\prod_{i=1}^m\frac{v_j-x_i-e_i/n}{ v_j-x_i}\right)^{s_j}
			\right] e^{-\sum_{j=1}^r s_j b(v'_j)}}{\bE_\bQ\left[\prod_{j=1}^r\left(\prod_{i=1}^m\frac{v_j-x_i-e_i/n}{ v_j-x_i}\right)^{s_j}
			\right]  e^{-\sum_{j=1}^r s_j b(v'_j)}}.
	\end{align}

	\noindent Let us verify that the denominator in \eqref{e:dAv0} is bounded above and below by constants (that is, it is of order $1$). To that end, observe that 
	\begin{align}\begin{split}\label{e:avs}
			&\phantom{{}={}}\bE_\bQ\left[\prod_{j=1}^r\left(\prod_{i=1}^m\frac{v_j-x_i-e_i/n}{ v_j-x_i}\right)^{s_j}
			\right]  e^{-\sum_{j=1}^r s_j b(v'_j)}\\
			& \qquad \qquad =\sum_\bme a_\bme e^{\OO(1/K)}\prod_{j=1}^r\left(\prod_{i=1}^m\frac{v'_j-x_i-e_i/n}{ v'_j-x_i}\right)^{s_j} e^{-\sum_{j=1}^r s_j b(v'_j)},
	\end{split}\end{align}
	
	\noindent where we used \eqref{e:changez} to replace $ v_1,v_2,\cdots, v_r$ by $v_1', v_2', \cdots, v_r'$, which gives an extra factor $e^{\OO(1/K)}$. The leading order term in \eqref{e:avs} will be given by \eqref{e:v'v} with $\hat v_j=v'_j$ and $\hat s_j=s_j$, and it is bounded below by $\mathfrak{C}^{-1}$ and above by $\mathfrak{C}$. 
	
	We show the remainder is of order $\OO(1/K)$. Observe that this remainder is bounded by a constant multiple of
	\begin{align}\label{e:tobbb}
		\frac{1}{K}\sum_\bme a_\bme \prod_{j=1}^r\left|\prod_{i=1}^m\frac{v'_j-x_i-e_i/n}{ v'_j-x_i}\right|^{s_j} e^{-\sum_{j=1}^r s_j\Re b(v'_j)}.
	\end{align}
	
	\noindent Letting $r'=\ceil{r/2}$, the Cauchy--Schwarz inequality gives 
	\begin{align}\begin{split}\label{e:tobbb1}
			\eqref{e:tobbb}
			&\leq \frac{1}{K}\sum_\bme a_\bme \prod_{j=1}^{r'}\left(\prod_{i=1}^m\frac{v'_j-x_i-e_i/n}{ v'_j-x_i}\right)^{s_j} \left(\prod_{i=1}^m\frac{\bar v'_j-x_i-e_i/n}{ \bar v'_j-x_i}\right)^{s_j} e^{-\sum_{j=1}^r (s_j b(v'_j)+s_j b(\bar v'_j))}\\
			& \qquad +\frac{1}{K}\sum_\bme a_\bme \prod_{j=r'+1}^{r}\left(\prod_{i=1}^m\frac{v'_j-x_i-e_i/n}{ v'_j-x_i}\right)^{s_j} \left(\prod_{i=1}^m\frac{\bar v'_j-x_i-e_i/n}{ \bar v'_j-x_i}\right)^{s_j} e^{-\sum_{j=r'+1}^r (s_j b(v'_j)+s_j b(\bar v'_j))}.
	\end{split}\end{align} 
	
	\noindent The first term in \eqref{e:tobbb1} is in the form \eqref{e:v'v} with $\hat r=2\ceil{r/2}\leq 2R$, 
	$\hat \bmv=(v_1', \bar v'_1, v_2', \bar v_2',\cdots,v_{r'}', \bar v_{r'}')$ and $\hat \bms=(s_1,s_1,\cdots, s_{r'}, s_{r'})$. It is bounded by $\fC/K$. Similarly, the second term in \eqref{e:tobbb1} is also of form \eqref{e:v'v}, and is bounded by $\fC/K$.
	By taking $K$ large enough, \eqref{e:avs} is then bounded below by $\mathfrak{C}^{-1} - \mathcal{O} (\fC/K) \geq (2 \mathfrak{C})^{-1}$ and above by $\mathfrak{C} + \mathcal{O} (\fC/K) \leq 2 \mathfrak{C}$. 
	
	For the numerator in \eqref{e:dAv0}, by the same argument we have
	\begin{align}\begin{split}\label{e:dzA}
			&\phantom{{}={}} \left|\del_{v_k}\bE_\bQ\left[\prod_{j=1}^r\left(\prod_{i=1}^m\frac{v_j-x_i-e_i/n}{ v_j-x_i}\right)^{s_j}
			\right] e^{-\sum_{j=1}^r s_j (b(v'_j)}\right|\\
			&=
			\left|\sum_\bme a_\bme\left(\sum_{i = 1}^m\frac{1}{v_k-x_i-e_i/n}-\sum_{i = 1}^m\frac{1}{v_k-x_i}\right)\prod_{j=1}^r\left(\prod_{i=1}^m\frac{v_j-x_i-e_i/n}{ v_j-x_i}\right)^{s_j}
			e^{-\sum_{j=1}^r s_j (b(v'_j)}
			\right|\\
			&\lesssim \frac{1}{\dist \big(v_k,[\fa',\fb'] \big)}\sum_\bme a_\bme \prod_{j=1}^r\left|\prod_{i=1}^m\frac{v'_j-x_i-e_i/n}{ v'_j-x_i}\right|^{s_j} e^{-\sum_{j=1}^r s_j\Re b(v'_j)}\\
			&\lesssim \frac{1}{\dist \big(v_k,[\fa',\fb'] \big)}\lesssim \frac{1}{\eta}.
	\end{split}\end{align}
	We conclude from \eqref{e:avs} and \eqref{e:dzA} that $|\eqref{e:dervv}|\lesssim 1/\eta$. This finishes the proof of \eqref{e:derbb1} and \eqref{e:derbb2}.

	The bounds \eqref{a1} and \eqref{a2} will follow as consequences of \eqref{e:derbb1} and \eqref{e:derbb2} by taking $K$ large enough. Let us verify this for the second estimate in \eqref{a1}, on $|\cA^{(\bmv, \bms)}(z)e^{-b(z)-\sum_{j=1}^r s_j (b(v_j)+\log \cG(v_j))}|$; the other bounds can be checked the same way. We take $z',v_1', v_2', \ldots , v_r' \in \mathscr{D} \cup \mathscr{D} (\mathfrak{r})$ such that $|z-z'|, |v_j'-v_j|\leq \eta/K$ and $\dist \big( z', [\fa',\fb'] \big), \dist \big(v'_j, [\fa',\fb'] \big) \geq \eta$ for $1\leq j\leq r$. Then,
	\begin{align*}\begin{split}
			\left|\log \cA^{(\bmv, \bms)}(z)-\log \cA^{(\bmv, \bms)}(z')\right|
			=\left|\int^{z}_{z'}\del_u \log \cA^{(\bmv, \bms)}(u)\rd u\right| & =\OO\left(\frac{1}{K}\right),\\
			\big |\log b(w)-\log b(w') \big|
			=\left|\int^{w}_{w'}\del_u \log b(u)\rd u\right| & = \mathcal{O} \bigg( \displaystyle\frac{1}{K} \bigg), \\
			\big|\log \cG(w)-\log \cG(w') \big|
			=\left|\int^{w}_{w'}\del_u \log \cG(u)\rd u\right| & =\OO\left(\frac{1}{K}\right),
	\end{split}\end{align*}
	for $w=z, v_1,v_2,\cdots, v_r$, where we used \eqref{e:derbb1}, \eqref{e:derbb2} and \Cref{l:fbbound} to bound the derivative. Then, by taking $K$ sufficiently large, we find 
	\begin{align*}
		\frac{1}{2\fC} \leq \displaystyle\frac{1}{\mathfrak{C}} + \mathcal{O} \bigg( \displaystyle\frac{1}{K} \bigg) \leq |\cA^{(\bmv, \bms)}(z)e^{-b(z)-\sum_{j=1}^r s_j (b(v_j)+\log \cG(v_j))}|\leq \mathfrak{C} + \mathcal{O} \bigg( \displaystyle\frac{1}{K} \bigg) \leq  2\fC,
	\end{align*}
	
	\noindent establishing the proposition.
\end{proof}

Now we can establish   \Cref{c:boost}.

\begin{proof}[Proof of \Cref{c:boost}]
	
	With the estimates \eqref{e:derbb1} and \eqref{e:derbb2}  as input, we can analyze the error term \eqref{e:defcE} and prove \eqref{e:cebb}.
	We recall $\cE^{(\bmv,\bms)}_0(w)$ from \eqref{e:defcE}, and rewrite it as
	\begin{align}\label{e:cc}
		\frac{f(w+1/n)}{f(w+1/n)+1} \left(\frac{\cA^{(\bmv,\bms)}(w+1/n)\cB(w+1/n)\tilde g(w)}{\cA^{(\bmv,\bms)}(w)\cB(w)\tilde g(w+1/n)}e^{\frac{1}{n}\psi(w)}-1\right) + \left(\frac{f(w+1/n)}{f(w+1/n)+1}-\frac{f(w)}{f(w)+1}\right).
	\end{align}
	For the first term on the right side of \eqref{e:cc}, using \eqref{e:derbb1} and \eqref{e:derbb2}, we have
	\begin{align}\begin{split}\label{e:setm}
			& \phantom{{}={}}\frac{{\cA^{(\bmv,\bms)}(w+1/n)\cB(w+1/n)}}{{\cA^{(\bmv,\bms)}(w)\cB(w)}}\frac{\tilde g(w)}{\tilde g(w+1/n)}e^{\frac{1}{n}\psi(w)}\\
			&\qquad = \exp \Bigg( \int_{w}^{w+1/n}\del_u \big( \log \cA^{(\bmv,\bms)}(u)+\log \cB(u) \big)
			\rd u \Bigg) \left(1+\OO\left(\frac{1}{n}\right)\right)
			=1+\OO\left(\frac{1}{n\eta}\right),
	\end{split}\end{align}
	for  $\dist \big(w,[\fa',\fb'] \big)\geq (1-1/K)\eta$, where we also used the bound $\big|\del_u \log \cB(u) \big|\lesssim 1/\eta$ from \Cref{l:fbbound}. For the second term on the right side of \eqref{e:cc}, using \Cref{l:fbbound},  we have
	\begin{align}\label{e:sec}
		\frac{f(w+1/n)}{f(w+1/n)+1}-\frac{f(w)}{f(w)+1}
		=\int_{w}^{w+1/n}\del_u \left(\frac{f(u)}{f(u)+1}\right)
		\rd u\lesssim \frac{1}{n\dist \big(w,[\fa',\fb'] \big)}.
	\end{align}
	
	\noindent For the last term on the right side of \eqref{e:defcE}, since $|f(w)+1|\gtrsim 1$ we have
	\begin{align}\label{e:lastt}
		\left| \frac{\psi^{(\bmv,\bms)}(w)-1}{f(w)+1}\right|\lesssim \frac{1}{n\eta},
	\end{align}
	
	\noindent for $\dist \big(w,[\fa',\fb'] \big)\geq (1-1/K)\eta$ and $\dist \big(w,\{v_1, v_2, \cdots, v_r\} \big)\geq \eta/2K$. 
	Combining \eqref{e:setm}, \eqref{e:sec} and \eqref{e:lastt}, we deduce the bound on $\cE^{(\bmv, \bms)}(w)$,
	\begin{align*}
		\big| \cE^{(\bmv, \bms)}(w) \big|\lesssim \frac{1}{n\eta},
	\end{align*}
	provided $\dist \big(w,[\fa',\fb'] \big)\geq (1-1/K)\eta$ and $\dist \big(w,\{v_1, v_2, \cdots, v_r\} \big)\geq \eta/2K$.
	This finishes the proof of \eqref{e:cebb}.

	It therefore remains to establish the second statement of the proposition. To that end, we can rewrite \eqref{e:ABCprecise} in the following linear form
	\begin{align}\label{e:dzC0}
		\frac{\del_z \cC^{(\bmv,\bms)}(z)}{\cC^{(\bmv,\bms)}(z)}=\frac{\del_z \cA^{(\bmv,\bms)}(z)}{\cA^{(\bmv,\bms)}(z)}+\frac{\del_z\cB(z)}{\cB(z)}+\frac{\del_z\cE^{(\bmv,\bms)}(z)}{1+\cE^{(\bmv,\bms)}(z)}.
	\end{align}
	In the following we first prove that $\del_z \cC^{(\bmv,\bms)}(z)/\cC^{(\bmv,\bms)}(z)$ is analytic for $z$ inside 
	\begin{align}\label{e:dom0}
		\Lambda_0 = \bigg( \mathscr D\cup \Big\{z:\dist \big(z,[\fa',\fb'] \big)\leq \fr \Big\} \bigg) \cap \Big\{z \in \Lambda : \dist \big(z, \{v_1, v_2,\cdots, v_r\} \big) \geq \eta/(2K) \Big\}\subseteq \Lambda.
	\end{align}
	The poles of $\del_z \cC^{(\bmv,\bms)}(z)/\cC^{(\bmv,\bms)}(z)$ are the zeros and poles of $\cC^{(\bmv,\bms)}(z)$. Thanks to the loop equation \eqref{e:sum1}, we know that $\cC^{(\bmv,\bms)}(z)$ is analytic and does not have poles in $\Lambda\setminus \big\{v_j-(s_j+1)/2n \big\}_{1\leq j\leq r}$. So, we only need to show that $\cC^{(\bmv,\bms)}(z)$ does not have zeros inside $\Lambda_0$.
	We recall the expression \eqref{e:ABCprecise},
	\begin{align}\label{e:ABCprecise0}
		\cC^{(\bmv,\bms)}(z)=\cA^{(\bmv,\bms)}(z)\cB(z)\left(1+\cE^{(\bmv,\bms)}(z)\right).
	\end{align}
	For $z\in \Lambda\setminus[\fa',\fb']$, by our \Cref{a:mz}, $\cB(z)\neq 0$. 
	
	The estimate \eqref{e:cebb} implies that  $1+\cE^{(\bmv,\bms)}(w)=1+\OO(1/n\eta)\neq 0$ for  $w\in \Lambda_0\cap \mathscr D$ and $\dist \big(w, [\mathfrak{a}', \mathfrak{b}'] \big) \geq(1 - 1/K) \eta$.
	Next we  show that $\mathcal{A}^{(\boldsymbol{v}, \boldsymbol{s})} (w)$ is nonzero when $w \in \mathscr D\cup \mathscr D(\fr)$ and $\dist \big(w, [\mathfrak{a}', \mathfrak{b}'] \big) \geq(1 - 1/K) \eta$. By \Cref{p:continuity}, \eqref{a1} holds for $z\in \mathscr D\cup \mathscr D(\fr)$ with $\dist \big(z, [\fa',\fb'] \big)\geq (1-1/K)\eta$. Thus $\cA^{(\bmv,\bms)}(w)\neq 0$ for $(\bmv,\bms)$ as in \eqref{e:term1}. For $(\bmv,\bms)$ as in \eqref{e:term2a}, the weights $a_\bme^{(\bmv,\bms)}$ are positive, the decomposition \eqref{e:Aexp} implies that $\Im \big[\cA^{(\bmv,\bms)}(w) \big]>0$ for $\Imaginary w>0$. For $\Imaginary w=0$ and $\dist \big(w, [\fa',\fb'] \big)\geq (1-1/K)\eta$, since each $x_i\in [\fa, \fb-1/n]$, we have that 
	\begin{align*}
		\prod_{i=1}^m \frac{w-x_i-e_i/n}{w-x_i}>0.
	\end{align*}
	
	\noindent So, the definition \eqref{e:defA} of $\mathcal{A}^{(\boldsymbol{v}, \boldsymbol{s})}$ and the positivity of the $a_{\boldsymbol{e}}^{(\boldsymbol{v}, \boldsymbol{s})}$, together imply that $\cA^{(\bmv,\bms)}(w)>0$. Thus on $\big\{w\in \mathscr{D} \cup \mathscr{D} (\mathfrak{r}): \dist(w, [\fa',\fb'])\geq (1-1/K)\eta \big\}$, we have $\cA^{(\bmv,\bms)}(w)\neq 0$. We conclude from \eqref{e:ABCprecise0} that $\cC^{(\bmv,\bms)}(w)\neq 0$ on $\Lambda_0 \cap \big\{ w\in \mathscr D: \dist(w, [\fa',\fb'])\geq (1-1/K)\eta \big\}$.
	
	For $\big\{w\in \Lambda_0: \dist(w, [\fa',\fb']) < (1-1/K)\eta \big\}$, we take a contour $\omega$ inside the region $\big\{ w\in \Lambda_0: \dist(w, [\fa',\fb'])\geq (1-1/K)\eta \big\}$, which encloses $[\fa',\fb']$ but not $v_1,v_2,\cdots, v_r$, and whose distance from the set $\{ v_1, v_2, \ldots , v_r \}$ is at least $\eta/(2K)$. We may count the zeros of $\cC^{(\bmv,\bms)}(w)$ inside $\omega$ by 
	\begin{align}\label{e:aC0}
		\frac{1}{2\pi\ri}\oint_{\omega}\frac{\del_w \cC^{(\bmv,\bms)}(w)}{\cC^{(\bmv,\bms)}(w)}\rd w=\sum_{p: \cC^{(\bmv,\bms)}(p)=0\atop\text{$p$ inside the contour}}1.
	\end{align}

	We recall the definition of $\cA^{(\bmv,\bms)}(w)$ from \eqref{e:defA} and the decomposition \eqref{e:Aexp},
	\begin{align}
		\label{asum}
		\cA^{(\bmv,\bms)}(w)=\sum_\bme a_\bme^{(\bmv,\bms)} \left(1+\sum_{i = 1}^m \frac{\Res_{w=x_i}}{w-x_i}\right).
	\end{align}
	It has the same number of zeros and poles, and we have shown that it does not have zeros or poles inside the region $\big\{ w\in \mathscr D\cup \mathscr D(\fr): \dist(w, [\fa',\fb'])\geq (1-1/K)\eta \big\}$. Therefore, $\cA^{(\bmv,\bms)}(z)$ has the same number of zeros and poles inside the contour $\omega$, so it follows that
	\begin{align}\label{e:aA0}
		\frac{1}{2\pi\ri}\oint_{\omega}\frac{\del_w \cA^{(\bmv,\bms)}(w)}{\cA^{(\bmv,\bms)}(w)}\rd w=0.
	\end{align}
	
	\noindent We perform the same contour integral on both sides of \eqref{e:dzC0}, we get
	\begin{align}\begin{split}\label{e:aC01}
			\sum_{p: \cC^{(\bmv,\bms)}(p)=0\atop\text{$p$ inside the contour}}1 =\frac{1}{2\pi\ri}\oint_{\omega}\frac{\del_w \cC^{(\bmv,\bms)}(w)}{\cC^{(\bmv,\bms)}(w)}\rd w & =\frac{1}{2\pi \ri}\oint_{\omega}\left(\frac{\del_w \cA^{(\bmv,\bms)}(w)}{\cA^{(\bmv,\bms)}(w)}+\frac{\del_w \cB(w)}{\cB(w)}+\frac{\del_w  \cE^{(\bmv,\bms)}(w)}{1+ \cE^{(\bmv,\bms)}(w)}\right)\rd w\\
			&=\frac{1}{2\pi \ri}\oint_{\omega}\del_w  \log \big( 1+\cE^{(\bmv,\bms)}(w) \big)\rd w,
	\end{split}\end{align}
	where in the last equality, we used \eqref{e:aB0} and \eqref{e:aA0}. Thanks to \eqref{e:cebb}, along the contour $\omega$, $|\cE^{(\bmv,\bms)}(w)|\lesssim 1/n\eta\ll 1$. Since the left side of \eqref{e:aC01} is an integer, we conclude that the above integral is $0$ and  $\del_w \cC^{(\bmv,\bms)}(w)/\cC^{(\bmv,\bms)}(w)$ is analytic for $w \in \Lambda_0$.

	From the discussion above, $\del_z\cC^{(\bmv,\bms)}(z)/\cC^{(\bmv,\bms)}(z)$ is analytic inside the contour $\omega$, we can use a contour integral to get rid of $\del_z\cC^{(\bmv,\bms)}(z)/\cC^{(\bmv,\bms)}(z)$ and recover $\del_z \cA^{(\bmv,\bms)}(z)/\cA^{(\bmv,\bms)}(z)$,
	\begin{align}
		\label{aintegralbc} 
		\begin{split}
			\frac{\del_z \cA^{(\bmv,\bms)}(z)}{\cA^{(\bmv,\bms)}(z)}=-\frac{1}{2\pi \ri}\oint_{\omega}\frac{\del_w \cA^{(\bmv,\bms)}(w)}{\cA^{(\bmv,\bms)}(w)}\frac{\rd w}{w-z} & = \frac{1}{2\pi \ri}\oint_{\omega}\left(\frac{\del_w \cB(w)}{\cB(w)}+\frac{\del_w \cE^{(\bmv,\bms)}}{1+\cE^{(\bmv,\bms)}}-\frac{\del_w \cC^{(\bmv,\bms)}(w)}{\cC^{(\bmv,\bms)}(w)}\right)\frac{\rd w}{w-z}\\
			&=\frac{1}{2\pi \ri}\oint_{\omega}\frac{\del_w \cB(w)}{\cB(w)}\frac{\rd w}{w-z}+\frac{1}{2\pi\ri }\oint_{\omega} \log \big( 1+\cE^{(\bmv,\bms)}(w) \big)\frac{\rd w}{(w-z)^2},
		\end{split}
	\end{align}
	where the contour $\omega$ encloses $[\fa',\fb']$ but not $\{z,v_1,v_2,\cdots, v_r\}$, and is at distance at least $\eta / 2K$ from $z$. Here, we have used the fact (which follows from \eqref{asum}) that $\big| \partial_w \mathcal{A}^{(\boldsymbol{v}, \boldsymbol{s})} (w) / \mathcal{A}^{(\boldsymbol{v}, \boldsymbol{s})} (w) \big|$ tends to $0$ as $w$ tends to $\infty$. Along the contour $\omega$, $\big| \cE^{(\bmv,\bms)}(w) \big| \lesssim 1/n\eta$, so we conclude that
	\begin{align}\label{e:firstodahi}
		\frac{\del_z \cA^{(\bmv,\bms)}(z)}{\cA^{(\bmv,\bms)}(z)}
		=\frac{1}{2\pi \ri}\oint_{\omega}\frac{\log \cB(w)\rd w}{(w-z)^2}+\OO\left(\frac{1}{n\eta^2}\right).
	\end{align}
	This finishes the proof of \eqref{e:firstoda}.

	By first integrating \eqref{aintegralbc} from $z_2$ to $z_1$, and then applying \eqref{e:cebb}, we get
	\begin{align}\begin{split}\label{e:z1z2}
			\phantom{{}={}}\log \cA^{(\bmv,\bms)} (z_1)-\log \cA^{(\bmv,\bms)}(z_2) &=\frac{1}{2\pi \ri}\oint_{\omega} \Big( \log \cB(w)+\log \big(1+\cE^{(\bmv,\bms)}(w) \big) \Big)\left(\frac{1}{w- z_1}-\frac{1}{w-z_2}\right)\rd w\\
			&=b(z_1)-b(z_2)+\OO\left(\frac{1}{n\eta}\right).
	\end{split}\end{align}
	\noindent We can also take $z_2=\infty$ in \eqref{e:z1z2} to obtain
	\begin{align}\label{e:zinf}
		&\log \cA^{(\bmv,\bms)}(z)-\log \cA^{(\bmv,\bms)}(\infty)=b(z)+\OO\left(\frac{1}{n\eta}\right).
	\end{align}

	Thanks to the relation \eqref{e:zinf}, \eqref{e:toap2} implies \eqref{e:toap1}. Thus we only need to prove \eqref{e:toap2}. The relation \eqref{e:zinf} is almost the same as \eqref{e:yao00}, but with an error depending on $\eta$. 
The same induction argument after 	\eqref{e:yao00} establishes \eqref{e:toap2}, and \Cref{c:boost}. 

\end{proof}

\section{Improved Edge Estimates}\label{s:improveEE}
In this section, we strengthen the weak estimate \Cref{p:firstod} through bootstrapping, and conclude the proof of Propositions \ref{p:improve} and \ref{p:improve2}. We collect some notations, and miscellaneous derivative bounds in Section \ref{EstimateE}. We prove the first statement \eqref{e:improve1c} of \Cref{p:improve2} in Section \ref{s:1state}, and the second statement \eqref{e:improve2c} in Section \ref{s:2state}. Finally, in Section \ref{s:finalproof}, we prove \Cref{p:improve} using \Cref{p:improve2} as input.

\subsection{Notations and Preliminary Estimates}

\label{EstimateE}
In this section, we collect some notations, and miscellaneous derivative bounds, which will be used to prove \Cref{p:improve2} in Sections \ref{s:1state} and \ref{s:2state}.

To prove \eqref{e:improve1c}, we take $r=1$, $s_1=-1$, and $v_1=v$ in \eqref{e:firstod},
\begin{align}\begin{split}\label{e:firstodc}
		\del_z \big(\log \cA^{(\bmv,\bms)}(z)-b(z) \big)
		&=\frac{1}{2\pi \ri}\oint_{\omega}\log \big(1+\cE^{(\bmv,\bms)}(w) \big) \frac{\rd w}{(w-z)^2},
\end{split}\end{align}
where the contour $\omega\subseteq \mathscr D$ encloses $[\fa',\fb']$ but not $z$, $\overline{z}$, or $v$. We recall from \eqref{e:dzlogA} that \eqref{e:improve1c} is equivalent to
\begin{align}\label{e:eqform}
	\bigg| \del_z \big(\log \cA^{(\bmv,\bms)}(z)-b(z) \big) \Big|_{z=v} \bigg|\lesssim \frac{\big| \Im[f(v)/(f(v)+1)] \big|}{n\Im[v] \dist(v,I)}.
\end{align}

By the definition \eqref{e:firstod}, $b(z)= \OO(1/z)$ as $z\rightarrow \infty$. Thus, 
by integrating \eqref{e:firstodc} from $z$ to $\infty$, we get
\begin{align}\label{e:nodelz}
	\log \cA^{(\bmv,\bms)}(z)-\log \cA^{(\bmv,\bms)}(\infty)-b(z)
	&=\frac{1}{2\pi \ri}\oint_{\omega}\log \big( 1+\cE^{(\bmv,\bms)}(w) \big)\frac{\rd w}{w-z}.
\end{align}

\noindent By taking a contour integral on both sides of \eqref{e:nodelz}, we get for any $k_1+k_2\geq 1$ that 
\begin{align}\begin{split}\label{e:toest}
		\frac{1}{2\pi\ri}\oint_{\omega'} \frac{\big( \log \cA^{(\bmv,\bms)}(w)-b(w) \big)\rd w}{(w-z)^{k_1}(w-\overline z)^{k_2}}
		&=\frac{1}{(2\pi \ri)^2}\oint_{w\in \omega}\oint_{w'\in \omega'}\frac{\log \big(1+\cE^{(\bmv,\bms)}(w) \big)\rd w\rd w'}{(w-w')(w'-z)^{k_1}(w'-\overline z)^{k_2}}\\
		&=-\frac{1}{2\pi \ri}\oint_{\omega}\frac{\log \big(1+\cE^{(\bmv,\bms)}(w) \big)\rd w}{(w-z)^{k_1}(w-\overline z)^{k_2}},
\end{split}\end{align}
where the contour $\omega'\subseteq \mathscr D$ encloses $\omega$ but not $z,\overline{z}$, or $v,\overline v$. In the following, we analyze \eqref{e:toest}, which is  more general than \eqref{e:firstodc}. In particular, by taking $k_1=2$ and $k_2=0$ in \eqref{e:toest}, we recover \eqref{e:firstodc}, where we are using the facts that $\log \mathcal{A}^{(\boldsymbol{v}, \boldsymbol{s})} (\infty)$ is bounded (in a way dependent on the $v_j$, by \eqref{e:ap2}); that $b(\infty) = 0$; and that $\log \mathcal{A}^{(\boldsymbol{v}, \boldsymbol{s})} (z) - b(z)$ has no poles outside of $\omega'$.

We recall the spectral region $\mathscr D(\fr)= \big\{z\in \mathscr D: \dist(z, [\fa',\fb'])\geq \fr \big\}$. Let $v\in \mathscr D\setminus \mathscr D(\fr)$. We will later take either $z \in \{ v, \overline v \}$, or $z\in \mathscr{D}\cap \mathscr{D} (\mathfrak{r})$ in \eqref{e:firstodc}. We divide them into three cases. 
If $z \in \{ v,\overline v \}$, then we deform the contour $\omega$ in \eqref{e:toest} to see that its right side equals
\begin{align}\label{e:toest1}
	\frac{1}{2\pi \ri}\oint_{\omega_{v,\overline v}}\frac{\log \big(1+\cE^{(\bmv,\bms)}(w) \big) \rd w}{(w-z)^{k_1}(w-\overline z)^{k_2}}-\frac{1}{2\pi \ri}\oint_{\omega+}\frac{\log \big( 1+\cE^{(\bmv,\bms)}(w) \big) \rd w}{(w-z)^{k_1}(w-\overline z)^{k_2}}.
\end{align}
If $z\in \mathscr D(\fr)\setminus \mathscr D(2\fr)$ we deform the contour $\omega$ in \eqref{e:toest} to
\begin{align}\label{e:toest2}
	\frac{1}{2\pi \ri}\oint_{\omega_{v,\overline v}}\frac{\log \big(1+\cE^{(\bmv,\bms)}(w) \big)\rd w}{(w-z)^{k_1}(w-\overline z)^{k_2}}+\frac{1}{2\pi \ri}\oint_{\omega_{z,\overline z}}\frac{\log \big(1+\cE^{(\bmv,\bms)}(w) \big)\rd w}{(w-z)^{k_1}(w-\overline z)^{k_2}}-\frac{1}{2\pi \ri}\oint_{\omega+}\frac{\log \big(1+\cE^{(\bmv,\bms)}(w) \big)\rd w}{(w-z)^{k_1}(w-\overline z)^{k_2}}.
\end{align}

\noindent If $z \in \mathscr{D} (2 \mathfrak{r})$, 
we deform the contour $\omega$ in \eqref{e:toest} to
\begin{align}\label{e:toest33}
	\frac{1}{2\pi \ri}\oint_{\omega_{v,\overline v}}\frac{\log \big(1+\cE^{(\bmv,\bms)}(w) \big)\rd w}{(w-z)^{k_1}(w-\overline z)^{k_2}}-\frac{1}{2\pi \ri}\oint_{\widehat\omega+}\frac{\log \big(1+\cE^{(\bmv,\bms)}(w) \big)\rd w}{(w-z)^{k_1}(w-\overline z)^{k_2}}.
\end{align}
Here, in all three cases, the contour $\omega_{v,\overline v}$ encloses $v$ and $\overline v$; the contour $\omega_{z,\overline z}$ encloses $z$ and $\overline z$; the contour $\omega+\subseteq \mathscr D\cap \mathscr D(2\fr)$ encloses $[\fa',\fb']$ and $v,\overline v, z,\overline z$;
and  the contour $\widehat \omega+\subseteq \mathscr D(\fr)\setminus \mathscr D(2\fr)$ encloses $[\fa',\fb']$ and $v,\overline v$, but not $z,\overline z$.  See Figure \ref{f:contour2}.

\begin{figure}
	
	\begin{center}		
		
		\begin{tikzpicture}[
			>=stealth,
			auto,
			style={
				scale = .52
			}
			]

			\draw[black, thick] (0,0) ellipse (5 and 1.8 );
			\draw[black, thick] (-2,0.6) circle (0.4 );
			\draw[black, fill=black](-2,0.6) circle (0.05 );
			\draw[] (-1.9,0.6)node[left, scale = .7]{$v$};
			\draw[black, thick] (-2,-0.6) circle (0.4 );
			\draw[black, fill=black](-2,-0.6) circle (0.05 );
			\draw(-1.9,-0.6)node[left, scale = .7]{$\bar v$};
			
			\draw[black] (-5.5,0) -- (5.5,0);
			\draw[black, very thick] (-3,0) -- (3,0);
			
			\draw[] (-1.6,-0.6)  node[right, scale = .7]{$\omega_{v,\bar v}$};
			\draw[] (-1.6,0.6)  node[right, scale = .7]{$\omega_{v,\bar v}$};
			\draw[] (0,1.8)  node[below, scale = .7]{$\widehat\omega+$};
			
			\draw[black, thick] (12,0) ellipse (5 and 1.8 );
			\draw[black, thick] (10,0.6) circle (0.4 );
			\draw[black, fill=black](10,0.6) circle (0.05 );
			\draw[] (10.1,0.6)node[left, scale = .7]{$v$};
			\draw[black, thick] (10,-0.6) circle (0.4 );
			\draw[black, fill=black](10,-0.6) circle (0.05 );
			\draw(10.1,-0.6)node[left, scale = .7]{$\bar v$};
			
			\draw[] (10.4,-0.6)  node[right, scale = .7]{$\omega_{v,\bar v}$};
			\draw[] (10.4,0.6)  node[right, scale = .7]{$\omega_{v,\bar v}$};
			
			\draw[black] (6.5,0) -- (17.5,0);
			\draw[black, very thick] (9,0) -- (15,0);

			\draw[] (12,1.8)  node[below, scale = .7]{$\omega+$};
			
			\draw[black, thick] (13,1.3) circle (0.4 );
			\draw[black, fill=black](13,1.3) circle (0.05 );
			\draw[] (13.1,1.3)node[left, scale = .7]{$z$};
			\draw[black, thick] (13,-1.3) circle (0.4 );
			\draw[black, fill=black](13,-1.3) circle (0.05 );
			\draw(13.1,-1.3)node[left, scale = .7]{$\bar z$};
			
			\draw[] (13.4,-1.3)  node[right, scale = .7]{$\omega_{z,\bar z}$};
			\draw[] (13.4,1.3)  node[right, scale = .7]{$\omega_{z,\bar z}$};

		\end{tikzpicture}
		
	\end{center}
	
	\caption{\label{f:contour2} Shown to the left is the case $z\in \mathscr D (2\fr)$, where the contour $\omega_{v,\overline v}$ consists of two small circles enclosing $v,\overline v$ respectively, and the contour $\widehat \omega+\subseteq\mathscr D(\fr)\setminus\mathscr D(2\fr)$ encloses $[\fa',\fb']$ and $v,\overline v$ but not $z$.
		Shown to the right is the case $z\in \mathscr D (\fr)\setminus\mathscr D (2\fr)$ or $z=v,\bar v$, where the contour $\omega_{v,\overline v}$ encloses $v,\overline v$, the contour $\omega_{z,\overline z}$ encloses $z,\overline z$, and the contour $\omega+\subseteq\mathscr D(2\fr)$ encloses $[\fa',\fb']$ and $v,\overline v, z,\overline z$.}		
\end{figure}
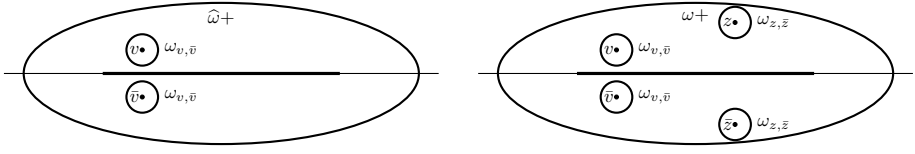

We recall $\cE_0^{(\bmv,\bms)}(z)$ from \eqref{e:defcE},
\begin{align*}\begin{split}
		&\cE^{(\bmv,\bms)}(z)=
		\cE^{(\bmv,\bms)}_0(z)
		+\frac{\psi^{(\bmv,\bms)}(z)-1}{f(z)+1},\\
		&\cE^{(\bmv,\bms)}_0(z)=\left(\frac{f(z+1/n)}{f(z+1/n)+1} \frac{\cA^{(\bmv,\bms)}(z+1/n)\cB(z+1/n)}{\cA^{(\bmv,\bms)}(z)\cB(z)}\frac{\tilde g(z)}{\tilde g(z+1/n)}e^{\frac{1}{n}\psi(z)}-\frac{f(z)}{f(z)+1}\right),
\end{split}\end{align*}

\noindent and $b(z)$ from \eqref{e:defbz},
\begin{align*}
	b(z)=\frac{1}{2\pi \ri}\oint_{\omega} \frac{\log \cB(w)\rd w}{w-z}
	=-\log\cB(z)+\frac{1}{2\pi \ri}\oint_{\omega+} \frac{\log \cB(w)\rd w}{w-z},
\end{align*}
where the new contour $\omega+\subseteq \mathscr D\cap \mathscr D(\fr)$  encloses both $[\fa',\fb']$ and $z$. We introduce
\begin{align*}
	\hat g(z)=\tilde g(z) \exp \Bigg( -{\frac{1}{2\pi \ri}\oint_{\omega+} \frac{\log \cB(w)\rd w}{w-z}} \Bigg),
\end{align*}
where the contour $\omega+\subseteq \mathscr D\cap \mathscr D(\fr)$  encloses both $[\fa',\fb']$ and $z$. Similar to $\tilde g(z)$ (by \Cref{a:mz}),  the function $\hat g(z)$ is analytic in $\Lambda$ and does not have zeros or poles in $\Lambda$. Moreover, we have
\begin{align}\label{e:dhgbound}
	\big| \del_z\log \hat g(z) \big|\leq \big| \del_z \log \tilde g(z) \big|+\frac{1}{2\pi}\left|\oint_{\omega+}\frac{\log \cB(w)\rd w}{(w-z)^2}\right|\lesssim 1,\quad \left|\frac{\hat g(z)}{\hat g(z+1/n)}-1\right|\lesssim \frac{1}{n},
\end{align}

\noindent by taking the contour $\omega+$ to be bounded away from $z$, together with the fact that $\mathcal{B} (w) = \big( f(w) + 1 \big) (w - \mathfrak{a})$ is bounded above and below for $w$ bounded away from $[\fa',\fb']$.

With $\hat g(z)$, we can rewrite $\cE^{(\bmv,\bms)}_0(z)$ as
\begin{align}\begin{split}\label{e:newcE0}
		\frac{f(z+1/n)}{f(z+1/n)+1} \left(\frac{\cA^{(\bmv,\bms)}(z+1/n)e^{-b(z+1/n)}}{\cA^{(\bmv,\bms)}(z)e^{-b(z)}}\frac{\hat g(z)}{\hat g(z+1/n)}e^{\frac{1}{n}\psi(z)}-1\right)+\frac{f(z+1/n)}{f(z+1/n)+1}-\frac{f(z)}{f(z)+1}.
\end{split}\end{align}

To estimate \eqref{e:toest1}, \eqref{e:toest2} and \eqref{e:toest33}, for any function $h$, which is analytic in  small neighborhoods of $z,\overline z$, and $k\geq 0$, we let $\cD_z^{(k)}h$ denote the set of values of the form
\begin{align}\label{e:defDz}
	\frac{1}{2\pi\ri}\oint_{\omega_{z,\overline z}}\frac{h(w)}{(w-z)^{k_1}(w-\overline z)^{k_2}} \mathrm{d} w, \quad  k_1+k_2=k+1,
\end{align}
where the contour $\omega_{z,\overline z}$ encloses $z,\overline z$. Observe that the expression \eqref{e:defDz}  contains the first $k$ derivatives of $h$ (up to some constant) at $z, \overline z$, by taking $k_1=k+1, k_2=0$ or $k_1=0, k_2=k+1$.
Although $\cD_z^{(k)}h$ is a set of numbers, with a slight abuse of notation, later we will simply treat $\cD_z^{(k)}h$ as a number,  which can be any number in $\cD_z^{(k)}h$.
We consider the product of two functions,
\begin{align*}
	\frac{1}{2\pi\ri}\oint_{\omega_{z,\overline z}}\frac{h_1(w)h_2(w)}{(w-z)^{k_1}(w-\overline z)^{k_2}} \mathrm{d}w.
\end{align*}
If $k_2=0$ (and, similarly, if $k_1=0$), by the chain rule we have that 
\begin{align*}
	&\phantom{{}={}}\frac{(k_1-1)!}{2\pi\ri}\oint_{\omega_{z, \overline{z}}} \frac{h_1(w)h_2(w)}{(w-z)^{k_1}}  \mathrm{d}w
	=\del_z^{k_1-1} \big( h_1(z) h_2(z) \big),
\end{align*}
is a linear combination of 
\begin{align*}
	\frac{1}{2\pi\ri}\oint_{\omega_{z, \overline{z}}} \frac{h_1(w)}{(w-z)^{k_1-a}} \mathrm{d}w \cdot \frac{1}{2\pi\ri}\oint_{\omega_{z, \overline{z}}} \frac{h_2(w)}{(w-z)^{a+1}} \mathrm{d} w,\quad 0\leq a< k_1.
\end{align*}
More generally, for $k_1, k_2\geq 1$ we have
\begin{align*}\begin{split}
		(k_1 - 1)! (k_2 - 1)! & \frac{1}{2\pi\ri}\oint_{\omega_{z, \overline{z}}} \frac{h_1(w)h_2(w)}{(w-z)^{k_1}(w-\overline z)^{k_2}} \mathrm{d}w \\
		& =\del_z^{k_1-1}\del_{\overline z}^{k_2-1}\left(\frac{h_1(z)h_2(z)-h_1(\overline z)h_2(\overline z)}{z-\overline z}\right)\\
		&=\del_z^{k_1-1}\del_{\overline z}^{k_2-1}\left(\frac{h_1(z)-h_1(\overline z)}{z-\overline z}h_2(z)+\frac{h_2(z)-h_2(\overline z)}{z-\overline z}h_1(\overline z)\right),
\end{split}\end{align*}
which can be written as a linear combination of terms in the form
\begin{align*}
	\frac{1}{2\pi\ri}\oint\frac{h_1(w) \mathrm{d}w}{(w-z)^{k_1-a}(w-\overline z)^{k_2}} \cdot \frac{1}{2\pi\ri}\oint\frac{h_2(w) \mathrm{d} w}{(w-z)^{a+1}}, \quad 0\leq a< k_1,
\end{align*}
and
\begin{align*}
	\frac{1}{2\pi\ri}\oint\frac{h_2 (w) \mathrm{d} w}{(w-z)^{k_1}(w-\overline z)^{k_2-a}} \cdot \frac{1}{2\pi\ri}\oint\frac{h_1(w) \mathrm{d} w}{(w-\overline z)^{a+1}}, \quad 0\leq a< k_2.
\end{align*}
Therefore, $\cD_z^{(k)}(h_1h_2)$ is a linear combination of $\cD_z^{(k_1)}(h_1)\cD_z^{(k_2)}(h_2)$, with $k_1+k_2= k$ (and uniformly bounded coefficients). We write this by
\begin{align}\label{e:lincomb}
	\cD_z^{(k)}(h_1h_2)\in \text{LinComb}_{k_1+k_2= k} \cD_z^{(k_1)}(h_1)\cD_z^{(k_2)}(h_2).
\end{align}
If $h(w)$ is uniformly bounded (that is, if $\big| h(w) \big| = \mathcal{O} (1)$) for $|w-z|, |w-\bar z|\lesssim \eta$, we can take the contour $\omega_{z,\overline z}$ in \eqref{e:defDz} as two small circles centered around $z, \bar z$ with radius of order $\eta$. Then we have the trivial bound for any element of $\cD_z^{(k)}h$, given by
\begin{align}\label{e:Dzhbb}
	|\cD_z^{(k)}h|=\left|\frac{1}{2\pi\ri}\oint_{\omega_{z,\overline z}}\frac{h(w) \mathrm{d} w}{(w-z)^{k_1}(w-\overline z)^{k_2}}\right|\lesssim \oint_{\omega_{z,\overline z}} \frac{\mathrm{d} w}{\eta^{k_1+k_2}}
	\lesssim \frac{1}{\eta^k}, \qquad \text{for $k_1+k_2=k+1$},
\end{align}
which will be used repeatedly in the remaining of this section.

For the complex slope $f(z)$, we recall from \eqref{e:defI} that $\Im \big[ f(x+0\ri)/(f(x+\ri)+1) \big]$ defines a negative measure on $\bR$ with support $I$. Thanks to \eqref{e:fbound}, as in the fourth statement of \Cref{p:gfbound} we have for any $z\in \mathscr D$ and $k\geq 1$ that
\begin{align}\label{e:ftest}
	\left|\frac{f(z)}{f(z)+1}\right|\lesssim1,\quad  \left|\cD_z^{(k)}\frac{f}{f+1}\right|\lesssim \frac{|\Im[f(z)/f(z)+1]|}{\Im[z] \dist(z,I)^{k-1}},\quad \frac{|\Im[f(z)/f(z)+1]|}{\Imaginary z }\lesssim \frac{1}{\dist(z,I)}.
\end{align}
We remark that the second estimate in \eqref{e:ftest} is slightly more general than \eqref{e:derf/f+1}, but it can be proven essentially the same way as in \eqref{ff1derivative}, by
\begin{align*}\begin{split}
		\left|\cD^{(k)}_z \frac{f}{f+1}\right|
		& =\left| \displaystyle\frac{1}{\pi} \int_{I}\frac{k! \Im \big[f(x+\ri0)/(f(x+\ri0)+1) \big]\rd x}{(x-z)^{k_1}(x-\bar z)^{k_2}}\right|+\OO(1) \\
		& \lesssim \Bigg| \displaystyle\frac{1}{\dist (z, I)^{k_1 + k_2 - 2}} \displaystyle\int_I \displaystyle\frac{k! \Im \big[ f(x + \mathrm{i} 0) / (f (x + \mathrm{i} 0) + 1) \big] \mathrm{d} x}{|x - z|^2} \Bigg| + \mathcal{O} (1) \\
		& \lesssim \frac{\big| \Im[f(z)/(f(z)+1)] \big|}{\Im[z]\dist(z,I)^{k-1}},
\end{split}\end{align*}
where $k_1+k_2=k+1$.

The last two relations in \eqref{e:ftest} together with  \eqref{e:lincomb} implies for any $k\geq 1$ and $j\geq 0$,
\begin{align}\begin{split}\label{e:secddd}
		\cD_z^{(k)}\left(\frac{f}{f+1}\right)^{j}&\lesssim \sum \left|\prod_{u=1}^j \cD_z^{(\tilde k_u)}\left(\frac{f}{f+1}\right)\right| \\
		& \lesssim \sum \left(\frac{|\Im[f(z)/f(z)+1]|}{\Im[z] }\right)^j\frac{1}{\dist(z, I)^{\sum_{u=1}^j (k_u-1)}}\\
		&\lesssim \sum \frac{|\Im[f(z)/f(z)+1]|}{\Im[z] } \frac{1}{\dist(z, I)^{j-1}}\frac{1}{\dist(z, I)^{k-j}}
		\lesssim  \frac{|\Im[f(z)/f(z)+1]|}{\Im[z] \dist(z,I)^{k-1}},
\end{split}\end{align}
where the sum is over index sets $(\tilde{k}_1, \tilde{k}_2, \ldots , \tilde{k}_j)$ such that $\sum_{u = 1}^j \tilde k_u=k$; and in the second line we used the last estimate in \eqref{e:ftest}.

\subsection{Proof of Statement \eqref{e:improve1}}\label{s:1state}

We recall from our \Cref{a:mz} that $\mathscr D\subseteq \big\{ z\in \Lambda: \dist(z,[\fa',\fb'])\geq n^{\fc-1} \big\}$. We fix a large constant 
\begin{align}\label{e:defq}
	q=4\ceil{R/\fc}.
\end{align}
For any $z\in \mathscr D$, we define the control parameters
\begin{align}\label{e:Phic}
	\begin{aligned}
		\Phi_k(z)= \frac{|\Im[f(z)/(f(z)+1)]|}{n\Im[z] \dist(z,I)^k}+\frac{1}{\dist \big(z,[\fa',\fb'] \big)^q}\left(\frac{\log n}{n}\right)^{q-k}, \qquad & \text{for $0\leq k\leq q$}; \\
		\Phi_k (v) = \dist \big( z, [\fa',\fb'] \big)^{-k}, \qquad & \text{for $k > q$}.
	\end{aligned}
\end{align}

\noindent Since $q \geq 4/\fc$, for $k=1$, we have
\begin{align}\label{e:Phi1}
	\Phi_1(z)= \frac{\Im[f(z)/(f(z)+1)]|}{n\Im[z]\dist(z,I)} +\OO\left(n\left(\frac{\log n}{n^\fc}\right)^q\right)\asymp \frac{\Im[f(z)/(f(z)+1)]|}{n\Im[z]\dist(z,I)},
\end{align}
where we used the lower bound $|\Im[f(z)/(f(z)+1)]|/\Im[z]\gtrsim 1$ from \eqref{e:imratio}. 

For $z\in \mathscr D$, $\Phi_k(z)\lesssim \dist \big(z,[\fa',\fb'] \big)^{-k}$. Thus for $k+k'>q$, we have 
$\Phi_{k}(z)\Phi_{k'}(z)\lesssim \dist \big(z,[\fa',\fb'] \big)^{-k - k'}=\Phi_{k+k'}(z)$. For $k+k'\leq q$, 
\begin{align*}\begin{split}
		\Phi_{k}(z)\Phi_{k'}(z)
		&\lesssim \frac{|\Im[f(z)/(f(z)+1)]|}{n\Im[z] \dist(z,I)^k}\frac{|\Im[f(z)/ (f(z)+1)]|}{n\Im[z] \dist(z,I)^{k'}}
		+\Phi_k(z)\frac{1}{\dist \big( z,[\fa',\fb'] \big)^q}\left(\frac{\log n}{n}\right)^{q-{k'}}\\
		& \qquad +\Phi_{k'}(z)\frac{1}{\dist \big(z,[\fa',\fb'] \big)^q}\left(\frac{\log n}{n}\right)^{q-{k}}\\
		&\lesssim\frac{|\Im[f(z)/(f(z)+1)]|}{n\Im[z] \dist(z,I)^{k+k'}}+\frac{1}{\dist \big(z,[\fa',\fb'] \big)^q}\left(\frac{\log n}{n}\right)^{q-k-k'}
		= \Phi_{k+k'}(z),
\end{split}\end{align*}
where for the last line, we used  $\Phi_k(z)\lesssim \dist \big(z,[\fa',\fb'] \big)^{-k} \ll (n/\log n)^k$. Thus $\Phi_k(z)$ satisfies
\begin{align}\label{e:phikk'}
	\Phi_{k}(z)\Phi_{k'}(z)\lesssim \Phi_{k+k'}(z).
\end{align}

The statement \eqref{e:eqform} will eventually be a consequence of the following proposition, which states that if we have some (weak) a priori estimate on $\cA^{(\bmv, \bms)}(z)$, then the dynamical loop equation can be used to obtain an improved estimate on $\cA^{(\bmv, \bms)}(z)$. 
\begin{prop}\label{c:boost2}
	
	Adopt \Cref{a:mz}. Suppose that $r=1$, $s_1=-1$, and $v_1=v\in \mathscr D\setminus \mathscr D(\fr)$; let $z$ be a complex number with either $z \in \{ v, \overline{v} \}$ or $z \in \mathscr{D}\cap \mathscr{D} (\mathfrak{r})$. Further assume for any $1\leq k\leq 2q$ (as defined in \eqref{e:defq}) we have 
	\begin{align}\label{e:hyp1}
		\Big|\cD_z^{(k)} \big( \log \cA^{(\bmv,\bms)} (w) - b(w) \big) \Big|\leq   \log n \cdot \Phi_k(v).
	\end{align}
	
	\noindent Then, we have the improved estimate
	\begin{align}\label{e:conc2}
		\Big| \cD_z^{(k)} \big( \log \cA^{(\bmv,\bms)} (w) - b(w) \big) \Big|\leq \fC \Phi_k(v),\quad 1\leq k\leq 2q,
	\end{align}
	
	\noindent for some constant $\fC = \mathfrak{C} (R) > 1$.
	
\end{prop}

\begin{proof}
	
	In the following we prove \eqref{e:conc2} for $z=v$; the proofs in the other cases when $z=\overline v$ or $z\in \mathscr D\cap \mathscr D(\fr)$ follow from essentially the same argument. We denote $\eta=\dist \big(v, [\fa',\fb'] \big)$. Thanks to \eqref{e:zinf}, for any $z\in \mathscr D$ with $\dist \big(z,[\fa',\fb'] \big) \gtrsim \eta$, we have that
	\begin{align}
		\label{e:logAbb}
		\left|\log \cA^{(\bmv,\bms)} (z)-\log \cA^{(\bmv,\bms)} (\infty) -b (z)\right|\lesssim \frac{1}{n\eta}\lesssim1.
	\end{align}
	
	For $q\leq k\leq 2q$, \eqref{e:conc2} follows from \eqref{e:logAbb} and the trivial bound \eqref{e:Dzhbb},
	\begin{align*}
		\Big| \cD_v^{(k)} \big(\log \cA^{(\bmv,\bms)} (w) -b (w) \big) \Big|
		&\lesssim \frac{1}{\eta^k} = \Phi_k(v).
	\end{align*}
	
	\noindent For $k\leq q$, we recall that from  \eqref{e:toest1}, $\cD_v^{(k)} \big(\log \cA^{(\bmv,\bms)} (w) + b (w) \big)$ is a collection of values in the form,
	\begin{align}\label{e:toest1copy}
		\frac{1}{2\pi \ri}\oint_{\omega_{v,\overline v}}\frac{\log \big( 1+\cE^{(\bmv,\bms)}(w) \big)\rd w}{(w-v)^{k_1}(w-\overline v)^{k_2}}-\frac{1}{2\pi \ri}\oint_{\omega+}\frac{\log \big( 1+\cE^{(\bmv,\bms)}(w) \big) \rd w}{(w-v)^{k_1}(w-\overline v)^{k_2}}, \quad k_1+k_2=k+1,
	\end{align}
	where the contour $\omega_{v,\overline v}$ encloses $v,\overline v$, and the contour $\omega+\subseteq\mathscr D(\fr)$ encloses $[\fa',\fb']$ and $v,\overline v$.
	
	From our construction, $\cE_0^{(\bmv,\bms)}(w)$  is analytic for $w\in \mathscr D$, since all the denominators in \eqref{e:newcE0} are bounded away from $0$. However, 
	$\psi^{(\bmv,\bms)}(w)-1=1/n(w-v)$ has a pole at $v$. Along $\big\{ w \in \mathbb{C} : |w-v| = \eta / 3 \big\}$ and $\big\{ w \in \mathbb{C} : |w-\overline v| = \eta / 3 \big\}$, \eqref{e:errorbb} gives $\big| \cE^{(\bmv,\bms)}(w) \big| \lesssim 1/n\eta$. So, we can expand $\log \big( 1+\cE^{(\bmv,\bms)}(w) \big)$  as
	\begin{align}\begin{split}\label{e:ttsa}
			\sum_{j=1}^q & \frac{(-1)^{j-1} \big( \cE^{(\bmv,\bms)}(w) \big)^j}{j}+\OO\left(\frac{1}{(n\eta)^{q+1}}\right)\\
			&=\sum_{j_1+j_2= 1}^q\frac{(-1)^{j_1+j_2-1}}{j_1+j_2}{j_1+j_2\choose j_1}
			\big( \cE_0^{(\bmv,\bms)}(w) \big)^{j_1}\left(\frac{\psi^{(\bmv,\bms)}-1}{f(w)+1}\right)^{j_2}+\OO\left(\frac{1}{(n\eta)^{q+1}}\right)\\
			&=\sum_{j_1+j_2= 1}^q\frac{(-1)^{j_1+j_2-1}}{j_1+j_2}{j_1+j_2\choose j_1} \big( \cE_0^{(\bmv,\bms)}(w) \big)^{j_1}\left(\frac{1}{f(w)+1}\right)^{j_2}\frac{1}{n^{j_2}(w-v)^{j_2}}+\OO\left(\frac{1}{(n\eta)^{q+1}}\right).
	\end{split}\end{align}
	Thanks to the trivial bound \eqref{e:Dzhbb}, the contribution from the error term $\OO((n\eta)^{-(q+1)})$ is bounded by $\OO((n\eta)^{-(q+1)}\eta^{-k})$, which is smaller than $\Phi_k(v)$.
	By plugging \eqref{e:ttsa} into the first term in \eqref{e:toest1copy}, we get expressions involving linear combinations (with bounded coefficients) of quantities  of the form
	\begin{align}\begin{split}\label{e:tt1}
			\phantom{{}={}}\frac{1}{2\pi\ri}  \oint_{\omega_{v,\overline v}} & \big( \cE_0^{(\bmv,\bms)}(w) \big)^{j_1}\left(\frac{1}{f(w)+1}\right)^{j_2}\frac{\mathrm{d} w}{n^{j_2}(w-v)^{k_1 + j_2}(w-\overline v)^{k_2}}\\
			& \qquad \qquad \in \frac{1}{n^{j_2}}\cD_v^{(k_1+k_2+j_2-1)}\left( \big(\cE_0^{(\bmv,\bms)}(w) \big)^{j_1}\left(\frac{1}{f(w)+1}\right)^{j_2}\right),
	\end{split}\end{align}
	where $1\leq j_1+j_2\leq q$. 
	Thanks to \eqref{e:lincomb}, \eqref{e:tt1} decomposes to sum of products involving
	\begin{align}\label{e:tt1copy}
		\frac{1}{n^{j_2}}\cD_v^{( \hat k_1)} \big( \cE_0^{(\bmv,\bms)}(w) \big)^{j_1}\cD_v^{(\hat k_2)}\left(\frac{1}{f(w)+1}\right)^{j_2}\quad \hat k_1+\hat k_2=k_1+k_2+j_2-1=k+j_2.
	\end{align}

	In the following we show that under our hypothesis \eqref{e:hyp1}, for $0\leq \hat k\leq 2q$,
	\begin{align}\label{e:DcE0}
		\big| \cD_v^{(\hat k)}(\cE_0^{(\bmv, \bms)}) \big|
		&\lesssim \Phi_{\hat k}(v),
	\end{align}
	
	\noindent which together with \eqref{e:phikk'} and \eqref{e:lincomb}, would imply
	\begin{align}\label{e:secdd}
		\cD_v^{(\hat k_1)}\left(\cE_0^{(\bmv,\bms)}\right)^{j_1}
		\lesssim \sum \left|\prod_{u = 1}^{j_1} \cD_v^{(\tilde k_u)}\left(\cE_0^{(\bmv,\bms)}\right)\right|\lesssim \prod_{u = 1}^{j_1} \Phi_{\tilde k_u}(v)\lesssim  \Phi_{\hat k_1}(v),\quad j_1\geq 1,
	\end{align}
	where the sum is over index sets $(\tilde{k}_1, \tilde{k}_2, \ldots , \tilde{k}_{j_1})$ such that $\sum_{u = 1}^{j_1}\tilde k_u=\hat k_1$.
	
	For $q\leq {\hat k}\leq 2q$, \eqref{e:DcE0} follows from the trivial bound \eqref{e:Dzhbb},
	$
	\cD_v^{({\hat k})}(\cE_0^{(\bmv, \bms)})
	\lesssim \eta^{-\hat k}=\Phi_{\hat k}(v)$.
	For ${\hat k}\leq q$, we recall the expression of $\cE_0^{(\bmv, \bms)}$ from \eqref{e:newcE0}. For the last term in \eqref{e:newcE0} we have
	\begin{align}\label{e:firsttf}
		\left|\cD_v^{({\hat k})}\left(\frac{f(w+1/n)}{f(w+1/n)+1}-\frac{f(w)}{f(w)+1}\right)\right|
		\lesssim \frac{|\Im[f(v)/f(v)+1]|}{n\Im[v] \dist(v, I)^{\hat k}}\leq \Phi_{\hat k}(v),\quad {\hat k}\geq 0,
	\end{align}
	where we used \eqref{e:ftest}.
	Next we need to estimate the first term in \eqref{e:newcE0}
	\begin{align*}
		\cD_v^{({\hat k})} \Bigg( \frac{f(w+1/n)}{f(w+1/n)+1} \left(\frac{\cA^{(\bmv,\bms)}(w+1/n)e^{-b(w+1/n)}}{\cA^{(\bmv,\bms)}(w)e^{-b(w)}}\frac{\hat g(w)}{\hat g(w+1/n)}e^{\frac{1}{n}\psi(w)}-1\right) \Bigg),
	\end{align*}
	which, thanks to \eqref{e:lincomb}, is a linear combination of
	\begin{align}\label{e:linear}
		\cD_v^{(\hat k-\hat k')}\left(\frac{f(w+1/n)}{f(w+1/n)+1}\right) \cD_v^{(\hat k')}\left(\frac{\cA^{(\bmv,\bms)}(w+1/n)e^{-b(w+1/n)}}{\cA^{(\bmv,\bms)}(w)e^{-b(w)}}\frac{\hat g(w)}{\hat g(w+1/n)}e^{\frac{1}{n}\psi(w)}-1\right).
	\end{align}
	If $\hat k'=0$, we can use \eqref{e:dhgbound} to bound the second term of \eqref{e:linear} as
	\begin{align}\begin{split}\label{e:k'=0i}
			\frac{\cA^{(\bmv,\bms)}(v+1/n)e^{-b(v+1/n)}}{\cA^{(\bmv,\bms)}(v)e^{-b(v)}}\frac{\hat g(v)}{\hat g(v+1/n)}e^{\frac{1}{n}\psi(v)}-1
			=\frac{\cA^{(\bmv,\bms)}(v+1/n)e^{-b(v+1/n)}}{\cA^{(\bmv,\bms)}(v)e^{-b(v)}}-1+\OO\left(\frac{1}{n}\right).
	\end{split}\end{align}
	For the first term on the right side of \eqref{e:k'=0i}, by a Taylor expansion, using \eqref{e:hyp1}, we have
	\begin{align}\begin{split}\label{e:dlogA}
			\Bigg| \log \bigg(  \frac{\cA^{(\bmv,\bms)}(v+1/n)e^{-b(v+1/n)}}{\cA^{(\bmv,\bms)}(v)e^{-b(v)}}\bigg)\Bigg|
			& \lesssim 
			\left|  \sum_{j=1}^{q}\frac{\del_v^j \big( \log\cA^{(\bmv, \bms)}(v)-b(v) \big)}{n^j}+\OO\left(\frac{1}{(n\eta)^{q+1}}\right)\right|\\
			&\lesssim \sum_{j=1}^{q}n^{-j} \log n \cdot \Phi_{j}(v) +\frac{1}{(n\eta)^{q+1}} \\
			& \lesssim \frac{ \log n  |\Im[f(v)/(f(v)+1) ]|}{n^2\Im[v]\dist(v, I)} +\OO\left(\frac{1}{n}\right)
			\lesssim \frac{\big| \Im[f(v)/(f(v)+1)] \big|}{n\Im v}\ll1 ,
	\end{split}\end{align}
	where we used $\Phi_j(v)/n^j$ is decreasing in $j$, $q\geq 1/\fc$ thus $1/(n\eta)^{q+1}\leq 1/n$,  and the bound $n\dist(v,I)\geq n^{\fc}\gg  \log n$. By plugging \eqref{e:dlogA} into \eqref{e:k'=0i}, we get
	\begin{align}\label{e:k'=0}
		\frac{\cA^{(\bmv,\bms)}(v+1/n)e^{-b(v+1/n)}}{\cA^{(\bmv,\bms)}(v)e^{-b(v)}}\frac{\hat g(v)}{\hat g(v+1/n)}e^{\frac{1}{n}\psi(v)}-1
		=\OO\left(\frac{|\Im[f(v)/(f(v)+1)]|}{n\Im[v]}\right).
	\end{align}
	
	\noindent Then for $\hat k'=0$, by plugging \eqref{e:k'=0} into \eqref{e:linear}, and using \eqref{e:ftest}, we have if $\hat k\geq 1$,
	\begin{align}\label{e:8110}
		|\eqref{e:linear}|\lesssim  \frac{|\Im[f(v)/f(v)+1]|}{\Im[v] \dist(v,I)^{\hat k-1}} \left(\frac{|\Im[f(v)/(f(v)+1)]|}{n\Im[v]}\right)\leq \Phi_{\hat k}(v),
	\end{align}
	or $\hat k=0$,
	\begin{align}\label{e:81102}
		|\eqref{e:linear}|\lesssim   \left(\frac{|\Im[f(v)/(f(v)+1)]|}{n\Im[v]}\right)\leq \Phi_{0}(v).
	\end{align}

	For $\hat k'\geq 1$, using \eqref{e:lincomb} again, the second term on the right side of \eqref{e:linear} is a linear combination of quantities of the form
	\begin{align}\begin{split}\label{e:llc}
			&\phantom{{}={}}\left|\cD_v^{(\hat k'-\hat k'')}\left(\frac{\hat g(w)}{\hat g(w+1/n)}e^{\frac{1}{n}\psi(w)}\right) \cD_v^{(\hat k'')}\left(\frac{\cA^{(\bmv,\bms)}(w+1/n)e^{-b(w+1/n)}}{\cA^{(\bmv,\bms)}(w)e^{-b(w)}}\right)\right|\\
			& \qquad \qquad \qquad \qquad \qquad \lesssim \left(\frac{1}{n}\right)^{\bm1(\hat k'-\hat k''\geq 1)}\left|\cD_v^{(\hat k'')}\left(\frac{\cA^{(\bmv,\bms)}(w+1/n)e^{-b(w+1/n)}}{\cA^{(\bmv,\bms)}(w)e^{-b(w)}}\right)\right|,
	\end{split}\end{align}
	where we used \eqref{e:dhgbound}, as well as the analyticity of $\hat{g}$ and $\psi$ on $\Lambda$ (enabling the use of \eqref{e:Dzhbb} with $\eta$ of order $1$), to bound the first term.
	
	If $\hat k''=0$, we can bound the last term on the right side of \eqref{e:llc} by $\OO(1)$ using \eqref{e:dlogA}.  By plugging \eqref{e:llc} into \eqref{e:linear} and using \eqref{e:ftest}, the sum of such terms with $\hat k''=0$ is bounded 
	\begin{align}\label{e:8111}
		\frac{\big| \Im[f(v)/(f(v)+1)] \big|}{\Im[v] \dist(v,I)^{\hat k-\hat k'-1}} \left(\frac{1}{n}\right)^{\bm1(\hat k'-\hat k''\geq 1)}\lesssim \Phi_{\hat k}(v). 
	\end{align}
	
	To estimate the second term on the right side of \eqref{e:llc} with $1\leq \hat k''\leq q$, we first show that
	\begin{align}\label{e:DvklogA}
		\left|\cD_v^{(\hat k'')} \log \left(\frac{\cA^{(\bmv,\bms)}(w+1/n)e^{-b(w+1/n)}}{\cA^{(\bmv,\bms)}(w)e^{-b(w)}}\right)\right|\lesssim \Phi_{\hat k''}(v).
	\end{align}
	By a Taylor expansion, we have
	\begin{align}\begin{split}\label{e:DvklogA2}
			\phantom{{}={}}\Bigg|\cD_v^{(\hat k'')} \log \bigg( & \frac{\cA^{(\bmv,\bms)}(w+1/n)e^{-b(w+1/n)}}{\cA^{(\bmv,\bms)}(w)e^{-b(w)}}\bigg)\Bigg|\\
			&\lesssim 
			\left|\cD_v^{({\hat k''})} \left( \sum_{j=1}^{q-{\hat k''}}\frac{\del_w^j \big( \log\cA^{(\bmv, \bms)}(w)-b(w) \big)}{n^j}+\OO\left(\frac{1}{(n\eta)^{q-{\hat k''}+1}}\right)\right)\right|\\
			&\lesssim  \sum_{j=1}^{q-{\hat k''}} n^{-j} \Big| \cD_v^{({\hat k''}+j)} \big(\log\cA^{(\bmv, \bms)}(w)-b(w) \big) \Big| +\left|\cD_v^{({\hat k''})} \OO\left(\frac{1}{(n\eta)^{q-{\hat k''}+1}}\right)\right|\\
			&\lesssim \sum_{j=1}^{q-{\hat k''}}n^{-j} \log n \cdot \Phi_{{\hat k''}+j}(v) +\frac{1}{n\eta}\frac{1}{\eta^{q}}\frac{1}{n^{q-{\hat k''}}}\lesssim \Phi_{\hat k''}(v),
	\end{split}\end{align}
	where in the third inequality we used our assumption \eqref{e:hyp1} to bound the sum, and the trivial bound \eqref{e:Dzhbb} for the last term. In the last line, we used that $n \dist(v,I)\geq n^{\fc}\gg  \log n $, so $ n^{-j} \log n \cdot \Phi_{{\hat k''}+j}(v) \lesssim \Phi_{\hat k''}(v)$ for $j\geq 1$.
	Thanks to \eqref{e:lincomb} and the relation \eqref{e:phikk'}, \eqref{e:DvklogA2} implies that 
	\begin{align}\label{e:DvklogA3}
		\Bigg|\cD_v^{({\hat k''})} \bigg( \log \Big(\frac{\cA^{(\bmv,\bms)}(w+1/n)e^{-b(w+1/n)}}{\cA^{(\bmv,\bms)}(w)e^{-b(w)}}\Big) \bigg)^j \Bigg|\lesssim \Phi_{\hat k''}(v),\quad j\geq 1.
	\end{align}
	Using \eqref{e:DvklogA3} as input, we get the following bound for the second term on the right side of \eqref{e:llc}, 
	\begin{align}\begin{split}\label{e:llc2}
			&\phantom{{}={}} \left|\cD_v^{(\hat k'')} \left(\frac{\cA^{(\bmv,\bms)}(w+1/n)e^{-b(w+1/n)}}{\cA^{(\bmv,\bms)}(w)e^{-b(w)}}\right)\right|\\
			&  \qquad  =\left|\cD_v^{(\hat k'')}\Bigg(1+\frac{1}{j!}\sum_{j=1}^{q-\hat k''} \bigg( \log \Big(\frac{\cA^{(\bmv,\bms)}(w+1/n)e^{-b(w+1/n)}}{\cA^{(\bmv,\bms)}(w)e^{-b(w)}} \Big) \bigg)^j+\OO \bigg( \frac{1}{(n\eta)^{q-\hat k''+1}}\bigg) \Bigg)\right|
			\lesssim \Phi_{\hat k''}(v), 
	\end{split}\end{align}
	
	\noindent for $1\leq \hat{k}'' \leq q$. By plugging \eqref{e:llc2} and \eqref{e:llc} into \eqref{e:linear}, the sum of such terms with $\hat k''\geq 1$ is bounded by
	\begin{align}\begin{split}\label{e:8112}
			&\sum_{\hat k''=1}^{\hat k'}\frac{\big| \Im[f(v)/(f(v)+1)]|}{\Im[v] \dist(v,I)^{\hat k-\hat k'-1}} \left(\frac{1}{n}\right)^{\bm1(\hat k'-\hat k''\geq 1)}\Phi_{\hat k''}(v)
			\lesssim \frac{\big| \Im[f(v)/(f(v)+1)] \big|}{\Im[v] \dist(v,I)^{\hat k-\hat k'-1}} \Phi_{\hat k'}(v)\\
			&\lesssim \frac{1}{\dist(v,I)^{\hat k-\hat k'}}\left( \frac{\big| \Im[f(v)/(f(v)+1)] \big|}{n\Im[v]\dist(v,I)^{\hat k'}}+\frac{1}{\eta^q}\left(\frac{\log n}{n}\right)^{q-\hat k'}\right)	\leq  \Phi_{\hat k}(v),
	\end{split}\end{align}
	where we used  that $\Phi_{\hat k''}(v)$ is nondecreasing in $\hat k''$ in the first line; in the second line we used the last estimate in \eqref{e:ftest}. 
	
	We conclude from \eqref{e:8110}, \eqref{e:81102}, \eqref{e:8111} and \eqref{e:8112} that for $0\leq \hat k\leq q$,
	\begin{align}\begin{split}\label{e:secondt}
			&\phantom{{}={}}\Bigg|\cD_v^{(\hat k)} \bigg( \frac{f(z+1/n)}{f(z+1/n)+1} \Big( \frac{\cA^{(\bmv,\bms)}(z+1/n)e^{-b(z+1/n)}}{\cA^{(\bmv,\bms)}(z)e^{-b(z)}}\frac{\hat g(z)}{\hat g(z+1/n)}e^{\frac{1}{n}\psi(z)}-1 \Big) \bigg) \Bigg| \lesssim \Phi_{\hat k}(v).
	\end{split}\end{align}
	The claim \eqref{e:DcE0} follows from combining \eqref{e:firsttf} and \eqref{e:secondt}.
	
	Now we can use \eqref{e:secddd} and \eqref{e:secdd} to upper bound \eqref{e:tt1copy}. There are several cases,
	If $\hat k_1=0$, then $\hat k_2=k+j_2$. We may assume that $j_2\geq 1$, for otherwise the second factor in \eqref{e:tt1copy} vanishes. In this case, \eqref{e:tt1copy} becomes
	\begin{align*}
		|\eqref{e:tt1copy}|\lesssim \frac{1}{n^{j_2}}\cD_v^{(k+j_2)}\left(\left(\frac{1}{f(w)+1}\right)^{j_2}\right) \lesssim \frac{\Im[f(v)/f(v)+1]}{n^{j_2}\Im[v]\dist(v, I)^{k+j_2-1}}\lesssim \Phi_{k}(v).
	\end{align*}
	If $\hat k_2=0$, then $\hat k_1=k+j_2$. We may assume that $j_1\geq 1$ (for otherwise the first factor in \eqref{e:tt1copy} vanishes). In this case, \eqref{e:tt1copy} becomes
	\begin{align*}
		|\eqref{e:tt1copy}|\lesssim \frac{1}{n^{j_2}}\cD_v^{k+j_2}\left((\cE_0^{(\bmv,\bms)}(w))^{j_1}\right)\lesssim \frac{\Phi_{k+j_2}(v)}{n^{j_2}}\lesssim \Phi_{k}(v).
	\end{align*}
	If $\hat k_1,\hat k_2\geq 1$, then we may assume that $j_1,j_2\geq 1$, and  \eqref{e:tt1copy} becomes
	\begin{align*}
		|\eqref{e:tt1copy}|\lesssim \Phi_{\hat k_1}(v)\frac{\Im[f(v)/f(v)+1]}{n^{j_2}\Im[v]\dist(v, I)^{\hat k_2-1}} \lesssim n^{-(j_2-1)} \Phi_{\hat{k}_1} (v) \Phi_{\hat{k}_2-1} (v) \lesssim n^{-(j_2-1)} \Phi_{k + j_2-1} (v) \lesssim \Phi_{k}(v).
	\end{align*}
	It follows that the first term in \eqref{e:toest1copy} is bounded by 
	\begin{align}\label{e:fafat}
		\left|\frac{1}{2\pi \ri}\oint_{\omega_{v,\overline v}}\frac{\log \big(1+\cE^{(\bmv,\bms)}(w) \big) \rd w}{(w-v)^{k_1}(w-\overline v)^{k_2}}\right|\lesssim \Phi_k(v).
	\end{align}
	For the second term in \eqref{e:toest1copy}, we may deform $\omega+$ so that it is bounded away from $v$. Then, any $w\in\omega+$ is bounded away from $v$, so \eqref{e:defcE} gives
	\begin{align*}
		\left|\frac{\psi^{(\bmv,\bms)}(w)-1}{f(w)+1}\right|\lesssim \frac{1}{n}, \qquad \text{meaning} \qquad  \big| \cE^{(\bmv,\bms)}(w) \big| =
		\big| \cE^{(\bmv,\bms)}_0(w) \big|
		+\OO\left(\frac{1}{n}\right).
	\end{align*}
	Thus it follows
	\begin{align}\label{e:fafat1}
		\left|\frac{1}{2\pi \ri}\oint_{\omega+}\frac{\log \big( 1+\cE^{(\bmv,\bms)}(w) \big)\rd w}{(w-v)^{k_1}(w-\overline v)^{k_2}}\right|\lesssim \max_{w\in \omega+} \big| \cE^{(\bmv,\bms)}(w) \big|\lesssim \frac{1}{n}+\max_{w\in \omega+} \big| \cE_0^{(\bmv,\bms)}(w) \big|.
	\end{align}
	
	\noindent For the last term on the right side, recall the expression of $\cE_0^{(\bmv,\bms)}(w)$ from \eqref{e:newcE0}
	\begin{align}\begin{split}\label{e:fafat2}
			|\cE_0^{(\bmv,\bms)}(w)|&\lesssim \left|\frac{\cA^{(\bmv,\bms)}(w+1/n)e^{-b(w+1/n)}}{\cA^{(\bmv,\bms)}(w)e^{-b(w)}}\frac{\hat g(w)}{\hat g(w+1/n)}e^{\frac{1}{n}\psi(w)}-1\right|+\left|\frac{f(w+1/n)}{f(w+1/n)+1}-\frac{f(w)}{f(w)+1}\right|\\
			&\lesssim  \frac{|\Im[f(v)/(f(v)+1)]|}{n\Im[v]\dist(v,I)}+\frac{1}{n}\lesssim\Phi_1(v)\lesssim \Phi_k(v),
	\end{split}\end{align}
	where in the first line, the first term on the right side is bounded as in  \eqref{e:k'=0}, the second term is bounded by using \eqref{e:ftest} and the fact that $w$ is bounded away from $[\fa',\fb']$; for the second line, we used \eqref{e:imratio}. We conclude from plugging \eqref{e:fafat}, \eqref{e:fafat1} and \eqref{e:fafat2} into \eqref{e:toest1} that
	\begin{align*}
		\Big| \cD_v^{(k)} \big(\log \cA^{(\bmv,\bms)} (w) - b(w) \big) \Big|=|\eqref{e:toest1}|\lesssim \Phi_k(v).
	\end{align*}
	This finishes the proof of \Cref{c:boost2}, by taking $\fC$ larger than the above implicit constant.
\end{proof}

Now we can establish the first statement of \Cref{p:improve2}, given by equation \eqref{e:improve1c}. 

\begin{proof}[Proof of \eqref{e:improve1c}]
	We recall from \eqref{e:eqform} and \eqref{e:Phi1} that \eqref{e:improve1c} is equivalent to 
	\begin{flalign*} 
		\Big|\cD_v^{(1)} \big( \log \cA^{(\bmv,\bms)} (w) -b (w) \big) \Big| \lesssim \Phi_1(v),
	\end{flalign*} 
	
	\noindent  where $\Phi_1(v)$ is explicitly given by \eqref{e:Phi1}.
	Thanks to \Cref{c:boost2}, a weak estimate of the form,
	\begin{flalign*} 
		\Big| \cD_v^{(k)} \big( \log \cA^{(\bmv,\bms)} (w) - b(w) \big) \Big|\leq  \log n \cdot \Phi_k(v), \qquad \text{for $v\in \mathscr D$},
	\end{flalign*}
	
	\noindent for $k \in \qq{1, 2q}$, implies the improved estimate 
	\begin{flalign*} 
		\Big | \cD_v^{(k)} \big(\log \cA^{(\bmv,\bms)} (w) - b(w) )\big) \Big|\leq \fC \Phi_k(v), \quad 1\leq k\leq 2q.
	\end{flalign*} 
	
	\noindent For $v\in \mathscr D\cap \mathscr D(\fr)$ then $\dist(v,[\fa',\fb'])\geq \fr$, \Cref{p:firstod} gives that
	\begin{align*}
		&\left|\frac{1}{2\pi\ri}\oint_\omega\frac{ \big( \log \cA^{(\bmv,\bms)}(w)-b(w) \big)\rd w}{(w-z)^{k_1}(w-\overline z)^{k_2}}\right|
		=\left|\frac{1}{2\pi \ri}\oint_{\omega}\frac{\log \big( 1+\cE^{(\bmv,\bms)}(w) \big)\rd w}{(w-z)^{k_1}(w-\overline z)^{k_2}}\right|\lesssim \frac{1}{n},
	\end{align*}
	
	\noindent so that $\big| \cD_z^{(k)}(\log \cA^{(\bmv,\bms)} - b) \big| \lesssim \Phi_k(v)$ for $z=v,\overline v$ or $z\in \mathscr D(\fr)$. 
	In particular, if we take $\eta_0=\fr$, then \eqref{e:conc2} holds for $v\in \mathscr D$ with $\dist \big( v,[\fa',\fb'] \big) \geq \eta_0$. We define the sequence 
	\begin{align*}
		\eta_{i+1}=  \eta_i-n^{-3q},\quad i\geq 0.
	\end{align*}
	We show by induction on $i$ that the claim \eqref{e:improve1c} holds for $v\in \mathscr D\cap\{w:\dist(w,[\fa',\fb'])\geq \eta_i\}$, whenever $\eta_i \gg 1/n$ (meaning $i = \mathcal{O} (n^{3q})$). Thanks to \eqref{e:logAbb} and the trivial bound \eqref{e:Dzhbb}, we have for any $z,v\in \mathscr D$, and $1\leq k\leq 2q$,
	\begin{align*}
		\big| \del_v\cD_z^{(k)}(\log \cA^{(\bmv,\bms)}-b) \big|\lesssim n^{k+1}\leq n^{2q+1}.
	\end{align*}
	In particular $\cD_z^{(k)}(\log \cA^{(\bmv,\bms)}-b)$ is Lipschitz in $v$ with Lipschitz constant at most $n^{2q+1}$. Therefore if \eqref{e:conc2} holds for $v\in \mathscr D$ with $\dist \big( v,[\fa',\fb'] \big) \geq \eta_i$, then \eqref{e:hyp1} holds for $v\in \mathscr D$ with $\dist \big( v,[\fa',\fb'] \big)\geq \eta_{i+1}$. And \Cref{c:boost2} implies that \eqref{e:conc2} also holds for $v\in \mathscr D$ with $\dist \big( v,[\fa',\fb'] \big)\geq \eta_{i+1}$. In this way, by repeatedly using \Cref{c:boost2}, we conclude that $\big| \cD_z^{(k)}(\log \cA^{(\bmv,\bms)}- b) \big|\lesssim \Phi_k(v)$ holds for any $v\in \mathscr D$ and $z\in \{v,\overline v\}$ or $z\in\mathscr D(\fr)$. 
\end{proof}

\subsection{Proof of Statement \eqref{e:improve2c}}\label{s:2state}
We next proceed to prove the second statement of \Cref{p:improve2}, given by \eqref{e:improve2c}. To that end, we take  $R=4p$, $s_1=s_2=\cdots=s_{2p-1}=1$, $s_{2p}=s_{2p+1}=\cdots=s_{4p-2}=-1$,  and $s_{4p-1}=-1, v_{4p-1}=-v$.
For any $0\leq r\leq 2p-1$, the derivative $\del_z \del_{v_1}\del_{v_2}\cdots \del_{v_{r}} \log\cA^{(\bmv,\bms)}$ after specializing at $z=v, v_j=v_{j+2p-1}\in \{v,\overline v\}$, gives the $(r+1)$st joint cumulants of 
\begin{align}\label{e:vcc}
	\sum_{i = 1}^m \frac{1}{v-x_i-e_i/n}-\sum_{i = 1}^m \frac{1}{v-x_i}, \quad \sum_{i = 1}^m \frac{1}{\overline v-x_i-e_i/n}-\sum_{i = 1}^m \frac{1}{\overline v-x_i}.
\end{align}
For any subset $J\subseteq \qq{1,2p-1}$, we denote 
$\del_{V_J}\log\cA^{(\bmv, \bms)}=(\prod_{j\in J}\del_{v_j})\log\cA^{(\bmv, \bms)}$.
In order to establish \eqref{e:improve2c}, it suffices to show that for any subset\footnote{Here, $J$ does not contain repeated elements (for otherwise the associated cumulant below is equal to $0$).} $J\subseteq\qq{1,2p-1}$ with $0\leq |J|=r\leq 2p-1$ we have 
\begin{align}\label{e:eqform2}
	\left|\del_z \del_{V_J} \log\cA^{(\bmv,\bms)}(w)|_{z=v, v_j=v_{j+2p-1}\in \{v,\overline v\}}\right|\lesssim \frac{|\Im[f(v)/(f(v)+1)]|}{n^{r}\Im[v] \dist(v,I)^{2r}}.
\end{align}

\noindent Indeed, assuming \eqref{e:eqform2}, we deduce that the $(r+1)$-th joint cumulants of \eqref{e:vcc} are bounded by 
\begin{align*}
	\displaystyle\frac{\big| \Im[f(v)/(f(v)+1)] \big|}{n^{r}\Im[v] \dist(v,I)^{2r}}.
\end{align*}

\noindent Then, the left side of \eqref{e:improve2c} is a combination of the first $(2p)$-th cumulants of \eqref{e:vcc}
\begin{align}\begin{split}\label{e:eqform3}
		\phantom{{}={}}\bE_\bQ & \left|\sum_{i = 1}^m \frac{1}{z-x_i-e_i/n}-\frac{1}{z-x_i}-\bE_\bQ\left[\sum_{i = 1}^m \frac{1}{z-x_i-e_i/n}-\frac{1}{z-x_i}\right]\right|^{2p}\\
		&\lesssim \sum_{\sum (r_u+1)=2p} \prod_u \frac{\big| \Im[f(z)/(f(z)+1)] \big|}{n^{r_u}\Im[z] \dist(z,I)^{2r_u}}\\
		&\lesssim \left(\frac{|\Im[f(z)/(f(z)+1)]|}{n\Im[z] \dist(z,I)^{2}}\right)^{p}+\frac{|\Im[f(z)/(f(z)+1)]|}{n^{2p-1}\Im[z] \dist(z,I)^{4p-2}},
\end{split}\end{align}
where the sum is over positive integer sequences $(r_1, r_2, \ldots , r_k)$ such that $\sum_{u = 1}^k (r_u+1)=2p$. To get the last line, we used the fact that the largest term in the sum is either obtained when $k=p$, $r_1=r_2=\cdots=r_p=1$ or when $k=1$, $r_1=2p-1$.  

Set $R = 4p$, and then define the parameter $q>1$ as in \eqref{e:defq}. Further denote the control parameters 
by setting
\begin{align}\label{e:Phicl}
	\begin{aligned} 
		\Phi_{k, 0} (z) = \Phi_k (z), \qquad & \text{for $k \geq 0$}; \\
		\Phi_{k,\ell}(z)= \frac{\big| \Im[f(z)/(f(z)+1)] \big|}{n^{\ell}\Im[z] \dist(z,I)^{k+2\ell-1}}+\frac{1}{\dist \big( z,[\fa',\fb'] \big)^q}\left(\frac{\log n}{n}\right)^{q-k-\ell}, \qquad & \text{for $\ell \geq 1$ and $q \leq  k + \ell$}; \\
		\Phi_{k, \ell} (z) = \dist \big( z, [\fa',\fb'] \big)^{-k - \ell}, \qquad & \text{for $\ell \geq  1$ and $k + \ell > q$}.
	\end{aligned}
\end{align}

\noindent Since $q \geq 4R/\fc\geq 16 p/\fc$, for $k=1$ and $\ell\leq 2p-1$, we have
\begin{align}\label{e:Phikl}
	\Phi_{1,\ell}(z)= \frac{\big| \Im[f(z)/(f(z)+1)] \big|}{n^\ell\Im[z]\dist(z,I)^{2\ell}} +\OO\left(n^{2p}\left(\frac{\log n}{n^\fc}\right)^q\right)\asymp \frac{\big| \Im[f(z)/(f(z)+1)] \big|}{n^\ell\Im[z]\dist(z,I)^{2\ell}},
\end{align}
where we used the lower bound $\big| \Im[f(z)/(f(z)+1)] \big| /\Im[z]\gtrsim 1$ from \eqref{e:imratio}. 
Moreover,  similarly to \eqref{e:phikk'}, one can check that $\Phi_{k,\ell}(z)$ satisfies
\begin{align}\label{e:phikk'l}
	\Phi_{k,\ell}(z)\Phi_{k',\ell'}(z)\lesssim \Phi_{k+k', \ell+\ell'}(z).
\end{align}

The statement \eqref{e:eqform2} will be deduced as a consequence of the following two propositions, which state that if we have some (weak) a priori estimates on $\del_{V_J}\log\cA^{(\bmv, \bms)}(z)$, then the dynamical loop equation can be used to obtain improved estimates on $\del_{V_J}\log\cA^{(\bmv, \bms)}(z)$. 
\begin{prop}\label{c:boost25}

	Take  $s_1=s_2=\cdots=s_{2p-1}=1$, $s_{2p}=s_{2p+1}=\cdots=s_{4p-2}=-1$,  and $s_{4p-1}=-1, v_{4p-1}=-v$, for some $v \in \mathbb{C}$. The following two statements hold.
	
	\begin{enumerate} 
	\item Suppose $v\in \mathscr D(\fr)$. Assume for any $z\in \mathscr D(\fr)$ and $J \subseteq \qq{1, 2p-1}$ with $1 \leq k + |J| \leq 2q$ that
	\begin{align}\label{e:hyp1025}
		\bigg| \cD_z^{(k)} \Big( \del_{V_J} \big( \log \cA^{(\bmv,\bms)} (w) -b (w) \big) \Big) \Big|_{ v_j=v_{j+2p-1}\in \{v,\overline v\}} \bigg| \lesssim \frac{ \log n }{n^{1\vee |J|}}.
	\end{align}
	
	\noindent Then we have the improved estimate
	\begin{align}\label{e:conc2l25}
		\bigg| \cD_z^{(k)} \Big( \del_{V_J} \big(\log \cA^{(\bmv,\bms)} (w) - b (w) \big) \Big) \Big|_{ v_j=v_{j+2p-1} \in \{v,\overline v\}} \bigg| \leq\frac{\fC}{n^{1\vee |J|}},
	\end{align}
	
	\noindent for some constant $\mathfrak{C} = \mathfrak{C} (R) > 1$.
	
	\item Suppose $v\in \mathscr D$. Assume for any $z=v,\overline v$ or $z\in \mathscr D(\fr)$ $J\subseteq \qq{1, 2p-1}$ with $1 \leq k + |J| \leq 2q$ that
	\begin{align}\label{e:hyp10}
		\bigg| \cD_z^{(k)} \Big( \del_{V_J} \big(\log \cA^{(\bmv,\bms)} (w) - b(w) \big) \Big) \Big|_{ v_j=v_{j+2p-1}\in \{v,\overline v\}} \bigg| \leq  \log n \cdot \Phi_{k, |J|}(v).
	\end{align}
	
	\noindent Then we have the improved estimate
	\begin{align}\label{e:conc2l}
		\bigg| \cD_z^{(k)} \Big( \del_{V_J}(\log \cA^{(\bmv,\bms)} (w) -  b(w) \big) \Big) \Big|_{ v_j=v_{j+2p-1}\in \{v,\overline v\}} \bigg| \leq\fC\Phi_{k, |J|}(v),
	\end{align}
	
	\noindent for some constant $\fC = \mathfrak{C} (R) > 1$.
	\end{enumerate} 
\end{prop}

Let us mention that the first statement of \Cref{c:boost25} is nearly a special case of the second statement, except that in the former we allow $\dist (v, I)$ to grow with $n$. As such, these two parts of the proposition are very closely related; the only reason why we separated their statements is that the definition of $\Phi_{k, |J|} (v)$ involves the function $f(v)$, whose domain is $\Lambda$ and is thus not defined for general $v \in \mathscr{D} (\mathfrak{r})$.

\begin{proof}[Proof of \Cref{c:boost25}]
	
	Here, we only establish the second statement of \Cref{c:boost25}. The proof of the first statement is entirely analogous, by replacing every appearance of $\Phi_{k, |J|} (v)$ below with the quantity $1/n^{1 \vee |J|}$. 
	
	In the following we prove \eqref{e:conc2l} for $z=v$; the proofs in the other cases that $z=\overline v$ or $z\in \mathscr D(\fr)$ follow from essentially the same argument. To simplify notation, we take $|J|=r\leq 2q-1$ and $J=\{1,2,\cdots,r\}$; we also denote $\eta=\dist \big( v, [\fa',\fb'] \big)$.
	For $q\leq k+r\leq 2q$, using \eqref{e:logAbb}, the trivial bound \eqref{e:Dzhbb} implies 
	\begin{align}\label{e:hyp30}
		\left|\cD^{(k)}_z\del_{v_1}\del_{v_2}\cdots\del_{v_r}\big( \log \cA^{(\bmv,\bms)} (w) - b(w) \big) \big|_{z=v, v_i=v_{j+2p-1}\in \{v,\overline v\}}\right|\lesssim \frac{1}{\eta^{k+r}}=\Phi_{k,r}(v).
	\end{align}
	
	It therefore suffices to consider the case $1 \leq k+r\leq q$, in which case the proof will be similar to that of \Cref{c:boost2}. Specifically, by taking derivative with respect to $v_1, v_2, \cdots, v_r$ on both sides of \eqref{e:toest}, then specializing $z=v$ and $v_i\in\{v, \overline v\}$, we get
	\begin{align}\label{e:toest3}
		\frac{1}{2\pi\ri}\oint_\omega\frac{\del_{v_1}\del_{v_2}\cdots\del_{v_r} \big( \log \cA^{(\bmv,\bms)}(w)-b(w) \big)}{(w-v)^{k_1}(w-\overline v)^{k_2}}
		=-\frac{1}{2\pi \ri}\oint_{\omega}\frac{\del_{v_1}\del_{v_2}\cdots\del_{v_r}\log \big(1+\cE^{(\bmv,\bms)}(w) \big) \rd w}{(w-v)^{k_1}(w-\overline v)^{k_2}}.
	\end{align}
	
	\noindent By expanding $\log \big(1+\cE^{(\bmv, \bms)}(w) \big)$ as in \eqref{e:ttsa}, we must estimate
	\begin{align}\begin{split}\label{e:toestterm}
			\frac{1}{2\pi\ri}\oint_\omega & \del_{v_1}\del_{v_2}\cdots\del_{v_r}\left( \big( \cE_0^{(\bmv,\bms)}(w) \big)^{j_1}\left(\frac{\psi^{(\bmv, \bms)}(w)-1}{f(w)+1}\right)^{j_2}\frac{1}{(w-v)^{k_1}(w-\overline v)^{k_2}}\right)\rd w\\
			&=
			\frac{1}{2\pi \ri}\oint_{\omega}\sum \prod_{u=1}^{j_1} \del_{V_{J_u}} \big( \cE_0^{(\bmv,\bms)}(w) \big)\prod_{u=1}^{j_2}\del_{V_{J'_u}} \big( \psi^{(\bmv, \bms)}(w)-1 \big) \frac{1}{\big( f(w)+1 \big)^{j_2}}\frac{\rd w}{(w-v)^{k_1}(w-\overline v)^{k_2}},
	\end{split}\end{align}
	where the sum is over all index sets $J_1, J_2, \ldots , J_{j_1}, J_1', J_2', \ldots , J_{j_2}'$ such that $\bigcup_{u=1}^{j_1} J_u \cup \bigcup_{u=1}^{j_2} J'_u=\qq{1,r}$, and $j_1+j_2\leq q$. 
	Thanks to the trivial bound \eqref{e:Dzhbb}, the contribution from the error term in \eqref{e:ttsa} is at most $\big| \cE^{(\bmv,\bms)}(w) \big|^{q+1} \eta^{-k-r}=\OO\big( (n\eta)^{-q-1}\eta^{-k-r} \big)$, which is smaller than $\Phi_{k,r}(v)$ for $\eta \geq n^{\mathfrak{c} - 1}$ and $q \geq 4R / \mathfrak{c}$. 
	 We recall the expression of $\psi^{(\bmv, \bms)}(w)$ from \eqref{e:defpsi},
	\begin{align*}
		\psi^{(\bmv,\bms)}(w)-1=  \displaystyle\prod_{j = 1}^{4p - 1} \bigg( 1 - \displaystyle\frac{1}{n (v_j - w)} \bigg)^{-s_j} - 1.
	\end{align*}
	
	\noindent If $|w-v_j|\gtrsim \eta$, then $\big| 1/n(v_j-w) \big|\lesssim 1/n\eta\ll1$, so $\psi^{(\bmv,\bms)}(w)-1$ can be Taylor expanded as a sum of terms in the form
	\begin{align}\label{e:jinJ}
		\prod_{j\in L}\frac{1}{n(w-v_j)}, 
	\end{align}
	where $L$ is a multi-set of $\qq{1,4p-1}$, with $1\leq |L|\leq q$, plus an error $\OO((n\eta)^{-(q+1)})$. Again, thanks to the trivial bound \eqref{e:Dzhbb}, the contribution from such error term  in \eqref{e:toestterm} is $\OO((n\eta)^{-(q+1)}\eta^{-(k+r)})$, which is smaller than $\Phi_{k,r}(v)$.  The derivative of \eqref{e:jinJ} with respect to $V_{J'_u}$ is nonzero only if $J'_u\subseteq L$. If this is the case, then
	\begin{align}\label{e:jinJ2}
		\del_{V_{J'_u}}\prod_{j\in L}\frac{1}{n(w-v_j)}=\prod_{j\in J'_u} \frac{\text{multiplicity of $j$ in $L$}}{n(w-v_j)^2}\prod_{j\in L\setminus J'_u}\frac{1}{n(w-v_j)},
	\end{align}
	where $L\setminus J'_u$ is the multi-set obtained from $L$ by removing a copy of each element in $J_u'$.
	
	After specializing $v_j \in \{v,\overline v\}$, the right side of \eqref{e:jinJ2} is in the form
	\begin{align*}
		\frac{1}{n^{\ell_{1u}+\ell_{2u}-|J_u'|}(w-v)^{\ell_{1u}}(w-\overline v)^{\ell_{2u}}}, \quad  2 |J'_u|\leq \ell_{1u}+\ell_{2u}\leq q+|J'_u|.
	\end{align*}	
	Therefore after specializing $v_j \in \{v,\overline v\}$, the leading order term of
	$\prod_{u=1}^{j_2}\del_{V_{J'_u}} \big( \psi^{(\bmv, \bms)}(w)-1 \big)$ is a linear combination (with bounded coefficients) of terms of the form
	\begin{align}\label{e:afterexp}
		\frac{1}{n^{\ell_1+\ell_2-\sum_u |J'_u|}(w-v)^{\ell_1}(w-\overline v)^{\ell_2}}, \quad  2\sum_u |J'_u|\leq \ell_1+\ell_2\leq j_2q+\sum_{u=1}^{j_2} |J_u'|,
	\end{align}
	where $\ell_1=\sum_u \ell_{1u}$ and $\ell_2=\sum_u \ell_{2u}$.
	By plugging \eqref{e:afterexp} into \eqref{e:toestterm}, we get
	\begin{align}\begin{split}\label{e:twoterm}
			\frac{1}{2\pi \ri} & \displaystyle\frac{1}{n^{\ell_1+\ell_2-\sum_u |J'_u|}} \oint_{\omega}\prod_{u=1}^{j_1} \del_{V_{J_u}} \big( \cE_0^{(\bmv,\bms)}(w) \big) \frac{1}{\big( f(w)+1 \big)^{j_2}}\frac{\rd w}{(w-v)^{k_1+\ell_1}(w-\overline v)^{k_2+\ell_2}}\\
			&=
			-\frac{1}{2\pi \ri} \displaystyle\frac{1}{n^{\ell_1+\ell_2-\sum_u |J'_u|}}\oint_{\omega_{v,\overline v}}\prod_{u=1}^{j_1} \del_{V_{J_u}} \big( \cE_0^{(\bmv,\bms)}(w) \big)\frac{1}{\big( f(w)+1 \big)^{j_2}}\frac{\rd w}{(w-v)^{k_1+\ell_1}(w-\overline v)^{k_2+\ell_2}}\\
			& \qquad +\frac{1}{2\pi \ri} \displaystyle\frac{1}{n^{\ell_1+\ell_2-\sum_u |J'_u|}}\oint_{\omega+}\prod_{u=1}^{j_1} \del_{V_{J_u}} \big(\cE_0^{(\bmv,\bms)}(w) \big) \frac{1}{\big( f(w)+1 \big)^{j_2}}\frac{\rd w}{(w-v)^{k_1+\ell_1}(w-\overline v)^{k_2+\ell_2}},
	\end{split}\end{align}
	
	\noindent where the contour $\omega_{v,\overline v}$ encloses $v,\overline v$, and the contour $\omega+\subseteq\mathscr D(2\fr)$ encloses $[\fa',\fb']$ and $v,\overline v$. See Figure \ref{f:contour2}.
	Thanks to \eqref{e:lincomb}, the first term on the right side of \eqref{e:twoterm} decomposes to sum of terms of the form
	\begin{align}\label{e:tt1copy2}
		\frac{1}{n^{\ell_1+\ell_2-\sum_u |J'_u|}}\cD_v^{( \hat k_1)}\left(\prod_{u=1}^{j_1}\del_{V_{J_u}}(\cE_0^{(\bmv,\bms)}(w))\right)\cD_v^{(\hat k_2)}\left(\frac{1}{f(w)+1}\right)^{j_2}, 
	\end{align}
	where $\hat k_1+\hat k_2=k_1+k_2+\ell_1+\ell_2-1=k+\ell_1+\ell_2$.

	In the following we show that under our hypothesis \eqref{e:hyp10}, for $0\leq \hat  k+|\hat J|\leq 2q$,
	\begin{align}\label{e:DcE0l}
		\left|\cD_v^{(\hat k)}(\del_{V_{\hat J}}\cE_0^{(\bmv, \bms)})\right|
		&\lesssim \Phi_{\hat k,|\hat J|}(v),\quad \hat J\subseteq\qq{1,r}.
	\end{align}
	Together with  \eqref{e:lincomb} and \eqref{e:phikk'l}, \eqref{e:DcE0l} would imply for $j_1\geq 1$ that
	\begin{align}\label{e:secddl}
		\cD_v^{( \hat k_1)}\left(\prod_{u=1}^{j_1}\del_{V_{J_u}} \big(\cE_0^{(\bmv,\bms)}(w) \big) \right) \lesssim\sum \prod_{u=1}^{j_1} \left|\cD_v^{(\tilde k_u)}\left(\del_{V_{J_u}}\cE_0^{(\bmv,\bms)}\right)\right|\lesssim \prod_{u=1}^{j_1} \Phi_{\tilde k_u,|J_u|}(v)\lesssim  \Phi_{\hat k,\sum_u |J_u|}(v),
	\end{align}
	where the sum is over index sets $(\tilde{k}_1, \tilde{k}_2, \ldots , \tilde{k}_{j_1})$ such that $\sum_{u=1}^{j_1}\tilde k_u=\hat k$.
	
	If $\hat J=\emptyset$, \eqref{e:DcE0l} is \eqref{e:DcE0}. In the following we assume that $|\hat J|\geq 1$. For $q\leq \hat k+|\hat J|\leq 2q$, \eqref{e:DcE0l} follows from the trivial bound
	$
	\left|\cD_v^{(\hat k)}(\del_{V_{\hat J}}\cE_0^{(\bmv, \bms)})\right|
	\lesssim 1/\eta^{\hat k+|\hat J|}=\Phi_{\hat k, |\hat J|}
	$.
	For $\hat k+|\hat J|\leq q$, we recall the expression of $\cE_0^{(\bmv, \bms)}$ from \eqref{e:newcE0}, and we need to estimate 
	\begin{align*}
		\cD_v^{(\hat k)}\left(\frac{f(w+1/n)}{f(w+1/n)+1}\frac{\hat g(w)}{\hat g(w+1/n)}e^{\frac{1}{n}\psi(w)} \del_{V_{\hat J}}\left(\frac{\cA^{(\bmv,\bms)}(w+1/n)e^{-b(w+1/n)}}{\cA^{(\bmv,\bms)}(w)e^{-b(w)}}\right)\right),
	\end{align*}
	which, thanks to \eqref{e:lincomb}, is a linear combination of
	\begin{align}\label{e:linearl}
		\cD_v^{(\hat k-\hat k')}\left(\frac{f(w+1/n)}{f(w+1/n)+1}\frac{\hat g(w)}{\hat g(w+1/n)}e^{\frac{1}{n}\psi(w)}\right) \cD_v^{(\hat k')}\del_{V_{\hat J}}\left(\frac{\cA^{(\bmv,\bms)}(w+1/n)e^{-b(w+1/n)}}{\cA^{(\bmv,\bms)}(w)e^{-b(w)}}\right).
	\end{align}
	
	\noindent By essentially the same argument as for \eqref{e:DvklogA3}, we have
	\begin{align}\label{e:DvklogA3l}
		\cD_v^{(\hat k')}\del_{V_{\hat J}}\left(\log \left(\frac{\cA^{(\bmv,\bms)}(w+1/n)e^{-b(w+1/n)}}{\cA^{(\bmv,\bms)}(w)e^{-b(w)}}\right)\right)^j\lesssim \Phi_{\hat k', |\hat J|}(v),\quad \hat k'\geq 0.
	\end{align}
	Using \eqref{e:DvklogA3l} as input, we get the following bound for the second term on the right side of \eqref{e:linearl}, 
	\begin{align}\begin{split}\label{e:linearl2}
			\phantom{{}={}}& \left|\cD_v^{(\hat k')}\del_{V_{\hat J}}\left(\frac{\cA^{(\bmv,\bms)}(w+1/n)e^{-b(w+1/n)}}{\cA^{(\bmv,\bms)}(w)e^{-b(w)}}\right)\right|\\
			& \qquad =\left|\cD_v^{(\hat k')}\left(\frac{1}{j!}\sum_{j=1}^{q-\hat k'} \del_{V_{\hat J}} \left(\log\left(\frac{\cA^{(\bmv,\bms)}(w+1/n)e^{-b(w+1/n)}}{\cA^{(\bmv,\bms)}(w)e^{-b(w)}}\right)\right)^j+\OO\left(\frac{1}{(n\eta)^{q-\hat k'+1}\eta^{|\hat J|}}\right)\right)\right|
			\lesssim \Phi_{\hat k', |\hat J|}(v),
	\end{split}\end{align}
	for any $\hat k'\geq 0$.
	
	By plugging \eqref{e:linearl2} into \eqref{e:linearl}, if $\hat k'=\hat k$, we have
	$
	|\eqref{e:linearl}|\lesssim \Phi_{\hat k, |\hat J|}(v)
	$.
	For $k'\leq \hat k-1$, using \eqref{e:dhgbound} and \eqref{e:ftest},  we can bound the first term in \eqref{e:linear} as
	\begin{align}\label{e:fttfirst}
		\left|\cD_v^{(\hat k-\hat k')}\left(\frac{f(w+1/n)}{f(w+1/n)+1}\frac{\hat g(w)}{\hat g(w+1/n)}e^{\frac{1}{n}\psi(w)}\right)\right|\lesssim \frac{\big| \Im[f(v)/(f(v)+1)] \big|}{\Im[v] \dist(v,I)^{\hat k-\hat k'-1}}.
	\end{align}
	By plugging \eqref{e:linearl2} and \eqref{e:fttfirst} into  \eqref{e:linearl},  the sum of terms with $\hat k'\leq \hat k-1$ is bounded by
	\begin{align*}
		\sum_{\hat k'=0}^{\hat k-1}\frac{\big| \Im[f(v)/ (f(v)+1)] \big|}{\Im[v] \dist(v,I)^{\hat k-\hat k'-1}}\Phi_{\hat k', |\hat J|}(v)\lesssim \Phi_{\hat k, |\hat J|}(v),
	\end{align*}
	
	\noindent where we have used the last inequality in \eqref{e:ftest}. This finishes the proof of \eqref{e:DcE0l}.

	Now we can use \eqref{e:secddd} and \eqref{e:secddl} to upper bound \eqref{e:tt1copy2}.  There are two cases, namely, either $\hat k_2=0$ or $\hat k_2\geq 1$:
	\begin{align}\begin{split}\label{e:inomega}
			|\eqref{e:tt1copy2}|&\lesssim \frac{\Phi_{\hat k_1, \sum_{u}|J_u|}(v)}{n^{\ell_1+\ell_2-\sum_u |J_u'|}} \left(\bm1(\hat k_2= 0)+\frac{\big| \Im[f(v)/(f(v)+1)] \big|\bm1(\hat k_2\geq 1)}{\Im[v]\dist(v, I)^{\hat k_2-1}}\right)\\
			&\lesssim \frac{\Phi_{\hat k_1+\hat k_2, \sum_{u}|J_u|}(v)}{n^{\ell_1+\ell_2-\sum_u |J_u'|}}
			=\frac{\Phi_{k+\ell_1+\ell_2, \sum_{u}|J_u|}(v)}{n^{\ell_1+\ell_2-\sum_u |J_u'|}}
			\leq \frac{\Phi_{k+2\sum_u |J'_u|, \sum_{u}|J_u|}(v)}{n^{\sum_u |J_u'|}}
			\lesssim \Phi_{k, r}(v),
	\end{split}\end{align}
	where for the second to last inequality we used that $\ell_1+\ell_2\geq 2\sum_u |J'_u|$ from \eqref{e:afterexp}, and $\Phi_{k,\ell}(v)/n^k$ is deceasing in $k$: for the last inequality, we used $\sum_u |J_u'|+\sum_u |J_u|=r$. For the second term in \eqref{e:twoterm}, since $w\in\omega+$ is bounded away from $v$, we have 
	\begin{align}\begin{split}\label{e:outomega}
			&\left|\frac{1}{2\pi \ri} \displaystyle\frac{1}{n^{\ell_1+\ell_2-\sum_u |J'_u|}}\oint_{\omega+}\prod_{u=1}^{j_1} \del_{V_{J_u}} \big( \cE_0^{(\bmv,\bms)}(w) \big)\frac{1}{\big( f(w)+1 \big)^{j_2}}\frac{\rd w}{(w-v)^{k_1+\ell_1}(w-\overline v)^{k_2+\ell_2}}\right|\\
			& \qquad\lesssim \max_{w\in \omega+} \bigg| \frac{\prod_{u=1}^{j_1}   \del_{V_{J_u}} \big(\cE_0^{(\bmv,\bms)}(w) \big) }{ n^{\ell_1+\ell_2-\sum_u |J'_u|}} \bigg|\lesssim \frac{\Phi_{0, \sum_u|J_u|}(v)}{ n^{\ell_1+\ell_2-\sum_u |J'_u|}}
			\lesssim \frac{\Phi_{0, \sum_u|J_u|}(v)}{ n^{\sum_u |J'_u|}}\lesssim \Phi_{k,r}(v),
	\end{split}\end{align}
	where we used \eqref{e:DcE0l} and $\ell_1+\ell_2\geq 2\sum_u |J'_u|$ from \eqref{e:afterexp} in the second line.
	We conclude from plugging \eqref{e:inomega} and \eqref{e:outomega}  into \eqref{e:twoterm} 
	\begin{align*}
		\Big| \cD_z^{(k)}\del_{v_1}\del_{v_2}\cdots \del_{v_r} \big( \log \cA^{(\bmv,\bms)} - b (w) \big) \big|_{z=v, v_j=v_{j+2p-1}\in \{v,\bar v\}} \Big|=|\eqref{e:toest3}|\lesssim \Phi_{k,r}(v).
	\end{align*}
	This finishes the proof of \Cref{c:boost2}.
\end{proof}

Now we can establish the second statement of \Cref{p:improve2}, given by \eqref{e:improve2c}.

\begin{proof}[Proof of \eqref{e:improve2c}]
	We recall from \eqref{e:eqform2} and \eqref{e:eqform3} that \eqref{e:improve2c} follows from the estimates of the derivatives of $\log \cA^{(\bmv,\bms)} - b$:
	$\big| \cD_z^{(1)}\del_{V_J}\log (\cA^{(\bmv,\bms)} - b (w) ) |_{z=v, v_j=v_{j+2p-1}\in \{v,\bar v\}} \big| \lesssim \Phi_{1,{|J|}}(v)$, for any $J\subseteq \qq{1,2p-1}$, where $\Phi_{1,|J|}(v)$ is explicitly given in \eqref{e:Phikl}.
	Thanks to Propositions \ref{c:boost25}, if we have some weak estimates for $\big|\cD_z^{(k)}\del_{V_J}(\log \cA^{(\bmv,\bms)} (w) - b (w) )|_{v_j=v_{j+2p-1}\in\{v,\overline v\}} \big|$, then it implies better estimates.
	Next we show that \eqref{e:hyp1025} holds for $\dist \big(\{z, v\},[\fa',\fb'] \big)\asymp n$. In fact, for $J=\emptyset$, \eqref{e:hyp1025} follows from \Cref{p:firstod}. For $|J|\geq 1$, thanks to \eqref{e:logAbb}, and the trivial bound \eqref{e:Dzhbb} (with $\eta\asymp n$), we have 
	\begin{flalign*} 
		\Big| \cD_z^{(k)}\del_{V_J} \big(\log \cA^{(\bmv,\bms)} (w) - b (w) \big) \big|_{v_j=v_{j+2p-1}\in \{v,\bar v\}} \Big|
	\lesssim \eta^{-k - |J|}\lesssim \displaystyle\frac{1}{n^{1\vee |J|}}.
	\end{flalign*} 
	
	\noindent Thus, \eqref{e:hyp1025} holds for $\dist \big( \{z,v\},[\fa',\fb'] \big) \asymp n$.

	Then a continuity argument quickly implies that 
	\begin{flalign*} 
		\Big| \cD_z^{(k)}\del_{V_J} \big( \log \cA^{(\bmv,\bms)} (w) - b(w) \big) \big|_{ v_j=v_{j+2p-1}\in\{v,\overline v\}} \Big| \lesssim \displaystyle\frac{1}{n^{1\vee |J|}}, \quad \text{holds for any $z,v\in \mathscr D(\fr)$};
	\end{flalign*} 
	
	\noindent we omit further details, since they are very similar to the one implemented in the proof of  \eqref{e:improve1c} (by repeatedly using the first statement of \Cref{c:boost25}). In particular, for any $v\in \mathscr D(\fr)$, the assumption \eqref{e:hyp10} in \Cref{c:boost25} holds. Then, by a continuity argument that is again very similar to the one used in the proof of \eqref{e:improve1c} (by repeated use of the second statement of \Cref{c:boost25}), we conclude that 
	\begin{flalign*} 
		\Big| \cD_z^{(k)}\del_{V_J} \big(\log \cA^{(\bmv,\bms)} (w) - b(w) \big) \big|_{v_j=v_{j+2p-1}\in\{v,\overline v\}} \Big|\lesssim \Phi_{k, |J|}(v), \quad \text{holds for any $v\in \mathscr D$ and $z\in \{v,\bar v\}\cup \mathscr D(\fr)$}. 
	\end{flalign*} 
	
	\noindent This finishes the proof of \eqref{e:improve2c}.
\end{proof}

\subsection{Proof of \Cref{p:improve}}\label{s:finalproof}
Finally, in Section \ref{s:finalproof}, we prove \Cref{p:improve} using \Cref{p:improve2} as input.
\begin{proof}[Proof of \Cref{p:improve}]
	We recall the relation between $\bP$ and $\bQ$ from \eqref{e:changem}, and $b(z)$ from \eqref{e:defbz}. Then 
	\begin{flalign*}
		\del_z b(z) = \displaystyle\frac{1}{2 \pi \mathrm{i}} \displaystyle\oint_{\omega} \displaystyle\frac{\log \mathcal{B} (w) dw}{(w - z)^2},
	\end{flalign*}
	
	\noindent where the contour $\omega \subseteq \Lambda$ contains $[\mathfrak{a}, \mathfrak{b}]$, but not $z$. The transition probability \eqref{e:defLtnew} is a special case of $\bP$, and so \Cref{p:improve} follows if we show
	\begin{align}\begin{split}\label{e:improve1P}
			&\phantom{{}={}}\bE_\bP\left[\sum_{i = 1}^m \frac{1}{z-x_i-e_i/n}-\frac{1}{z-x_i}\right]=\del_zb(z)+\OO\left(\frac{|\Im[f(z)/(f(z)+1)]|}{n\Im[z] \dist(z,I)}\right).
	\end{split}\end{align}
	and 
	\begin{align}\begin{split}\label{e:improve2P}
			&\phantom{{}={}}\bE_\bP\left|\sum_{i = 1}^m \frac{1}{z-x_i-e_i/n}-\frac{1}{z-x_i}-\bE_\bP\left[\sum_{i = 1}^m \frac{1}{z-x_i-e_i/n}-\frac{1}{z-x_i}\right]\right|^{2p}\\
			& \qquad \qquad \qquad \qquad \qquad \lesssim \left(\frac{|\Im[f(z)/(f(z)+1)]|}{n\Im[z] \dist(z,I)^{2}}\right)^{p}+\frac{|\Im[f(z)/(f(z)+1)]|}{n^{2p-1}\Im[z] \dist(z,I)^{4p-2}}.
	\end{split}\end{align}
	
	We notice the trivial bound for the exponent involved in $\mathbb{P}$, given by
	\begin{align}\label{e:kbb}
		\left|\sum_{1\leq i,j\leq m}\frac{e_ie_j}{n^2}\kappa(x_i,x_j)+\OO(1/n)\right|\lesssim 1.
	\end{align} 
	Then, by \eqref{e:improve1c} and \eqref{e:changem}, \eqref{e:improve1P} is equivalent to 
	\begin{align}\begin{split}\label{e:smallg}
			\phantom{{}={}}\bE_\bQ & \Bigg[ \bigg(\sum_{i = 1}^m \frac{1}{z-x_i-e_i/n}-\frac{1}{z-x_i}-\del_z b(z) \bigg) \\ 
			&  \quad \times \bigg(e^{\sum_{i,j}\frac{e_ie_j}{n^2}\kappa(x_i,x_j)+\OO(1/n)}-\bE_\bQ\left[e^{\sum_{i,j}\frac{e_ie_j}{n^2}\kappa(x_i,x_j) + \mathcal{O} (1/n)}\right] \bigg) \Bigg] = \OO\left(\frac{|\Im[f(z)/(f(z)+1)]|}{n\Im[z] \dist(z,I)}\right).
	\end{split}\end{align}	
	
	\noindent We will prove \eqref{e:smallg} by the Cauchy--Schwarz inequality. In the following, we estimate the variances of the first and second term in \eqref{e:smallg}.
	
	We notice that $\big| \Im[f(z)/(f(z)+1)]/(n\Imaginary z) \big| \ll1$ for $z \in \mathscr{D}$. We can replace the inner expectation in \eqref{e:improve2c} by $\del_z b(z)$, and the error is negligible:
	\begin{align}\begin{split}\label{e:improve2w}
			&\phantom{{}={}}\bE_\bQ\left|\sum_{i = 1}^m \frac{1}{z-x_i-e_i/n}-\frac{1}{z-x_i}-\del_z b(z)\right|^{2p} \\
			& \qquad \qquad \lesssim \left|\bE_\bQ\left[\sum_{i = 1}^m \frac{1}{z-x_i-e_i/n}-\frac{1}{z-x_i}\right]-\del_z  b(z)\right|^{2p} \\
			& \qquad \qquad \qquad \qquad  + 
			\bE_\bQ\left|\sum_{i = 1}^m \frac{1}{z-x_i-e_i/n}-\frac{1}{z-x_i}-\bE_\bQ\left[\sum_{i = 1}^m \frac{1}{z-x_i-e_i/n}-\frac{1}{z-x_i}\right]\right|^{2p}
			\\
			& \qquad \qquad \lesssim\left(\frac{|\Im[f(z)/(f(z)+1)]|}{n\Im[z] \dist(z,I)^{2}}\right)^{p}+\frac{|\Im[f(z)/(f(z)+1)]|}{n^{2p-1}\Im[z] \dist(z,I)^{4p-2}}.
	\end{split}\end{align}
	In particular, by taking $p=1$, we have
	\begin{align}\begin{split}\label{e:improvep=1}
			\bE_\bQ\left|\sum_{i = 1}^m \frac{1}{z-x_i-e_i/n}-\frac{1}{z-x_i}-\del_z  b(z)\right|^{2}
			\lesssim\frac{|\Im[f(z)/(f(z)+1)]|}{n\Im[z] \dist(z,I)^{2}}.
	\end{split}\end{align}
	
	For the second term in \eqref{e:smallg}, we recall that $\kappa(z,w)$ is analytic in a neighborhood of $[\fa,\fb]$. The Laplace transform of $\sum_{ij}e_ie_j \kappa(x_i,x_j)/n^2$ can be computed as an integral of its expectation under deformation of $\bQ$ as
	\begin{align*}
		\bE_{\mathbb{Q}} \left[ \exp \Bigg( t \sum_{1 \le i, j \le m}\frac{e_ie_j}{n^2} \kappa(x_i,x_j) \Bigg) \right]
		= 1 + \int_0^t \bE_{\mathbb{Q}} \left[\sum_{1 \leq i, j \leq n}\frac{e_ie_j}{n^2} \kappa(x_i,x_j) \exp \Bigg( \theta \sum_{1 \le i,j \le m}\frac{e_ie_j}{n^2} \kappa(x_i,x_j) \Bigg) \right]\rd \theta.
	\end{align*} 
	This can be analyzed the same way as in \cite[Section 8.2]{NRWLT}, using the loop equation of the deformed measure, and it gives
	\begin{align}\label{e:kvar}
		\bE_\bQ\left| \exp \bigg( \sum_{1 \le i, j \le m} \frac{e_ie_j}{n^2}\kappa(x_i,x_j) \bigg) -\bE_\bQ \Bigg[ \exp \bigg( \sum_{1 \le i, j \le m}\frac{e_ie_j}{n^2}\kappa(x_i,x_j) \bigg) \Bigg] \right|^2\lesssim \frac{1}{n}.
	\end{align}
	Then using \eqref{e:kbb}, we conclude from \eqref{e:kvar} that
	\begin{align}\label{e:kvar2}
		\bE_\bQ\left| \exp \bigg( \sum_{1 \le i, j \le m} \frac{e_ie_j}{n^2}\kappa(x_i,x_j)+\OO(1/n) \bigg) -\bE_\bQ \Bigg[ \exp \bigg( \sum_{ij}\frac{e_ie_j}{n^2}\kappa(x_i,x_j) \bigg) \Bigg] \right|^2\lesssim \frac{1}{n}.
	\end{align}

	\noindent Using \eqref{e:improvep=1} and \eqref{e:kvar2}, \eqref{e:smallg} follows from the Cauchy-Schwarz inequality,
	\begin{align*}
		|\eqref{e:smallg}|\lesssim \sqrt{\frac{|\Im[f(z)/(f(z)+1)]|}{n\Im[z] \dist(z,I)^{2}}}\frac{1}{\sqrt n}=\OO\left(\frac{|\Im[f(z)/(f(z)+1)]|}{n\Im[z] \dist(z,I)}\right).
	\end{align*}
	This finishes the proof of \eqref{e:improve1P}.
	
	To prove \eqref{e:improve2P}, similarly to in \eqref{e:improve2w}, using \eqref{e:improve1P} as input, we can replace the inner expectation in \eqref{e:improve2P} by $\del_z b(z)$. In this way, \eqref{e:improve2P} is equivalent to
	\begin{align}\label{e:improve2ww}
		\bE_{ \bP}\left|\sum_{i = 1}^m \frac{1}{z-x_i-e_i/n}-\frac{1}{z-x_i}-\del_z b(z)\right|^{2p}
		\lesssim \left(\frac{|\Im[f(z)/(f(z)+1)]|}{n\Im[z] \dist(z,I)^{2}}\right)^{p}+\frac{|\Im[f(z)/(f(z)+1)]|}{n^{2p-1}\Im[z] \dist(z,I)^{4p-2}}.
	\end{align}	
	
	\noindent Thanks to the relation \eqref{e:changem} between $\bP$ and $\bQ$, and the trivial bound \eqref{e:kbb}, we conclude \eqref{e:improve2ww} from \eqref{e:improve2w} and \eqref{e:kvar2}.
\end{proof}

\appendix

\section{Thin Slice of Polygonal Domains}

\label{s:thin}

In this section, we check that thin slices of polygonal domains with their (tilted) limiting continuum height functions satisfy Assumption \ref{xhh}.
Throughout this section, we recall the notation from \Cref{HeightLimit}. 
We will focus on polygonal domains $\mathfrak{P} \subset \mathbb{R}^2$, given as follows. 
 	
 	\begin{definition} 
 		
 	\label{p} 
 	
 	A subset $\mathfrak{P} \subset \mathbb{R}^2$ is \emph{polygonal} if it is simply-connected and its boundary $\partial \mathfrak{P}$ consists of a finite union of line segments, each of which is parallel to an axis of $\mathbb{T}$. The domain\footnote{We assume throughout that all vertices of $n \mathfrak{P}$ are in $\mathbb{Z}^2$.} $\mathsf{P} = \mathsf{P}_n = n \mathfrak{P} \subset \mathbb{T}$, is then tileable and is therefore associated with a (unique, up to global shift) boundary height function $\mathsf{h} = \mathsf{h}_n$. By translating $\mathfrak{P}$ if necessary, we will assume that $\mathfrak{P} \subset \mathbb{R} \times \mathbb{R}_{\geq 0}$ and that $(0, 0) \in \partial \mathfrak{P}$. Then by shifting $\mathsf{h}$ if necessary, we will further suppose that $\mathsf{h} (0, 0) = 0$. Under this notation, we set $h: \partial \mathfrak{P} \rightarrow \mathbb{R}$ by $h (u) = n^{-1} \mathsf{h} (nu)$ for each $u \in \partial \mathfrak{P}$. Moreover, we abbreviate $\Adm (\mathfrak{P}) = \Adm (\mathfrak{P}; h)$; $\mathfrak{L} (\mathfrak{P}) = \mathfrak{L} (\mathfrak{P}; h)$; and $\mathfrak{A} (\mathfrak{P}) = \mathfrak{A} (\mathfrak{P}; h)$, they do not depend on the above choice of global shift fixing $h$.
 	
 	\end{definition} 	 
	 
  	The following result from \cite{LSCE,DMCS} describes properties of the limit shape $H^*$ and complex slope $f_t (x)$ for polygonal domains $\fP$ as in \Cref{p}.
  	 	
	 \begin{prop}[{\cite{LSCE,DMCS}}]
	 
	 \label{pa1}
	 
	 Adopt the notation of \Cref{p}, and assume that the domain $\mathfrak{R} = \mathfrak{P}$ is polygonal with at least $6$ sides. Then following statements hold.
	 
	 \begin{enumerate}
	 	
	\item For any $(x, t) \in \mathfrak{P} \setminus \mathfrak{L} (\mathfrak{P})$, we have $\nabla H^* (x, t) \in \big\{ (0, 0), (1, 0), (1, -1) \big\}$. 
	 \item The arctic boundary $\mathfrak{A} = \mathfrak{A} (\mathfrak{P})$ is an algebraic curve, and its singularities are all either ordinary cusps or tacnodes.
	 \item Fix $(x_0, t_0) \in \overline{\mathfrak{L}}$. There exists a neighborhood $\mathfrak{U} \subset \mathbb{C}^2$ of $(x_0, t_0)$ and a real analytic function $Q_0 : \mathfrak{U} \rightarrow \mathbb{C}$ such that, for any $(x, t) \in \mathfrak{U} \cap \overline{\mathfrak{L}}$, we have
			\begin{flalign}
				\label{e:defQ0} 
				Q_0 \big( f_t (x) \big) = x \big( f_t (x) + 1 \big) - t f_t (x).
			\end{flalign}
			 There exists a nonzero rational function $Q$ such that, for any $(x,t)\in \overline{\mathfrak{L}}$, we have   
	 	\begin{flalign}\label{e:qfh}
	 		Q \bigg( f_t (x), x - \displaystyle\frac{t f_t (x)}{f_t (x) + 1} \bigg) = 0.
	 	\end{flalign}

 		\item For any $(x, t) \in \overline{\mathfrak{L}}$, $(x,t)\in \mathfrak{A}(\fP)$ if and only if $f_t(x)$ is a double root of \eqref{e:defQ0}.
 		
	 \end{enumerate}
	 
	 \end{prop}

\begin{proof}
		
		\Cref{pa1} follows essentially from results in \cite{LSCE,DMCS}. Its first statement is \cite[Theorem 1.10]{DMCS}, and the second statement follows from \cite[Theorem 1.2, Theorem 1.10]{DMCS}. It follows from \cite[Theorem 1.5]{DMCS} that the complex slope $f_t(x)$ extends continuously to a real number on the arctic boundary.

For any connected component $\fL'\subset \fL(\fP)$ of the liquid region, with arctic boundary $\fA'$,
we consider the following map from the closure of the liquid region $\overline\fL'=\fL'\cup\fA'$ to $\mathbb{CP}^2$
\begin{align}\label{e:emb}
(x,t)\in \overline\fL\mapsto (f(x,t), z(x,t)):=\left(f_t(x), x-t\frac{f_t(x)}{f_t(x)+1}\right)\in \mathbb{CP}^2.
\end{align}

For $(x,t)\in \fL'$, the claim \eqref{e:defQ0} follows from \cite[Theorem 10.5]{RT}. Assuming $\fL'$ is simply connected (we will show it in \Cref{l0}), \cite[Theorem 5.1]{DMCS} gives a decomposition of the map $(x,t)\mapsto f(x,t)$.  It implies that for any $(x_0,t_0)\in \fA'$, there exists a small neighborhood $\fU$ of it, such that the map $(x,t)\mapsto f(x,t)$ is injective for $(x,t)\in \fU\cap \overline\fL'$. Thus the same argument as above, there exists a continuous map $Q_0$ such that $z(x,t)=Q_0(f(x,t))/(f(x,t)+1)$ for  $(x,t)\in \fU\cap \overline{\fL'}$, where $\fU$ is a small neighborhood of $(x_0,t_0)\in \fA'$, and $Q_0$ is analytic on $f(\fU\cap \fL')$. By Schwarz reflection principle, we can extends $Q_0$ to a real analytic function in a neighborhood of $f(x_0,t_0)$. This and the discussion above gives \eqref{e:defQ0}.

Since the image of the arctic boundary $\fA'$ under the map \eqref{e:emb} is real, we can glue the image of the map \eqref{e:emb} restricted to the liquid region $\fL'$, and its complex conjugate along the image of the arctic boundary to get an immersed Riemann surface in $\mathbb{CP}^2$:
$\{(f,z)=(f(x,t),z(x,t))\in \mathbb{CP}^2: (x,t)\in \overline\fL'\}\cup\{(f,z)=(\overline{f(x,t)},\overline{z(x,t)})\in \mathbb{CP}^2: (x,t)\in \overline\fL'\}$. Since $Q_0$ is real analytic, the local charts around any point $(f(x_0,t_0),z(x_0,t_0))$  for $(x_0,t_0)\in \fA'$ on the arctic boundary are given by $z=Q_0(f)/(f+1)$.
$f(x,t)$ with its complex conjugate  and $z(x,t)$ with its complex conjugate are two meromorphic functions on this Riemann surface. By \cite[Theorem 5.8.1]{jost2013compact},  there exists a rational function $Q'$ such that $Q'(f(x,t), z(x,t))=0$ for any $(x,t)\in \overline{\fL'}$, which gives \eqref{e:qfh} by taking product of $Q'$ corresponding to each connected component of $\fL(\fP)$.

The fourth statement follows from the facts that $(x, t) \in \mathfrak{A} (\mathfrak{P})$ if and only if $f_t (x) \in \mathbb{R}$, by \eqref{fh}, and that any root of \eqref{e:defQ0} is real if and only if it is a double root, as $Q_0$ is real anaytic (see also the discussion at the end of \cite[Section 1.6]{LSCE}).

\end{proof}

\begin{lem} 
	
	\label{l0}
	
	Adopting the notation of \Cref{pa1}, any connected component $\mathfrak{L}' \subseteq \mathfrak{L} (\mathfrak{P})$ is simply connected. 
	
\end{lem} 

\begin{proof}
Assume to the contrary that  $\fL'$ is not simply connected. Then there exists a closed curve $\gamma\subseteq\fL'$, which is not homotopy equivalent to a point in $\fL'$. It follows that the curve $\gamma$ encloses a frozen region $\fF\subseteq \fP\setminus \overline{\fL(\fP)}$. From \Cref{p}, the polygonal domain $\fP$ is simply connected, so the curve $\gamma$ separates the frozen region $\fF$ from the boundary $\partial \fP \subset \overline{\fP}$; see Figure \ref{f:hole}. We denote the boundary of $\fF$ by $\fA' = \partial \mathfrak{F} \subseteq \fA(\fP)$.

	\begin{figure}
		
		\begin{center}		
			
			\begin{tikzpicture}[
				>=stealth,
				auto,
				style={
					scale = .52
				}
				]
				\draw[black,dashed] (4,0) circle (2);
				\draw[black, very thick] (0,-4)--(4,-4)--(8,0)--(8,4)--(4,4)--(0,0)--(0,-4);
				\draw[black, thick] (2.5, 0) arc (-90:0:1.5);
				\draw[black, thick] (2.5, 0) arc (90:0:1.5);
				\draw[black, thick] (5.5, 0) arc (-90:-180:1.5);
				\draw[black, thick] (5.5, 0) arc (90:180:1.5);
			
				\draw[black] (2, 2) node[above left, scale=1]{$\partial \fP$};
				\draw[black] (4, 0) node[ scale=1]{$\fF$};
				\draw[black] (6, 2) node[below left, scale=1]{$\gamma$};
		
				\draw[black, thick] (12, 0) arc (-90:0:4);
				\draw[black, thick] (12, 0) arc (90:0:4);
				\draw[black, thick] (20, 0) arc (-90:-180:4);
				\draw[black, thick] (20, 0) arc (90:180:4);
				
				\draw[black] (16, 0) node[scale=1]{$\fF$};
				\draw[black] (18, 2) node[below left, scale=1]{$\fA'$};
				\draw[black] (13.5, -2.5) node[scale=1]{$\fL(\fP)$};

				\filldraw[fill = black] (14.85,-1.15) circle [radius = .1] node[above right, scale = .8]{$\xi$};	
				\filldraw[fill = black] (14,-0.55) circle [radius = .1] node[above right, scale = .8]{$u$};	
				\filldraw[fill = black] (15.45,-2) circle [radius = .1] node[above right, scale = .8]{$v$};	
				\draw[black, dashed] (14,-0.55)-- (15.45,-2);
			\end{tikzpicture}
			
		\end{center}
		
		\caption{\label{f:hole} Shown to the left is the frozen region $\fF$, which is separated from the boundary $\partial \fP$ by a curve $\gamma\in \fL'$. Shown to the right are two points $u,v$ on the boundary $\fA'$ of $\fF$ such that $H^*(v)>H^*(u)$.}
		
	\end{figure}
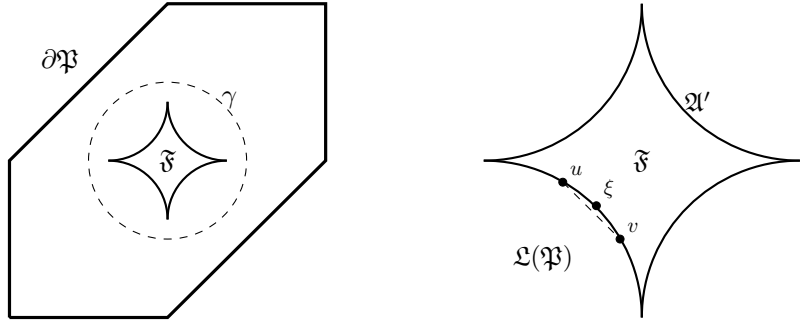

By the first part of \Cref{pa1}, for almost any $(x, t) \in \mathfrak{F}$ we have $\nabla H^*\in \big\{(0,0), (1,0),(1,-1) \big\}$. We consider two cases for the behavior of $\nabla H^*$ inside $\fF$; either $\nabla H^* |_{\mathfrak{F}}$ is continuous or discontinuous. If $\nabla H^*$ is discontinuous at some point $w \in \fF$, then \cite[Theorem 1.3]{MCFARS} implies that there exists a segment connecting $w$ and $\partial \fP$ staying outside  the liquid region, which contradicts the fact that $\gamma\in \fL'$ separates $\mathfrak F$ from $\partial \fP$. In the second case, $\nabla H^*$ is continuous inside $\fF$; thus, it takes a constant value in $\big\{ (0, 0), (1, 0), (1, -1) \big\}$. Without loss of generality, we may assume that $\nabla H^*\equiv (0,0)$ on $\fF$, so $H^* |_{\mathfrak{F}}$ is constant. By continuity, $H^*$ is the same constant on the boundary $\fA'$ of $\fF$. 

By the second part, the arctic boundary $\fA'$ is an algebraic curve with finitely many singularities, which are all either first order cusps or tacnodes; by \cite[Theorem 1.3(c)]{DMCS}, at every nonsingular point $\xi\in \fA'$, the boundary $\fA'$ is locally strictly convex; and, by \cite[Theorem 1.5, Theorem 1.10]{DMCS}, the unit tangent vector $\tau(\xi)$ to $\fA'$ at any $\xi \in \mathfrak{A}'$ depends continuously on $\xi$.  Therefore, when $\xi$ moves along $\fA'$, the tangent vector angle winds around the unit circle $\mathbb S^1$ at least once. Hence, there exists a nonsingular point $\xi \in \fA'$ such that the unit tangent vector $\tau(\xi)$ has positive $x$-coordinate and negative $y$-coordinate. 

Shift the tangent line at $\xi$ in the normal direction to $\mathfrak{A}'$, slightly into the liquid region, such that it intersects $\fA'$ at two points, $u$ and $v$. By the 
local strict convexity of $\mathfrak{A}'$ at $\xi$, the segment from $u$ to $v$ (which is parallel to $\tau(\xi)$) remains strictly inside the liquid region $\mathfrak{L} (\mathfrak{P})$; see Figure \ref{f:hole}. In the liquid region, $\nabla H^*$ stays inside strictly the triangle $\mathcal{T}$ from \eqref{t}. Since $v - u$ is parallel to $\tau(\xi)$, which has positive $x$-coordinate and negative $y$-coordinate, it follows that $\nabla H^* (z) \cdot \tau(\xi)>0$ for any $z \in \mathfrak{L} (\mathfrak{P})$. Thus,
\begin{align}
H^*(v)-H^*(u)=\int_0^1 \nabla H^* \big(\theta v+(1-\theta) u \big) \cdot(v-u)\rm d \theta>0,
\end{align}

\noindent which contradicts the fact that $H^*$ is constant on $\fA'$. Hence,  $\fL'$ is simply connected.
\end{proof}

	\begin{assumption} 
		
		\label{pa} 
		
		Under the notation of \Cref{p}, assume the following four properties hold. 
		\begin{enumerate} 
			\item The arctic boundary $\mathfrak{A} = \mathfrak{A} (\mathfrak{P})$ has no tacnode singularities. 
			\item No cusp singularity of $\mathfrak{A}$ is also a tangency location of $\mathfrak{A} (\mathfrak{P})$. 
			\item  There exists an axis $\ell$ of $\mathbb{T}$ such that any line connecting two distinct cusp singularities of $\mathfrak{A}$ is not parallel to $\ell$. 
			\item  Any intersection point between $\mathfrak{A}$ and $\partial \mathfrak{P}$ must be a tangency location of $\mathfrak{A}$.
		\end{enumerate}
	
	\end{assumption}

 It has been proven in \cite[Lemma 6.2]{AH2}, 
under \Cref{pa}, for any $u\in \overline{\fL(\fP)}$, there exists a trapezoid $\fD(u)\subseteq \fP\cap (\bR\times [\ft_1,\ft_2])$ in the form \eqref{d}, containing $u$ (flip $\fP$ if necessary), with liquid region $\fL$, and complex slope $f:\fL \mapsto \bH^-$ satisfying \ref{ftx}. For all $\ft_1\leq t\leq \ft_2$, we define slices of the liquid region (along the horizontal line $y = t$) by $I_t = \big\{ x : (x, t) \in \overline{\mathfrak{L}} \big\}$.
	Then the following constraints hold:

		\begin{enumerate} 
			\item The boundary height function $h$ is constant along both $\partial_{\ea} (\mathfrak{D})$ and $\partial_{\we} (\mathfrak{D})$.
			\item Either $\partial_{\north} (\mathfrak{D})$ is packed with respect to $h$ or there exists $\mathfrak{t}' > \mathfrak{t}_2-\ft_1$ such that $H^*$ admits an extension to time $\ft_1+\mathfrak{t}'$.
			\item There exists $\tilde{\mathfrak{t}} \in [\mathfrak{t}_1, \mathfrak{t}_2]$ such that the following holds. For $t \in [\mathfrak{t}_1, \tilde{\mathfrak{t}}]$, the set $I_t$ consists of one nonempty interval, and for $t \in (\tilde{\mathfrak{t}}, \mathfrak{t}_2]$, the set $I_t$ consists of two nonempty disjoint intervals.\footnote{Observe if $\tilde{\mathfrak{t}} = \mathfrak{t}_2$, then $I_t$ always only consists of one nonempty interval.}
			\item Any tangency location along $\partial \overline{\mathfrak{L}}$ is of the form $\min I_t$ or $\max I_t$, for some $t \in (\mathfrak{t}_1, \mathfrak{t}_2)$. At most one tangency location is of the form $\min I_t$, and at most one is of the form $\max I_t$. There tangency locations are on the west and east boundary of $\fD(u)$.
			\item There exists an algebraic curve $Q$ such that, for any $(x, t) \in \overline{\mathfrak{L}}$, we have 
			\begin{flalign}
		\label{e:complexss0}		Q \bigg( f_t (x), x - \displaystyle\frac{t f_t (x)}{f_t (x) + 1} \bigg) = 0.
			\end{flalign}
				
			\noindent Furthermore, the curve $Q$ ``approximately comes from a polygonal domain'' in the following sense. There exists a polygonal domain $\mathfrak{P}$ satisfying \Cref{pa} with liquid region $\mathfrak{L} (\mathfrak{P})$; a connected component $\mathfrak{L}_0 (\mathfrak{P}) \subseteq \mathfrak{L} (\mathfrak{P})$; and a real number $\alpha \in \mathbb{R}$ with $|\alpha - 1| < n^{-\delta}$ such that, if $Q_{\mathfrak{L}_0}$ is the algebraic curve associated with $\mathfrak{L}_0$ from \Cref{pa1}, then
			\begin{flalign}
			\label{e:complexss}	Q (u, v) = \displaystyle\frac{u + 1}{\alpha^{-1} u + 1} Q_{\mathfrak{L}_0} (\alpha^{-1} u, v).
			\end{flalign}
			
		\end{enumerate}

We remark that in \eqref{e:complexss}, if $\al=1$, then the height function $H^*$ on $\fD$ is the restriction of the limiting continuum height function on $\fP$. For $\al\neq 1$ and $|\alpha - 1| < n^{-\delta}$, $H^*$ is a slightly tilted version of that.

With a  slight abuse of notation, we denote the shifted version of $\fD, \fP, H^*, f$ by time $\ft_1$, still by $\fD, \fP, H^*, f$, and $\ft=\ft_2-\ft_1$. Then the trapezoid domain $\fD$ and limiting continuum height function $H^*$ satisfies the first two statements and Item (a),(b),(c) and (d) in the third statement in \Cref{xhh}. In particular, thanks to Assumption \ref{pa}, the cusp location on $\fA(\fD)$ (if it exists) is not a tangency location. 
We denote the extended trapezoid domain by $\fD'\subseteq \bR\times [0,\ft']$. Let $\fD^t=\fD'\cap (\bR\times [t,\ft'])$. And we still denote the extended complex slope  constructed using \eqref{e:complexss0} and \eqref{e:complexss} by $f$, and the extended liquid region $\tilde \fL=\fL(\fD')$. 
Then it follows from \Cref{pa1}, Item (c) and (d) in the third statement of \Cref{xhh} holds. Next we show that \eqref{e:injecthi} is a bijection to its image provided $\ft'$ is small enough.
\begin{prop}\label{p:injective}
There exists a constant $c = c (\fP) > 0$. If $\mathfrak{t}' < c$, then for any $t\in [0, \ft']$ the map 
	\begin{align}\label{e:injectcopy}
			\pi: (x,r)\in \fL(\fD)\cap \fD^t \mapsto x+(t-r)\frac{ f_r(x)}{ f_r(x)+1}\in \bH^+
			\end{align}
 is a bijection.
\end{prop}

\subsection{Proof of Propositions \ref{p:injective}}\label{s:injection}
To show \eqref{e:injectcopy} is a bijection, we will use the following criterion for the global homemorphism from \cite[Theorem 1]{meisters1963locally}.
\begin{thm}\label{t:hom}
	Fix an integer $d \geq 2$. Let $X \subset \mathbb{R}^d$ denote a compact subset, whose boundary $\del X$ is an irreducible separating set of $\bR^d$ (A separating set $\Omega$ of $\bR^d$ is said to be an irreducible separating
set of $\bR^d$ provided no proper closed subset of $\Omega$ separates $\bR^d$). Let $f: X \rightarrow \mathbb{R}^d$ be a continuous mapping that is locally one-to-one on $X\setminus Z$, where $Z \subset X$ is such that $Z\cap X^\circ$ ($X^\circ$ is the interior of $X$) is a discrete set and $\del X \setminus Z$ is not empty. If $f|_{\del X}$ is one-to-one, then $f$ is a homeomorphism of $X$ onto $f(X)$.
\end{thm}

The map $\pi$ in \eqref{e:injectcopy} satisfies the assumptions of \Cref{t:hom} with $Z$ the critical point of $\pi$. It is a bijection, if it is one-to-one on the boundary. The  boundary of the liquid region $\fL(\fD')\cap \fD^t$ consists of three parts: the north boundary, the south boundary, and the arctic curve boundary
\begin{align}\label{e:fLct}
\big\{(x,\ft')\in  \overline{\fL(\fD')} \big\}\cup \big\{ (x,t)\in \overline{\fL(\fD')} \big\}\cup \overline{\big\{ (x,r)\in \del {\fL(\fD')}: r\in(t,\ft') \big\}}.
\end{align}

\begin{lem}\label{l:bottombb}
The  map \eqref{e:injectcopy} maps the south boundary and the arctic curve bijectively to an interval in $\bR$.
\end{lem}
\begin{proof}

Any $(x,t)$ on the south boundary $\big\{ (x,t)\in\overline{\fL({\fD'})} \big\}$ is mapped to $z=x$ by \eqref{e:proj}.
At any point $(x,r)$ on the arctic boundary of $\fL(\fD')\cap \fD^t$ such that $(x,r)\in  \mathfrak{A} (\fD')$ with $r\in(t,\ft')$, the tangent vector to $(x,r)\in  \mathfrak{A} (\fD')$ has slope 
\begin{align}\label{e:slope2}
\frac{f_r(x)+1}{f_r(x)},
\end{align}
 by \eqref{e:slope}. 
$(x,r)$ is mapped to $(f,z)$ by \eqref{e:injectcopy} as given by
\begin{align}\label{e:deriv}
f=f_r(x),\quad \frac{r-t}{x-z}=\frac{f_r(x)+1}{f_r(x)}.
\end{align}
Comparing with \eqref{e:slope2}, the right side of the second expression in  \eqref{e:deriv} is the slope of the tangent vector of arctic curve at $(x,r)$. Thus $(z,t)$ is the intersection of the tangent line to the arctic curve at $(x,r)$ with the bottom boundary of $\fD^t$; See Figure \ref{f:tangent}.
When we move along the south boundary and the
arctic curve counterclockwise, its image $z\in \bR$ also moves in negative direction. The projection map \eqref{e:injectcopy} maps the south boundary and the arctic curve bijectively to an interval in $\bR$.

\end{proof}

\begin{figure}
		
		\begin{center}		
			
			\begin{tikzpicture}[
				>=stealth,
				auto,
				style={
					scale = .52
				}
				]
				
				\draw[black, very thick] (-4.9, 0) node[left, scale = .7]{$y = t$}-- (3.45, 0)--(3.45,3)--(-1.9,3)node[left, scale = .7]{$y = \ft'$}-- (-4.934877870428601, 0);
				\draw[](-4.9,0) node[below, scale=0.7]{$\fa(t)$};
				\draw[](3.45,0) node[below, scale=0.7]{$\fb(t)$};

				\draw[black,fill=black] (-2.65,2.25) circle (.1);

				\draw[black,fill=black] (3.45,1.5) circle (.1);
				\draw[black, thick] (-3.5, 0) arc (180:110:3.2);
				
				\draw[black,fill=black] (-3.2,1.4) circle (.1);
				\draw[black,fill=black] (-4,0) circle (.1);
				
				\draw[black] (-3.2,1.4)--(-4,0);
				\draw[] (-3.2,1.4) node[right, scale=0.7]{$(x,r)$};
				\draw[] (-4,0) node[below, scale=0.7]{$(z,t)$};

				\draw[black, thick] (3, 0) arc (-30:30:3);

				\draw[black, thick] (-1, 0) arc (-60:0:1.732);
				\draw[black, thick] (.75, 0) arc (240:180:1.732);
				
				\draw[] (0, 2)  node[scale=0.7]{$\mathfrak{L}(\fD')\cap \fD^t$};

				
%
%
%
%
%
				
%
%
%
%
%
%
%
%
%
%
%
%
%

			\end{tikzpicture}
			
		\end{center}
		
		\caption{\label{f:tangent} 
		Any point $(x,r)$ on the arctic curve is mapped to $(f,z)$ by \eqref{e:proj}, where $(z,t)$ is the intersection of the tangent to the arctic curve at $(x,r)$, and the bottom boundary of $\fD^t$ as shown in the left figure.}
		
	\end{figure}
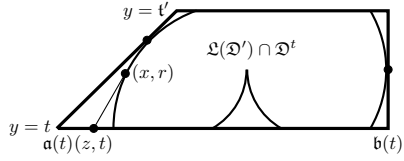

\begin{proof}[Proof of \Cref{p:injective}]
	By our assumption, the north boundary $\big\{ (x,\ft')\in \overline{\fL({\fD'})} \big\}$ is a single interval. It is mapped by \eqref{e:injectcopy} to a curve in the upper complex plane (namely, we have $\Im z>0$), except for at its two end points (which are mapped to the real line). 
Thanks to Theorem \ref{t:hom} and Lemma \ref{l:bottombb}, \Cref{p:injective} follows if we show that there exists a universal constant $c(\fP)>0$ such that, for any $0\leq \delta\leq c(\fP)$,
	\begin{align}\label{e:boundary2}
		(x,\ft')\in \fL(\fD') \mapsto x-\delta \frac{f_{\ft'}(x)}{f_{\ft'}(x)+1}\in \bH^+,
	\end{align}
	is a bijection onto its image.
	Thanks to the third statement of Proposition \ref{pa1},  locally around $\big( f_t(x), x \big)$, there exists  a real analytic function $Q_0$ in one variable such that,
	\begin{align}\label{e:q0fcopy2}
		Q_0 \big(f_t(x) \big)=x \big(f_t(x)+1 \big)-tf_t(x).
	\end{align}
	In the following, only consider the case $\alpha=1$, for the general case when $|\alpha-1|=\oo(1)$ is entirely analogous. There are two cases: the north boundary $\del_{\rm no} \fL(\fD')=\{(x,\ft')\in \fL(\fD')\}$ is close to a horizontal tangency location; or $\del_{\rm no} \fL(\fD')$ is bounded away from horizontal tangency locations. We discuss these cases separately. Fix  a small constant $\fc_0>0$.

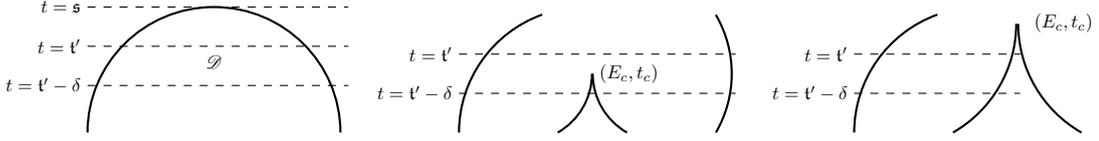
\begin{figure}
	
	\begin{center}		
		
		\begin{tikzpicture}[
			>=stealth,
			auto,
			style={
				scale = .52
			}
			]

			\draw[black, thick] (7, 0) arc (0:180:3.2);
			\draw[black, dashed] (0.6, 3.2) node[left, scale=0.7]{$t=\fs$} --(7.2, 3.2) ;
			\draw[black, dashed] (0.6, 2.2) node[left, scale=0.7]{$t=\ft'$} --(7.2, 2.2) ;
			
			\draw[black, dashed] (0.6, 1.2) node[left, scale=0.7]{$t=\ft'-\delta$} --(7.2, 1.2) ; 
			\draw[] (3.8, 1.8) node[scale=0.7]{$\mathscr D$};

			\draw[black, thick] (10, 0) arc (180:110:3.2);
			\draw[black, thick] (16.5, 0) arc (-30:30:3);
			\draw[black, thick] (12.5, 0) arc (-60:0:1.732);
			\draw[black, thick] (14.25, 0) arc (240:180:1.732);		
			\draw[black, dashed] (10, 2) node[left, scale=0.7]{$t=\ft'$} --(17, 2) ;
			\draw[black, dashed] (10, 1) node[left, scale=0.7]{$t=\ft'-\delta$} --(17, 1) ;
			\draw[] (14.3, 1.5)  node[scale=0.7]{$(E_c, t_c)$};

			\draw[black, thick] (20, 0) arc (180:110:3.2);

			\draw[black, thick] (22.5, 0) arc (-60:0:3.2);
			\draw[black, thick] (25.75, 0) arc (240:180:3.2);
			\draw[] (25.3, 2.8)  node[scale=0.7]{$(E_c, t_c)$};
			\draw[black, dashed] (20, 2) node[left, scale=0.7]{$t=\ft'$} --(24.2, 2) ;
			\draw[black, dashed] (20, 1) node[left, scale=0.7]{$t=\ft'-\delta$} --(24.2, 1) ;


		\end{tikzpicture}
		
	\end{center}
	
	\caption{\label{f:trapill} Shown to the left is the case that the north boundary $\del_{\rm no} \fL(\fD')=\{(x,\ft')\in \fL(\fD')\}$ is close to a horizontal tangency location. Shown to the middle is the case that $\del_{\rm no} \fL(\fD')$ is close to a cusp location $(E_c, t_c)$ and $\ft'>t_c$. Shown to the right is the case that $\del_{\rm no} \fL(\fD')$ is close to a cusp location $(E_c, t_c)$ and $\ft'<t_c$.
	}
\end{figure}

	First suppose that $\del_{\rm no} \fL(\fD')$ is close to a horizontal tangency location at time $\fs$ with distance less than $\fc_0$; see Figure \ref{f:trapill} (flip $\fD'$ if necessary). We can consider the liquid region $\mathscr D$ between times $\ft'-\delta$ and $\fs$, and the projection map
	\begin{align}\label{e:proj2}
		(x,t)\in \mathscr D \mapsto x+( \ft' - \delta - t) \frac{f_{t}(x)}{f_{t}(x)+1}\in \bH^+.
	\end{align}
	
	\noindent Similarly to Lemma \ref{l:bottombb}, \eqref{e:proj2} maps the boundary $\partial \mathscr D$ bijectively to $\bR\cup \{\infty\}$. To use \Cref{t:hom} we further compose \eqref{e:proj2} with a map $\bH^+\cup \{\infty\}$ to the unit disk $\bD$. Then by taking $X=\bD$ in \Cref{t:hom} we conclude that \eqref{e:proj2} is a bijection. By taking $t=\ft'$, we get that \eqref{e:boundary2} is a bijection.
	
	We next address the remaining two cases, namely, when $\del_{\rm no} \fL(\fD')$ is distance at least $\fc_0$ away from horizontal tangency locations. Then, there exists a universal constant $\fC_0>0$ depending on $\fc_0$ such that, for $(x,t)$ at distance $\fc_0$ away from any horizontal tangency location, we have
	\begin{align}\label{e:ftCbb}
		\big| f_t (x)+1 \big|\geq \fC_0^{-1}; \qquad \left|\frac{f_t (x)}{f_t (x)+1}\right|\leq \fC_0.
	\end{align}
	
	\noindent In what follows, we fix small constants $0<\fc<\fc'$, which will be chosen later (and will depend only on $\mathfrak{P}$).

	For $(x,t)$ at least distance $\fc$ away from the arctic curve, there is a universal constant $\fC$ (depending on $\fc$) such that
	\begin{align}\label{e:derftCbb}
		\left|\del_x \frac{f_t (x)}{f_t (x)+1}\right|\leq \fC.
	\end{align}
	
	\noindent Assume that there exist $(X, \ft')$ and $(X',\ft')$ with $X<X'$ such that 
	\begin{align}\label{e:eqXX}
		X-\delta \frac{f_{\ft'} (X)}{f_{\ft'}(X)+1}= X'-\delta \frac{f_{\ft'} (X')}{f_{\ft'}(X')+1},
	\end{align}
	then \eqref{e:ftCbb} gives that
	\begin{align}
		\label{xdistance}
		|X-X'|=\delta \left|\frac{f_{\ft'} (X)}{f_{\ft'}(X)+1}-\frac{f_{\ft'} (X')}{f_{\ft'}(X')+1}\right|
		\leq 2\delta \fC_0\leq \displaystyle\frac{\fc}{2},
	\end{align}
	provided we take $\delta \leq \fc/8\fC_0$. If the interval  $\big\{ (y,\ft'): X\leq y\leq X' \big\}$ is distance $\fc$ away from the arctic curve, then 
	\begin{align*}
		|X-X'|=\delta \left|\frac{f_{\ft'} (X)}{f_{\ft'}(X)+1}-\frac{f_{\ft'} (X')}{f_{\ft'}(X')+1}\right|
		\leq \delta |X-X'|\sup_{X\leq y\leq X'}\left|\del_y \frac{f_{\ft'} (y)}{f_{\ft'} (y)+1}\right|
		\leq \delta \fC|X-X'|<|X-X'|,
	\end{align*}
	which leads to a contradiction, provided we take $\delta \leq 1/2\fC$.

	It remains to consider the case when the interval  $\big\{ (y,\ft'): X\leq y\leq X' \big\}$ is at distance at most $\mathfrak{c}$ from the arctic curve.  Let us first consider the case that $\big\{ (y,\ft'): X\leq y\leq X' \big\}$ is at distance at most $\fc$ from a cusp location $(E_c, t_c)$ of the arctic boundary for $\mathfrak{P}$; see Figure \ref{f:trapill} (by our assumption, the cusp point upward). There are two cases, either $\ft'\geq t_c$ or $\ft'\leq t_c$. The proofs are essentially the same, so we will only discuss the former case.  We construct a curve $\mathscr C$ connecting two points on $t=\ft'-\delta$, such that $(X,\ft')$ and $(X',\ft')$ are inside the liquid region $\fL(\fP)$ between $\mathscr C$ and $t=\ft'-\delta$. We denote the liquid region between $\mathscr C$ and $t=\ft'-\delta$ by $\mathscr D$; see Figure \ref{f:injcontour}. We will show that the map
	\begin{align}\label{e:proj3}
		(x,t)\in \mathscr D\subseteq\fL(\fP) \mapsto x+(\ft'-\delta-t) \frac{f_{t} (x)}{f_{t}(x)+1}\in \bH^+,
	\end{align}
	is a bijection. By taking $t=\ft'$ and $x=X,X'$, this will rule out \eqref{e:eqXX}. By Lemma \ref{l:bottombb}, \eqref{e:proj3} maps the bottom boundary of $\mathscr D$ bijectively to an interval in $\bR$, and $\mathscr C$ to a curve in the upper half plane. Thanks to \Cref{t:hom}, the claim that \eqref{e:proj3} is a bijection would follow if we show that the map \eqref{e:proj3} restricted to $\mathscr C$ is one-to-one.
	
	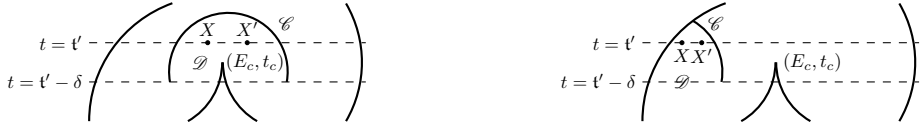
\begin{figure}
		
		\begin{center}		
			
			\begin{tikzpicture}[
				>=stealth,
				auto,
				style={
					scale = .52
				}
				]

				\draw[black, thick] (10, 0) arc (180:110:3.2);
				\draw[black, thick] (16.5, 0) arc (-30:30:3);
				\draw[black, thick] (12.5, 0) arc (-60:0:1.732);
				\draw[black, thick] (14.25, 0) arc (240:180:1.732);		
				\draw[black, dashed] (10, 2) node[left, scale=0.7]{$t=\ft'$} --(17, 2) ;
				\draw[black, dashed] (10, 1) node[left, scale=0.7]{$t=\ft'-\delta$} --(17, 1) ;
				\draw[] (14.2, 1.5)  node[scale=0.7]{$(E_c, t_c)$};
				\draw[] (12.8, 1.5)  node[scale=0.7]{$\mathscr D$};
				\draw[] (15, 2.5)  node[scale=0.7]{$\mathscr C$};
				
				\draw[black, thick] (15, 1) arc (-10:190:1.5);
				\draw[black, fill] (13, 2) node[above, scale=0.7]{$X$} circle (0.05);
				\draw[black, fill] (14, 2) node[above, scale=0.7]{$X'$} circle (0.05);

				\draw[black, thick] (24, 0) arc (180:110:3.2);
				\draw[black, thick] (30.5, 0) arc (-30:30:3);
				\draw[black, thick] (26.5, 0) arc (-60:0:1.732);
				\draw[black, thick] (28.25, 0) arc (240:180:1.732);		
				\draw[black, dashed] (24, 2) node[left, scale=0.7]{$t=\ft'$} --(31, 2) ;
				\draw[black, dashed] (24, 1) node[left, scale=0.7]{$t=\ft'-\delta$} --(31, 1) ;
				\draw[] (28.3, 1.5)  node[scale=0.7]{$(E_c, t_c)$};
				
				\draw[] (25, 1)  node[scale=0.7]{$\mathscr D$};
				\draw[] (25.5, 2.5)  node[ right, scale=0.7]{$\mathscr C$};
				\draw[black, thick] (26, 1) arc (-10:60:1.5);
				\draw[black, fill] (25, 2) node[below, scale=0.7]{$X$} circle (0.05);
				\draw[black, fill] (25.5, 2) node[below, scale=0.7]{$X'$} circle (0.05);

			\end{tikzpicture}
			
		\end{center}
		
		\caption{\label{f:injcontour} Shown to the left 
			is the curve $\mathscr C$ around a cusp location, connecting two points on $t=\ft'-\delta$, and $(X,\ft'), (X',\ft')$ are inside the liquid region between $\mathscr C$ and $t=\ft'-\delta$.
			Shown to the right is the curve $\mathscr C$ around a edge point $(E(\fs),\fs)$, connecting one point on $t=\ft'-\delta$ and one point on the arctic curve, and $(X,\ft'), (X',\ft')$ are inside the liquid region between $\mathscr C$ and $t=\ft'-\delta$.
			.
		}
	\end{figure}

	Let $f_c=f_{t_c} (E_c)$, and denote the map
	\begin{align}\label{e:mapfw}
		(x,t)\in \fL(\fP)\mapsto (f,w)=\left(f_t ( x ),  x +(t_c-t)\frac{   f_t ( x)}{ f_t ( x )+1}\right).
	\end{align}
	Here, $w=w(x,t)$ is a function of $x,t$; if the context is clear we will simply write it as $w$.
	We can interpret $f=f_{t_c}(w)$ as a function of $w$. We will soon see that after restriction to a fan-shaped neighborhood of $E_c$,  $f_{t_c}(w)$ is a single valued function.
	With the above notations, locally around $(f_c, E_c)$ we can rewrite \eqref{e:q0fcopy2} as
	\begin{align}\label{e:Q1f}
		Q_1(f):=\frac{Q_0(f)}{f+1}+\frac{t_c f}{f+1}=w.
	\end{align}
	Since $(E_c, t_c)$ is a cusp of the arctic boundary for $\mathfrak{P}$, it is characterized by the equations 
	\begin{flalign*} 
		Q_0(f_c)=E_c(f_c+1)-t_cf_c; \qquad Q_0'(f_c)=E_c-t_c; \qquad Q_0''(f_c)=0.
	\end{flalign*} 

	\noindent We also have that $Q_0'''(f_c)\asymp 1$, since $Q_0$ does not admit any quadruple roots by \Cref{pa1}. It  follows from \eqref{e:Q1f} (and the fact that no cusp is simultaneously a tangency location) that 
	\begin{flalign*} 
		Q_1(f_c)=E_c; \qquad Q_1'(f_c)=0; \qquad Q_1''(f_c)=0; \qquad |Q_0'''(f_c)|\asymp 1.
	\end{flalign*} 

	\noindent This, together with \eqref{e:Q1f}, implies that, locally around $E_c$, $f=f_{t_c}(w)$ has cube root behavior, namely, 
	\begin{flalign}
		\label{ftcw} 
		f_{t_c}(w)=f_c+\sqrt[3]{C(w)(w-E_c}),
	\end{flalign}  

	\noindent for some real analytic function $C(w)$. Let us discuss the choice for the cube root in \eqref{ftcw}; recall that in \eqref{e:mapfw}, we must have $\Im f = \Im[f_t (x)]<0$ and $\Im w =(t_c-t)\Im \big[f_t (x)/(f_t (x)+1) \big]>0$. Therefore, for $t>t_c$ in \eqref{e:mapfw} (here we used that the cusp points upward), its image corresponds to the branch of $f_{t_c}(w)$ in the lower half plane. Also $\Im[f_{t_c}(x)]=-\pi \del_x H_{t_c} (x)<0$ for $x\in (E_c-\varepsilon, E_c)\cup (E_c, E_c+\varepsilon)$ for some small $\varepsilon>0$. Thus the branch of $f_{t_c}(w)$ has to be the one in the lower half plane corresponding to the fan-shaped region with argument $(-2\pi/3, -\pi/3)$; See Figure \ref{f:branchcut}. Then $f_{t_c}(w)$ is a single valued function restricted to a neighborhood of $E_c$ after removing the branch cut.
	
	Next we consider the preimage of \eqref{e:mapfw} for $\big( f_{t_c} (w), w \big)$ with $w$ on the circle $w=E_c+\fc' e^{\ri \theta}$, where $\mathfrak{c}'$ is chosen so that $f(w)$ is single valued on this circle (and $\fc$ chosen to be smaller than $\fc'$). Since $\big| f_t (x)/ (f_t (x) + 1) \big| \leq \fC_0$ by \eqref{e:ftCbb}, we have
	\begin{align*}
		\fc'=|w-E_c|=\left|x-E_c+(t_c-t)\frac{f_t (x)}{f_t (x)+1}\right|\leq \fC_0 \big( |x-E_c|+|t_c-t| \big).
	\end{align*}
	Thus $(x,t)$ is bounded away from $(E_c,t_c)$, with $\dist \big( (x,t), (E_c,t_c) \big) \geq \fc'/ 2\fC_0$. Therefore, if we take $\fc\leq \fc'/ 4 \fC_0$, then by \eqref{xdistance} $(X,\ft')$ and $(X',\ft')$ will be inside the preimage curve. For $\Im[w]\leq -\fc'/2$, the corresponding $(x,t)$ satisfies
	\begin{align*}
		t_c-t =\frac{\Im[w]}{\Im[f/(f+1)]}=\frac{\Im[w]}{\Im[f]}|f+1|^2\geq \frac{\fc'}{2\fC_0^2 \Im[\sqrt[3]{C(w)\fc' e^{\ri \theta}}]}\geq \frac{\fc'}{C\fC_0^2},
	\end{align*}
	for some large constant $C$ depending on $C(w)$. If we take $\delta\leq c(\fP)\leq \fc'/C\fC_0^2$, then the preimage will intersect $t=\ft'-\delta$. This gives the desired curve $\mathscr C$; see Figure \ref{f:branchcut}.
	
	Next we consider the image of $\mathscr C$ under the map \eqref{e:proj3}
	\begin{align}\label{e:proj3copy}
		(x,t)\in \mathscr C \mapsto x+(\ft'-\delta-t) \frac{f_{t} (x)}{f_{t}(x)+1}
		=w-(t_c-\ft'+\delta)\frac{f_{t_c} (w)}{f_{t_c} (w)+1},
	\end{align}
	where $w=w(x,t)$ is from \eqref{e:mapfw}. For $(x,t)\in \mathscr C$, $w$ is on the circle $w=E_c+\fc' e^{\ri\theta}$. Since it is bounded away from $E_c$, namely $|w-E_c|=\fc'$, we have 
	\begin{align*}
		\left|\del_w \frac{f_{t_c} (w)}{f_{t_c} (w)+1}\right|
		\leq \frac{\del_wf_{t_c} (w)}{|f_{t_c} (w)+1|^2}
		\leq \fC_0^2 \big|\del_w \sqrt[3]{C(w)(w-E_c)} \big|\leq C\fC_0^2(\fc')^{-2/3},
	\end{align*}
	where the constant $C$ depends on $C(w)$.
	Therefore, if we take $|t_c-\ft'+\delta|\leq \fc+c(\fP)$ much smaller than $(\fc')^{2/3}/C\fC_0^2$ then, by the implicit function theorem, the map
	\begin{align*}
		w\mapsto w-(t_c-\ft'+\delta)\frac{f_{t_c} (w)}{f_{t_c} (w)+1},
	\end{align*}
	is a bijection from the curve $\big\{ w(x,t): (x,t)\in \mathscr C \big\}$ onto its image. We conclude that the map \eqref{e:proj3} is one-to-one after restriction to $\mathscr C$. Thus, the map \eqref{e:proj3} is a bijection.

	\begin{figure}
		
		\begin{center}		
			
			\begin{tikzpicture}[
				>=stealth,
				auto,
				style={
					scale = .52
				}
				]

				\draw[gray!40!white, fill] (4.875, 1.5) arc (0:180:1.5);
				\draw[gray!20!white, fill] (3.375, 1.5)--(4.875, 1.5) arc (0:-35:1.5)--(3.375, 1.5);
				\draw[gray!20!white, fill] (3.375, 1.5)--(1.875, 1.5) arc (180:215:1.5)--(3.375, 1.5);

				\draw[black, thick] (0, 0) arc (180:110:3.2);
				\draw[black, thick] (6.5, 0) arc (-30:30:3);
				\draw[black, thick] (2.5, 0) arc (-60:0:1.732);
				\draw[black, thick] (4.25, 0) arc (240:180:1.732);		
				\draw[black, dashed] (0, 2) node[left, scale=0.7]{$t=\ft'$} --(7, 2) ;
				\draw[black, dashed] (0, 1) node[left, scale=0.7]{$t=\ft'-\delta$} --(7, 1) ;			
				\draw[black, dashed] (0, 1.5) node[left, scale=0.7]{$t=t_c$} --(7, 1.5) ;	
				\draw[black, thick] (4.8,1) arc (-22:202:1.5);
				
				\draw[] (5, 2.5)  node[scale=0.7]{$\mathscr C$};
				
				\draw[] (3.5, 3.5)  node[scale=0.7]{$(x,t)$};

				\draw[] (9, 1.5)  node[scale=0.7]{$\mapsto$};

				\draw[gray!40!white, fill] (14.5, 1.5) arc (0:180:1.5);
				\draw[gray!20!white, fill] (13,1.5)--(13, 0) arc (-90:-180:1.5)--(13,1.5);
				\draw[gray!20!white, fill] (13,1.5)--(13, 0) arc (-90:0:1.5)--(13,1.5);
				
				\draw[black, thick] (13,1.5) circle(1.5);
				\draw[black, dashed] (11,1.5)--(15, 1.5);
				\draw[white, very thick] (13,1.5)--(13,-0.2);
				\draw[] (13, 3.5)  node[scale=0.7]{$w=w(x,t)=E_c+\fc' e^{\ri\theta}$};
				\draw[] (13.5, 1.5)  node[below, scale=0.7]{$E_c$};
				\draw[] (16.5, 1.5)  node[scale=0.7]{$\mapsto$};

				\draw[gray!40!white, fill] (20,1.5)--(20.75, 0.2) arc (-60:-120:1.5)--(20,1.5);
				\draw[gray!20!white, fill] (20,1.5)--(20.75, 0.2) arc (-60:-30:1.5)--(20,1.5);
				\draw[gray!20!white, fill] (20,1.5)--(19.25, 0.2) arc (-120:-150:1.5)--(20,1.5);
				
				\draw[gray!40!white, fill] (20,1.5)--(21.5, 1.5) arc (0:60:1.5)--(20,1.5);
				\draw[gray!40!white, fill] (20,1.5)--(18.5, 1.5) arc (180:120:1.5)--(20,1.5);
				\draw[black, dashed] (18,1.5)--(22, 1.5);
				\draw[] (20, 3.5)  node[scale=0.7]{$f_{t_c} (w)=f_c+\sqrt[3]{C(w)(w-E_c)}$};

			\end{tikzpicture}
			
		\end{center}
		
		\caption{\label{f:branchcut} For $t>t_c$ in \eqref{e:mapfw} (dark shaded area), its image satisfies $\Im[w]>0$ and $\Im[f_{t_c} (w)]<0$. It must be the branch of $f_{t_c}(w)$ in the lower half plane. After restriction to this branch, the preimage of the circle $w=E_c+\fc'e^{\ri\theta}$,  is bounded away from $(E_c,t_c)$ and intersects $t=\ft'-\delta$. The preimage between two intersections with $t=\ft'-\delta$ gives the desired curve $\mathscr C$. 				}
	\end{figure}
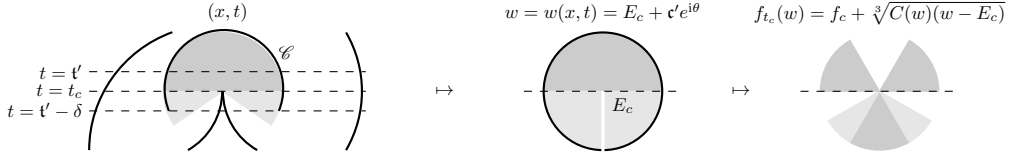

	Next we consider the case when $(X,\ft')$ and $(X',\ft')$ are close to a point $\big( E(\fs), \fs \big)$ (which is bounded away from a cusp location) on the arctic curve,  namely, $\dist \big( \big\{(y,\ft'): X\leq y\leq X' \big\}, (E(\fs),\fs) \big)\leq 2\fc$. Then we can construct a curve $\mathscr C$ connecting a point on the horizontal level $t=\ft'-\delta$ to a point on the arctic curve, such that  $(X,\ft')$ and $(X',\ft')$ are inside the liquid region $\fL(\fP)$ between $\mathscr C$ and $t=\ft'-\delta$. We denote this region by $\mathscr D$; see Figure \ref{f:injcontour}. We can consider the same map \eqref{e:mapfw}, by replacing the cusp location $(E_c,t_c)$ by $\big( E(\fs), \fs \big)$. Instead of cube root behavior \eqref{ftcw}, $f_{\fs} (w)$ has square root behavior $f_{\mathfrak{s}} (w) = f_{\mathfrak{s}} (w) + \sqrt{C(w) \big(w - E (\mathfrak{s}) \big)}$ in a neighborhood of $E(\fs)$. Then, by following essentially the same argument as used in the cusp case, we will have that the map \eqref{e:proj3} is also a bijection on $\mathscr D$.

\end{proof}

\section{Initial Estimates} \label{s:initialEs}
In this section we study the image of the map \eqref{e:proj0}, and prove Proposition \ref{p:fdecompose} for the special case $\beta=H_t^*$. This will serve as the initial estimates for the general case by a deformation argument.

By Item (e) in the third statement of our \Cref{xhh}, the projection map \eqref{e:proj0} is a bijection. Together with its complex conjugate, it maps two copies of the liquid region $\cL(\fD')\cap \fD^t$ gluing along the south boundary and arctic curve to a domain $\mathscr D_0$. We can interpret $f$ in \eqref{e:xrmap} as a function of $z\in \mathscr D_0$ as $f=f_t^*(z)$.
The north boundary $I_{\ft'}^*=[E_1(\ft'), E_2(\ft')]$ of $\cL(\fD')\cap \fD^t$ is a single interval,  it is mapped to a curve
\begin{align}\label{e:curve}
	\mathscr C^*(x):=x-(\ft'-t)\frac{f_{\ft'}^*(x)}{f_{\ft'}^*(x)+1}, \quad x\in I_{\ft'}^*.
\end{align}
We denote the contour $\mathscr C_0:= \big\{\mathscr C^*(x) \cup \overline{\mathscr C^*(x)}: x\in I_{\ft'}^* \big\}$, which is also the boundary curve of $\mathscr D_0$.

\begin{prop}\label{p:ftiinitial}
The contour $\mathscr C_0$ defined as in \eqref{e:curve} is an analytic Jordan curve. There exists a closed set  $\mathscr A_0$ containing an annulus neighborhood of $\mathscr C_0$ (shrinking $\ft'$ if necessary) such that the following holds for $z\in \mathscr A_0$
\begin{align}\label{e:ftiinitial}
	1/\fC\leq |f_t^*(z)|\leq \fC, \quad |f_t^*(z)+1|\geq 1/\fC,\quad |\del_z f_t^*(z)|\leq \fC
\end{align}
and $1+(\ft'-t)\del_z (f_t^*(z)/(f_t^*(z)+1))=0$ has two zeros (counting multiplicity) at $z=\mathscr  C^*(E_1(\ft')), \mathscr  C^*(E_1(\ft'))$.
\end{prop}

\begin{proof}
By shrinking $\ft'$ if necessary, we can extend $f^*_t(z)$ to a neighborhood of the contour $\mathscr C_0$. Moreover, we can also choose the  time slice $t=\ft'$ bounded away from tangency locations. Thus by taking $\mathscr A_0$ the closure of a sufficiently small annulus neighborhood of $\mathscr C_0$, \eqref{e:ftiinitial} holds.

Locally around $(f_r^*(x),x)$ for $(x,r)\in \fL(\fD')\cap \fD^t$ there exists an analytic function $Q_0$ such that
\begin{align}\label{e:introduceQ}
	Q_0(f_r^*(x))=x(f_r^*(x)+1)-rf_r^*(x).
\end{align}
Using \eqref{e:xrmap}, we can rewrite \eqref{e:introduceQ} in terms of $(f_t^*(z), z)$,
\begin{align}\label{e:introduceQ2}
	Q_0(f_t^*(z))=x(f_r^*(x)+1)-rf_r^*(x)=z(f_t^*(z)+1)-tf_t^*(z).
\end{align}
By Item (d) in the third statement of \Cref{xhh}, 
the two end points $E_1(\ft'), E_2(\ft')$ of the $I_{\ft'}^*$ are characterized by $\del_f Q_0(f_{\ft'}^*(E_i(\ft')))=E_i(\ft')-\ft'$, $i\in\{1,2\}$. Since they are not cusp locations, we have $\del^2_f Q_0(f_{\ft'}^*(E_i(\ft')))\neq 0$. Using the relation \eqref{e:introduceQ2}, we can write the derivatives of $Q_0$ in terms of the derivatives of $f_t^*(z)$. By taking derivatives with respect to $z$ on both sides of \eqref{e:introduceQ2}, and rearranging, we get
\begin{align}\label{e:derQ}
	\del_f Q_0(f_t^*(z))=z-t+\frac{f_t^*(z)+1}{\del_z f_t^*(z)},\quad 
	\del^2_f Q_0(f_t^*(z))=-\frac{(f_t^*(z)+1)^3}{(\del_z f_t^*(z))^3}\del_z^2\left(\frac{f_t^*(z)}{f_t^*(z)+1}\right).
\end{align}
For $z=\mathscr C^*(x)=x+(t-\ft')f_t^*(z)/(f_t^*(z)+1)$ for some $x\in I_\ft'^*$, we can rewrite the first relation in \eqref{e:derQ} as
\begin{align}\label{e:delfQ0}
	\del_f Q_0(f_t^*(z))=\del_f Q_0(f_{\ft'}^*(x))=(x-\ft')+\frac{1}{f_t^*(z)+1}\left(\ft'-t+\frac{1}{\del_z (f_t^*(z)/(f_t^*(z)+1)) }\right).
\end{align}
The second term on the righthand side of \eqref{e:delfQ0} vanishes only if $x$ is one of the end points $E_1(\ft'), E_2(\ft')$. 
Thus we conclude that along the contour $\mathscr C_0$, $1+(\ft'-t)\del_z (f_t^*(z)/(f_t^*(z)+1))=0$ only at $z=\mathscr  C^*(E_1(\ft')), \mathscr  C^*(E_1(\ft'))$. Since $(E_i(\ft'),\ft')$ are not cusp locations, the second relation in \eqref{e:derQ} implies that $\del_z^2(f_t^*(z)/(f_t^*(z)+1))\neq 0$ at these two points. By taking  $\mathscr A_0$ the closure of a sufficiently small annulus neighborhood of $\mathscr C_0$, $1+(\ft'-t)\del_z (f_t^*(z)/(f_t^*(z)+1))=0$ has two zeros (counting multiplicity) at $z=\mathscr  C^*(E_1(\ft')), \mathscr  C^*(E_1(\ft'))$.

Next we show the contour $\mathscr C_0$ is real analytic. 
In fact, from the defining relation \eqref{e:curve},  the contour $\mathscr C_0$ can also be constructed from $f_t^*(z)$ as 
$\{w\in \bC: \Im[w]+(\ft'-t)\Im[f_t^*(w)/(f_t^*(w)+1) ]=0\}$. For $z=\mathscr C^*(x)$ and $x\neq E_1(\ft'), E_2(\ft')$, we have $\del_z (z+(\ft'-t)(f_t^*(z)/(f_t^*(z)+1)))\neq 0$. Analytic implicit function theorem implies that $\{w\in \bC: \Im[w]+(\ft'-t)\Im[f_t^*(w)/(f_t^*(w)+1) ]=0\}$ is analytic around $w=z$. Next we discuss the case $z=\mathscr C^*(E_i(\ft'))$, with $i\in\{1,2\}$. By Taylor expansion around $z$, $f_t^*(w)/(f_t^*(w)+1)=f_t^*(z)/(f_t^*(z)+1)-(w-z)/(\ft'-t)+a_2(w-z)^2+a_3(w-z)^3\cdots$, where $a_2\neq 0$. Let $w=z+u+\ri v$. In a neighborhood of $z$, the contour $\mathscr C_0$ is characterized by 
\begin{align*}
	\frac{\Im[w]+(\ft'-t)\Im[f_t^*(w)/(f_t^*(w)+1) ]}{(\ft'-t)\Im[w]}=2a_2 u + a_3\frac{\Im[(u+\ri v)^2]}{v}+\cdots=0,
\end{align*}
where by diving $\Im[w]$, we remove the trivial solution that $w\in \bR$. Since $a_2\neq 0$, the derivative with respect to $u$ is nonzero. The analytic implicit function theorem gives that $\mathscr C_0$ is real analytic around $x=E_i(\ft')$.
\end{proof}

\begin{prop}
Proposition \ref{p:fdecompose} holds for $\beta=H_t^*$.
\end{prop}

\begin{proof}
For $\beta=H_t^*$, the first statement in \Cref{p:fdecompose} follows from the third statement of \Cref{xhh}. 
$\tilde {\mathscr U}_t:=\tilde{\mathscr U}_t^{H_t^*}=\mathscr D_0$ from the discussion above \eqref{e:curve}, and 
$\mathscr U_t:=\mathscr U_t^{H_t^*}$ is the regime from $\mathscr D_0$ by removing the interval $I_t^*=\{x: (x,t)\in \cL(\fD')\}$. Moreover, $f_t^*(z)=f(z;H_t^*,t)$ is analytic on $\mathscr U_t$.

Next, we show that for $x\in \tilde {\mathscr U}_t\cap \bR$, 
	\begin{align}\label{e:extendslope}
		\arg^* f^*_t(x+0\ri)=-\pi (\del_x H^*_t(x)+\bm1_{(-\infty, \fa(t)]}(x)+\bm1_{( \fb(t),\infty]}(x)).
	\end{align}
	For south boundary of $\fL(\fD')\cap \fD^t$, $x\in I_t^*=\{x: (x,t)\in \cL(\fD')\}$, the argument of the complex slope $f^*_t(x)$ is given explicitly by the defining relation \eqref{fh},
	\begin{align}\label{e:argft}
		\arg^* f_t^*(x+0\ri)=\arg^* f^*_t(x)=-\pi \del_x H_t^*(x).
	\end{align}
	For $z\in (\tilde {\mathscr U}_t\cap \bR) \setminus I_t^*$, it corresponds to a point on the arctic curve $(x,r)\in \del \overline{\fL({\fD'})}$ with $r\in(t,\ft')$, given by \eqref{e:deriv}; See Figure \ref{f:tangent},
	\begin{align}\label{e:derivcopy}
		f^*_t(z)=f^*_r(x),\quad \frac{r-t}{x-z}=\frac{f^*_r(x)+1}{f^*_r(x)}.
	\end{align}
	Since the segment from $(z,t)$ to $(x,r)$ stays in the frozen region (we interpret the region on the left of $\del_{\rm we}(\fD^t)$ a frozen region consisting of lozenges of type $2$, and the region on the right of  $\del_{\rm es}(\fD^t)$ a frozen region consisting of lozenges of type $3$), there are several cases. 
	If the frozen region consists of type 1 lozenges, then $\del_x H^*_t(z)+\bm1_{(-\infty, \fa(t)]}(z)+\bm1_{( \fb(t),\infty]}(z)=0$ and $f^*_t(z)=f^*_r(x)\in [0,+\infty]$;
	If the frozen region consists of type 2 lozenges, then $\del_x H^*_t(z)+\bm1_{(-\infty, \fa(t)]}(z)+\bm1_{( \fb(t),\infty]}(z)=1$ and $f^*_t(z)=f^*_r(x)\in[-\infty, -1]$;
	If the frozen region consists of type 3 lozenges, then $\del_x H^*_t(z)+\bm1_{(-\infty, \fa(t)]}(z)+\bm1_{( \fb(t),\infty]}(z)=1$ and $f^*_t(z)=f^*_r(x)\in[-1,0]$.
	In all cases we have
	\begin{align}\label{e:argft2}
		\arg^* f^*_t(z)=\arg^* f^*_r(x)=-\pi (\del_x H^*_t(z)+\bm1_{(-\infty, \fa(t)]}(z)+\bm1_{( \fb(t),\infty]}(z)).
	\end{align}
	This finishes the proof of \eqref{e:extendslope}.

	The above discussion gives that for $x\in \mathscr U_t\cap \bR$ with $x<\fa(t)$, $f^*_t(x)\in (-\infty, -1)$ and for $x\in \mathscr U_t\cap \bR$ with $x>\fb(t)$, we have $f^*_t(x)\in (0,-1)$. 
	Thanks to \eqref{e:xrmap}, we have
\begin{align*}
\frac{|\Im[f^*_t(x)/(f^*_t(x)+1)]|}{|\Im x|}=\frac{1}{r-t}\geq \frac{1}{\ft'-t}\gtrsim 1.
\end{align*}
This gives the second statement of  \Cref{p:fdecompose}.

	We denote the Stieltjes transform of $\del_x H_t^*$ as $m_t^*(z)=m_t(z;H_t^*,t)$, and its modification $\tilde  m^*_t(z)$, which is the Stieltjes transform of the measure $\del_x H^*_t(x)+\bm1_{(-\infty, \fa(t)]}(x)+\bm1_{( \fb(t),\infty]}(x)$. Then for any $x\in \tilde {\mathscr U}_t\cap \bR$, \eqref{e:argft2} implies 
	\begin{align}\begin{split}\label{e:nojump}
			\lim_{\eta\rightarrow 0+}\Im[\tilde m^*_{ t }( x \pm\ri\eta)]=\mp\pi (\del_{x} H^*_t( x )+\bm1_{(-\infty, \fa(t)]}(x)+\bm1_{( \fb(t),\infty]}(x)) =\lim_{\eta\rightarrow 0+}\arg^* f^*_t(x).
	\end{split}\end{align}
	We take the ratio of $e^{\tilde m^*_{ t }(z)}$ and $ f_t(z;\beta, s)$,
	\begin{align}\label{e:gszmutcopy}
		f^*_t(z)=e^{\tilde m^*_{ t }(z)}\tilde g^*_ t (z)=e^{ m^*_{ t }(z)} g^*_ t (z),
	\end{align}
	then $\tilde g^*_ t (z)$ is well defined in the domain $\tilde {\mathscr U}_t$.  It is analytic on $\tilde {\mathscr U}_t\setminus [\fa(t), \fb(t)]$, and \eqref{e:nojump} implies that it is continuous on $[\fa(t),\fb(t)]$. Morera's theorem implies that $\tilde g^*_t(z)$ extends to an analytic function on $\tilde {\mathscr U}_t$. From the construction, $\tilde g^*_t(z)$ is real analytic, i.e. $\overline{\tilde g^*_t(z)}=\tilde g^*_t(\bar z)$, and is positive on $\tilde {\mathscr U}_t\cap \bR$. 
	Moreover, in ${\mathscr U}_t$, $f^*_t(z)$ and $e^{\tilde m^*_t(z)}$ does not have zeros or poles.
	Thus $\tilde g^*_t( z)$ does not have zeros or poles in ${\mathscr U}_t$. Using the defining relation \eqref{e:deftgt}, we conclude that $g^*_t(z)$ is a meromorphic function in the  domain $\tilde {\mathscr U}_t$. 
	It is positive on the interval $(\fa(t), \fb(t))$, has a simple pole at $\fa(t)$, has a simple zero at $\fb(t)$, and does not have other poles or zeros. This gives the third statement of  \Cref{p:fdecompose}.

	For $\beta=H_t^*$, the fourth statement of \Cref{p:fdecompose} is an empty statement.

\end{proof}

\section{Proof of Propositions  \ref{p:fdecompose}}
\label{s:proof}
In this Section we prove Proposition  \ref{p:fdecompose}.
 To study general boundary profile $\beta$, we interpolate $\beta$ with $H_t^*$. We collect facts about Shiffer kernels in Section \ref{s:shiffer}. In Section \ref{s:deformation}, we construct a deformation $f(z;\theta)$ parameterized by $0\leq \theta \leq 1$, starting from $f(z;0)=f_t^*(z)$ and show if the solution of the deformation satisfies certain properties, then $f(z;1)=f(z;\beta,t)$ as in Proposition \ref{p:fdecompose}, and it satisfies all the properties in Proposition \ref{p:fdecompose}.  Finally in
Section \ref{s:decomposition}, we analyze these deformation formulas, and conclude the proof of Proposition \ref{p:fdecompose}.

\subsection{Schiffer kernel}\label{s:shiffer}

Let $\mathscr D$ be a simply connected  domain in the complex plane $\bC$ with boundary given by an analytic curve $\mathscr C$. We further assume $\mathscr D$ is symmetric with respect to the real axis. 
Let $\phi(z)$ be the Riemann map from $\mathscr D\cap \bH^+$ to the upper half plane $\bH^+$.
By the Schwarz reflection principle, it extends to $\mathscr D$, satisfying $\phi(\bar z)=\overline{\phi(z)}$. The Green's function of $\mathscr D$ is given by
\begin{align*}
	G(z,z')=\log \big|\phi(z)-\phi(z') \big|.
\end{align*}
The Schiffer kernel $B(z,z')\rd z\rd z'$ from \cite{BE2018, schiffer1946hadamard} on $\mathscr D$ is given in terms of the Riemann map $\phi(z)$ or the derivatives of the Green's function by
\begin{align}\label{e:Schiffer}
	B(z,z')=\del_z \del_{z'}G(z,z')=\frac{\phi'(z)\phi'(z')}{\big( \phi(z)-\phi(z') \big)^2}. 
\end{align}
It is a symmetric meromorphic bilinear differential (namely, a meromorphic $1$-form in $z$ tensored by one in $z'$, satisfying $B(z,z')=B(z',z)$); it has  double poles on the diagonal, 
\begin{align}\label{e:defK}
	B(z,z')=\frac{1}{(z-z')^2}
	+ K(z,z'),
\end{align}
where $K(z,z')$ is holomorphic. 

Since the curve $\mathscr C$ is real analytic, we show that $\phi(z)$ and the Schiffer kernel $B(z,z')$ can be analytically extended to  the contour $\mathscr C$. For $z\in \mathscr C\cap \bH^+$, namely $\Im[z]>0$, this follows from that Riemann mapping function can be extended analytically to the boundary which contains an analytic arc. If $z$ is one of points where $\mathscr C\cap \bR$, then in a small neighborhood of $z$, we can parametrize $\mathscr C$ as $\gamma(t)$ for $t\in (-\varepsilon, \varepsilon)$, with $\gamma(0)=z$ and $\gamma(-t)=\overline{ \gamma(t)}$. $\gamma(t)$ is an analytic function on a small disk $\mathbb D$ centered around $0$. Then we can consider $\phi(\gamma(z))$, and extend it to $\mathbb D$ by Schwarz reflection principle; See Figure \ref{f:reflection}. Then locally around $z$,
\begin{align}\label{e:phiz}
	\phi(z+\Delta z)=\phi(z)+C(\Delta z)\Delta z^2,
\end{align}
for some analytic function $C(\Delta z)$. Using the relation \eqref{e:Schiffer}, we can also extend the Schiffer kernel to the contour $\mathscr C$, and $B(z, z')=0$ if $z$ is one of points where $\mathscr C\cap \bR$.

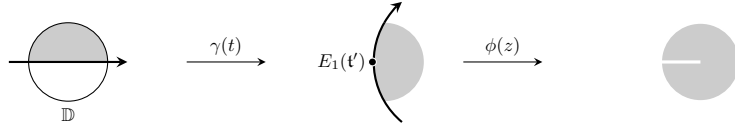
\begin{figure}
	
	\begin{center}		
		
		\begin{tikzpicture}[
			>=stealth,
			auto,
			style={
				scale = .52
			}
			]

			\fill[gray!40!white] (-3,1.5) arc (180:0:1) -- cycle;
			\draw[black] (-2,1.5) circle (1); 
			\draw[] (-2,0.5)  node[below, scale=0.7]{$\mathbb D$};
			\draw[-stealth, black, thick] (-3.5, 1.5)--(-0.5,1.5);
			
			\draw[-stealth, black] (1, 1.5)--(3,1.5);
			\draw[](2, 1.5) node[above, scale=0.7]{$\gamma(t)$};

			\fill[gray!40!white]  (6,0.5) arc (-90:90:1)--(6,2.5)arc (150:210:2);
			\draw[black, thick] (6,2.5) arc (150:230:2); 
			\draw[-stealth, black, thick] (6,2.5) arc (150:130:2);
			\draw[white, fill=black]  (5.7,1.5) circle (0.1);
			\draw[](5.7, 1.5) node[left, scale=0.7]{$E_1(\ft')$};
			

			\draw[-stealth, black] (8, 1.5)--(10,1.5);
			\draw[](9, 1.5) node[above, scale=0.7]{$\phi(z)$};

			\draw[white, thick, fill=gray!40!white] (14,1.5) circle (1); 
			\draw[white, very thick] (12,1.5)--(14,1.5);

		\end{tikzpicture}
		
	\end{center}
	
	\caption{\label{f:reflection}  
		We can extend $\phi(z)$ analytically to a neighborhood of $z$ where $\mathscr C$ intersect $\bR$, by the Schwarz reflection principle, such that $\phi(z+\Delta z)=\phi(z)+C(\Delta z)\Delta z^2,$.
	}
	
\end{figure}

\subsection{Deformation Formulas}\label{s:deformation}

To analyze $f(z;\beta,t)$ for boundary profile $\beta$ close to $H_t^*$, we interpolate $\beta$ with $H_t^*$,
\begin{align*}
	\beta^\theta=(1-\theta)\beta^0+\theta \beta^1, \quad 0\leq \theta\leq 1,\quad \beta^0=H_t^*,\quad \beta^1=\beta.
\end{align*}
We show that for boundary profile $\beta$ sufficiently close to $H_t^*$, $f(z;\beta,t)$ can be obtained by solving the following system of equations: Set $f(z;0)=f_t^*(z)$ and $\mathscr C(x;0)=\mathscr C^*(x)$ from \eqref{e:curve}. We construct their deformations as $f(z;\theta), \mathscr C(x;\theta)$ satisfying the following differential equations:
		\begin{align}\begin{split}\label{e:dthetaw}
				&\del_\theta \mathscr C(x;\theta)= F(z;\theta)\del_\theta \log f_s(z;\theta)|_{z=\mathscr C(x;\theta)}\\
				& F(z;\theta):=\frac{-(\ft'-t)f(z;\theta)/(f( z;\theta)+1) ^2)}{1+(\ft'-t)\del_z (f(z;\theta)/(f( z;\theta)+1))}.
		\end{split}\end{align}
		and
		\begin{align}\label{e:dlogf}
			\del_\theta \log f(z;\theta) =-\int_{\fa(t)}^{\fb(t)} B(z,y;\theta)( \beta^1(y)- \beta^0(y))\rd y,\quad z\in \mathscr D_\theta\setminus[\fa(t),\fb(t)].
		\end{align}
		where $\mathscr C_\theta=\{\mathscr C(x;\theta)\cup \overline{\mathscr C(x;\theta)}: x\in I_{\ft'}^*\}$, $\mathscr D_\theta$ is the region enclosed by $\mathscr C_\theta$, and $B(z,z';\theta)$ (recall from \eqref{e:Schiffer}) is the Schiffer kernel of the domain $\mathscr D_\theta$.
		We remark that if $\mathscr C_\theta$ is an analytic Jordan curve, and the righthand side of \eqref{e:dthetaw}, $F(z;\theta)\del_\theta \log f_s(z;\theta)$, is analytic in a neighborhood of $\mathscr C_\theta$, then the differential equations \eqref{e:dthetaw} and \eqref{e:dlogf} are well defined.

To understand the righthand side of \eqref{e:dthetaw},  we define a function
\begin{align}\label{e:defGzz}
	G(z,z';\theta):=\frac{B(z,z';\theta)}{1+(\ft'-t)\del_z (f(z;\theta)/(f(z;\theta)+1))},\quad z,z'\in \mathscr D_\theta\setminus[\fa(t),\fb(t)].
\end{align}
Then using \eqref{e:dlogf} and \eqref{e:defGzz}, for $z\in \mathscr C_\theta$, we have
\begin{align}\label{e:rewriteG}
	F(z;\theta) \del_\theta \log f(z;\theta)
	=\frac{(\ft'-t) f(z;\theta)}{(f(z;\theta)+1)^2}\int_{\fa(t)}^{\fb(t)} G(z,y;\theta)( \beta^1(y)- \beta^0(y))\rd y.
\end{align}
\begin{prop}\label{p:deform}
 If $\mathscr C_\theta$ is an analytic Jordan curve, and the righthand side of \eqref{e:dthetaw}, $F(z;\theta)\del_\theta \log f_s(z;\theta)$, is analytic on a closed set $\mathscr A_\theta$ containing an annulus neighborhood of $\mathscr C_\theta$, the derivative of the Schiffer kernel $B(z,z';\theta)$ with respect to $\theta$ is given by
		\begin{align}\begin{split}\label{e:bbcurve}
				\del_\theta B(z,z';\theta)
				&=-\frac{1}{2\pi\ri }\int_{\omega}  B(z,w;\theta)F(w;\theta)\del_\theta \log f(w;\theta) B(w,z';\theta) \rd w\\
				&-\bm1(\text{$z$ outside $\omega$}) \del_z(F(z;\theta)\del_\theta \log f(z;\theta) B(z,z';\theta))\\
				&-\bm1(\text{$z'$ outside $\omega$}) \del_{z'}(F(z';\theta) \del_\theta \log f_s(z';\theta)B(z,z';\theta)), \quad z,z'\in \mathscr D_\theta,
		\end{split}\end{align}
		where the contour $\omega\subseteq\mathscr A_\theta$ encloses $[\fa(t),\fb(t)]$. 
		The derivative of $G(z,z';\theta)$ with respect to $\theta$ is given by
		\begin{align}\begin{split}\label{e:derGzz}
		&\phantom{{}={}}\del_\theta G(z,z';\theta)
		=-\frac{1}{2\pi\ri }\int_{\omega}  G(z,w;\theta)F(w;\theta)\del_\theta \log f_s(w;\theta) B(w,z';\theta) \rd w\\
		&-\frac{(\ft'-t)f(z;\theta)}{(f(z;\theta)+1)^2}\del_z G(z,z';\theta)\int_{\fa(t)}^{\fb(t)} G(z,y;\theta)  ( \beta^1(y)- \beta^0(y))\rd y\\
		&-\bm1(\text{$z'$ outside $\omega$}) \del_{z'}(F(z';\theta) \del_\theta \log f_s(z';\theta)G(z,z';\theta))\\
\end{split}\end{align}
where the contour $\omega\subseteq\mathscr A_\theta$ encloses $[\fa(t),\fb(t)]$ but not $z$.
\end{prop}

\begin{proof}
	We recall the defining relation \eqref{e:Schiffer} of the Schiffer kernel, given as the derivatives of the Green's function. The variational formula of Green's function has been studied in \cite{schiffer1946hadamard} under more general conditions. By taking derivatives of the variational formula of Green's function, see \cite[Theorem 3.7]{huang2020edge},  we get for $z,z'\in \mathscr D_\theta$,
	\begin{align*}\begin{split}
			&\phantom{{}={}}\del_\theta B(z,z';\theta)
			=\frac{1}{2\pi\ri }\int_{x\in I^*_{\ft'}} B(z,\mathscr C(x;\theta);\theta)\del_\theta \mathscr C(x;\theta) B(\mathscr C(x;\theta),z';\theta)\rd \mathscr C(x;\theta)\\
			&-\frac{1}{2\pi\ri }\int_{x\in I^*_{\ft'}} B(z,\overline{\mathscr C(x;\theta)};\theta)\del_\theta \overline{\mathscr C(x;\theta)} B(\overline{\mathscr C(x;\theta)},z';\theta)\rd \overline{\mathscr C(x;\theta)}\\
			&=-\frac{1}{2\pi\ri }\int_{w\in \mathscr C_\theta} B(z,w;\theta)F(w;\theta) \del_\theta \log f(w;\theta)B(w,z';\theta)\rd w,
	\end{split}\end{align*}
	where we used 
	\eqref{e:dthetaw} for the last line. By our assumption $F(w;\theta) \del_\theta \log f(w;\theta)$ is analytic for $w\in \mathscr A_\theta$. We can get the case when $z$ or $z'$ is outside the contour $\omega$, by deforming the contour $\mathscr C_\theta$. Since the Schiffer kernel $B(z,w;\theta)$ has a double pole at $w=z$, see \eqref{e:defK}, we have
	\begin{align*}\begin{split}
			&\phantom{{}={}}-\frac{1}{2\pi\ri }\int_{ \omega} B(z,w;\theta)F(w;\theta) \del_\theta \log f(w;\theta) B(w,z';\theta)\rd w\\
			&=-\frac{1}{2\pi\ri }\int_{ \omega-} B(z,w;\theta)F(w;\theta)\del_\theta \log f(w;\theta) B(w,z';\theta)\rd w-\del_z(F(z;\theta)\del_\theta \log f(z;\theta) B(z,z';\theta)),
	\end{split}\end{align*}
	where the contour $\omega$ encloses $[\fa(t),\fb(t)]$ and $z,z'$, and the contour $\omega-$ encloses $[\fa(t),\fb(t)]$ and $z'$ but not $z$. If we further deform the contour $\omega-$ to not enclose $z'$, we will have one more term on the righthand side $-\del_{z'}(F(z';\theta)\del_\theta \log f(z';\theta) B(z,z';\theta))$. This gives \eqref{e:bbcurve}.

The derivative of $G(z,z';\theta)$ as in \eqref{e:defGzz} consists of two terms. If the derivative hits the numerator, 
\begin{align}\begin{split}\label{e:num}
		&\phantom{{}={}}\frac{\del_\theta B(z,z';\theta)}{1+(\ft'-t)\del_z (f(z;\theta)/(f(z;\theta)+1))}=-\frac{1}{2\pi\ri }\int_{\omega}  G(z,w;\theta)F(w;\theta)\del_\theta \log f(w;\theta) B(w,z';\theta) \rd w\\
		&-\bm1(\text{$z'$ outside $\omega$}) \del_{z'}(F(z';\theta) \del_\theta \log f(z';\theta)G(z,z';\theta))- \frac{\del_z(F(z;\theta) \del_\theta \log f(z;\theta)B(z,z';\theta))}{1+(\ft'-t)\del_z (f(z;\theta)/(f(z;\theta)+1))},
\end{split}\end{align}
where $z$ is outside the contour $\omega$. We recall the expression of $F(z;\theta)$ from \eqref{e:dthetaw}, we can rewrite the last term on the righthand side of \eqref{e:num} as
\begin{align}\begin{split}\label{e:num2}
		&\frac{\del_z((\ft'-t) \del_\theta  (f(z;\theta)/(f(z;\theta)+1))G(z,z';\theta))}{1+(\ft'-t)\del_z (f(z;\theta)/(f(z;\theta)+1))}\\
		&=\frac{(\ft'-t) \del_z\del_\theta  (f(z;\theta)/(f(z;\theta)+1))G(z,z';\theta))}{1+(\ft'-t)\del_z (f(z;\theta)/(f(z;\theta)+1))}+\frac{(\ft'-t) \del_\theta  (f(z;\theta)/(f(z;\theta)+1))\del_z G(z,z';\theta)}{1+(\ft'-t)\del_z (f(z;\theta)/(f(z;\theta)+1))}.
\end{split}\end{align}
If the derivative with respect to $\theta$ hits the denominator of $G(z,z';\theta)$ in \eqref{e:defGzz} we get
\begin{align}\label{e:deno}
	-G(z,z';\theta) \frac{(\ft'-t)\del_\theta\del_z (f(z;\theta)/(f(z;\theta)+1))}{1+(\ft'-t)\del_z (f(z;\theta)/(f(z;\theta)+1))}.
\end{align}
It follows from combining \eqref{e:num}, \eqref{e:num2} and \eqref{e:deno}, and noticing the cancellation between \eqref{e:num2} and \eqref{e:deno}, we get 
\begin{align*}\begin{split}
		&\phantom{{}={}}\del_\theta G(z,z';\theta)
		=-\frac{1}{2\pi\ri }\int_{\omega}  G(z,w;\theta)F(w;\theta)\del_\theta \log f(w;\theta) B(w,z';\theta) \rd w\\
		&-\frac{(\ft'-t)f(z;\theta)}{(f(z;\theta)+1)^2}\del_z G(z,z';\theta)\int_{\fa(t)}^{\fb(t)} G(z,y;\theta)  ( \beta^1(y)- \beta^0(y))\rd y\\
		&-\bm1(\text{$z'$ outside $\omega$}) \del_{z'}(F(z';\theta) \del_\theta \log f_s(z';\theta)G(z,z';\theta))\\
\end{split}\end{align*}
where we used \eqref{e:rewriteG} for the last term in \eqref{e:num2}. 

\end{proof}

Next we show that if the solution of  \eqref{e:dthetaw}, \eqref{e:dlogf} exists up to $\theta=1$, the $f(z;1)=f(z;\beta,t)$ as in Proposition  \ref{p:fdecompose} and Proposition  \ref{p:fdecompose} holds.

\begin{prop}\label{p:fdec2}
We assume the solution of  \eqref{e:dthetaw}, \eqref{e:dlogf} exists up to $\theta=1$, and for 
$0\leq \theta\leq 1$, $\mathscr C_\theta$ is analytic Jordan curve enclosing the interval $[\fa(t), \fb(t)]$, and it holds $|\del_z^k K(z,z';\theta)|\lesssim 1$ for any $0\leq k\leq 2$, $z\in \mathscr D_\theta$, $z'$ in a small neighborhood of $[\fa(t),\fb(t)]$.  Then $f(z;1)=f(z;\beta,t)$, and Proposition  \ref{p:fdecompose} holds.
\end{prop}

We have the following lemma which gives general conditions for a function $f(z)$ to be the limiting complex slope of lozenge tilings on a strip.
\begin{lem}\label{l:construct}
Let $\mathscr D$ be a simply connected  domain in the complex plane $\bC$ with boundary given by an analytic curve $\mathscr C$. We further assume $\mathscr D$ is symmetric with respect to the real axis. There exists a function $f(z)$ the following holds
\begin{enumerate}
\item $f(z)$ is analytic on $\mathscr D\setminus \bR$, $f(\bar z)=\overline{f(z)}$ and $\Im[f(z)]<0$ for $z\in \mathscr D\cap \bH^+$.
\item Let $\chi(z)=f(z)/(f(z)+1)$. $f(z)$ and $\chi(z)$ extend analytically to the boundary curve $\mathscr C$. Moreover, there exists some $t>0$, such that $\Im[z]+t\Im[\chi(z)]=0$ for $z\in \mathscr C$. 
\end{enumerate}
Then we have
 $f(z)$ is the limiting complex slope of lozenge tiling on the strip $[0,t]$, with height function at time $0$ given by $ -\pi \del_x H_0(x)=\arg^* f(x+\ri 0)$ almost surely for $x\in \mathscr D\cap \bR$, and at time $t$ given by $-\pi\del_x H_t(x)|_{x=z+t\chi(z)}=\arg^* f(z)$ with $z\in \mathscr C\cap \bH^+$.
\end{lem}

\begin{proof}[Proof of \Cref{p:fdec2}]
We check $f(z;1)$ satisfies the assumptions in Lemma \ref{l:construct}.
If the solution of  \eqref{e:dthetaw}, \eqref{e:dlogf} exists up to $\theta=1$, then $f(z,\theta)$ is analytic on $\mathscr D_\theta\setminus \bR$, extends analytically to the boundary curve $\mathscr C_\theta$, and $f(\overline z,\theta)=\overline{f(z,\theta)}$.

To check the second assumption in Lemma \ref{l:construct}, we compute
\begin{align*}\begin{split}
&\phantom{{}={}}\del_\theta \left(\mathscr C(x;\theta)+(\ft'-t)\frac{f(\mathscr C(x;\theta);\theta)}{f(\mathscr C(x;\theta);\theta)+1}\right)\\
&=\del_\theta \mathscr C(x;\theta)\left(1+(\ft'-t)\left.\del_z \frac{f(z;\theta)}{f(z;\theta)+1}\right|_{z=\mathscr C(x;\theta)} \right)
+(\ft'-t)\left.\del_\theta \frac{f(z;\theta)}{f(z;\theta)+1}\right|_{z=\mathscr C(x;\theta)} =0,
\end{split}\end{align*}
where we used \eqref{e:dthetaw} and \eqref{e:dlogf}.
Thus we have
\begin{align}\label{e:Cfix}
\mathscr C(x;\theta)+(\ft'-t)\frac{f(\mathscr C(x;\theta);\theta)}{f(\mathscr C(x;\theta);\theta)+1}=\mathscr C(x;0)+(\ft'-t)\frac{f(\mathscr C(x;0);0)}{f(\mathscr C(x;0);0)+1}=x.
\end{align}
where in the last inequality, we used the definition of $\mathscr C(x;0)=\mathscr C^*(x)$ from \eqref{e:curve}. By taking $\theta=1$, $f(z;\theta)$ satisfies the second assumption in Lemma \eqref{l:construct}.

Next we show that $\Im[ \log f(\mathscr C(x;\theta); \theta)]$ does not change.
\begin{align}\begin{split}\label{e:fftt}
&\phantom{{}={}}\del_\theta \log f(\mathscr C(x;\theta); \theta)
=\del_\theta \log f(z;\theta)|_{z=\mathscr C(x;\theta)}+\del_z \log f(z;\theta)|_{z=\mathscr C(x;\theta)}\del_\theta \mathscr C(x;\theta)\\
&=\frac{\del_\theta \log f(z;\theta)|_{z=\mathscr C(x;\theta)}}{1+(\ft'-t)\del_z(f(z;\theta)/(f(z;\theta)+1))|_{z=\mathscr C(x;\theta)}}
=\frac{-\int B(\mathscr C(x;\theta),y;\theta)(\beta^1(y)-\beta^0(y))\rd y}{1+(\ft'-t)\del_z(f(z;\theta)/(f(z;\theta)+1))|_{z=\mathscr C(x;\theta)} }.
\end{split}\end{align}
By taking derivative with respect to $x$ on both sides of \eqref{e:Cfix} and rearranging we have 
\begin{align*}
\del_x \mathscr C(x;\theta)=\frac{1}{1+(\ft'-t)\del_z(f(z;\theta)/(f(z;\theta)+1))|_{z=\mathscr C(x;\theta)}}.
\end{align*}
We recall the definition of Schifer kernel from \eqref{e:Schiffer},
\begin{align}\begin{split}\label{e:fftt2}
\frac{B(\mathscr C(x;\theta),y;\theta)}{1+(\ft'-t)\del_z(f(z;\theta)/(f(z;\theta)+1))|_{z=\mathscr C(x;\theta)}}
&=\frac{\phi'(\mathscr C(x;\theta))\phi'(y)}{(\phi(\mathscr C(x;\theta))-\phi(y;\theta))^2}\del_x \mathscr C(x;\theta)\\
&=\del_x \frac{\phi'(y;\theta)}{(\phi(y;\theta )-\phi(\mathscr C(x;\theta);\theta))}\in \bR,
\end{split}\end{align}
where $\phi(z;\theta)$ is the Riemann mapping from $\mathscr D_\theta\cap \bH^+$ to $\bH^+$, and we used that it maps the boundary curve to real, namely $\phi(\mathscr C(x;\theta);\theta),\phi(y;\theta)\in \bR$.
We conclude by plugging \eqref{e:fftt2} into \eqref{e:fftt} that 
\begin{align*}
\del_\theta \Im[ \log f(\mathscr C(x;\theta); \theta)]=0,
\end{align*}
and 
\begin{align*}
\Im[ \log f(\mathscr C(x;1); 1)]=\Im[ \log f(\mathscr C(x;0); 0)]=-\pi H^*_{\ft'}(x)\in [-\pi,0].
\end{align*}

Next we show that 
$
\lim_{z\in \bH^+\rightarrow \mathscr D_1\cap \bR} \Im[\log f(z;1)]\in[-\pi,0].
$
Then it follows that $\Im[f(z;1)]\leq 0$ on the boundary of $\mathscr D_1\cap \bH^+$, and $\Im[f(z;1)]<0$ for $z\in \mathscr D_1\cap \bH^+$ by the strong maximum principle of $\Im[f(z;1)]$. Thus $f(z;1)$ also satisfies the first assumption in Lemma \ref{l:construct}.

We recall the decomposition of the Schiffer kernel from \eqref{e:defK},		
\begin{align}\begin{split}\label{e:decomp}
\del_\theta \log f(z;\theta) 
&=-\int B(z,y;\theta)( \beta^1(y)- \beta^0(y))\rd y\\
&=-\int \left(\frac{1}{(z-y)^2}+K(z,y;\theta)\right)( \beta^1(y)- \beta^0(y))\rd y\\
&=-\int \left(K(z,y;\theta)\right)( \beta^1(y)- \beta^0(y))\rd y+\int\frac{\del_y \beta^1(y)-\del_y\beta^0(y)}{z-y}.
\end{split}\end{align}
We recall the decomposition of $f_t^*(z)=f(z;0)$ from \eqref{e:gszmutcopy}. By integrating \eqref{e:decomp} from $0$ to $1$, we get
\begin{align}\begin{split}\label{e:decomp2}
 \log f(z;1) 
 &=m(z;\beta,t)
+ \log g^*_t(z)-\int_0^1 \int \left(K(z,y;\theta)\right)( \beta^1(y)- \beta^0(y))\rd y\rd \theta\\
&=\tilde m(z;\beta,t)
+ \log \tilde g^*_t(z)-\int_0^1 \int \left(K(z,y;\theta)\right)( \beta^1(y)- \beta^0(y))\rd y\rd \theta,
\end{split}\end{align}
where $m(z;\beta,t),\tilde m(z;\beta,t)$ are Stieltjes transform of $\del_y\beta^1=\del_y\beta$, and $\del_y\beta+\bm1_{(-\infty, \fa(t)]}(y)+\bm1_{( \fb(t),\infty]}(y))$ respectively. The last two terms on the righthand side of \eqref{e:decomp2} are real analytic, and the density $\del_y\beta+\bm1_{(-\infty, \fa(t)]}(y)+\bm1_{( \fb(t),\infty]}(y))$ is bounded by $1$, 
we conclude that 
$
\lim_{z\in \bH^+\rightarrow \mathscr D_1\cap \bR} \Im[\log f(z;1)]\in[-\pi,0].
$

So far we have checked that $f(z;1)$ satisfies the assumptions of Lemma \ref{l:construct}, it is the limiting complex slope of lozenge tiling on the strip $[t, \ft']$ (by shifting time by $t$), with the height function at time $t$ given by $\beta$, and at time $\ft'$ given by $H_{\ft'}$. In particular $f(z;1)=f(z;\beta,t)$ as in Proposition \ref{p:fdecompose}. The first statement in Proposition \ref{p:fdecompose} follows by taking  $\tilde {\mathscr U}_t^\beta=\mathscr D_1$ and $\mathscr U_t^\beta=\mathscr D_1\setminus \{x:(x,t)\in \overline{\fL(\fD^t;\beta)}\}$. The second statement  in Proposition \ref{p:fdecompose} follows from the decomposition \eqref{e:decomp}, and the maximum principle of $\Im[f(z;1)/(f(z;1)+1)]/\Im[z]$. For the third statement in Proposition \ref{p:fdecompose}, we take
\begin{align}\begin{split}\label{e:chooseg}
& \log \tilde g(z;\beta,t)= \log \tilde g^*_t(z)-\int_0^1 \int \left(K(z,y;\theta)\right)( \beta(y)- H_t^*(y))\rd y\rd \theta\\
&\log g(z;\beta,t)= \log  g^*_t(z)-\int_0^1 \int \left(K(z,y;\theta)\right)( \beta(y)- H_t^*(y))\rd y\rd \theta,
\end{split}\end{align} 
as in \eqref{e:decomp2}.

Finally for the fourth statement in Proposition \ref{p:fdecompose}, we can rewrite \eqref{e:chooseg} as a contour integral
	\begin{align}\begin{split}\label{e:lngs}
			&\phantom{{}={}}\log g(z;\beta,t)  - \log g(z;H_t^*,t)
			=-\int_{0}^1\left(\int  K(z,y;\theta)( \beta(y)- H_t^*(y)) \rd y\right)\rd \theta\\
			&= \int_{0}^1\left(\int \int_{z'=\fa(t)}^ {z'=y} K(z,z';\theta)\rd y (\del_y \beta(y)- \del_yH_t^*(y)) \rd y\right)\rd \theta\\
			&=\int_{0}^1  \left(\frac{1}{2\pi\ri}\oint_\omega  \int_{z'=\fa(t)}^{z'=w} K(z,z';\theta)\rd y( m(w;\beta,t)- m_t^*(w)) \rd w\right)\rd \theta,
	\end{split}\end{align}
	where $\omega$ is a contour enclosing $[\fa(t),\fb(t)]$.
	The claim \eqref{e:lngsbound} follows from taking derivative with respect to $z$, then taking absolute value on both sides of \eqref{e:lngs}, and using that $|\del_z^k K(z,z';\theta)|\lesssim1$.
	\end{proof}

\begin{proof}[Proof of Lemma \ref{l:construct}]
One can check that $\chi(z)$ is analytic on $\mathscr D\setminus \bR$, $\chi(\bar z)=\overline{\chi(z)}$ and $\Im[\chi(z)]<0$ for $z\in \mathscr D\cap \bH^+$. By Nevanlinna representation, $\Im[\chi(x+\ri 0)]$ defines a negative measure $\mu$ on $I\subset \mathscr D\cap \bR$,
\begin{align}\label{e:dec}
\chi(z)=\int\frac{\rd \mu(x)}{z-x}+\xi(z),
\end{align}
where $\xi(z)$ is real analytic. In particular, $\chi(z)$ extends analytically to $\mathscr D\setminus I$. 

For any $z\in \mathscr D\setminus I$, we define
\begin{align}\label{e:recov}
s(z)=-\frac{\Im[z]}{\Im[\chi(z)]},\quad x(z)=z+s(z) \chi(z)\in \bR, \quad f_{s(z)}(x(z))=f(z),
\end{align}
we remark that for $z\in (\mathscr D\cap \bR)\setminus I$, we should interpret the first statement in \eqref{e:recov} as $s(z)=-1/\chi'(z)$.
Next we show that the map \eqref{e:recov} restricted to $\bH^+\cup \bR$, $z\mapsto (x(z), s(z))$ is an injection, thus it gives the desired limiting complex slope of lozenge tiling on the strip $[0,t]$. 

We first notice that $\Im[\chi(z)]/\Im[z]$ satisfies a strong maximum principle inside $\mathscr D\setminus I$. Let $z=x+\ri y$ and $v(x,y)=\Im[\chi(z)]$, then  $v(x,y)$ is harmonic and for $|y|\neq 0$, $v(x,y)/y$ satisfies the elliptic equation
\begin{align*}
\Delta \frac{v(x,y)}{y}= \frac{\Delta v(x,y)}{y}-2\left(\frac{\del_y v(x,y)}{y^2}-\frac{v(x,y)}{y^3}\right)=\frac{2}{y}\del_y \frac{ v(x,y)}{y}.
\end{align*} 
Thus the strong maximum principle holds for $|y|\neq 0$. Next we show that $z=x\in (\mathscr D\setminus I)\cap \bR$ is not a local extrema. We can directly Taylor expand $\chi(z)$ around $x$, $\chi(x+\Delta x+\Delta y\ri)=a_0+a_1(\Delta x+\Delta y\ri)+a_2(\Delta x+\Delta y \ri)^2+\cdots$ with $a_i\in \bR$. Say $a_k$ is the first nonvanishing coefficient with $k\geq 2$, then taking $\Delta x=\la \Delta y$
\begin{align*}\begin{split}
\frac{\Im[\chi(x+\Delta x+\Delta y\ri)]}{\Delta y}
&=a_1+\frac{a_k \Im [(\la \Delta y+ \Delta y\ri)^k]}{\Delta y}+\OO(|\Delta y|^k)\\
&=a_1+\Im [(\la+ \ri)^k]a_k\Delta y^{k-1}+\OO(|\Delta y|^k).
\end{split}\end{align*}
For $k\geq 2$, there are choices of $\la$ such that $\Im [(\la+ \ri)^k]$ is positive or negative. We conclude that $z=x\in (\mathscr D\cap \bR)\setminus I$ is not a local extrema.  

By our assumption for $z\in \mathscr C$, $\Im[\chi(z)]/\Im[z]=-1/t$ and for $z\rightarrow I$ (the support of $\mu$), the decomposition \eqref{e:dec} gives that $\Im[\chi(z)]/\Im[z]\rightarrow-\infty$. Thus by maximum principle we have $\Im[\chi(z)]/\Im[z]\in (-\infty, -1/t)$ for $z\in \mathscr D\setminus I$. In particular $s(z)\in (0,t)$ as in \eqref{e:recov} for any $z\in \mathscr D\setminus I$.

For any $0<s<t$, let $\mathscr C_s:=\{z\in \mathscr D\setminus I: s=s(z)\}$ be the boundary curve of the domains
\begin{align}\label{e:D_s+-}
\mathscr D_s^{-}:= \{z\in \mathscr D\setminus I: -\Im[z]/\Im[\chi(z)]<s\},\quad \mathscr D_s^{+}:= \{z\in \mathscr D\setminus I: -\Im[z]/\Im[\chi(z)]>s\}.
\end{align}
Thanks to the maximum principle, $\mathscr D_s^{+}$ is connected, in fact each point in $\mathscr D_s^{+}$ is connected to the boundary curve $\mathscr C$.
$\mathscr D_s^{-}$ may consist of many disconnected components, but each component of $\mathscr D_s^{-}$ contains part of $I$.

The set $\mathscr C_s$ is characterized by $\Im[z+s\chi(z)]=0$. The critical points are given by $1+s\del_z \chi(z)=0$, which consists of  isolated points. So $\mathscr C_s$ consists of piecewise analytic curves. Next we show the critical points are on the real axis. Otherwise, assume we have a critical point at $z_c$ with $\Im[z_c]>0$. By symmetric, $\bar z_c$ is also a critical point, such $\Im[z_c+s\chi(z_c)]=0$ and $1+s\del_z \chi(z_c)=0$. Then locally around $z_c$, we have $z+s\chi(z)=z_c+s\chi(z_c)+a_2(z-z_c)^2+a_3(z-z_c)^3+\cdots$, and $\Im[z+s\chi(z)]=\Im[a_2(z-z_c)^2+a_2(z-z_c)^2+\cdots]$. We take a small disk $\bD$ centered around $z_c$, the set $\mathscr C_s$ divides it into at least four pieces, on each piece, the sign of $s+\Im[z]/\Im[\chi(z)]$ alternates. The same statement holds locally around $\bar z_c$.  See Figure \ref{f:critical}. We recall $\mathscr D_s^{-}, \mathscr D_s^{+}$ from \eqref{e:D_s+-}. Since each connected component of $\mathscr D_s^{-}$ contains part of $I$, the two regions marked by ``$*$" in Figure \ref{f:critical} belong to the same connected component of $\mathscr D_s^{-}$, so do the two regions marked by ``$-$". They divide $\mathscr D_s^{+}$ into at least two disconnected components, which leads to a contradiction.

	\begin{figure}
		
		\begin{center}		
			
			\begin{tikzpicture}[
				>=stealth,
				auto,
				style={
					scale = .52
				}
				]

				\draw[black]  (8,1.5) ellipse (4 and 1.5 );
				
				\draw[black, dashed]  (7.56,1.5) ellipse (0.6 and 1.2 );
				\draw[black, dashed]  (3.8,1.5)--(12.2, 1.5); 
				\draw(6.7,0.7)--(7.3, 1.3);
				\draw(6.7,1.3)--(7.3, 0.7);
				\draw(6.7,2.3)--(7.3, 1.7);
				\draw(6.7,1.7)--(7.3, 2.3);
				\draw[] (7,2)node[above, scale=0.7]{$-$};
				\draw[] (7,2)node[below, scale=0.7]{$*$};
				\draw[] (7,2)node[left, scale=0.7]{$z_c$};
				\draw[] (7,1)node[above, scale=0.7]{$*$};
				\draw[] (7,1)node[below, scale=0.7]{$-$};
				\draw[] (7,1)node[left, scale=0.7]{$\bar z_c$};
				\draw[] (5,2.4) node[above, scale=0.7]{$\mathscr C$};
				\draw[] (5,2) node[scale=0.7]{$\mathscr D$};

				\draw[black]  (18,1.5) ellipse (4 and 1.5 );
				
				\draw[] (15,2.4) node[above, scale=0.7]{$\mathscr C$};
				\draw[] (15,2) node[scale=0.7]{$\mathscr D_s^+$};
				
				\draw[black,  fill=gray!40!white]  (17,1.5) ellipse (1 and 0.8 );
				\draw[black,  fill=gray!40!white]  (20,1.5) ellipse (0.8 and 0.7 );
				\draw[black, dashed]  (13.8,1.5)--(22.2, 1.5); 
				
				\draw[] (17,1.5) node[below, scale=0.7]{$\mathscr D_s^-$};
				\draw[] (20,1.5) node[below, scale=0.7]{$\mathscr D_s^-$};
				\draw[] (18.2,1.5) node[above, scale=0.7]{$\mathscr C^1_s$};
				\draw[] (21,1.5) node[above, scale=0.7]{$\mathscr C^2_s$};

				\draw[black, thick] (24, 0) arc (180:110:3.2);
				\draw[black, thick] (30.5, 0) arc (-30:30:3);
				\draw[black, thick] (26.5, 0) arc (-60:0:1.732);
				\draw[black, thick] (28.25, 0) arc (240:180:1.732);		
			
				\draw[black, dashed] (24, 1) node[left, scale=0.7]{$s$} --(31, 1) ;
				
					\draw[] (24.7,1) node[below, scale=0.7]{$x_1(s)$};
				\draw[] (26.5,1) node[below, scale=0.7]{$x'_1(s)$};
				\draw[] (28.3,1) node[below, scale=0.7]{$x_2(s)$};
				\draw[] (30.2,1) node[below, scale=0.7]{$x'_2(s)$};

			\end{tikzpicture}
			
		\end{center}
		
		\caption{\label{f:critical} The first plot illustrates that the critical points $z_c$ must be on the real axis. Shown to the second and third plots,  $\mathscr C_s:=\{z\in \mathscr D\setminus I: s=s(z)\}$ consists of contours, $\mathscr C_s^1, \mathscr C_s^2,\cdots, \mathscr C_s^r$, each surrounds a connected component of $\mathscr D_s^{-}$. $\mathscr C_s^k\cap \bH^+$ are mapped bijectively to disjoint intervals, corresponding the slice of liquid region at time $s$. }
	\end{figure}
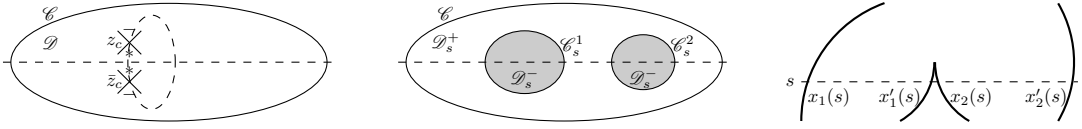

We conclude that all the critical points are on the real axis, and $\mathscr C_s:=\{z\in \mathscr D\setminus I: s=s(z)\}$ consists of contours, each surrounds a connected component of $\mathscr D_s^{-}$. These contours can only touch each other on the real axis. We order them from left to right as $\mathscr C_s^1, \mathscr C_s^2,\cdots, \mathscr C_s^r$. 
For any $z\in \mathscr C_s^k\cap \bH^+$, we parametrize $\mathscr C_s^k$ around $z$ as $\gamma: (-\varepsilon, \varepsilon)\mapsto  \mathscr C_s^k$ with $\gamma(0)=z$ in clockwise direction. Then the direction of $\gamma'(0)$ is given by $1/(1+s\del_z\chi(z))$, and $\gamma'(0)=c/(1+s\del_z\chi(z))$ with some $c>0$. We recall $x(z)=z+s\chi(z)$ from \eqref{e:recov}, $\del_\tau x(\gamma(\tau))|_{\tau=0}=c>0$. We conclude that $x(z)$ maps each arc $\mathscr C_s^k \cap (\bH^+\cup \bR)$ bijectively to an interval $[x(w_k(s)), x(w_k'(s))]=[x_k(s), x'_k(s)]$, where $w_k(s), w_k'(s)$ are the left and right end point of the arc $\mathscr C_s^k \cap (\bH^+\cup \bR)$. Next we show that $x'_k(s)\leq x_{k+1}(s)$. 
If $w_k'(s)=w_{k+1}(s)$, then $x'_k(s)= x_{k+1}(s)$. Otherwise, for any  $z\in (w_k'(s), w_{k+1}(s))\in \mathscr D_s^+$, the maximum principle gives that $\chi'(z)>-1/s$. It follows that the map  $z\mapsto z+s\chi(z)$ is  increasing,  since $(z+s\chi(z))'=1+s\chi'(z)>0$. We conclude that the map $z\mapsto z+s\chi(z)$ maps the interval $[w_k'(s), w_{k+1}(s)]$ bijectively to the interval $[x'_k(s), x_{k+1}(s)]$. In particular $x_{k+1}(s)>x'_k(s)$.

It follows the map $z\mapsto (s(z), x(z))$ in \eqref{e:recov} restricted to  $(\bH^+\cup \bR)$ is an injection, which gives a liquid region $\fL$, and the slice at time $s$ is given by the intervals $[x_1(s), x_1'(s)]\cup \cdots [x_r(s), x_r'(s)]$.  See Figure \ref{f:critical}. \eqref{e:recov} also gives a complex slope $f_s(x)$ for any $(x,s)\in \overline\fL$, characterized by
\begin{align}
f_s(x)=f\left(x-s\frac{f_s(x)}{f_s(x)+1}\right),
\end{align}
which is essentially \eqref{q0f}. Moreover, the complex slope $f_s(x)$ extends outside the liquid region: for any $x\in (x'_k(s), x_{k+1}(s))$, there exists some $z\in(w_k'(s), w_{k+1}(s))$ such that $x=z+s\chi(z)$. We can simply take $f_s(x)=f(z)$. 

The height function $H_s(x)$ is determined by $-\pi\del_x H_s(x)=\arg^*(f_s(x))$ and $\pi\del_s H_s(x)=\arg^*(f_s(x)+1)$ for $(x,s)\in \fL$. It also extends continuously outside the liquid region $\fL$.
It follows from \cite[Theorem 8.3 and Remark 8.6]{DMCS}, that $H_s(x)$ is the limiting height function of lozenge tiling on the strip $[0,t]$ with complex slope given by $f_s(x)$. Moreover, at time $0$ the height function is given by $\del_x H_0(x)=-\pi\arg^* f(x+\ri 0)$ almost surely for $x\in \mathscr D\cap \bR$, and at time $t$ given by $\del_x H_t(x)|_{x=z+t\chi(z)}=-\pi\arg^* f(z)$ with $z\in \mathscr C\cap \bH^+$.

\end{proof}

\subsection{Proof of \Cref{p:fdecompose}}\label{s:decomposition}

The system of differential equations for $\mathscr C(x;\theta)$ and $\log f(z;\theta)$  in \eqref{e:dthetaw} and \eqref{e:dlogf} can be analyzed using the ``method of majorant" originated from Cauchy. In the following we recall some notations from \cite[Section 2]{lax1953nonlinear}. A formal power series in $k$ veriables: $A=\sum_{i_r\geq0} a_{i_1i_2\cdots i_k}y_1^{i_1}\cdots y_k^{i_k}$ is said to majorize another one
$B=\sum_{0\leq i_r< \infty} b_{i_1i_2\cdots i_k}(x_1-\tilde x_1)^{i_1}\cdots (x_k-\tilde x_k)^{i_k}$ at $(\tilde x_1, \cdots, \tilde x_k)$ denoted by $B\subset_{(\tilde x_1, \cdots, \tilde x_k)}A$ if $|b_{i_1i_2\cdots i_k}|\leq a_{i_1i_2\cdots i_k}$ for all $0\leq i_1, i_2,\cdots, i_k$. The following rules will be used repeatedly without mentioning: 
(i) If $B_i\subset_{\tilde \bmx} A_i$, then $\sum B_i\subset_{\tilde \bmx} \sum A_i$, and $\prod B_i\subset_{\tilde \bmx} \prod A_i$;
(ii) Let $x_i(s_1, \cdots, s_\ell)$ and $y_i(t_1, \cdots, t_\ell)$, $1\leq i\leq k$, be power series with $x_i(s_1, \cdots, s_\ell)\subset_{\tilde \bms} y_i(t_1, \cdots, t_\ell)$, and $(x_1(\tilde \bms), \cdots x_k(\tilde \bms))=\tilde \bmx$. Let $A$ and $B$ be power series in $k$ variables satisfying $B\subset_{\tilde \bmx}A$. Then it holds that 
$B(x_1(s_1,\cdots, s_\ell), \cdots, x_k(s_1,\cdots, s_\ell))\subset_{\tilde s}A(y_1(t_1,\cdots, t_\ell), \cdots, y_k(t_1,\cdots, t_\ell))$.

The following power series play important roles in \cite[Section 2]{lax1953nonlinear},
\begin{align}\label{e:introf0}
	f^0(u)=\sum_{j\geq 0}\frac{u^j}{(j+1)^2},\quad f^1(u)=1+\sum_{j\geq 1}\frac{u^j}{j^3},\quad f^0(u)=\del_u f^1(u).
\end{align}
Any convergent power series is majorized by $Mf^0(ru)$ and  $Mf^1(ru)$, if $M$ and $r$ are chosen sufficiently large. The power series $f^0(u), f^1(u)$ satisfy
\begin{align}\label{e:multiply}
	\del_u f^0(u)f^0(u)\subset_0 4 \del_u f^0(u),\quad   f^0(u)f^0(u)\subset_0 4 f^0(u), \quad f^1(u)f^1(u)\subset_0 4 f^1(u). 
\end{align}
More importantly, composition also preserves them,
\begin{align}\label{e:composition}
	f^0(au f^1(bu))\subset_0 f^0(bu),\quad
	f^0(au f^0(bu))\subset_0 f^0(bu),
\end{align}
provided that $4a\leq b$. In the rest of this section, we will use $u,u', v, v'$ for formal variables.

Fix a small constant $\fr>0$. We recall that $\mathscr C_0=\{\mathscr C^*(x) \cup \overline{\mathscr C^*(x)}: x\in I_{\ft'}^*\}$ from \eqref{e:curve} and its annulus neighborhood of $\mathscr A_0$ from Proposition \ref{p:ftiinitial}. We denote $\mathscr C^0=\mathscr C_0$ and take more contours $\mathscr C^1,\mathscr C^2, \mathscr C^3\subseteq \mathscr A_0$, such that $\mathscr C^i$ encloses $\mathscr C_{i+1}$. They are distance at least $\fr$ bounded away from each other, and the interval $[\fa(t), \fb(t)]$.  See Figure \ref{f:cont}.

\begin{figure}
	
	\begin{center}		
		
		\begin{tikzpicture}[
			>=stealth,
			auto,
			style={
				scale = .52
			}
			]

			\draw[black] (-8, 0) node[left, scale = .7]{$y = s$}-- (8, 0);
			\draw[black, very thick] (-2, 0) node[below, scale=0.7]{$\fa(t)$} -- (2, 0)node[below, scale=0.7]{$\fb(t)$};

			\draw[black] (0,0) ellipse (7 and 3);
			\draw[black, dashed] (0,0) ellipse (7.2 and 3.2);
			\draw[black, dashed] (0,0) ellipse (6.8 and 2.8);
			\draw (5.8,0.6) node[right, scale=0.7]{$\mathscr C^{0-}$};
			\draw (6.7,1.2) node[right, scale=0.7]{$\mathscr C^{0+}$};

			\draw[black, dashed] (0,0) ellipse (6 and 2.5);
			
			\draw[black, dashed] (0,0) ellipse (5 and 2);
			\draw[black, dashed] (0,0) ellipse (4 and 1.5);
			
			\draw (0,4) node[below,scale=0.7]{$\mathscr C_0=\{\mathscr C^*(x) \cup \overline{\mathscr C^*(x)}: x\in I_{\ft'}^*\}$};
			\draw (-6,0.9) node[scale=0.7]{$\mathscr C^1$};
			\draw (-5,0.8) node[scale=0.7]{$\mathscr C^2$};
			\draw (-4,0.7) node[scale=0.7]{$\mathscr C^3$};
			

			%
			
		\end{tikzpicture}
		
	\end{center}
	
	\caption{\label{f:cont} We take contours $\mathscr C^0=\{\mathscr C^*(x) \cup \overline{\mathscr C^*(x)}: x\in I_{\ft'}^*\}$ (recall from \eqref{e:curve}) and $\mathscr C^1,\mathscr C^2, \mathscr C^3$ in an annulus neighborhood of $\mathscr C^0$, such that $\mathscr C^i$ encloses $\mathscr C^{i+1}$. They are distance at least $\fr$ bounded away from each other, and the interval $[\fa(t), \fb(t)]$. 
	We take $\mathscr C^{0\pm}$ to be in a radius $1/(10r)$ neighborhood of $\mathscr C^{0}$, and distance $1/(100r)$ away from $\mathscr C^{0}$, where $r$ is from \eqref{e:initial}. We define $\mathscr A^\pm$ be the closure of the annulus between the contours $\mathscr C^{0\pm}$ and $\mathscr C^{2}$. $\mathscr C_\theta$ will stay between $\mathscr C^{0\pm}$.
	}
	
\end{figure}
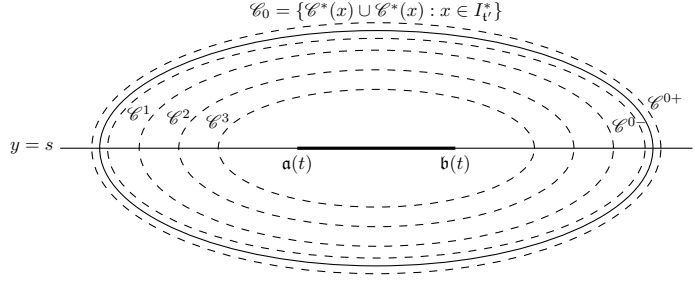

The equations \eqref{e:dlogf}, \eqref{e:bbcurve} and \eqref{e:derGzz} form a closed system for $\log f(z;\theta)$ with $z\in \mathscr C^i, 0\leq i\leq 2$, 
$B(z,z'; \theta)$ with $(z,z')\in \mathscr C^i\times \mathscr C^j, 0\leq i<j\leq 3$,
and $G(z,z';\theta)$ with $(z,z')\in \mathscr C^0\times \mathscr C^j, 1\leq j\leq 3$. More precisely, we can rewrite \eqref{e:dlogf}  as a contour integral
\begin{align}\label{e:dlogf2}
	\del_\theta \log f(z;\theta) =-\frac{1}{2\pi\ri}\oint_{\mathscr C^3} B(z,w;\theta) \int_{\fa(s)}^{\fb(s)} \frac{ \beta^1(x)- \beta^0(x)}{w-x}\rd x\rd w,\quad z\in \mathscr C^i,\quad 0\leq i\leq 2.
\end{align}
For $0\leq i<j\leq 3$, we take $0\leq k\leq 2$ such that $k\neq i,j$. Then for $(z,z')\in \mathscr C^i\times \mathscr C^j$, we can rewrite \eqref{e:bbcurve} as
\begin{align}\begin{split}\label{e:bbcurve2}
		\del_\theta B(z,z';\theta)
		&=-\frac{1}{2\pi\ri }\int_{\mathscr C^k}  B(z,w;\theta)F(w;\theta)\del_\theta \log f(w;\theta) B(w,z';\theta) \rd w\\
		&-\bm1(\text{$z$ outside ${\mathscr C^k}$}) \del_z(F(z;\theta) \del_\theta \log f(z;\theta)B(z,z';\theta))\\
		&-\bm1(\text{$z'$ outside ${\mathscr C^k}$}) \del_{z'}(F(z';\theta) \del_\theta \log f(z';\theta)B(z,z';\theta)).
\end{split}\end{align}
For $(z,z')\in \mathscr C^0\times \mathscr C^j$ with $1\leq j\leq 3$, we take $k\in\{1,2\}$ such that $k\neq j$. Then for $(z,z')\in \mathscr C^0\times \mathscr C^j$, we can rewrite \eqref{e:derGzz} as
\begin{align}\begin{split}\label{e:derGzzcopy}
		&\phantom{{}={}}\del_\theta G(z,z';\theta)
		=-\frac{1}{2\pi\ri }\int_{\mathscr C^k}  G(z,w;\theta)F(w;\theta)\del_\theta \log f(w;\theta) B(w,z';\theta) \rd w\\	
		&-\frac{(\ft'-t)}{2\pi\ri}\frac{f(z;\theta)}{(f(z;\theta)+1)^2}\del_z G(z,z';\theta)\oint_{\mathscr C^3} G(z,w;\theta) \int_{\fa(s)}^{\fb(s)} \frac{ \beta^1(x)- \beta^0(x)}{w-x}\rd x\rd w\\
		&-\bm1(\text{$z'$ outside ${\mathscr C^k}$}) \del_{z'}(F(z';\theta) \del_\theta \log f(z';\theta)G(z,z';\theta)),
\end{split}\end{align}
where we rewrite  the second term in  \eqref{e:derGzz}  as a contour integral. 

We notice that if $j=3$, then $z'\in \mathscr C^3$ is inside $\mathscr C^k$, and the last terms in \eqref{e:bbcurve2} and \eqref{e:derGzzcopy} vanish. Thus \eqref{e:bbcurve2} and \eqref{e:derGzzcopy} only involves $F(z;\theta)$ for $z\in \mathscr C^i$ with $0\leq i\leq 2$, which are functions of $f(z;\theta)$ for $z\in \mathscr C^i$ with $0\leq i\leq 2$.

We can use $f^0(u), f^1(u)$ introduced in \eqref{e:introf0} to majorize $\log f(z;0)$ and $B(z, z';0)$: there exist large constants $M, r>0$ such that
\begin{align}\label{e:initial}
	\log f(z;0) \subset_{\tilde z} \frac{M}{20 r^2} f^1(r u),\quad 
	B(z, z';0)\subset_{\tilde z, \tilde z'} \frac{M}{2} f^0(r u) f^0(r u').
\end{align}
for any $(\tilde z,\tilde z')\in \mathscr C^i\times \mathscr C^j$ with $0\leq i<j\leq 3$. 
We remark that $\tilde z, \tilde z'$ are bounded away from each other, namely $|\tilde z-\tilde z'|\geq \fr$, we are away from the double poles of $B(z, z';0)$ along the diagonal.

We recall $G(z,z';0)$ from \eqref{e:defGzz}. For $\theta=0$, $f(z;\theta)=f_t^*(z)$, from Proposition \ref{p:ftiinitial}, its denominator has only two zeros at $z=\mathscr C^*(E_1(\ft')), \mathscr C^*(E_2(\ft'))$. From the discussion after \eqref{e:phiz}, its numerator $B(z,z';0)$ also has two zeros at $z=\mathscr C^*(E_1(\ft')), \mathscr C^*(E_2(\ft'))$. They cancel out, and $G(z,z';0)$ is analytic for $(z,z')\in \mathscr C^0\times \mathscr C^i$ with $1\leq i\leq 3$.  The same as in \eqref{e:initial}, we can use  $f^0(u)$ to majorize $G(z,z';0)$ as in \eqref{e:initial} 
\begin{align}\label{e:initialG}
	G(z, z';0)\subset_{\tilde z, \tilde z'} \frac{M}{2} f^0(r u) f^0(r u'), 
\end{align}
for any $(\tilde z, \tilde z')\in \mathscr C^0\times \mathscr C^j$ with $1\leq j\leq 3$. 

To use the deformation formula Proposition \ref{p:deform}, we need that the righthand side of \eqref{e:dthetaw} is analytic in a neighborhood of $\mathscr C_\theta= \{\mathscr C(x;\theta)\cup \overline{\mathscr C(x;\theta)}: x\in I_{\ft'}^*\}$. Later we will see that $\mathscr C_\theta$ will stay close to the contour $\mathscr C^0$. For this we define two more contours $\mathscr C^{0+}$ and $\mathscr C^{0-}$, such that $\mathscr C^{0+}$ encloses $\mathscr C^{0}$, and 
$\mathscr C^{0}$ encloses $\mathscr C^{0-}$; See Figure \ref{f:cont}. Moreover, we take $\mathscr C^{0\pm}$ to be in a radius $1/(10r)$ neighborhood of $\mathscr C^{0}$, and distance $1/(100r)$ away from $\mathscr C^{0}$, where $r$ is from \eqref{e:initial} and \eqref{e:initialG}. We define $\mathscr A^\pm$ be the closure of the annulus between the contours $\mathscr C^{0\pm}$ and $\mathscr C^{2}$, and will take $\mathscr A_\theta=\mathscr A^+$ in Proposition \ref{p:deform}. 

Later, we will show that $G(z, z';\theta)$ is analytic for $(\tilde z, \tilde z')\in \mathscr C^0\times \mathscr C^j$ with $1\leq j\leq 3$, this will guarantee that the righthand side of \eqref{e:dthetaw} is analytic in a radius $\OO(1/r)$ neighborhood of $\mathscr C_0$. 
For $z$ between the two contours $ \mathscr C^{0-}, \mathscr C^2$, we will show that 
$1+(\ft'-t)\del_z (f(z;\theta)/(f(z;\theta)+1))$ is not $0$. 

Thanks to Proposition \ref{p:ftiinitial}, there exists a large constant $\fC$, 
	\begin{align}\label{e:ftiinitialcopy}
	1/\fC\leq |f(z;0)|\leq \fC, \quad |f(z;0)+1|\geq 1/\fC,\quad  |\del_z f(z;0)|\leq \fC, \quad z\in \mathscr C^i,\quad 0\leq i\leq 2,
	\end{align}
	and (by increasing $\fC$ if necessary), we also have
	\begin{align}\label{e:ftiinitialcopy2}
	|1+(\ft'-t)\del_z (f(z;0)/(f(z;0)+1))|\geq 1/\fC, \quad z\in \mathscr C^{1}\cup \mathscr C^2,
\end{align}
since $\mathscr C^1, \mathscr C^2$ are bounded away from $\mathscr C^0$.

We recall $F(z;\theta)$ from \eqref{e:dthetaw}, we can rewrite it as a function of $\log f(z;\theta)$ and $\del_z\log  f(z;\theta)$
\begin{align*}
	F(z;\theta)=\tilde F(\log f(z;\theta),\del_z\log  f(z;\theta)),\quad \tilde F(x,y)=\frac{-(\ft'-t)e^x/(1+e^x)^2}{1 +(\ft'-t)ye^x/(1+e^x)^2}.
\end{align*}
For any $(\tilde x, \tilde y)= (\log f(z;0),\del_z\log  f(z;0))$ with $z\in \mathscr C^2$, $\tilde F(x, y)$ is analytic in a small neighborhood of $(\tilde x, \tilde y)$. There exist small $\fb>0$ and large  $r'>0, M>0$ (we will take $r$ in \eqref{e:initial} much bigger than $r'$) such that
\begin{align}\label{e:Fexp}
	\tilde F(x,y)\subset_{\tilde x,\tilde y}M f^0(r'u) f^0(r'v),
\end{align}
for any $(\tilde x, \tilde y)$ satisfying $|\tilde x-\log f(\tilde z;0)|, |\tilde y-\del_z\log  f(\tilde z;0))|\leq \fb$ for some $\tilde z\in \mathscr C^i$, $i\in\{1,2\}$.

Next we analyze the system of equations \eqref{e:dlogf2},\eqref{e:bbcurve2} and \eqref{e:derGzzcopy}, show the assumptions of Proposition \ref{p:fdec2} hold, and Proposition \ref{p:fdecompose} follows. 

\begin{proof}[Proof of Propositon \ref{p:fdecompose}]
	We recall the contour $\mathscr C^0, \mathscr C^1,\mathscr C^2, \mathscr C^3, \mathscr C^{0\pm}$; see Figure \ref{f:cont}. Define $r_\theta=r e^\theta$ (where $r$ is from \eqref{e:initial} and \eqref{e:initialG}) for $0\leq \theta\leq 1$, and $\mathscr A^\pm$ the closure of the annulus between $\mathscr C^{0\pm}$ and $\mathscr C^2$. We prove that for $\fc$ small enough, the following holds for $0\leq\theta \leq 1$,
	\begin{enumerate}
		\item The righthand side of \eqref{e:dthetaw} is analytic on $\mathscr A^+$.
		\item For any $z\in \mathscr C^i$ with $0\leq i\leq 2$,
		\begin{align*}
		1/2\fC\leq |f(z;\theta)|\leq 2\fC, \quad |f(z;\theta)+1|\geq 1/2\fC, \quad
		|\del_zf(z;\theta)|\leq 2\fC, 
		\end{align*}
		and for any $z\in \mathscr C^1\cup \mathscr C^2$:
		\begin{align}\label{e:boundaryc}
			|1+(\ft'-t)\del_z (f(z;\theta)/(f(z;\theta)+1))|\geq 1/2\fC.
		\end{align}
		\item For $z\in \mathscr C^i$ with $0\leq i\leq 2$, it holds
		\begin{align}\begin{split}\label{e:fshift}
				|\log f(z;\theta)-\log f(z;0)|, 
				|\del_z \log f(z;\theta)-\del_z\log f(z;0)|
				\leq \fb,
		\end{split}\end{align}
		where $\fb$ is from \eqref{e:Fexp}.
		\item
		For any $(\tilde z,\tilde z')\in \mathscr C^i\times \mathscr C^j$ with $0\leq i<j\leq 3$,
		\begin{align}\label{e:fBexp}
			\log f(z;\theta) \subset_{\tilde z} (M/r^2_\theta) f^1(r_\theta u),\quad 
			B(z, z';\theta)\subset_{\tilde z, \tilde z'} M f^0(r_\theta u) f^0(r_\theta u'),
		\end{align}
		and for any $(\tilde z,\tilde z')\in \mathscr C^0\times \mathscr C^j$ with $1\leq j\leq 3$,
		\begin{align}\label{e:Gexp}
			G(z, z';\theta)\subset_{\tilde z, \tilde z'} M f^0(r_\theta u) f^0(r_\theta u').
		\end{align}
	\end{enumerate}
	In the following proof, $C$ will represent some constant which may differ from line to line.
	It may depend on  $r,r',M,\fb, \fr, \fC, (\ft'-t)$ and the length of the contours $\mathscr C^i$, but independent of $\fc$. 
	We denote $\sigma=\sigma(\beta)$ the first time any of the above conditions fails. 
	Then for $\theta\leq \sigma$ and $\tilde z\in \mathscr C^i$ with $0\leq i\leq 2$, we have from \eqref{e:fBexp} 
	\begin{align}\begin{split}\label{e:logfdiff}
			&\log f(z;\theta)-\log f(\tilde z;\theta)\subset_{\tilde z} (M/r_\theta)uf^1(r_\theta u),\quad 
			\del_z\log f(z;\theta)\subset_{\tilde z} (M/r_\theta)f^0(r_\theta u),\\ 
			&\del_z \log f(z;\theta)-\del_z\log f(\tilde z;\theta)\subset_{\tilde z}  M u f^0(r_\theta u).
	\end{split}\end{align} 
	For $\theta\leq \sigma$,  we have $|f(z;\theta)+1|\geq 1/2\fC$. Thus for $\tilde z\in \mathscr C^i$ with $0\leq i\leq 2$, \eqref{e:fBexp}  also implies that
	\begin{align}\label{e:frak}
		\frac{f(z;\theta)}{(f(z;\theta)+1)^2}\subset_{\tilde z} Cf^0(r_\theta u),
	\end{align}
	where we view the lefthand side as a function of $\log f(z;\theta)$, composed by $e^x/(e^x+1)^2$, and used \eqref{e:composition}; the constant $C$ depends on $\fC$. 
	
	It follows from plugging  \eqref{e:fBexp} into \eqref{e:dlogf2}, we get
	\begin{align}\begin{split}\label{e:deltheta2}
			&\phantom{{}={}}\del_\theta \log f(z;\theta)
			=-\frac{1}{2\pi\ri}\oint_{\mathscr C^3} B(z,w;\theta) \int_{\fa(s)}^{\fb(s)} \frac{ \beta^1(x)- \beta^0(x)}{w-x}\rd x\rd w\\
			&\subset_{\tilde z} M f^{(0)}(r_\theta u)\max_{w\in \mathscr C^3}\oint_{\mathscr C^3}\left|\int_{\fa(s)}^{\fb(s)} \log (w-x)(\del_x\beta^1(x)- \del_x\beta^0(x))\rd x\right|\rd w
			\subset_{0} \fc C  f^{(0)}(r_\theta u),
	\end{split}\end{align} 
	where for the last inequality, we used our assumption that $d(\del_x \beta^1, \del_x \beta^0)=d(\del_x \beta, \del_x H_s^*)\leq \fc$, and $w\in \mathscr C^3$ is distance $\fr$ away from $[\fa(t),\fb(t)]$. So $\log(w-x)$ as a function of $x$, is Lipschitz with Lipschitz constant $1/\fr$, and the integral is bounded by $\OO(\fc/\fr)$; the constant $C$ depends on $\fr,M$ and the contour $\mathscr C^3$.

	Thanks to \eqref{e:Fexp}, \eqref{e:fshift} and  \eqref{e:logfdiff}, for $\theta\leq \sigma$ and $\tilde z\in \mathscr C^1\cup \mathscr C^2$ we have 
	\begin{align*}\begin{split}
			&\phantom{{}={}}F(z;\theta)
			= \tilde F(\log f(z;\theta),\del_z\log  f(z;\theta))\\
			&\subset_{\tilde z} M f^0(r'(\log f(z;\theta)-\log f(\tilde z;\theta)))
			f^0(r'(\del_z\log  f(z;\theta)-\del_z\log  f(\tilde z;\theta)))\\
			&\subset_{\tilde z} M f^0(z r' (M/r_\theta)f^1(r_\theta u)) f^0(z r' M  f^0(r_\theta u)) 
	\end{split}\end{align*}
	Then \eqref{e:multiply} and \eqref{e:composition} implies that for $\theta\leq \sigma$ and $\tilde z\in \mathscr C^1\cup \mathscr C^2$
	\begin{align}\label{e:Fbound}
		F(z;\theta)
		\subset_{\tilde z} M f^0(r_\theta u),\quad F(z; \theta)\del_\theta \log f(z;\theta)\subset_{\tilde z} \fc C Mf^0(r_\theta u)f^0(r_\theta u)\subset_{0}4\fc C Mf^0(r_\theta u),
	\end{align}
	provided that  $r_\theta=r e^\theta\geq 4r'M$. 
	For $\tilde z\in \mathscr C^0$, we can write $F(z; \theta)\del_\theta \log f(z;\theta)$ in terms of $G(z;w;\theta)$ as in \eqref{e:rewriteG}:
\begin{align*}
	F(z;\theta) \del_\theta \log f(z;\theta)
	=\frac{1}{2\pi\ri}\frac{(\ft'-t) f(z;\theta)}{(f(z;\theta)+1)^2}	\oint_{\mathscr C^3} G(z,w;\theta) \int_{\fa(s)}^{\fb(s)} \frac{ \beta^1(x)- \beta^0(x)}{w-x}\rd x\rd w.
\end{align*}
	Using  \eqref{e:Gexp} and \eqref{e:frak}, the same as for \eqref{e:deltheta2}, we get for $\tilde z\in \mathscr C^0$,
	\begin{align}\label{e:Fbound2}
		F(z; \theta)\del_\theta \log f(z;\theta)\subset_{\tilde z}
		\fc C f^0(r_\theta u),
	\end{align}
	where $C$ depends on $\fC, \fr, (\ft'-t)$.

	Next can estimate $\del_\theta B(z,z';\theta)$ using the expression \eqref{e:bbcurve2}, for the first term on the righthand side of \eqref{e:bbcurve2}, using \eqref{e:multiply}, \eqref{e:fBexp},  \eqref{e:Fbound} and \eqref{e:Fbound2}, we get for any $(\tilde z, \tilde z')\in \mathscr C^i\times \mathscr C^j$ with $0\leq i<j\leq 3$,
	\begin{align*}\begin{split}
			&\phantom{{}={}}\frac{1}{2\pi\ri }\int_{\mathscr C^k}  B(z,w;\theta)F(w;\theta)\del_\theta \log f(w;\theta) B(w,z';\theta) \rd w\\
			&\subset_{\tilde z,\tilde z'}
			f^0(r_\theta u)  f^0(r_\theta u')\int_{\mathscr C^k}4\fc C M^3 |\rd w| 
			= 4\fc C M^3 |\mathscr C^k | f^0(r_\theta u)f^0(r_\theta u'),
	\end{split}\end{align*}
	where $|\mathscr C^k|$ is the length of the contour. For the second term on the righthand side of \eqref{e:bbcurve2}, similarly we have
	\begin{align*}\begin{split}
			&\phantom{{}={}}\bm1(\text{$z$ outside ${\mathscr C^k}$}) \del_z(F(z;\theta) \del_\theta \log f(z;\theta)B(z,z';\theta))\\
			&\subset_{\tilde z,\tilde z'} \del_u( 4 \fc C  f^{(0)}(r_\theta u)M^2 f^0(r_\theta u) f^0(r_\theta u'))=16\fc C M^2\del_u f^0(r_\theta u)  f^0(r_\theta u').
	\end{split}\end{align*}
	We have the same estimates for the last term on the righthand side of \eqref{e:bbcurve2}. Combining them together, 
	we have the following upper bound of $\del_\theta B(z,z';\theta)$  for any $(\tilde z, \tilde z')\in \mathscr C^i\times \mathscr C^j$ with $0\leq i<j\leq 3$,
	\begin{align}\begin{split}\label{e:bbcurve3}
			\del_\theta B(z,z';\theta)\subset_{\tilde z,\tilde z'} \fc C( f^0(r_\theta u)f^0(r_\theta u') +\del_ u f^0(r_\theta u) f^0(r_\theta u')+ f^0(r_\theta u) \del_ {u'}f^0(r_\theta u')  ),
	\end{split}\end{align}
	where $C$ depends on $\fr, \fC, M$ and the contours.
	
	By the same argument as for \eqref{e:bbcurve3}, we can analyze the differential equation \eqref{e:derGzzcopy} of $\del_\theta G(z,z';\theta)$ for any $(\tilde z, \tilde z')\in \mathscr C^0\times \mathscr C^j$ with $1\leq j\leq 3$,
	\begin{align}\label{e:dGbb0}
		\del_\theta G(z,z';\theta)\subset_{\tilde z,\tilde z'} \fc C( f^0(r_\theta u)f^0(r_\theta u') +\del_ u f^0(r_\theta u) f^0(r_\theta u')+ f^0(r_\theta u) \del_ {u'}f^0(r_\theta u')  ).
	\end{align}
	
	We recall that $r_\theta=r e^\theta$. By integrating \eqref{e:deltheta2} from $\theta=0$ to $\sigma$,  we get for any $\tilde z\in \mathscr C^i$ with $0\leq i\leq 2$,
	\begin{align}\begin{split}\label{e:logf2}
			\log f(z;\sigma)-\log f(\tilde z;0)
			&\subset_{\tilde z}
			\int_0^{\sigma} \fc C f^{(0)}(r_\theta u)\rd \theta
			= \fc C\int_0^{\sigma} \sum_{j\geq 0}\frac{e^{\theta j} (r u)^j}{(j+1)^2}\rd \theta\\
			&= \fc C\left(\sigma+ \sum_{j\geq 1}\frac{(e^{\sigma j}-1) (r u)^j}{j(j+1)^2}\rd \theta\right)
			\subset_{0} 2\fc Cf^1(r_{\sigma} u).
	\end{split}\end{align}
	Similarly by taking derivative on both sides of \eqref{e:deltheta2}, then integrating from $0$ to $\sigma$, for any $\tilde z\in \mathscr C^i$ with $0\leq i\leq 2$ we have
	\begin{align}\label{e:derlogf2}
		\del_z\log f(z;\sigma)-\del_z\log f(\tilde z;0)
		\subset_{\tilde z}2\fc C\del_u f^1(r_{\sigma} u)\subset_{0}2er\fc C f^0(r_{\sigma} u).
	\end{align}
	By integrating \eqref{e:bbcurve3} from $\theta=0$ to ${\sigma}$, we get for any $(\tilde z, \tilde z')\in \mathscr C^i\times \mathscr C^j$ with $0\leq i<j\leq 3$,
	\begin{align}\begin{split}\label{e:bbcurve4}
			&\phantom{{}={}}B(z,z';\sigma)- B(\tilde z,\tilde z';0)\\
			&\subset_{\tilde z,\tilde z'} 
			\fc C \int_0^{\sigma}(f^0(r_\theta u) f^0(r_\theta u')+ \del_ u f^0(r_\theta u) f^0(r_\theta u')+ f^0(r_\theta u) \del_ {u'}f^0(r_\theta u')  )\rd \theta\\
			&= 
			\fc C \int_0^{\sigma}\sum_{j,k\geq 0}\frac{e^{\theta(j+k)}(ru)^j (ru)^k+re^{\theta(j+k)}j(ru)^{j-1} (ru)^k+re^{\theta(j+k)}(ru)^{j} k(ru)^{k-1}}{(j+1)^2(k+1)^2}\rd \theta\\
			& \subset_{0,0}
			3r\fc C f^0(r_{\sigma} u) f^0(r_{\sigma} u').
	\end{split}\end{align}
	Similarly by integrating \eqref{e:dGbb0} from $\theta=0$ to ${\sigma}$, we get for any $(\tilde z, \tilde z')\in \mathscr C^0\times \mathscr C^i$ with $1\leq i\leq 3$,
	\begin{align}\begin{split}\label{e:dGbb1}
			G(z,z';\sigma)- G(\tilde z,\tilde z';0)
			\subset_{\tilde z, \tilde z'}
			3r\fc C f^0(r_{\sigma} u) f^0(r_{\sigma} u').
	\end{split}\end{align}

	The estimates \eqref{e:logf2} and \eqref{e:derlogf2} imply that for any $ z\in \mathscr C^i$ with $0\leq i\leq 2$, we have 
	\begin{align}\label{e:fzsmall}
		|\log f(z;\sigma)-\log f(z; 0)|\leq 2\fc C,\quad
		|\del_z \log f(z;\sigma)-\del_z \log f(z; 0)|\leq 2er\fc C.
	\end{align}
	Thus if we take $\fc$ sufficiently small, \eqref{e:fzsmall} together with the initial estimates \eqref{e:ftiinitialcopy} give that 
	\begin{align}\begin{split}\label{e:mid}
			&|f(z;\sigma)-f(z;0)|\leq 1/3\fC,\quad 
			|\del_z f(z;\sigma)-\del_z f(z;0)|\leq \fC/3,\\
			&\left|\frac{\del_z f(z;\sigma)}{(f(z;\sigma)+1)^2}-\frac{\del_z f(z;0)}{(f(z;0)+1)^2}\right|\leq \frac{1}{3(\ft'-t)\fC}.
	\end{split}\end{align}
	Thus \eqref{e:mid} and \eqref{e:ftiinitialcopy} give for any $z\in \mathscr C^i$ with $0\leq i\leq 2$, 
	\begin{align}\begin{split}\label{e:fullbb1}
			&1/2\fC< |f(z;\sigma)|< 2\fC, \quad |f(z;\sigma)+1|> 1/2\fC,\quad | \del_z \log f(z;\sigma)|< 2\fC.
	\end{split}\end{align}
	And \eqref{e:mid} and \eqref{e:ftiinitialcopy2} implies that for any $z\in \mathscr C^1\cup \mathscr C^2$ 
	\begin{align}\label{e:ftbb}
		|1+(\ft'-t)\del_z (f(z;\sigma)/(f(z;\sigma)+1))|> 1/2\fC.
	\end{align}

	We in fact have \eqref{e:ftbb} for any $\theta\leq \sigma$, this implies that as we deform from $\theta=0$ to $\theta =\sigma$, no zero of $1+(\ft'-t)\del_z (f(z;\theta)/(f(z;\theta)+1))$ enters $\mathscr A^-$ from its inner boundary curve $\mathscr C^2$. 
	Thanks to the estimate \eqref{e:Fbound2}, for any $\theta \leq \sigma$, $F(z; \theta)\del_\theta \log f(z;\theta)$ defines an analytic field in a radius $1/(2er)$ tube neighborhood of $\mathscr C^0$. 
	We recall the differential equation of $\mathscr C(x;\theta)$ from \eqref{e:dthetaw}, 
	then
	\begin{align}\label{e:delw}
		|\del_\theta \mathscr C(x;\theta)|=\left|F(\mathscr C(x;\theta); \theta)\del_\theta \log f(\mathscr C(x;\theta);\theta)\right|\leq \fc M C,
	\end{align}
	for any $\mathscr C(x;\theta)$ inside this radius $1/(2er)$ tube neighborhood of $\mathscr C^0$.
	By integrating both sides of \eqref{e:delw} from $\theta=0$ to $\theta\leq \sigma$, we get
	\begin{align}\label{e:fullbb0}
		|\mathscr C(x;\theta)-\mathscr C(x;0)|\leq \fc M C\leq 1/(200r), 
	\end{align}
	provided we take $\fc\leq \fr/200rMC$. We conclude that $\mathscr C_\theta=\{\mathscr C(x;\theta)\cup \overline{\mathscr C(x;\theta)}:x\in I_{\ft'}^*\}$ stays in a radius $1/(200r)$ tube neighborhood of $\mathscr C^0$. In particular it stays in between the contours $\mathscr C^{0+}$ and $\mathscr C^{0-}$ (recall from Figure \ref{f:cont}). 
	
	The above discussion also implies that no zero of $1+(\ft'-t)\del_z (f(z;\theta)/(f(z;\theta)+1))$ enters $\mathscr A^-$ from its outer boundary $\mathscr C^{0-}$.  In fact, from \eqref{e:Gexp}, $G(z,z';\theta)$ is analytic for $\dist(z, \tilde z)\leq 1/(2er)$, where $\tilde z\in \mathscr C^0$, and $z'\in \mathscr C^3$ . In particular, it is analytic for $z\in \mathscr C^{0-}$. We recall the expression of $G(z,z';\theta)$ from \eqref{e:defGzz}, for $z\in \mathscr C^{0-}, z'\in \mathscr C^3$, \eqref{e:fullbb0} implies that both $z,z'$ are inside the contour $\mathscr C_\theta$, and the Schiffer kernel $B(z,z';\theta)\neq 0$ from its definition \eqref{e:Schiffer}. Thus the denominator $1+(\ft'-t)\del_z (f(z;\theta)/(f(z;\theta)+1))$ in $G(z,z';\theta)$ is nonzero for $z\in \mathscr C^{0-}$, and $0\leq \theta\leq \sigma$.

	 We conclude that $1+(\ft'-t)\del_z (f(z;\theta)/(f(z;\theta)+1))\neq 0$ for any $\mathscr A^-$. Together with \eqref{e:Fbound2}, we conclude that the righthand side of \eqref{e:dthetaw} is analytic on $\mathscr A^+$ for $0\leq \theta \leq \sigma$.
	 
	 By comparing \eqref{e:logf2}, \eqref{e:bbcurve4}, \eqref{e:dGbb1} with \eqref{e:fBexp} and \eqref{e:Gexp}, and using the estimates of the initial data \eqref{e:initial}, \eqref{e:initialG} and the estimate \eqref{e:fzsmall}, if we take $2er\fc C \leq \fb/2$ and $3\fc C\leq M/(2e^2r^2)$, we have for any $(\tilde z,\tilde z')\in \mathscr C^i\times \mathscr C^j$ with $0\leq i<j\leq 3$,
	\begin{align}\begin{split}\label{e:fullbb2}
			&\log f(z;\sigma) \subset_{\tilde z} (M/2r^2_{{\sigma}}) f^1(r_{\sigma} u),\quad 
			B(z, z';\sigma)\subset_{\tilde z, \tilde z'} (M/2) f^0(r_{\sigma} u) f^0(r_{\sigma} u'),\\
			&|\log f(\tilde z;\sigma)-\log f(\tilde z;0)|, 
			|\del_z \log f(\tilde z;\sigma)-\del_z\log f(\tilde z;0)|
			\leq \fb/2. 
	\end{split}\end{align}
	and for $(\tilde z,\tilde z')\in \mathscr C^0\times \mathscr C^i$ with $1\leq i\leq 3$,
	\begin{align}\begin{split}\label{e:fullbb3}
			G(z, z';\sigma)\subset_{\tilde z, \tilde z'} (M/2) f^0(r_{\sigma} u) f^0(r_{\sigma} u').
	\end{split}\end{align}
	We conclude from  \eqref{e:fullbb1}, \eqref{e:ftbb} and \eqref{e:fullbb2} and \eqref{e:fullbb3} that $\sigma=1$, provided we take $\fc$ small enough.

	Therefore 
	if the boundary data $\beta$ is sufficiently close to $H_t^*(x)$ such that $d(\del_x \beta, H_t^*)\leq \fc$, then $\sigma=1$. The map as constructed from \eqref{e:delw} using the analytic field $F(z; \theta)\del_\theta \log f(z;\theta)$
	\begin{align*}
		x\in I_{\ft'}^*\mapsto \mathscr C(x;0)\mapsto \mathscr C(x; 1),
	\end{align*}
	gives an analytic Jordan curve. Moreover, $|\del_z^k B(z,z';\theta)|\lesssim 1$ uniformly for any $(z,z')\in \mathscr C^0\times \mathscr C^1$. 
	We recall the defining relation of $K(z,z';\theta)$ from \eqref{e:defK}. For $(z,z')\in \mathscr C^0\times \mathscr C^1$, we have $|z-z'|\geq \fr$, and 
	$|\del_z^k K(z,z';\theta)|\lesssim |\del_z^k B(z,z';\beta,s)|+1/|z-z'|^{2+k}\lesssim 1$. Since $\del_z^k K(z,z';\theta)$ is analytic for $z,z'$ inside the contours $\mathscr C^0$, by the maximal principle, we conclude that $|\del_z^k K(z,z';\theta)|\lesssim 1$ for any $0\leq k\leq 2$ and $z$ inside the contour $\mathscr C^0$, $z'$ inside the contour $\mathscr C^1$. 
	Thus the assumptions in Proposition \ref{p:fdec2} hold, and Proposition \ref{p:fdecompose} follows.

\end{proof}

\section{Proof of \Cref{p:ftzbehave}}

\label{QEquation}

In this section we establish  \Cref{p:ftzbehave}, which are quantitative properties of the complex Burgers equation that are used throughout this paper. 
Throughout, we without loss of generality restrict ourselves to $z \in \overline{\mathbb{H}^+}$ (namely, $\Im z \geq 0$), for the proof when $z \in \mathbb{H}^-$ follows from conjugation. Then, 
	\begin{flalign*} 
		\Im[f^*_t(z)]\leq 0; \qquad \arg^* \big(f^*_t(z) \big)=\Im \big[\log f_t^*(z) \big]; \qquad \sin \big(\Im [\log f_t^*(z)] \big)= \displaystyle\frac{\Im \big[ f_t^*(z) \big]}{\big|f_t^*(z) \big|}.
	\end{flalign*} 
	
	\noindent Recall the decompositions from \eqref{e:decompft*}  and \eqref{e:decompft*2}, given by
	\begin{align}\label{e:decompft*3}
		\log f_t^*(z)=m_t^*(z)+\log g_t^*(z),\quad 
		\log f_t^*(z)=\tilde m_t^*(z)+\log\tilde g_t^*(z),
	\end{align}
	where $\log\tilde g_t^*(z)$ is uniformly bounded and real analytic in a neighborhood of $[\fa(t),\fb(t)]$, and $\log g_t^*(z)= \log \tilde g_t^*(z)-\log \big(z-\fa(t) \big)+\log \big(\fb(t)-z \big)$. In particular, uniform boundedness and analyticity of  $\log \tilde{g}_t^* (z)$ imply that the derivative of $\log\tilde g_t^*(z)$ is also uniformly bounded. 
	Thus we have $\Im[\log \tilde g_t^*(z)]=\Im[\log \tilde g_t^*(\Re z)]+\OO(|z-\Re z|)=\OO(\Im z)$, where the last equality is from that $\log \tilde g_t^*(z)$ is real analytic. 	
	Together with \eqref{e:decompft*3}, we get 
	\begin{flalign} 
		\label{mtft}
		\Im \big[ \tilde m_t^*(z) \big]=\Im \big[\log f_t^*(z) \big]+\OO(\Im z).
	\end{flalign} 

	\noindent We will see that in all the cases the error term $\OO(\Im[z])$ is negligible.  
	
	Further recall from Item (d) of the third part of \Cref{xhh} that locally around $(x,r)$ (which was mapped to $(f_t^*(z),z)=(f_r^*(x), x+(t-r)f_r^*(x)/ (f_r^*(x)+1))$ under the map \eqref{e:xrmap}, there exists a real analytic function $Q_0$ in one variable such that
	\begin{align}\label{e:q0fcopy}
		Q_0 \big(f^*_r(x) \big)=Q_0 \big(f^*_t(z) \big)=x \big(f^*_r(x)+1 \big)-r f^*_r(x)=z\big(f^*_t(z)+1 \big)-t f^*_t(z).
	\end{align}
%
	The first statement \eqref{e:ft+1b2} of the proposition follows from \eqref{e:ft+1b}. 
	
	Next, let 
	\begin{flalign*} 
		\chi_t^*(z)= \displaystyle\frac{f^*_t(z)}{f^*_t(z)+1}; \qquad Q_1(u)=(1-u)Q_0 \bigg( \displaystyle\frac{u}{1 - u} \bigg).
	\end{flalign*} 

	\noindent We can rewrite \eqref{e:q0fcopy} as an equation of $\chi_t^*(z)$:
	\begin{align}
		\label{chiq1} 
		Q_1 \big(\chi^*_t(z) \big)=z-t \chi^*_t(z).
	\end{align}
	At point $(E(t),t)$ on the arctic curve $\fA(\fD')$, $f_t^*(E(t))$ is  a double root of $Q_0 \big(f^*_t(E(t)) \big)=E(t) \big(f^*_t(E(t))+1 \big)-t f^*_t(E(t))$ and are thus characterized by the equations
	\begin{align}\label{chiq2}
		Q_1 \Big(\chi^*_t \big( E(t) \big) \Big)=E(t)-t \chi^*_t \big(E(t) \big); \qquad Q_1' \Big( \chi^*_t \big( E(t) \big) \Big)=-t.
	\end{align}

	\noindent Recalling that $\big( E(t), t \big)$ is a cusp singularity of the arctic boundary if and only if $\chi_t^* (z)$ is a triple root of \eqref{chiq1} (and that there are no quadruple or higher roots of \eqref{chiq1}), we have that 
	\begin{flalign}
		\label{qchi0}
		\begin{aligned} 
		\bigg| Q_1'' \Big( \chi_t^* \big( E(t) \big) \Big) \bigg| \asymp 1, \qquad & \text{if $\big( E(t), t \big)$ is bounded away from a cusp}; \\
		\bigg| Q_1'' \Big( \chi_t^* \big( E(t) \big) \Big) \bigg| = 0 \quad \text{and} \quad \bigg| Q_1''' \Big( \chi_t^* \big( E(t) \big) \Big) \bigg| \asymp 1, \qquad & \text{if $\big( E(t), t \big) = \big( E(t_c), t_c \big)$ is a cusp}.
		\end{aligned} 
	\end{flalign}

	Now let us establish the three parts of \Cref{p:ftzbehave}.

	\begin{proof}[Proof of the first statement of \Cref{p:ftzbehave}]
	
	By the first statement of \eqref{qchi0}, if $\big( E(t), t \big)$ is bounded away from the cusp point, then $\chi_t^*(z)$ has square root behavior in a small neighborhood of $E(t)$, namely,
	\begin{align}\label{e:localfth}
		\chi^*_t(z)=\chi_t^*(E(t))+\sqrt{C_t(z) \big( z-E(t) \big)}, \qquad \text{where $C_t(z)$ is analytic for $\big| z-E(t) \big|\leq \fc$},
	\end{align} 

	\noindent where the square root has negative imaginary part for $\Imaginary z>0$. We further have $C_t(z)=\overline{C_t(z)}$ and
	\begin{align*}
		C_t(E(t))=\frac{2}{Q_1'' \big(\chi^*_t(E(t)) \big)}, \qquad \Big|C_t \big(E(t) \big) \Big| \asymp 1.
	\end{align*}

	\noindent This gives \eqref{e:squareeq}. For $(z,t)$ bounded  away from the tangency locations and cusp points, the claim \eqref{e:mtsquare} follows from \eqref{e:localfth} and \eqref{mtft}. This establishes the first statement of the proposition. 
	\end{proof}

	\begin{proof}[Proof of the second statement of \Cref{p:ftzbehave}]
	
	We next establish the second statement  concerning the behavior of the complex slope near the tangency locations. At the tangent point $\big( E_1(t_1),t_1 \big) = \big( \fa(t_1),t_1 \big)$ of slope $1$, the complex slope $f^*_{t_1} \big( E_1(t_1) \big)=\infty$, so that $\chi^*_{t_1} \big(E_1 \big(t_1) \big)=1$. Since no tangency location is also a cusp, we also have by \eqref{qchi0} that $-Q_1''(1)\asymp 1$. Hence, \eqref{e:localfth} reduces to
	\begin{align}\label{e:fatt1}
		\frac{f_{t_1}^*(z)}{f_{t_1}^*(z)+1}=1+\sqrt{C_{t_1}(z)\big( z-\fa(t_1) \big)},\qquad \text{for $\big| z-\fa(t_1) \big| \leq \fc$}.
	\end{align}
	
	\noindent  By Taylor expansion of $Q_1$ and $Q_1'$ from \eqref{chiq2}  around $\chi_{t_1}^*(E_1(t_1))=1$, we get
	\begin{align}\begin{split}\label{e:perturb}
			E_1(t)-E_1(t_1)
			& =t-t_1 -\frac{(t-t_1)^2}{2Q_1''(1)} + \mathcal{O} \big( |t - t_1|^3 \big),\\
			\chi_t^* \big( E_1(t) \big)-\chi_{t_1}^* \big( E_1(t_1) \big)
			& =\chi_t^* \big( E_1(t) \big)-1=-\frac{t-t_1}{Q_1''(1)}+ \mathcal{O} \big( |t - t_1|^2 \big).
	\end{split}\end{align}

	We recall that $\fa(t)=\fa(t_1)+ t-t_1=E_1(t_1)+ t-t_1$. Thus \eqref{e:perturb} implies that  $E_1(t)-\fa(t)\asymp |t_1-t|^2$, and together with \eqref{e:localfth}
	\begin{align}\label{e:f/f1}
		\frac{f_t^*(z)}{f_t^*(z)+1}=\chi_t^*(z)=1+\frac{t_1-t}{Q_1''(1)}+\sqrt{C_t(z) \big(z-E_1(t) \big)} + \mathcal{O} \big( |t_1 - t|^2 \big), \qquad \text{for $\big| z-E_1(t) \big|\leq \fc$}.
	\end{align}

	\noindent  We can use \eqref{e:f/f1} to solve for $f_t^*(z)$,
	\begin{align}\label{e:ftzexp}
		f_t^*(z)= \bigg( \frac{t-t_1}{Q_1''(1)}-\sqrt{C_t(z) \big(z-E_1(t) \big)} + \mathcal{O} \big( |t_1 - t|^2 \big) \bigg)^{-1} - 1.
	\end{align}
	
	For $0\leq t-t_1\leq \fc$ and $\big| z-E_1(t) \big|\leq \fc$, the right side of \eqref{e:ftzexp} is large, and so $\arg^* \big( f^*_t(z) \big)$ is of the same order as  $\arg^* \big( f_t^* (z) + 1 \big)$. This gives 
	\begin{align}\begin{split}\label{e:mtbeh3}
			\pi+\arg^* \big( f_t^*(z) \big) \asymp \pi + \arg^* \big( f_t^* (z) + 1 \big) 
			&\asymp\pi-\arg \left( \displaystyle\frac{t-t_1}{Q_1''(1)} - \sqrt{C_t (z) \big(z - E_1 (t) \big)} +\OO \big( |t_1-t|^2 \big) \right).
	\end{split}\end{align}
	
	\noindent We notice that $-Q_1''(1)\asymp 1$ and $(t-t_1)/Q_1''(1)+\OO(|t_1-t|^2)\asymp t_1 - t$ is real. Thus, denoting $z = E_1 (t) - \kappa + \mathrm{i}\eta$, for $\kappa\geq 0$, the denominator in \eqref{e:ftzexp} behaves like $t_t - t - \sqrt{\kappa+\eta}+\ri  \eta/\sqrt{\kappa+\eta}$ and, for $\kappa \leq 0$, it behaves like $t_1 - t -\eta/\sqrt{\kappa+\eta}+ \mathrm{i} \sqrt{\kappa+\eta}$. Thus, we can further simplify \eqref{e:mtbeh3} as
	\begin{align}\begin{split}\label{e:mtbehave3}
			\eqref{e:mtbeh3}\asymp 
			\left\{
			\begin{array}{cc}
				\frac{\eta}{\sqrt{|\kappa|+\eta}(\sqrt{|\kappa|+\eta}+|t-t_1|)}, & \kappa\geq 0,\\
				\frac{\sqrt{|\kappa|+\eta}}{|t-t_1|}\wedge 1 & \kappa \leq 0.
			\end{array}
			\right.
	\end{split}\end{align}

	\noindent By plugging \eqref{e:mtbehave3} into \eqref{e:decompft*2}, and recalling $\arg \big(\log \tilde g_t^*(z) \big) = \mathcal{O} (\eta)$, we conclude that
	\begin{align}\label{e:mtbehave4}
		\pi+\Im[\tilde m_t^*(z)]
		=\pi+\Im[\log f_t^*(z)]-\Im[\log \tilde g_t^*(z)]\asymp
		\left\{
		\begin{array}{cc}
			\frac{\eta}{\sqrt{|\kappa|+\eta}(\sqrt{|\kappa|+\eta}+|t-t_1|)}, & \kappa\geq 0,\\
			\frac{\sqrt{|\kappa|+\eta}}{|t-t_1|} \wedge 1& \kappa \leq 0.
		\end{array}
		\right.
	\end{align}
	
	\noindent The estimates \eqref{e:mtbehave3} and \eqref{e:mtbehave4} together implies \eqref{e:mtbehave3copy}. 	
	
	Next we assume that $0\leq t_1-t\leq \fc$ and $\big| z-E_1(t) \big|\leq \fc$. Since the line through $\big( \mathfrak{a} (t), t \big)$ with slope $1$ is tangent to the arctic boundary, \eqref{e:slope} implies that $\chi_t^* \big(\mathfrak{a} (t) \big) = 1$. Thus, \eqref{e:localfth} implies that
	\begin{align}\label{e:newft}
		f_t^*(z)=\Big( \sqrt{C_t \big( \fa(t) \big) \big(\fa(t)-E_1(t) \big)} -\sqrt{C_t(z) \big( z-E_1(t) \big)} \Big)^{-1} - 1.
	\end{align}
	
	\noindent Multiplying both sides of \eqref{e:newft} by $z-\fa(t)$ gives
	\begin{align}\begin{split}\label{e:mtbehave1}
			-\arg^* \Big( f_t^*(z) \big(z-\fa(t) \big) \Big)
			&=
			-\arg^*\left(\fa(t) - z +\frac{\sqrt{C_t \big(\fa(t) \big) \big(\fa(t)-E_1(t) \big)}+\sqrt{C_t(z) \big(z-E_1(t) \big)}}{-C_t \big(\fa(t) \big)+\OO \big( \big|z-\fa(t) \big| +|\fa(t)-E_1(t)|\big)}\right)\\
			&\asymp
			\left\{
			\begin{array}{cc}
				\frac{\eta}{\sqrt{|\kappa|+\eta}(\sqrt{|\kappa|+\eta}+|t-t_1|)}, & \kappa\geq 0,\\
				\frac{\sqrt{|\kappa|+\eta}}{t_1-t} \wedge 1& \kappa \leq 0,
			\end{array}
			\right.
	\end{split}\end{align}
	where we have again denoted $z=E_1(t)-\kappa+\ri \eta$, where in the first line we used  $C_t (\mathfrak{a}(t)) (\mathfrak{a}(t) - E_1) - C_t (z) (z - E_1) = (\mathfrak{a}(t) - z) (C_t(\mathfrak{a}(t))+\OO(|\fa(t)-E_1(t)|) + \mathcal{O} \big( |z - \mathfrak{a}(t)|^2 \big)$, and in the second line we used the square root behavior of $\sqrt{C_t(z)(z-E_1(t))}$. By rearranging \eqref{e:mtbehave1}, we conclude
	\begin{align}\label{e:mtbehave2}
		-\arg^* f_t^*(z)\asymp
		\left\{
		\begin{array}{cc}
			\frac{\eta}{\sqrt{|\kappa|+\eta}(\sqrt{|\kappa|+\eta}+|t-t_1|)} +\arg \big(z-\fa(t) \big), & \kappa\geq 0,\\
			\frac{\sqrt{|\kappa|+\eta}}{|t-t_1|} \wedge 1& \kappa \leq 0,
		\end{array}
		\right.
	\end{align}
	where for the second line, we used that for $\kappa\leq 0$ we have
	\begin{flalign*} 
		\arg \big(z-\fa(t) \big)\asymp \displaystyle\frac{\eta}{|\kappa|+ \big| E_1 (t) - \mathfrak{a} (t) \big| + \eta} \lesssim \displaystyle\frac{\eta}{|\kappa| + (t - t_1)^2 + \eta} \lesssim \displaystyle\frac{\sqrt{|\kappa|+\eta}}{|t-t_1|},
	\end{flalign*} 

	\noindent is smaller than the leading term. Moreover, 
	\begin{align}\label{e:mtbehave5}
		-\Im \big[m_t^*(z) \big]
		=\Im \big[\log g_t^*(z)(z-\fa(t)) \big]-\Im \big[\log f_t^*(z) \big(z-\fa(t) \big) \big] \asymp
		\left\{
		\begin{array}{cc}
			\frac{\eta}{\sqrt{|\kappa|+\eta}(\sqrt{|\kappa|+\eta}+|t-t_1|)}, & \kappa\geq 0,\\
			\frac{\sqrt{|\kappa|+\eta}}{|t-t_1|}\wedge 1 & \kappa \leq 0,
		\end{array}
		\right.
	\end{align}
	where we used that $\log g_t^*(z) \big(z-\fa(t) \big)$ is real analytic in a small neighborhood of $\fa(t)$, thus its imaginary part is of size $\OO(\Im z)$, which is small compared with the leading term.
	The estimates \eqref{e:mtbehave2} and \eqref{e:mtbehave5} together imply \eqref{e:mtbehave2copy} and \eqref{e:mtbehave1copy}.
	
	The proof is very similar close to the other tangency location $\big( E_2(t_2),t_2 \big) = \big( \fb(t_2),t_2 \big)$, so in this case we only state the analogous bounds. Here, the tangency location $\big( E_2 (t_2), t_2) \big)$ is of slope $\infty$, and so the tangent line through $\big( \mathfrak{b} (t), t \big)$ to the arctic boundary has slope $0$; thus, $\chi^*_{t_2}(\fb(t_2))=0$. We also again have the estimate $\fb(t)-E_2(t)\asymp |t_2-t|^2$. For $0\leq t-t_2\leq \fc$ and $|z-E_2(t)|\leq \fc$ we have
	\begin{align}\label{e:mtbehave6}
		\pi+\Im[\tilde m_t^*(z)]
		\asymp\pi+\arg^*(f_t^*(z))\asymp
		\left\{
		\begin{array}{cc}
			\frac{\eta}{\sqrt{|\kappa|+\eta}(\sqrt{|\kappa|+\eta}+|t-t_1|)}, & \kappa\geq 0,\\
			\frac{\sqrt{|\kappa|+\eta}}{|t-t_1|}\wedge 1, & \kappa \leq 0,
		\end{array}
		\right.
	\end{align}
	where $z=E_2(t)+\kappa+\ri \eta$. 
	For $0\leq t_2-t\leq \fc$ and $|z-E(t)|\leq \fc$, we have 
	\begin{align*}
		-\Im[m_t^*(z)]
		=\Im[\log g_t^*(z)/(\fb(t)-z)]-\Im[\log f_t^*(z)/(\fb(t)-z)]\asymp
		\left\{
		\begin{array}{cc}
			\frac{\eta}{\sqrt{|\kappa|+\eta}(\sqrt{|\kappa|+\eta}+|t-t_1|)}, & \kappa\geq 0,\\
			\frac{\sqrt{|\kappa|+\eta}}{|t-t_1|}\wedge 1,  & \kappa \leq 0,
		\end{array}
		\right.
	\end{align*}
	and 
	\begin{align*}
		-\arg^*(f_t^*(z))\asymp 
		\left\{
		\begin{array}{cc}
			\frac{\eta}{\sqrt{|\kappa|+\eta}(\sqrt{|\kappa|+\eta}+|t-t_1|)} +\arg(\fb(t)-z), & \kappa\geq 0,\\
			\frac{\sqrt{|\kappa|+\eta}}{|t-t_1|}\wedge 1,  & \kappa \leq 0.
		\end{array}
		\right.
	\end{align*}
	where $z=E_2(t)+\kappa+\ri \eta$. 
	\end{proof}

	\begin{proof}[Proof of the third statement of \Cref{p:ftzbehave}]
	
	The cusp point $(E_c,t_c)$ is characterized by the equations 
	\begin{align}\label{e:cuspbehave}
		Q_1 \big(\chi^*_{t_c}(E_c) \big)=E_c-t_c \chi^*_{t_c}(E_c);\qquad Q_1' \big(\chi^*_{t_c}(E_c) \big) =-t_c; \qquad Q_1''\big( \chi^*_{t_c}(E_c) \big)=0.
	\end{align}
	We recall from \eqref{qchi0}, $|Q_1'''\big( \chi^*_{t_c}(E_c) \big)|\asymp 1$. 
	In the following we use \eqref{e:cuspbehave} to study the complex slope $f_t^*(z)$ and $\chi_t^*(z)$ in a neighborhood of $(E_c,t_c)$.  For $t\leq t_c$, the arctic curves $\big( E_1'(t),t \big), \big(E_2'(t),t \big)$ are characterized by
	\begin{align}\label{e:cuspcurve}
		Q_1 \Big(\chi^*_t \big(E(t) \big) \Big)=E(t)-t \chi^*_t \big( E(t) \big),\quad Q_1' \Big(\chi^*_t \big( E(t) \big) \Big)=-t,\qquad \text{for $E(t)$ equal to either $E_1'(t)$ or $E_2'(t)$}.
	\end{align}

	Using \eqref{e:cuspbehave}, we can solve \eqref{e:cuspcurve}. By taking the difference between \eqref{e:cuspbehave} and \eqref{e:cuspcurve}, and expanding around $(E_c,t_c)$ we get  
	\begin{align}\begin{split}
		\label{tct}
		t_c-t= Q_1' \Big( \chi_t^* \big( E(t) \big) \Big) - Q_1' \big( \chi_{t_c}^* (E_c) \big) & = \frac{Q_t'''(\chi_{t_c}^*(E_c))}{2} \Big(\chi_t^* \big( E(t) \big)-\chi_{t_c}^*(E_c) \Big)^2 \\
		& \qquad +\OO \bigg( \Big|\chi_t^* \big(E(t) \big)-\chi_{t_c}^*(E_c) \Big|^3 \bigg).
	\end{split}\end{align} 

	\noindent Moreover, recalling $c(t) = E_c + (t - t_c) \chi_{t_c}^* (E_c)$, we have that
	\begin{flalign*}
		E(t)-c(t) & = E(t) - E_c + (t_c - t) \chi_{t_c}^* (E_c) \\
		& = t \chi_t^* \big( E(t) \big) - t \chi_{t_c}^* (E_c) + Q_1 \Big (\chi_t^* \big( E(t) \big) \Big) - Q_1 \big( \chi_{t_c}^* (E_c) \big) \\
		& = t \big( \chi_t^* \big( E(t) \big) - \chi_{t_c}^* (E_c) \big) + Q_1' \big( \chi_{t_c}^* (E_c) \big) \Big( \chi_t^* \big( E(t) \big) - \chi_{t_c}^* (E_c) \Big) \\
		& \qquad + \displaystyle\frac{Q_1'' \big( \chi_{t_c}^* (E_c) \big)}{2} \Big( \chi_t^* \big( E(t) \big) - \chi_{t_c}^* (E_c) \Big)^2 + \displaystyle\frac{Q_1''' \big( \chi_{t_c}^* (E_c) \big)}{6} \Big( \chi_t^* \big( E(t) \big) - \chi_{t_c}^* (E_c) \Big)^3 \\
		& \qquad + \mathcal{O} \bigg( \Big| \chi_t \big( E(t) \big) - \chi_{t_c} (E_c) \Big|^4 \bigg) 
	\end{flalign*}

	\noindent and thus 
	\begin{flalign}
		\label{etct}
		\begin{aligned}
		E(t) - c(t) 	& = (t-t_c) \Big(\chi_t^* \big( E(t)  \big)-\chi_{t_c}^*(E_c) \Big)+\frac{Q_t''' \big(\chi_{t_c}^*(E_c) \big)}{6} \Big( \chi_t^* \big(E(t) \big)-\chi_{t_c}^*(E_c) \Big)^3 \\
		& \qquad +\OO \bigg( \Big|\chi_t^* \big( E(t) \big)-\chi_{t_c}^*(E_c) \Big|^4 \bigg). 
		\end{aligned} 
	\end{flalign}

	\noindent It follows from \eqref{tct} that $\big| \chi_t^*(E(t))-\chi_{t_c}^*(E_c) \big|\asymp |t_c-t|^{1/2}$ and, more precisely, 
	\begin{align}\begin{split}\label{e:chiEc}
			\chi^*_t(E(t)) & =\chi^*_{t_c}(E_c)\pm \sqrt{\frac{2(t_c-t)}{Q'''_1 \big(\chi^*_{t_c}(E_c) \big)}}+\OO \big( |t_c-t| \big),\\
			E(t)-c(t)& =\mp \sqrt{\frac{8(t_c-t)^3}{9Q_1''' \big(\chi^*_t(E_c) \big)}}+\OO \big( |t_c-t|^2 \big),
	\end{split}\end{align}

	\noindent where the second equality follows from inserting the first into \eqref{etct}. We conclude that $E_2'(t)-E_1'(t)\asymp (t_c-t)^{3/2}$. In a small neighborhood of $E_1'(t)$, namely, $0\leq t_c-t\leq \fc$ and $\big| z-E_1'(t) \big|\leq \fc(t_c-t)^{3/2}$, we can solve the equation $Q_1 \big(\chi^*_{t}(z) \big)=z-t\chi_t^* \big( E(t) \big)$ by expanding around $\big( E_1'(t), t \big)$. This gives
	\begin{flalign*}
			\phantom{{}={}}z-E_1'(t) & = t \chi_t^* (z) - t \chi_t^* \big( E_1' (t) \big) + Q_1 \big( \chi_t^* (z) \big) - Q_1 \Big( \chi_t \big( E_1' (t) \big) \Big) \\
			& = t \chi_t^* (z) - t \chi_t^* \big( E_1' (t) \big) + Q_1' \Big( \chi_t^* \big( E_1' (t) \big) \Big) \Big( \chi_t^* (z) - \chi_t \big( E_1' (t) \big) \Big) \\
			& \qquad + \displaystyle\frac{Q'' \big( \chi_t^* \big( E_1' (t) \big) \big)}{2} \Big( \chi_t^* (z) - \chi_t^* \big( E_1' (t) \big) \Big)^2 + \mathcal{O} \bigg( \Big| \chi_t^* (z) - \chi_t^* \big( E_1 (t) \big) \Big|^3 \bigg) \\
			& =\frac{Q_1'' \big(\chi_{t}^*  \big(E_1'(t) \big) \big)}{2} \Big( \chi_t^*(z)-\chi_{t}^* \big(E_1'(t) \big) \Big)^2+\OO \bigg( \Big| \chi_t^*(z)-\chi_{t}^* \big(E_1'(t) \big) \Big|^3 \bigg). 
		\end{flalign*} 
	
		\noindent Thus, since $Q_1'' \big( \chi_{t_c}^* (E_c) \big) = 0$, we deduce by Taylor expansion and  \eqref{e:chiEc} that
		\begin{flalign*}
			z - E_1' (t) & = \bigg( \frac{Q_1''' \big( \chi_{t_c}^* (E_c) \big)}{2} \Big( \chi_t^* \big( E_1' (t) \big) - \chi_{t_c}^* (E_c) \Big) + \mathcal{O} \big( |t - t_c| \big) \bigg) \Big(\chi_t^*(z)-\chi_{t}^* \big(E_1'(t) \big) \Big)^2 \\
			& \qquad +\OO \bigg( \Big| \chi_t^*(z)-\chi_{t}^* \big( E_1'(t) \big) \Big|^3 \bigg)\\
			&=\sqrt{\frac{Q_1''' \big(\chi_{t_c}^*(E_c) \big)}{2} (t_c - t)} \cdot \Big(\chi_t^*(z)-\chi_t^* \big( E_1'(t) \big) \Big)^2 \\
			& \qquad +\OO \bigg( |t_c-t| \Big|\chi_t^*(z)-\chi_{t}^* \big(E_1'(t) \big) \Big|^2+ \Big|\chi_t^*(z)-\chi_{t}^* \big( E_1'(t) \big) \Big|^3 \bigg).
	\end{flalign*}
	
	By rearranging, we conclude that
	\begin{align}\label{e:blowcoeff}
		\chi^*_t(z)-\chi^*_t \big(E'_1(t) \big)=
		\bigg( \displaystyle\frac{2}{Q_1''' \big( \chi_{t_c}^* (E_c) \big)} \bigg)^{1/4}\frac{\sqrt{z-E'_1(t)}}{(t_c-t)^{1/4}} +\OO\left((t_c-t)^{1/4} \big| z-E_1'(t) \big|^{1/2}+\frac{\big| z-E_1'(t) \big|}{t_c-t}\right).
	\end{align}

	\noindent It therefore follows that 
	\begin{align}\label{e:mtcubiccopy}
		-\Im \big[f_t^*(z) \big]\asymp-\Im[\chi_t^*(z)] \asymp
		\left\{
		\begin{array}{cc}
			\frac{\eta}{\sqrt{|\kappa|+\eta}(t_c-t)^{1/4}}, & \kappa\geq 0,\\
			\frac{\sqrt{|\kappa|+\eta}}{(t_c-t)^{1/4}} & \kappa \leq 0,
		\end{array}
		\right.
	\end{align}
	where we have denoted $z=E_1(t)+\kappa+\ri\eta$, with $\big| z-E_1'(t) \big|\leq \fc(t_c-t)^{3/2}$. The claims \eqref{e:cubiceq} and \eqref{e:mtcubic} then follow from \eqref{e:blowcoeff}, \eqref{e:mtcubiccopy}, and the decomposition \eqref{e:decompft*3}.
	
	In the following we study $\chi_t^*(z)$ away from the two points $E_1'(t)$ and $E_2'(t)$. 
	By a Taylor expansion, using \eqref{e:cuspbehave}
	in a neighborhood of $(E_c,t_c)$, we obtain
	\begin{align}\begin{split}\label{e:cuspexp}
			z-E_c-(t-t_c)\chi^*_{t_c}(E_c)
			& =Q_1 \big(\chi_t^*(z) \big)+t\chi_t^*(z)-E_c-(t-t_c)\chi^*_{t_c}(E_c)\\
			&=(t-t_c) \big(\chi_t^*(z)-\chi_{t_c}^*(E_c) \big)+\frac{Q_1'''(\chi_{t_c}^*(E_c))}{6} \big(\chi_t^*(z)-\chi_{t_c}^*(E_c) \big)^3 \\
			& \qquad +\OO \Big( \big|  \chi_t^*(z)-\chi_{t_c}^*(E_c) \big|^4 \Big).
	\end{split}\end{align}
	If we plug in $z=c(t)=E_c+(t-t_c)\chi^*_{t_c}(E_c)$ in \eqref{e:cuspexp}, we get that 
	$
	Q_1 \big(\chi_t^* \big(c(t) \big) \big)=c(t)-t\chi_t^* \big(c(t) \big).
	$
	By comparing with \eqref{e:q0fcopy}, we conclude that $\chi_t^* \big(c(t) \big)=\chi_{t_c}^*(E_c)$, which by \eqref{e:cuspexp} gives  
	\begin{align}\begin{split}\label{e:cuspexp2}
			&\phantom{{}={}}z-c(t)
			=(t-t_c) \Big( \chi_t^*(z)-\chi_{t}^* \big(c(t) \big) \Big) + \frac{Q_1'''(\chi_{t_c}^*(E_c))}{6} \Big(\chi_t^*(z)-\chi_{t}^* \big( c(t) \big) \Big)^3+\OO \bigg( \Big|  \chi_t^*(z) -\chi_{t_c}^* \big(c(t) \big) \Big|^4 \bigg).
	\end{split}\end{align}

		By the first statement of \eqref{e:chiEc}, for $z$ in a small neighborhood of $c(t)$ satisfying $\dist \big(z, [E'_1(t), E'_2(t)] \big)\geq \fc(t_c-t)^{3/2}$ (equivalently, $\big|z - c(t) \big| \lesssim (t_c-t)^{3/2}$), the linear term on the right side of \eqref{e:cuspexp2} is the leading term, and so $\chi^*_t(z)-\chi_{t}^* \big( c(t) \big)$ behaves like $\big( z-c(t) \big)/(t-t_c)$. 
	For $z$ in the larger neighborhood of $c(t)$ satisfying $\fc(t_c-t)^{3/2}\leq \dist \big(z, [E'_1(t), E'_2(t)] \big)\leq \fc$ (equivalently, $\mathfrak{c} (t_c - t)^{3/2} \lesssim | z - c(t) | \lesssim 1$), the cubic term on the right side of \eqref{e:cuspexp2} dominates $
	\chi_t(z)-\chi_{t}^* \big( c(t) \big)$ behaves like $\big(z-c(t) \big)^{1/3}$, and we have
	\begin{align*}
		-\Imaginary \big[\chi_t^*(z) \big]\asymp -\Im\left[ \displaystyle\frac{f^*_t(z)}{f^*_t(z)+1} \right]\asymp -\Imaginary \big[f_t^*(z) \big]\asymp \big|z-c(t) \big|^{1/3}.
	\end{align*}
	This finishes the proof of \eqref{e:localft3}.
\end{proof}

\section{Solving for the Correction Term}\label{s:ansatz}

In this section we sketch the proof of \Cref{p:formula}, which closely follows \cite[Section 9]{NRWLT}.
For any time $s\in [0,\ft)\cap \bZ_n$, and any particle configuration $(x_1,x_2,\cdots,x_{m})\subseteq [\fa(s),\fb(s)]$, we make the following ansatz
\begin{align*}
	N_{s}(x_1, x_2, \cdots, x_m)=\frac{V(\bmx)e^{n^2Y_{s}(\bmx)}E_{s}(x_1, x_2, \cdots, x_m)}{\prod_{j=1}^m \Gamma_n(x_j-\fa(s)))\Gamma_n(\fb(s)-x_j-1/n)},
\end{align*}
where $\Gamma_n(k/n)=(k/n)((k-1)/n)\cdots (1/n)=k!/n^k$ and 
$\del_x \beta(x)=\rho(x;\bmx)$. Then from our construction \eqref{e:deftN}, we have
\begin{align*}
	E_\ft(x_1,x_2,\cdots, x_m)=1, \quad (x_1,x_2,\cdots, x_m)\in [\fa(\ft), \fb(\ft)].
\end{align*}

Then we can rewrite the recursion \eqref{e:recurss} as a recursion for $E_{s}(x_1,x_2,\cdots,x_m)$: for any time $s\in[0,\ft)\cap \bZ_{n}$, and $\bmx=(x_1,x_2,\cdots,x_m)\in [\fa(s),\fb(s)]$
\begin{align}\label{e:cEeq}
	E_{s}(\bmx)
	=\sum_{\bme\in\{0,1\}^m}a_s(\bme;\bmx) E_{(s+1/n)}(\bmx+\bme/n),
\end{align}
where
\begin{align}\label{e:defat}
	a_s(\bme;\bmx)=\frac{V(\bmx+\bme/n)}{V(\bmx)}\prod_{j=1}^m
	(x_j-\fa(s))^{1-e_j}
	(\fb(s)-x_j-1/n)^{e_j}e^{n^2(Y_{(s+1/n)}(\bmx+\bme/n)-Y_{s}(\bmx))}.
\end{align}

\noindent We denote the associated partition function by 
\begin{align*}
	\tilde Z_s(\bmx)\deq \sum_{\bme\in\{0,1\}^m}a_s(\bme;\bmx).
\end{align*}

The equation \eqref{e:cEeq} can be solved using the Feynman--Kac formula.  For any time $s\in[0,\ft)\cap \bZ_{n}$, to solve for $E_{s}(\bmx)$, we construct a Markov process $\{\bmx_t\}_{t\in[s, \ft]\cap \bZ_n}$ starting from the configuration $\bmx_s=\bmx$ with a time dependent generator given by 
\begin{align}\label{e:gener}
	\cL_{t} f(\bmx)=\sum_{\bme \in \{0,1\}^{m}}\frac{a_t(\bme;\bmx)}{\tilde Z_t(\bmx)}(f(\bmx+\bme/n)-f(\bmx)),
	\quad \bmx=(x_1, x_2, \cdots, x_{m})\in [\fa(t),\fb(t)].
\end{align}

We remark that the boundary constraints of the double-sided trapezoid domain $\fD'$ (recall from \eqref{e:defD'}) is encoded in the jump probabilities $a_t(\bme;\bmx)$. In particular, observe under this jump process that, if a particle $x_j$ is at $\fa(t)$, then it has to jump to the right, for otherwise $e_j=0$, implying $a_t(\bme;\bmx)=0$. Similarly, if a particle $x_j$ is at $\fb(t)-1/n$, then it has to stay, for otherwise $e_j=1$, implying $a_t(\bme;\bmx)=0$.
Thus, for non-intersecting Bernoulli random walks in the form \eqref{e:gener}, particles (the $x_j$) are constrained to remain inside $\fD'$.

Using the Markov process $\{\bmx_t\}_{t\in[s, \ft]\cap \bZ_n}$ defined above, $E_{s}$ can be solved using the Feynman-Kac formula.
\begin{prop}[{\cite[Proposition 9.3]{NRWLT}}]
	
	\label{p:FK}
	For any $s\in[0,\ft]\cap \bZ_n$ and $\bmx \subset [\fa(s), \fb(s)]$, we have
	\begin{align}\begin{split}\label{e:FK}
			E_{s}(\bmx)
			&=\bE\left[\left.\prod_{t\in[s,\ft]\cap \bZ_n}\tilde Z_t(\bmx_t)
			\right| \bmx_s=\bmx\right].
	\end{split}\end{align}
\end{prop}

To use \Cref{p:FK} to solve for the correction terms $E_{s}$, we need to understand the Markov process $\{\bmx_t\}_{t\in [s,\ft]\cap \bZ_n}$ with generator $\cL_t$ as given in \eqref{e:gener}. For the weights $a_s(\bme,\bmx)$ as in \eqref{e:defat}, all the terms are explicit, except for $Y_{(s+1/n)}(\bmx+\bme/n)-Y_s(\bmx)$. We need to derive a $1/n$-expansion of $Y_{(s+1/n)}(\bmx+\bme/n)-Y_s(\bmx)$. 
We recall the variational problem from \eqref{e:varWt} 
\begin{align}\label{e:varWtcopy}
	W_ s ( \beta )=\sup_{ H\in \Adm({\fD^s};\beta)}\int_ s ^{\ft'}\int_\bR\sigma(\nabla  H)\rd x \rd t.
\end{align}
with minimizer  $H_t(x;\beta,s)$ and the complex slope $ f_t( x ;\beta, s)$. The complex slope can be extended to a neighborhood of $[\fa(s),\fb(s)]$. We also recall the decomposition $f(z;\beta,s)=e^{\tilde m(z;\beta,s)}\tilde g(z;\beta,s)$ from \eqref{e:gszmut}.

\begin{prop}[{\cite[Proposition 6.3]{NRWLT}}]
	
	\label{p:YYdif}
	There exists a universal constant $\fc>0$. For any boundary profile $\beta=\beta(x;\bmx)$ sufficiently close to $H_s^*$, i.e., $d(\del_x\beta (x;\bmx), \del_x H_s^*(x))\leq \fc$, the following holds
	\begin{align}\label{e:YYdif}
		n^2(Y_{(s+1/n)}(\bmx+\bme/n)-Y_{s}(\bmx))=\sum_{i = 1}^m e_i \chi(x_i;\beta,s)+\frac{1}{n^2}\sum_{1 \leq i,j \leq m} e_i e_j \kappa(x_i,x_j;\beta,s) +\zeta(\bmx,s)+\OO(1/n),
	\end{align}
	where the error $\OO(1/n)$ is uniform in $\beta$, and 
	\begin{align}\begin{split}\label{e:defka}
			\chi(x_i;\beta,s)&=
			\log \tilde g(x_i;\beta,s)+\frac{1}{n} \left(\del_s(\log  \tilde g(x_i;\beta,s))+\frac{1}{2}\del_z \log  \tilde g(z;\beta,s)\right)\\
			\kappa(x_i,x_j;\beta,s)&=-\frac{1}{2}K(x_i,x_j;\beta,s),
	\end{split} \end{align}
	the functions $ \tilde g_s(x;\beta,s), K(x,y;\beta,s)$ are as in \eqref{e:gszmut}, \eqref{e:defK}, and
	\begin{align}\label{e:defzeta}
		\zeta(\bmx,s)=(n\del_s Y_s(\bmx)+\frac{1}{2}\del_s^2 Y_s(\bmx)).
	\end{align}
\end{prop}

Using \Cref{p:YYdif}, we can rewrite $a_s(\bme;\bmx)$ in the following form
\begin{align}\label{e:newat0}
	a_s(\bme;\bmx)&=\frac{V(\bmx+\bme/n)}{V(\bmx)}\prod_{j = 1}^m \phi^+(x_j;\beta,s)^{e_i}\phi^-(x_j;\beta,s)^{1-e_i}
	e^{\frac{1}{n^2}\sum_{i, j} e_i e_j \kappa(x_i,x_j;\beta,s) +\zeta(\bmx,s)+\OO(1/n)},
\end{align}
where $\beta=\beta(x;\bmx)$, and 
\begin{align*}
	\phi^+(x;\beta,s)= (\fb(s)-x-1/n)e^{ \chi(x;\beta,s)},\quad 
	\phi^-(x;\beta,s)= x-\fa(s).
\end{align*}

Instead of solving \eqref{e:FK} directly, we make another ansatz. 
The formula for the error term $E_s(\bmx)$ from \eqref{e:FK} is expressed in terms of the non-intersecting Bernoulli random walk $\{\bmx_t\}_{t\in[s,\ft)\cap \bZ_n}$ starting from $\bmx_s=\bmx$. Let $\beta=\beta(x;\bmx)$. We make the following ansatz that the leading order term of $E_s(\bmx)$ is given by the expression \eqref{e:FK} with $\bmx_t$ replaced by $H_t(x;\beta,s)$:
\begin{align}\begin{split}\label{e:ansatz2}
		&E_{s}(\bmx)=e^{F_s(\bmx)}E_s^{(1)}(\bmx),\quad F_s(\bmx)= \log\left(\prod_{t\in [s,\ft)\cap \bZ_n}\tilde Z_t(H_t(x;\beta,s))\right).
\end{split}\end{align}

The following proposition states that \eqref{e:ansatz2} is a good approximation.

\begin{prop}[{\cite[Proposition 9.3]{NRWLT}}]
	
	\label{p:Et1bb}
	
	There exists a constant $\mathfrak{c} > 0$, for any time $s\in [0,T)\cap \bZ_n$, and any particle configuration $(x_1,x_2,\cdots,x_{m})\in [\fa(s),\fb(s)]$ such that $\beta=\beta(x;\bmx)$ is sufficiently close to $H_s^*$, i.e., $d(\del_x\beta (x;\bmx), \del_x H_s^*(x))\leq \fc$, $E_s^{(1)}(\bmx)$ as in \eqref{e:ansatz2} satisfies
	\begin{align}\label{e:Et1bb}
		\frac{E_t^{(1)}(\bmx+\bme/n)}{E_t^{(1)}(\bmx)}=e^{\OO(1/n)}.
	\end{align}
\end{prop}

Given these results,  \Cref{p:formula} follows from plugging \eqref{e:newat0}, Propositions \ref{p:YYdif} and \ref{p:Et1bb} into \eqref{e:wNBB2}.
\begin{proof}[Proof of \Cref{p:formula}]
	The transition probability is given by
	\begin{align}\label{e:wNBB2}
		\bP(\bmx_{t+1/n}=\bmx+\bme/n|\bmx_t=\bmx)=\frac{N_{t+1/n}(\bmx+\bme/n)}{N_t(\bmx)}
		=a_t(\bme;\bmx)\frac{E_{t+1/n}(\bmx+\bme/n)}{E_{t}(\bmx)}.
	\end{align}
	Let $\beta=\beta(x;\bmx)$. 
	The quantity $F_s(\bmx)$ is a sum of $\log \tilde Z_t$. The formula of $\log \tilde Z_t$ for $t\in[s,\ft)\cap \bZ_n$ is explicitly given in \cite[Proposition 9.1]{NRWLT} up to error $\OO(1/n^2)$. It also gives the functional derivative $\del_\beta F_s(\beta)(z)$ of $F_s$, with respect to $\beta$ up to error $\OO(1/n)$,
	\begin{align}\label{e:fuct}
		F_s(\bmx+\bme/n)-F_s(\bmx)=(1/n)\sum_{i=1}^m e_i\del_\beta  F_s(\beta)(x_i)+\OO(1/n).
	\end{align} 
	Thanks to \eqref{e:newat0}, Propositions \ref{p:YYdif} and \ref{p:Et1bb}, we can rewrite \eqref{e:wNBB2} as
	\begin{align*}
		\frac{1}{Z_t(\bmx)}\frac{V(\bmx+\bme/n)}{V(\bmx)}\prod_{i=1}^m\phi^+(x_i;\beta,t)^{e_i}\phi^-(x_i;\beta,t)^{1-e_i}e^{\sum_{i,j} \frac{e_ie_j}{n^2}\kappa(x_i,x_j;\beta,t)+\OO(1/n)},
	\end{align*}
	where $Z_t(\bmx)=\tilde Z_t(\bmx)/e^{\zeta(\bmx,t)}$,
	\begin{align*}\begin{split}
			&\phi^+(z;\beta,t)=\phi^+(z;\beta,t)e^{\frac{1}{n}\del_\beta F_t(\beta)(z)}
			=\tilde g(z;\beta,t) (\fb(t)-1/n-z)e^{\frac{1}{n}\psi(z;\beta,t)},\\
			&\psi(z;\beta,t)=\frac{1}{2}\del_z \log  \tilde g(z;\beta,t)+\del_t(\log \tilde g(z;\beta,t))+ \del_\beta F_t(\beta)(z)+\OO(1/n),
	\end{split}\end{align*}
	$\del_\beta F_t(\beta)(z)$ is the functional derivative of $F_t$ with respect $\beta$ as in \eqref{e:fuct}.
	This finishes the proof of \Cref{p:formula}.
\end{proof}

\bibliography{TilingsConcentration.bib}
\bibliographystyle{alpha}

\end{document}